\documentclass[11pt,a4paper,colorlinks]{article}

\usepackage[a4paper, left=2.5cm, right=2.5cm, bottom=2.5cm]{geometry}

\usepackage{amsmath}   

\usepackage[page]{appendix}

\usepackage{mathtools} 
\usepackage{amsfonts}
\usepackage{amstext}
\usepackage{amssymb}

\usepackage[ngerman,english]{babel}

\usepackage{fancyhdr}

\usepackage[hyperref,amsmath,thmmarks]{ntheorem}

\usepackage{mathbbol}

\usepackage{tikz}
\usetikzlibrary{matrix,arrows,calc,snakes,patterns,decorations.markings}

\usepackage{stmaryrd}
\usepackage{mathrsfs}
\usepackage{stackrel}

\usepackage{hyperref}

\usepackage[utf8]{inputenc}

\newcommand{\D}{\ensuremath{\textbf{D}}}

\renewcommand{\H}{\ensuremath{\text{H}}}
\newcommand{\HoH}{\ensuremath{\textbf{HH}}}
\newcommand{\Ho}{\ensuremath{\textbf{Ho}}}
\newcommand{\Acyc}{\ensuremath{\textbf{Acyc}}}

\newcommand{\incl}{\ensuremath{\text{incl}}}

\newcommand{\LR}     {\ensuremath{\text{LR}}}
\newcommand{\SBM}[1] {\ensuremath{\text{B}}_{#1}}
\newcommand{\KL} [1] {\ensuremath{\ul{\text{H}}_{#1}}}

\newcommand  {\CC}{\ensuremath{{\mathbb C}}}
\newcommand  {\FF}{\ensuremath{{\mathbb F}}}
\newcommand  {\GG}{\ensuremath{{\mathbb G}}}
\newcommand  {\LL}{\ensuremath{{\mathbb L}}}
\newcommand  {\NN}{\ensuremath{{\mathbb N}}}
\newcommand  {\PP}{\ensuremath{{\mathbb P}}}

\newcommand  {\VV}{\ensuremath{{\mathbb V}}}
\newcommand  {\WW}{\ensuremath{{\mathbb W}}}
\newcommand  {\XX}{\ensuremath{{\mathbb X}}}
\newcommand  {\YY}{\ensuremath{{\mathbb Y}}}
\newcommand  {\ZZ}{\ensuremath{{\mathbb Z}}}

\newcommand{\calA}{\ensuremath{{\cal A}}}
\newcommand{\calB}{\ensuremath{{\cal B}}}
\newcommand{\calC}{\ensuremath{{\cal C}}}
\newcommand{\calD}{\ensuremath{{\cal D}}}
\newcommand{\calE}{\ensuremath{{\cal E}}}
\newcommand{\calI}{\ensuremath{{\cal I}}}
\newcommand{\calJ}{\ensuremath{{\cal J}}}
\newcommand{\calO}{\ensuremath{{\cal O}}}
\newcommand{\calP}{\ensuremath{{\cal P}}}
\newcommand{\calU}{\ensuremath{{\cal U}}}

\newcommand  {\frl}{\ensuremath{{\mathfrak l}}}
\newcommand  {\frm}{\ensuremath{{\mathfrak m}}}
\newcommand  {\frp}{\ensuremath{{\mathfrak p}}}
\renewcommand{\frq}{\ensuremath{{\mathfrak q}}}
\newcommand  {\frs}{\ensuremath{{\mathfrak s}}}
\newcommand  {\frS}{\ensuremath{\mathfrak S}}

\newcommand  {\scE}{\ensuremath{{\mathscr E}}}
\newcommand  {\scS}{\ensuremath{{\mathscr S}}}

\newcommand{\C}        {\ensuremath{\text{C}}}

\newcommand{\PoincPoly}{\ensuremath{\calP}}

\newcommand{\differential}      {\ensuremath{\text{d}}}

\newcommand  {\us}    {\ensuremath{^{\star}}}
\newcommand  {\uss}   {\ensuremath{^{\star\star}}}

\newcommand  {\pr}    {\ensuremath{\text{pr}}}

\newcommand  {\p}     {\ensuremath{^{\prime}}}
\newcommand  {\pp}    {\ensuremath{^{\prime\prime}}}

\newcommand  {\uc}    {\ensuremath{^{\circ}}}

\newcommand  {\ua}    {\ensuremath{^{\ast}}}
\newcommand  {\la}    {\ensuremath{_{\ast}}}

\newcommand  {\ud}    {\ensuremath{^{\bullet}}}

\newcommand  {\ld}    {\ensuremath{_{\centerdot}}}

\newcommand{\wt}   [1]{\ensuremath{\widetilde{#1}}}

\newcommand  {\ol} [1]{\ensuremath{\overline{#1}}}
\newcommand  {\ul} [1]{\ensuremath{\underline{#1}}}

\newcommand  {\ui}    {\ensuremath{^{-1}}}

\newcommand  {\op}    {\ensuremath{^{\text{op}}}}

\newcommand  {\id}    {\ensuremath{\text{id}}}

\newcommand{\bfD}     {\ensuremath{\textbf{D}}}
\newcommand{\bfH}     {\ensuremath{\textbf{H}}}
\newcommand{\bfK}     {\ensuremath{\textbf{K}}}
\newcommand{\bfL}     {\ensuremath{\textbf{L}}}
\newcommand{\bfM}     {\ensuremath{\textbf{M}}}
\newcommand{\bfR}     {\ensuremath{\textbf{R}}}

\newcommand{\bfx}     {\ensuremath{\textbf{x}}}
\newcommand{\bfy}     {\ensuremath{\textbf{y}}}

\renewcommand{\sl}     {\ensuremath{\frs\frl}}

\renewcommand{\dim}      {\ensuremath{\text{dim}}}     
\newcommand  {\gdim}     {\ensuremath{\text{g.dim}}}   
\newcommand  {\projdim}  {\ensuremath{\text{proj.dim}}}
\newcommand  {\injdim}   {\ensuremath{\text{inj.dim}}} 
  
\newcommand  {\gldim}    {\ensuremath{\text{gl.dim}}}  
\newcommand  {\perf}     {\ensuremath{\text{Perf}}}
\newcommand  {\Perf}     {\ensuremath{\text{Perf}}}
     
\newcommand  {\ass}      {\ensuremath{\text{Ass}}}     

\renewcommand{\mod}   {\ensuremath{\text{-mod}}}  
\newcommand{\cell}    {\ensuremath{\text{-cell}}}

\newcommand{\Mod}     {\ensuremath{\text{-Mod}}}  
\newcommand{\bimod}   {\ensuremath{\text{-Mod-}}}

\newcommand {\Heck}   {\ensuremath{\textbf{H}}}

\newcommand  {\stab}[1]{\ensuremath{^{\left\{#1\right\}}}}
\newcommand  {\Id}     {\ensuremath{\text{Id}}}
\newcommand  {\fpd}    {\ensuremath{\textbf{fpd}}}    
\newcommand  {\ufpd}   {\ensuremath{\underline{\textbf{fpd}}}}
\newcommand  {\nfpd}   {\ensuremath{\text{fpd}}}    
\newcommand  {\nmcm}   {\ensuremath{\text{MCM}}}    
\newcommand  {\fold}   {\ensuremath{\text{fold}}}    

\newcommand  {\mor}[1]{\ensuremath{\stackrel{#1}{\longrightarrow}}}

\newcommand  {\rk}    {\ensuremath{\text{rk}}}
\newcommand  {\spec}  {\ensuremath{\text{Spec}}}
\renewcommand{\hom}   {\ensuremath{\text{Hom}}}      % Hom
\newcommand  {\End}   {\ensuremath{\text{End}}}

\newcommand  {\KR}  {\ensuremath{\textbf{KR}}}
\newcommand  {\B}   {\ensuremath{\textbf{B}}}
\newcommand  {\MCM} {\ensuremath{\textbf{MCM}}}
\newcommand  {\uMCM}{\ensuremath{\underline{\MCM}}}

\newcommand  {\Sym} {\ensuremath{\scS}}

\newcommand  {\HMF}    {\ensuremath{\textbf{HMF}}}
\newcommand  {\grsplit}{\ensuremath{\textbf{K}_0^{\oplus}}}
\newcommand  {\length} {\ensuremath{\text{len}}}
\newcommand  {\Ass}    {\ensuremath{\text{Ass}}}
\newcommand  {\Ann}    {\ensuremath{\text{Ann}}}
\newcommand  {\Ch}     {\ensuremath{\textbf{Ch}}}
\newcommand  {\can}    {\ensuremath{\text{can}}}

\newcommand  {\Cof}    {\ensuremath{\text{Cof}}}
\newcommand  {\ev}     {\ensuremath{\text{ev}}}
\newcommand  {\Span}   {\ensuremath{\text{span}}}

\newcommand  {\fr}   {\ensuremath{\text{fr}}}
\newcommand  {\fg}   {\ensuremath{\text{fg}}}

\newcommand  {\sqm}   {{\scriptsize\text{?}}}

\newcommand  {\MF}   {\ensuremath{\textbf{MF}}}
\newcommand  {\ldg}   {\ensuremath{_\text{dg}}}

\newcommand  {\mult}   {\ensuremath{\text{mult}}}
\newcommand  {\Iso}    {\ensuremath{\text{Iso}}}
\newcommand  {\Indec}  {\ensuremath{\text{Indec}}}

\newcommand  {\cone}   {\ensuremath{\text{Cone}}}    

\newcommand  {\depth} {\ensuremath{\text{depth}}}

\newcommand  {\rg}    {\ensuremath{\text{rank}}}    
\newcommand  {\rank}    {\ensuremath{\text{rank}}}  
\newcommand  {\coker} {\ensuremath{\text{coker}}}   
\newcommand  {\image} {\ensuremath{\text{im}}}      
\renewcommand{\ker}   {\ensuremath{\text{ker}}}     

\newcommand  {\ueven}{\ensuremath{^\text{even}}}
\newcommand  {\uodd} {\ensuremath{^\text{odd}}}

\newcommand  {\pro}  {\ensuremath{\text{Pro}}}     

\newcommand  {\supp}  {\ensuremath{\text{Supp}}}

\newcommand  {\tor}   {\ensuremath{\text{Tor}}}      
\newcommand  {\ext}   {\ensuremath{\text{Ext}}}      

\newcommand  {\Fl}    {\ensuremath{\text{Fl}}}

\renewcommand{\projlim}[1]{\ensuremath{\varinjlim\limits_{#1}}}

\newcommand  {\veps}  {\ensuremath{\varepsilon}}

\theoremstyle      {plain}
\theoremheaderfont {\bfseries}
\theorembodyfont   {\normalfont}
\theoremseparator  {}
\theoremnumbering  {arabic}

\newtheorem{prop}       {Proposition}[subsection]
\newtheorem{lem}        [prop]{Lemma}      
\newtheorem{ex}         [prop]{Example}    
\newtheorem{definition} [prop]{Definition} 
\newtheorem{cor}        [prop]{Corollary}  
  
\newtheorem{theorem}    [prop]{Theorem}    
\newtheorem{fact}       [prop]{Fact}       

\theoremsymbol {$\lozenge$}  

\newtheorem{rem}   [prop]{Remark}    

\theoremstyle      {nonumberplain}
\theoremheaderfont {\itshape}
\theoremseparator  {.}
\theorembodyfont   {\normalfont}
\theoremsymbol     {$\Box$}
\newtheorem{proof} {Proof}

\numberwithin{equation}{subsection}
\renewcommand{\theequation}{\thesubsection-\arabic{equation}}

\tikzstyle{fat}     = [circle, draw, fill=black!0, inner sep=0pt, minimum width=4pt]
\tikzstyle{big}     = [circle, fill=black, inner sep = 0pt, minimum width=0pt]
\tikzstyle{txt}     = [pos=0.75, rectangle, scale=0.6, fill=white, inner sep = 1pt]
\tikzstyle{txtp}    = [pos=0.25, rectangle, scale=0.6, fill=white, inner sep = 1pt]
\tikzstyle{txtq}    = [pos=0.25, rectangle, scale=0.6, fill=white, inner sep = 3pt]
\tikzstyle{oredged} = [-,postaction={decorate}, decoration={markings,mark=at position .4 with {\arrow{>}}}]
\tikzstyle{oredgeu} = [-,postaction={decorate}, decoration={markings,mark=at position .6 with {\arrow{>}}}]
\tikzstyle{oredgefe} = [-,postaction={decorate}, decoration={markings,mark=at position .7 with {\arrow{>}}}]
\tikzstyle{oredgefa} = [-,postaction={decorate}, decoration={markings,mark=at position .3 with {\arrow{>}}}]
\tikzstyle{oredgem} = [-,postaction={decorate}, decoration={markings,mark=at position .5 with {\arrow{>}}}]
\tikzstyle{oredgee} = [-,postaction={decorate}, decoration={markings,mark=at position 1.0 with {\arrow{>}}}]

\begin{document}

 \title{Khovanov-Rozansky homology via Cohen-Macaulay approximations and Soergel bimodules}

 \author{Hanno Becker}
 \date{}
 \maketitle
 \abstract{
We describe a simplification in the construction of Khovanov-Rozansky's categorification of quantum
$\sl(n)$ link homology using the theory of maximal Cohen-Macaulay modules over hypersurface singularities and the
combinatorics of Soergel bimodules. More precisely, we show that the matrix factorizations associated to basic
MOY-graphs equal Cohen-Macaulay approximations of certain Soergel bimodules, and prove that
taking Cohen-Macaulay approximation commutes with tensor products as long as the MOY-graph under consideration does
not possess oriented cycles. It follows that the matrix factorization associated to a MOY-braid equals the
Cohen-Macaulay approximation of the Soergel bimodule corresponding to the endofunctor on BGG-category $\calO$ associated
to the braid by Mazorchuk and Stroppel. This reduces certain computations in the category of matrix factorizations to
 known combinatorics of the Hecke-algebra. Finally, we describe braid closure as some kind of Hochschild
 cohomology and prove that the indecomposable Soergel bimodules corresponding to Young tableaux with more than 
 $n$ rows have trivial Cohen-Macaulay approximation, in analogy to the fact that the corresponding projective functors
 on category $\calO$ vanish on restriction to parabolics with at most $n$ parts.
}

\tableofcontents
\newpage

%%%%%%%%%%%%%%%%% EINLEITUNG %%%%%%%%%%%%%%%

\renewcommand{\theequation}{\arabic{equation}}

\section*{Introduction}
\selectlanguage{english}
\addcontentsline{toc}{section}{Introduction}

\renewcommand\theprop{\arabic{prop}}

In \cite{KR1}, Khovanov and Rozansky constructed a categorification of quantum $\sl(n)$ polynomial knot invariants for
all $n>0$. More precisely, their construction categorifies the Reshetikhin-Turaev link invariant (see \cite{RTInvariant} and
\cite[Part III]{KasselQuantumGroups}) for links whose components are labeled by the vector representation of
$\calU_q(\sl(n))$. This construction was recently extended by Wu and Yonezawa in their articles 
\cite{WuLinkHomology} and \cite{Yonezawa}, where they provided a categorification for links with components labeled
with arbitrary exterior powers of the vector representation. For now it is not important how these knot invariants are
constructed in detail; we only care about what objects they associate to a tangle, namely matrix factorizations. Next we
recall what this means.    

Let $S$ be a commutative ring and $w\in S$ be arbitrary. A matrix factorization (originally due to Eisenbud, see
\cite{EisenbudMatrixFactorizations}) of type $(S,w)$ is a pair of maps
$M\mor{\alpha} N$, $N\mor{\beta} M$ between free $S$-modules $M$, $N$ such that
$\alpha\beta=\beta\alpha=w\cdot\id$. Therefore, one might think of a matrix factorization of type $(S,w)$ as some
$2$-periodic complex of free $S$-modules where the usual condition $\delta^2=0$ for the differential is weakened to
$\delta^2=w\cdot\id$. In this description, morphisms of matrix factorizations are morphisms of $2$-periodic complexes,
and they can be written as the $0$-cocycles in some $2$-periodic complex (in the usual sense) of morphisms, defined as
in the case of ordinary complexes over some additive category, yielding a differential-graded category
$\MF\ldg(S,w)$. It turns out that this dg-category is pretriangulated, 
i.e. there is a reasonable notion of shift and cones, so that we have a canonical triangulated structure on its homotopy
category. It is this homotopy category of matrix factorizations $\HMF(S,w)$ where Khovanov-Rozansky's link homology
theory takes its values.  

In case $S$ is a regular local ring and $w\in\frm\setminus\{0\}$, it is known that the homotopy category of matrix
factorizations is triangle equivalent to what is called the singularity category of the ring $S/(w)$. The singularity
category can be defined for any local Noetherian ring $R$ and has several equivalent definitions (see \cite{OrlovGrad}),
the usual one being the Verdier quotient $\bfD^b(R\mod)/\perf$ of the bounded derived category of finitely generated
$R$-modules by the subcategory of perfect complexes, i.e. those which are quasi-isomorphic to bounded complexes of
projectives. In view of Serre's Theorem stating that a Noetherian local ring is regular if and only if every module has
a finite projective resolution, this is a quite intuitive measure for the failure of $R$ to be regular. However, for us
the description as the stable category $\uMCM(R)$ of the category $\MCM(R)$ of maximal Cohen-Macaulay modules over
$S/(w)$ (originally due to \cite{Buchweitz} and \cite{Happel}) is of interest, which we now recall. A finitely generated
module over a local ring $R$ is called Cohen-Macaulay if its depth (i.e. the maximal length of a regular sequence in
$M$) equals its dimension (the dimension of the topological space $\supp(M)\subset\spec(R)$); it is called maximal
Cohen-Macaulay if $\depth(M)=\dim(M)=\dim(R)$. A ring is called Cohen-Macaulay if it is Cohen-Macaulay considered as a
module over itself. Denote the category of maximal Cohen-Macaulay modules over $R$ by $\MCM(R)$. Being maximal
Cohen-Macaulay is a stable property in the following sense: given a finitely-generated module $M$ 
over a Cohen-Macaulay ring $R$, its depth increases as one takes syzygies of $M$, as long as the depth does not get
bigger than $\depth(R)=\dim(R)$. In particular, the $k$-th syzygy of $M$ is maximal Cohen-Macaulay for
$k\geq\depth(R)-\depth(M)$, and this is why one can think of maximal Cohen-Macaulayness as a stable property. Now, in
case where $R$ is not only Cohen-Macaulay but even Gorenstein, a small miracle occurs: once a module belongs, after
sufficiently many projective resolving steps to the left, to the ``stable'' range of maximal Cohen-Macaulay modules, it
can even be projectively resolved to the \textit{right} with all syzygies again maximal Cohen-Macaulay. In precise terms,
this is known as the fact that the category of maximal Cohen-Macaulay modules is a Frobenius category, i.e. an exact
category with enough projectives and injectives where in addition the classes of projective and injective objects
coincide. Annihilating all morphisms which factor through a projective object in such a
category yields a canonically triangulated category (see \cite[Section 3.3]{KellerDGKats} and references therein), and
so in particular we get the stable category of maximal Cohen-Macaulay modules endowed with a canonical triangulated
structure. If we work over some hypersurface $R = S/w$, i.e. $S$ is  
regular and $w\in\frm\setminus\{0\}$, then this stable category of maximal Cohen-Macaulay modules is both triangle
equivalent to the  singularity category of $R$ and the homotopy category of matrix factorizations. 
\begin{equation*}\begin{tikzpicture}[description/.style={fill=white,inner sep=2pt}]
    \matrix (m) [matrix of math nodes, row sep=3em,
                 column sep=2.5em, text height=1.5ex, text depth=0.25ex,
                 inner sep=0pt, nodes={inner xsep=0.3333em, inner ysep=0.3333em}]
    {
       & \uMCM(R) & \\ \D^b(R\mod)/\perf && \HMF(S,w)\\
    };
    \draw[->] (barycentric cs:m-1-2=0.8,m-2-1=0.2) -- node[above]{$\cong$} (barycentric cs:m-1-2=0.2,m-2-1=0.8);
    \draw[->] (barycentric cs:m-1-2=0.8,m-2-3=0.2) -- node[above]{$\cong$} (barycentric cs:m-1-2=0.2,m-2-3=0.8);
\end{tikzpicture}\end{equation*}
To sum up, we have the following situation: the Khovanov-Rozansky link invariant takes values in homotopy categories of
matrix factorizations, and those are equivalent to stable categories of maximal Cohen-Macaulay modules over the
corresponding quotient singularities. This naturally leads to the following question, which we want to study in this
article: 
\vskip 3mm
\noindent\textbf{Main Question:} \textit{How can we construct (and simplify?) KR-homology using the stable category of
  maximal  Cohen-Macaulay modules instead of the homotopy category of matrix factorizations?} 
\vskip 3mm

To be able to describe our attempt to answer this question, we first sketch Khovanov and Rozansky's original
construction. 

Given a link $L$ we first replace any crossing of $L$ either by the uncrossing or the wide edge, as
depicted in Figure \ref{fig:resolve}. 
\begin{figure}[h]\begin{center}
\begin{tikzpicture}[scale=0.75]
\draw[->] (+1,-1) -- (-1,+1);
\draw[draw=white, line width=4mm] (-1,-1) -- (1,1);
\draw[->] (-1,-1) -- (1,1);
\begin{scope}[xshift=6cm,yshift=2cm]
\draw[oredgem] (-1,-1) -- (+0,-0.5);
\draw[oredgem] (+1,-1) -- (+0,-0.5);
\draw (+0,-0.5) -- node[right,scale=0.75]{$2$} (+0,0.5);
\draw[oredgem] (+0,0.5) -- (-1,1);
\draw[oredgem] (+0,0.5) -- (+1,1);
\end{scope}
\begin{scope}[xshift=6cm,yshift=-2cm]
\draw[->] (-0.75,-1) -- (-0.75,+1);
\draw[->] (+0.75,-1) -- (+0.75,+1);
\end{scope}
\draw[->,decorate,decoration=snake] (2cm,1cm)  -- (4cm,2cm);
\draw[->,decorate,decoration=snake] (2cm,-1cm) -- (4cm,-2cm);
\end{tikzpicture}
\end{center}
\caption{\small{Resolving a crossing}}
\label{fig:resolve}
\end{figure}
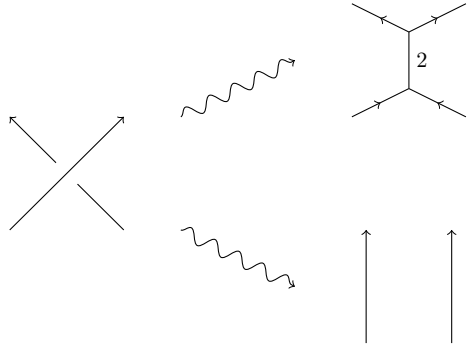
A graph with is composed of subgraphs as in Figure \ref{fig:resolve} is called a MOY-graph (see \cite{MOY}). The
MOY-graph obtained from $L$ by some choice of replacement for each crossing is called a smoothing of $L$. The main part  
in the construction of the Khovanov-Rozansky link homology of $L$ is to associate to each smoothing of all $n$ crossings
of $L$ a matrix factorization. Having done this, all these $2^n$ matrix factorizations are finally patched together to a
complex of matrix factorizations, whose homotopy type is the value of $L$ under Khovanov-Rozansky homology. We 
will not consider this patching construction, for which we refer the reader to the original article \cite{KR1} for
details. Instead, let us look at the steps through which Khovanov-Rozansky construct the matrix factorization $\KR(\Gamma)$
associated to a smoothing $\Gamma$ of the link: 
\begin{enumerate}
\item First decompose $\Gamma$ into basic MOY-graphs $\Gamma_1^1$ and $\Gamma_m^m$ as depicted in Figure
  \ref{fig:buildingblocksintro}, and to each of these building blocks associate certain explicit matrix factorizations. 
\item Glue them together by tensoring. The result is $\KR(\Gamma)$.
\end{enumerate}
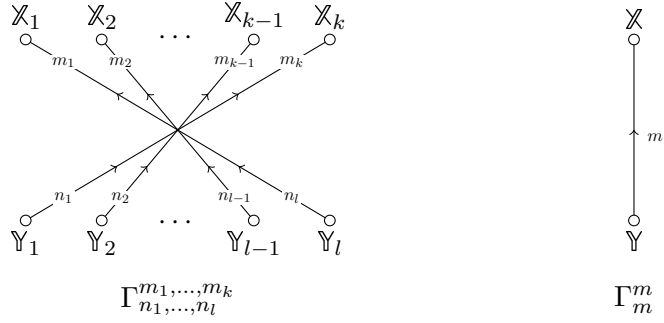
\begin{figure}[h]\begin{center}
\begin{tikzpicture}
\draw[oredged] (0,0) -- node[txt]{${m_1}$}   (-2,1.2) node[fat]{} node[above]{$\XX_1$};
\draw[oredged] (0,0) -- node[txt]{$m_2$}     (-1,1.2) node[fat]{} node[above]{$\XX_2$};
\draw[oredged] (0,0) -- node[txt]{$m_{k-1}$} (+1,1.2) node[fat]{} node[above]{$\XX_{k-1}$};
\draw[oredged] (0,0) -- node[txt]{$m_k$}     (+2,1.2) node[fat]{} node[above]{$\XX_k$};
\draw[oredgeu] (-2,-1.2) node[fat]{} node[below]{$\YY_1$} -- node[txtp]{$n_1$} (0,0);
\draw[oredgeu] (-1,-1.2) node[fat]{} node[below]{$\YY_2$} -- node[txtp]{$n_2$}     (0,0) ;
\draw[oredgeu] (+1,-1.2) node[fat]{} node[below]{$\YY_{l-1}$} -- node[txtp]{$n_{l-1}$} (0,0);
\draw[oredgeu] (+2,-1.2) node[fat]{} node[below]{$\YY_l$} -- node[txtp]{$n_l$}  (0,0)node[big]{} ;
\node at (0,+1)[above] {$\cdots$};
\node at (0,-1)[below] {$\cdots$};
\node at (0,-2.25) {$\Gamma_{n_1,...,n_l}^{m_1,...,m_k}$};
\node at (6cm,-2.25) {$\Gamma_m^m$};
\begin{scope}[xshift=6cm]
\draw[oredgem] (0,-1.2) node[fat]{} node[below]{$\YY$} -- node[right=8pt,txt,pos=0.5,below=1pt]{$m$}(0,1.2) node[fat]{}
node[above]{$\XX$}; 
\end{scope}
\end{tikzpicture}
\end{center}
\caption{\small{Building blocks}}
\label{fig:buildingblocksintro}
\end{figure}
According to our main question, this yields the following two steps in our desired construction of Khovanov-Rozansky
homology through matrix factorizations:
\begin{enumerate}
\item[(1)] Describe the maximal Cohen-Macaulay modules corresponding to the matrix factorizations associated to building
  blocks $\Gamma_{n_1,...,n_l}^{m_1,...,m_k}$ and $\Gamma_m^m$.
\item[(2)] Try to understand what the tensor product on the homotopy category of matrix factorizations looks like in the
  stable category of maximal Cohen-Macaulay modules.
\end{enumerate}
The answer to (1) is as follows. Consider again a hypersurface $R = S/(w)$, so $S$ is regular and
$w\in\frm\setminus\{0\}$. Further, denote $\ul{R\mod}$ and $\uMCM(R)$ the stable categories of all resp. maximal
Cohen-Macaulay modules over $R$, i.e. the categories obtained from $R\mod$ and  $\MCM(R)$ by annihilating morphisms
factoring through a projective; this  
annihilation is necessary for the syzygy $\Omega M$ of an $R$-module $M$ to be well-defined up to isomorphism and to be
functorial in $M$. Now, the embedding $\uMCM(R)\hookrightarrow\ul{R\mod}$ has a right adjoint 
$\bfM:\ul{R\mod}\to\uMCM(R)$, given by $\bfM(M) := \lim\limits_{n\gg 0} \Omega^{2n} M$. Here the right hand side means
that one has to choose $n\gg 0$ such that $\Omega^{2n} M$ is maximal Cohen-Macaulay and set $\bfM(M) := \Omega^{2n}
M$; the particular choice of $n$ does not matter, because $\Omega^2\cong\id$ on $\uMCM(R)$. This yields the following
stabilization functor (see \cite{Krause}), which is fundamental in the present paper: 
$$(-)\stab{w}\quad:=\quad R\mod\longrightarrow\ul{R\mod}\xrightarrow{\ \ \lambda\ \ }\uMCM(R)\xrightarrow{\ \ \cong\
  \ }\HMF(S,w).$$ 
Now we can formulate our first theorem:
\begin{theorem}[see \ref{thm_KRviaStab}]\label{thm_1}
There is a homotopy equivalence
$$\KR\left(\Gamma_{1,1}^{1,1}\right)\ =\ C\left(\ \begin{tikzpicture}[baseline=3mm,scale=0.25]
\draw[oredgem] (0,0) -- (1,1);
\draw[oredgem] (2,0) -- (1,1);
\draw[oredgem] (1,1) -- node[scale=0.4, right=3pt]{$2$} (1,2); % wide edge
\draw[oredgem] (1,2) -- (0,3);
\draw[oredgem] (1,2) -- (2,3);
\end{tikzpicture}\ \right)\ \ \simeq\ \
\left(\CC[x_1,x_2]\stackrel[\text{Sym}]{}{\otimes}\CC[y_1,y_2]\langle
  1\rangle\right)^{\{x_1^{n+1}+x_2^{n+1}-y_1^{n+1}-y_2^{n+2}\}}$$  
Here $\otimes_\text{Sym}$ means that the symmetric polynomials in $x_1,x_2$ and $y_1,y_2$ are identified, and the base
ring for the stabilization is the (regular local graded) polynomial ring $\CC[x_1,x_2,y_1,y_2]$.

More generally, for the basic building block $\Gamma_{n_1,...,n_l}^{m_1,...,m_k}$ in Figure
\ref{fig:buildingblocksintro}, we have the following homotopy equivalence (for the notation, see Section
\ref{notation}), where $r := \sum\limits_{1\leq i<j\leq k} m_i m_j$:
\begin{align*}
\KR\left(\Gamma_{n_1,...,n_l}^{m_1,...,m_k}\right)\ & =\ \KR\left(\ \begin{tikzpicture}[scale=0.5,baseline=-0.6mm]
\draw[oredged] (0,0) -- node[txt,scale=0.8]{$m_1$}   (-2,1.2) node[fat]{} node[above,scale=0.75]{$\XX_1$}; 
\draw[oredged] (0,0) -- node[txt,scale=0.8]{$m_2$}     (-1,1.2) node[fat]{} node[above,scale=0.75]{$\XX_2$};
\draw[oredged] (0,0) -- node[txt,scale=0.8]{$m_{k-1}$} (+1,1.2) node[fat]{} node[above,scale=0.75]{$\XX_{k-1}$};
\draw[oredged] (0,0) -- node[txt,scale=0.8]{$m_k$}     (+2,1.2) node[fat]{} node[above,scale=0.75]{$\XX_k$};
\draw[oredgeu] (-2,-1.2) node[fat]{} node[below,scale=0.75]{$\YY_1$} -- node[txtp,scale=0.8]{$n_1$} (0,0);
\draw[oredgeu] (-1,-1.2) node[fat]{} node[below,scale=0.75]{$\YY_2$} -- node[txtp,scale=0.8]{$n_2$}     (0,0) ;
\draw[oredgeu] (+1,-1.2) node[fat]{} node[below,scale=0.75]{$\YY_{l-1}$} -- node[txtp,scale=0.8]{$n_{l-1}$} (0,0);
\draw[oredgeu] (+2,-1.2) node[fat]{} node[below,scale=0.75]{$\YY_l$} -- node[txtp,scale=0.8]{$n_l$}  (0,0)node[big]{} ;
\node at (0,+1)[above,scale=0.75] {$\cdots$};
\node at (0,-1)[below,scale=0.75] {$\cdots$};\end{tikzpicture}\ \right)\\ & \simeq\
\left(\text{Sym}(\XX_1|...|\XX_n)\stackrel[\text{Sym}]{}{\otimes}\text{Sym}(\YY_1|...|\YY_m)  
\langle r\rangle\right)^{\left\{\Sigma\XX^{n+1}-\Sigma\YY^{n+1}\right\}}.\end{align*}
\end{theorem}

Given the statement of Theorem \ref{thm_1} it is natural to ask to what extend the stabilization functor commutes with tensor
products. Informally, our result can be stated as follows:

\begin{theorem}[see \ref{thm_glue} and \ref{thm_acyclicgraphs}]\label{thm_2}
As long as there are no oriented cycles in the graph, the stabilization functor commutes with tensor products.
\end{theorem}
As a special case, we obtain the description of the matrix factorizations associated to MOY-braids through Soergel
bimodules. Here by a MOY-braid we mean a concatenation of MOY-graphs as in Figure \ref{fig:moybraid}. For
$i_1,...,i_l\in\{1,2,...,m-1\}$ we write $s_{i_1} s_{i_2}\cdots s_{i_l}$ for the concatenation from top to bottom of
  $s_{i_1}$, $s_{i_2}$, ..., $s_{i_l}$. For example, in this notation we have $\Gamma_0 = s_1 s_2 s_1$, where $\Gamma_0$
  is the MOY-graph depicted in Figure \ref{fig:examplerelation}. 
\begin{figure}[h]\begin{center}
\begin{tikzpicture}
\draw[oredgem] (0cm,0cm) node[scale=0.75,below]{$1$} -- (0cm,3cm);
\draw[oredgem] (1cm,0cm) node[scale=0.75,below]{$2$} -- (1cm,3cm);
\draw[oredgem] (2cm,0cm) node[scale=0.75,below]{$i-1$} -- (2cm,3cm);
\begin{scope}[xshift=3cm]
\draw[oredgem] (0cm,0cm) node[scale=0.75,below]{$i$} -- (1cm,1cm);
\draw[oredgem] (2cm,0cm) node[scale=0.75,below]{$i+1$} -- (1cm,1cm);
\draw[oredgem] (1cm,1cm) -- node[scale=0.4, right=3pt]{$2$} (1cm,2cm); % wide edge
\draw[oredgem] (1cm,2cm) -- (0cm,3cm);
\draw[oredgem] (1cm,2cm) -- (2cm,3cm);
\end{scope}
\draw[oredgem] (6cm,0cm) node[scale=0.75,below]{$i+2$} -- (6cm,3cm);
\draw[oredgem] (7cm,0cm) node[scale=0.75,below]{$m-1$} -- (7cm,3cm);
\draw[oredgem] (8cm,0cm) node[scale=0.75,below]{$m$} -- (8cm,3cm);
\node[scale=0.5] at (1.5cm,1.5cm) {$\cdots$};
\node[scale=0.5] at (6.5cm,1.5cm) {$\cdots$};
\end{tikzpicture}
\end{center}
\caption{\small{Basic MOY-braid $\sigma_i$}}
\label{fig:moybraid}
\end{figure}
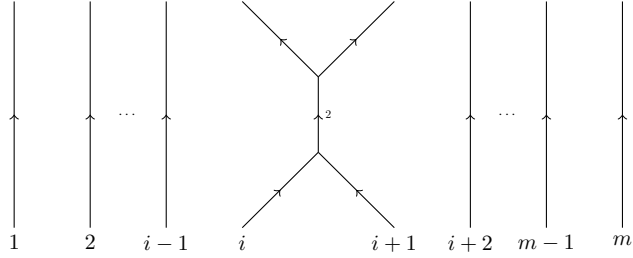

\begin{cor*}[see Corollary \ref{cor:krofamoybraid}]
The matrix factorization associated to a MOY-braid $s_{i_1} s_{i_2} \cdots s_{i_l}$ is canonically homotopy equivalent to the
stabilization of the Soergel bimodule $B_{i_1}\otimes B_{i_2}\otimes...\otimes B_{i_l}$ associated to the braid. 
\end{cor*}

Later in Section \ref{sec_description} we will give the definition of the $B_i$ and the category of Soergel bimodules;
for now it is not necessary to know them. The important thing to realize is that the corollary implies a
bunch of relations up to homotopy between the matrix factorizations associated to MOY-braids, namely those which are
already true on the level of the corresponding Soergel bimodules. The combinatorics of these modules is quite well
understood in terms of the Hecke algebra $\bfH_m(q)$ of the symmetric group: sending the Kazhdan-Lusztig basis element
$\ul{\H}_i$ to the class of the Soergel bimodule $B_i$ induces an isomorphism of rings between the (generic)
Hecke-algebra and the split Grothendieck ring of the category of Soergel bimodules. Relations in the Hecke algebra
therefore correspond to relations between Soergel bimodules, and when applying the stabilization functor these yield
relations between matrix factorizations appearing in the construction of Khovanov and Rozansky.

For example, it follows directly from Theorem \ref{thm_2} and the known equality 
\begin{align}\label{eq:heckealgs3}
\KL{s}\KL{t}\KL{s} - \KL{s}\ & =\ \KL{t}\KL{s}\KL{t} - \KL{t}\ =\ \KL{sts}
\end{align}
for the Hecke-algebra of $\frS_3 = \langle s,t\ |\ s^2=t^2=e,\ sts=tst\rangle$ 
that there is a homotopy equivalence $$\KR(\Gamma_0)\oplus \KR(\Gamma_1)\ \simeq\ \KR(\Gamma_2)\oplus \KR(\Gamma_3),$$
where $\Gamma_i$, $i=0,1,2,3$ are depicted in Figure \ref{fig:examplerelation}. Note that though the relation is
elementary in the Hecke algebra, it requires a substantial amount of direct calculations to verify it in $\HMF$.

\begin{figure}[h]\begin{center}
\begin{tikzpicture}[scale=0.5]
\draw[oredgem] (0,0) -- (1,1);
\draw[oredgem] (2,0) -- (1,1);
\draw[oredgem] (1,1) -- node[scale=0.6, right=3pt]{$2$} (1,2); % wide edge
\draw[oredgem] (1,2) -- (0,3);
\draw[oredgee] (1,2) -- (2,3);
\draw[oredgem] (4,0) -- (4,3);
\begin{scope}[xshift=2cm, yshift=3cm]
\draw[oredgem] (-2,0) -- (-2,3);
\draw (0,0) -- (1,1);
\draw[oredgem] (2,0) -- (1,1);
\draw[oredgem] (1,1) -- node[scale=0.6, right=3pt]{$2$} (1,2); % wide edge
\draw[oredgee] (1,2) -- (0,3);
\draw[oredgem] (1,2) -- (2,3);
\end{scope}
\begin{scope}[yshift=6cm]
\draw[oredgem] (0,0) -- (1,1);
\draw (2,0) -- (1,1);
\draw[oredgem] (1,1) -- node[scale=0.6, right=3pt]{$2$} (1,2); % wide edge
\draw[oredgem] (1,2) -- (0,3);
\draw[oredgem] (1,2) -- (2,3);
\draw[oredgem] (4,0) -- (4,3); 
\end{scope}
\node at (2,-1){$\Gamma_0$};
\node at (5.5,4.5){$\oplus$};
\begin{scope}[xshift=7cm]
\draw[oredgem] (1,0) -- (2,1);
\draw[oredgem] (3,0) -- (2,1);
\draw[oredgem] (2,1) -- node[scale=0.6, right=3pt]{$2$} (2,8);
\draw[oredgem] (2,8) -- (1,9);
\draw[oredgem] (2,8) -- (3,9);
\draw[oredgem] (0,0) -- (0,9);
\end{scope}
\node at (8.5,-1){$\Gamma_1$};
\node at (12,4.5) {$=$};
\begin{scope}[xshift=14cm]
\draw[oredgem] (2,0) -- (3,1);
\draw[oredgem] (4,0) -- (3,1);
\draw[oredgem] (3,1) -- node[scale=0.6, right=3pt]{$2$} (3,2); % wide edge
\draw[oredgee] (3,2) -- (2,3);
\draw[oredgem] (3,2) -- (4,3);
\draw[oredgem] (0,0) -- (0,3);
\begin{scope}[yshift=3cm]
\draw[oredgem] (0,0) -- (1,1);
\draw (2,0) -- (1,1);
\draw[oredgem] (1,1) -- node[scale=0.6, right=3pt]{$2$} (1,2); % wide edge
\draw[oredgem] (1,2) -- (0,3);
\draw[oredgee] (1,2) -- (2,3);
\draw[oredgem] (4,0) -- (4,3);
\end{scope}
\begin{scope}[yshift=6cm]
\draw (2,0) -- (3,1);
\draw[oredgem] (4,0) -- (3,1);
\draw[oredgem] (3,1) -- node[scale=0.6, right=3pt]{$2$} (3,2); % wide edge
\draw[oredgem] (3,2) -- (2,3);
\draw[oredgem] (3,2) -- (4,3);
\draw[oredgem] (0,0) -- (0,3);
\end{scope}
\node at (2,-1){$\Gamma_2$};
\node at (5.5,4.5){$\oplus$};
\begin{scope}[xshift=7cm]
\draw[oredgem] (0,0) -- (1,1);
\draw[oredgem] (2,0) -- (1,1);
\draw[oredgem] (1,1) -- node[scale=0.6, right=3pt]{$2$} (1,8);
\draw[oredgem] (1,8) -- (0,9);
\draw[oredgem] (1,8) -- (2,9);
\draw[oredgem] (3,0) -- (3,9);
\end{scope}
\node at (8.5,-1){$\Gamma_3$};
\end{scope}
\end{tikzpicture}
\end{center}
\caption{Basic MOY-relation}
\label{fig:examplerelation}
\end{figure}
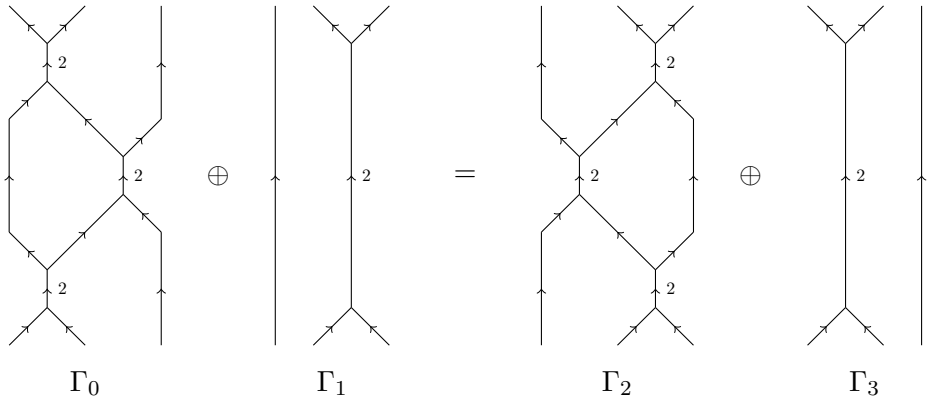

Theorem \ref{thm_2} reveals the following parallel between the constructions of Khovanov-Rozansky homology and the knot
invariant of Mazorchuk-Stroppel (see \cite{StroppelMazorchuk}): In the construction of Mazorchuk and Stroppel they
associate to a MOY-braid a projective functor on some graded version of the Bernstein-Gelfand-Gelfand category $\calO$,
and this projective functor corresponds to the associated Soergel bimodule under Soergels combinatorial functor
$\VV$ (see \cite{SoergelHarishChandra}). Hence, both constructions begin by associating to a MOY-braid something which is equivalent to the Soergel bimodule
associated to the braid. Next, the construction of Khovanov-Rozansky proceeds by stabilizing the given Soergel bimodule,
while Mazorchuk-Stroppel restrict the projective functor under consideration to some parabolic subcategories 
$\calO^\frp$. We will come back to the meaning of these second steps in the construction of Khovanov-Rozansky and
Mazorchuk-Stroppel soon. 

The Hecke algebra relations we get from Theorem \ref{thm_2} are not enough to category the Reshetikhin-Turaev link
invariant for $\calU_q(\sl(n))$. To understand this, recall the description 
of the endomorphism algebra of the $m$-th tensor power of the vector representation $V$ of the quantum group
$\calU_q(\sl(n))$ for generic $q$. We have a surjective map
\begin{equation*}\begin{tikzpicture}[description/.style={fill=white,inner sep=2pt}]
    \matrix (m) [matrix of math nodes, row sep=3em,
                 column sep=2.5em, text height=1.5ex, text depth=0.25ex,
                 inner sep=0pt, nodes={inner xsep=0.3333em, inner ysep=0.3333em}]
    {
      \bfH_m(q) &&  \End_{\calU_q(\sl(n))}(V^{\otimes m})\\
    };
    \draw[->]  (m-1-1) -- node[above,scale=0.75]{$\tau_k$} (m-1-3);
\end{tikzpicture}\end{equation*}
sending the Kazhdan-Lusztig element $\ul{\H}_i$ to the intertwiner
\begin{equation*}
 V^{\otimes m}\xrightarrow{\id\otimes
  \left(\begin{tikzpicture}[baseline=1.8mm,scale=0.25] 
\draw[oredgem] (0,0) -- (1,1);
\draw[oredgem] (2,0) -- (1,1);
\draw[oredgem] (1,1) -- node[scale=0.4, right=3pt]{$2$} (1,2); % wide edge
\end{tikzpicture}\right)\otimes\id} V^{\otimes (i-1)}\otimes (V\wedge V)\otimes V^{\otimes
(m-i-1)}
\xrightarrow{\id\otimes\left(
\begin{tikzpicture}[baseline=1.8mm,scale=0.25]
\draw[oredgem] (1,0) -- node[scale=0.4, right=3pt]{$2$} (1,1);
\draw[oredgem] (1,1) -- (2,2);
\draw[oredgem] (1,1) -- (0,2);
\end{tikzpicture}\right)\otimes\id} V^{\otimes m}
\end{equation*}
Thus, the special intertwiners
$$\begin{tikzpicture}[baseline=2.4mm,scale=0.25]
\draw[oredgem] (0cm,0cm) -- (0cm,3cm);
\draw[oredgem] (1cm,0cm) -- (1cm,3cm);
\draw[oredgem] (3cm,0cm) -- (3cm,3cm);
\begin{scope}[xshift=4cm]
\draw[oredgem] (0cm,0cm) -- (1cm,1cm);
\draw[oredgem] (2cm,0cm) -- (1cm,1cm);
\draw[oredgem] (1cm,1cm) -- node[scale=0.4, right=3pt]{$2$} (1cm,2cm); % wide edge
\draw[oredgem] (1cm,2cm) -- (0cm,3cm);
\draw[oredgem] (1cm,2cm) -- (2cm,3cm);
\end{scope}
\draw[oredgem] (7cm,0cm) -- (7cm,3cm);
\draw[oredgem] (9cm,0cm) -- (9cm,3cm);
\draw[oredgem] (10cm,0cm) -- (10cm,3cm);
\node[scale=0.5] at (2cm,1.4cm) {$\cdots$};
\node[scale=0.5] at (8cm,1.4cm) {$\cdots$};
\end{tikzpicture}\ =\ \tau_m(\H_i)$$
satisfy the Hecke algebra relations, and so should their categorifications. Theorem
\ref{thm_2} shows that this is indeed true for Khovanov-Rozansky's construction. However, since $\tau_m$ is not
injective in general, there are more relations in $\End_{\calU_q(\sl(n))}(V^{\otimes m})$, namely those coming from
elements of $\ker(\tau_m)$; these should be fulfilled in the categorification, too. We call the relations coming from
$\bfH_m(q)/\ker(\tau_m)$ MOY-relations for short. On the level of Soergel bimodules, the MOY-relations are not
satisfied, because the combinatorics of Soergel bimodules, i.e. the Grothendieck ring of the category of Soergel
bimodules, is given by $\bfH_m(q)$. If we want all MOY-relations to be fulfilled in Khovanov-Rozansky homology, we
therefore have to show that by stabilizing Soergel bimodules we obtain the missing relations from
$\ker(\tau_m)$. Concretely, the kernel of $\tau_m$ is generated by those Kazhdan-Lusztig basis elements $\ul{\H}_w$ for
permutations $w\in\frm_m$ whose Robinson-Schensted tableau has more than $n$ rows; we therefore have to show that the
stabilizations of the Soergel bimodules corresponding to these elements vanish, and this is the content of the following 
theorem: 
\begin{theorem}[see Theorem \ref{thm_killbadbimodule}]
Fix $n\geq 2$ and let $w\in\frS_m$ be such that the Robinson-Schensted tableau of $w$ has more than $n$ rows. Then the
indecomposable Soergel bimodule $B_w$ is of finite projective dimension considered as a module over the ring
$$\CC[x_1,...,x_m,y_1,...,y_m] / \left(\sum\limits_{i=1}^{m} x_i^{n+1} - y_i^{n+1}\right),$$ and hence
$$B_w\stab{\sum\limits_{i=1}^{m} x_i^{n+1} - y_i^{n+1}}\ \simeq\ 0.$$
In particular, the stabilizations of Soergel bimodules satisfy the MOY-relations.
\end{theorem}

As an example, take $n=2$ and $m=3$. In this case, \eqref{eq:heckealgs3} yields $\ul{\H}_{s}\ul{\H}_{t}\ul{\H}_{s} =
\ul{\H}_{sts} + \ul{\H}_{s}$ in the Hecke algebra of $\frS_3$. The Robinson-Schensted tableau of $sts$ has $3$ rows,
so we get $C(\Gamma_0)\simeq C(\Gamma_3)$ and $C(\Gamma_1)\simeq C(\Gamma_2)$ (see Figure \ref{fig:examplerelation}).

For the invariant of Mazorchuk and Stroppel the situation is similar: the projective functors associated to MOY-braids
 satisfy the Hecke-algebra relations, but not the extra relations coming from $\ker(\tau_m)$. To obtain the missing
 relations, the functors have to be restricted to certain parabolic subcategories $\calO^\frp$ of $\calO$.

\vskip3mm Theorems 1-3 provide the first steps for a connection between Khovanov-Rozansky- and Stroppel-Mazorchuk
homology. However, we cannot state a precise comparison theorem. On the Hecke algebra level of Soergel
bimodules/projective functors on $\calO$, the connection is clear. However, it is not clear to the author in which way
restriction from $\calO_0$ to parabolic subcategories $\calO\frp$ corresponds to the stabilization of the
corresponding Soergel bimodule with respect to $\Sigma x_i^{n+1} - \Sigma y_i^{n+1}$, even though the \textit{effect} of
both operations is the same.

We can informally summarize the results of this work in the commutative diagram \ref{fig:bigpicture}.
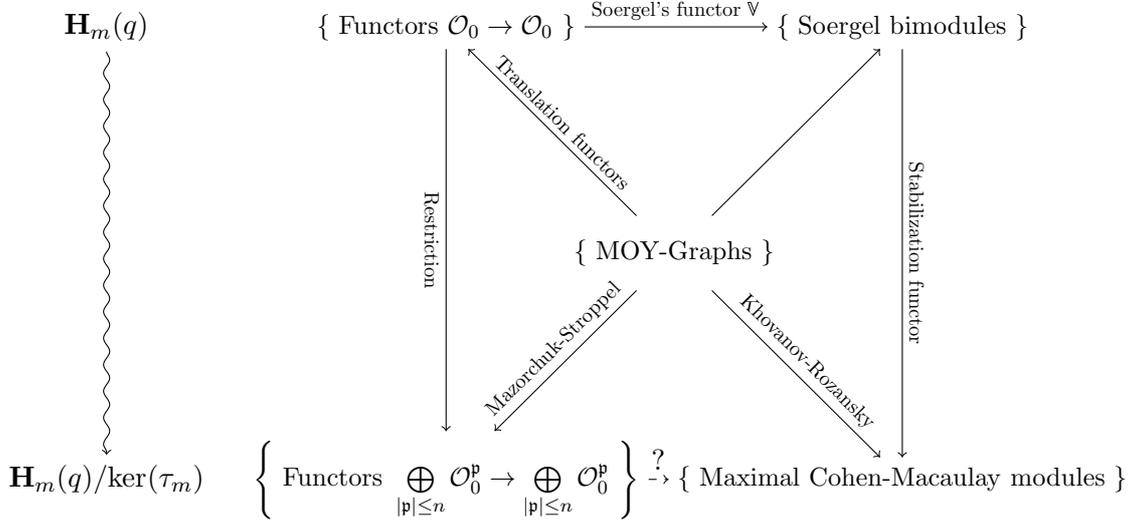
\begin{figure}[h]\begin{center}
\begin{tikzpicture}[scale=0.3]
\begin{scope}[xshift=6cm]
\node (funcprincipalblock) [fill=white, scale=0.9] at (-10,+10)   {$\left\{\ \text{Functors }\calO_0\to\calO_0\ \right\}$};
\node (funcparabolic)      [fill=white,scale=0.9] at (-10,-10)   {$\left\{\ \text{Functors }\bigoplus\limits_{|\frp|\leq
    n}\calO_0^\frp\to\bigoplus\limits_{|\frp|\leq n} \calO_0^\frp\ \right\}$}; 
\node (soergelbimodules)   [fill=white,scale=0.9] at (+10,+10)   {\{\ Soergel bimodules\ \}};
\node (moygraph)           [fill=white,scale=0.9, inner sep=10pt] at (0,0)       {\{\ MOY-Graphs\ \}};
\node (mcmmodules)         [fill=white,scale=0.9] at (+10,-10)   {\{\ Maximal Cohen-Macaulay modules\ \}};
\draw[oredgee] (moygraph) -- node[sloped,scale=0.7, above] {Translation functors} (funcprincipalblock);
\draw[oredgee] (moygraph) -- node[sloped,scale=0.7, above] {Khovanov-Rozansky}    (mcmmodules);
\draw[oredgee] (moygraph) -- node[sloped,scale=0.7, above] {Mazorchuk-Stroppel}   (funcparabolic);
\draw[oredgee] (moygraph) -- (soergelbimodules);
\draw[oredgee] (funcprincipalblock) -- node[sloped,scale=0.7,below] {Restriction} (funcparabolic);
\draw[dashed,oredgee] (funcparabolic) -- node[sloped,above] {?} (mcmmodules);
\draw[oredgee] (soergelbimodules) -- node[sloped,scale=0.7,above] {Stabilization functor} (mcmmodules);
\draw[oredgee] (funcprincipalblock) -- node[sloped,scale=0.7,above] {Soergel's functor $\VV$}
(soergelbimodules);

\node (parab) at ($(funcparabolic.west) + (-6cm,0)$) {$\bfH_m(q)/\ker(\tau_m)$};
\node (hecke) at (parab |- funcprincipalblock) {$\bfH_m(q)$};

\draw[->,decorate,decoration={snake,amplitude=0.4mm}] (hecke)  -- (parab);
\end{scope}
\end{tikzpicture}
\end{center}
\caption{Overview over the results of this work.}
\label{fig:bigpicture}
\end{figure}

\textbf{Structure: } The paper is organized as follows. In Section \ref{sec_commalg} we recall some basics about local
graded commutative rings, focusing on how to relate it to the better known case of ungraded local commutative rings. In
Section \ref{sec_stabilization} we introduce the notion of graded maximal Cohen-Macaulay modules over Gorenstein rings
and recall the well-known connection between graded maximal Cohen-Macaulay modules over a hypersurface and graded matrix
factorizations. We then introduce the stabilization functor and study the compatibility of stabilization with tensor
products of matrix factorizations. In Section \ref{sec_description} we use the techniques developed 
so far to simplify the construction of Khovanov-Rozansky using the stabilization functor, proving Theorems
\ref{thm_1} and \ref{thm_2}. In Section
\ref{sec_duality} we study the compatibility of the stabilization functor with the duality for matrix factorizations,
and apply the results we get in Section \ref{sec_closing} to describe braid closure as some kind of stabilized
Hochschild-cohomology (see \cite{WebsterCanopolis}). In all these sections we focus on motivation, examples and explicit calculations, while not
trying to give all results in the greatest possible generality. In contrast to that, there is an appendix where we
reprove almost all statements in a much more general situation using the language of derived categories. This appendix
can be read almost independently of the rest of the paper; however, its bigger generality and abstraction might prevent
the reader from getting the motivation for what is done, and this is why we didn't work in this more abstract setting
right from the beginning.

\vskip3mm\textbf{Acknowledgements:}\ I want to thank all people who helped and supported me during the process of
writing this thesis. My special thanks go to my advisor Prof.~Dr.~Catharina~Stroppel for the countless helpful and
interesting discussions about the subject. 

%%%%%%%%%%%%%%%%%% KAPITEL %%%%%%%%%%%%%%%%%

\newpage

\renewcommand{\theequation}{\thesubsection-\arabic{equation}}

\section{Basics on local graded commutative algebra}\label{sec_commalg} 

\setcounter{prop}{0}
\renewcommand\theprop{\arabic{section}.\arabic{subsection}.\arabic{prop}}
In this section, we will give a short introduction to local graded commutative algebra. All of the results we recall
here are well-known at least in the ungraded case, so we concentrate on explaining how they can rigorously be upgraded
to the graded case.  
\subsection{Notation}
In the following, we always denote by $R\ld = \bigoplus_{n\in\ZZ} R_n$ a Noetherian graded commutative ring which is
\textit{local} in the sense that there is precisely one graded maximal ideal $\frm$, and we let $k\ld := R\ld /
\frm$ be its residue class ring. Note that any ungraded local ring can be considered as a local graded ring concentrated
in degree zero, so the ungraded situation is a special case of the graded situation. Next, let $R$ be the ungraded
ring underlying $R\ld$, and let $R\ld\Mod$ denote the abelian category of all graded $R\ld$-modules with grading
preserving morphisms of $R$-modules. The set of morphisms between graded $R\ld$-modules $M\ld$ and $N\ld$ is denoted by
$\hom_{R\ld}(M\ld,N\ld)$. The subcategory of finitely generated graded $R\ld$-modules is denoted by $R\ld\mod$. Next, let
$\langle d\rangle: R\ld\Mod\to R\ld\Mod$ be the automorphism given by grading shift, i.e. $M\langle d\rangle_k :=
M_{k+d}$. If $M\ld$ is a graded $R\ld$-module, we denote by $M$ the underlying $R$-module. An $R\ld$-module $M\ld$ is
called \textit{free} (\textit{of finite rank}) if it is isomorphic to a (finite) direct sum of modules of the form 
$R\langle d\rangle\ld$ for some $d\in\ZZ$.

For $R\ld$-modules $M\ld$, $N\ld$ there is a \textit{graded homomorphism space} $\hom_{R}(M\ld,N\ld)\ld$ defined by
$\hom_{R}(M\ld,N\ld)_k := \hom_{R\ld}(M\ld,N\ld\langle k\rangle)$. An element $f\in\hom_{R}(M\ld,N\ld)_k$ is
called a \textit{homomorphism of degree $k$}; this is just a homomorphism of the underlying $R$-modules raising
the degree of each element precisely by $k$. There is a natural action of $R\ld$ on $\hom_{R}(M\ld,N\ld)\ld$ making
it into a graded $R\ld$-module. Note that by our convention we have
$\hom_{R\ld}(M\ld,N\ld)=\hom_{R}(M\ld,N\ld)_0$, but $\hom_{R}(M\ld,N\ld) = \bigoplus_{k\in\ZZ}
\hom_{R\ld}(M\ld,N\ld\langle k\rangle)$.  Note also that there is a natural homomorphism
$\hom_{R}(M\ld,N\ld)\hookrightarrow \hom_{R}(M,N)$ whose image consists of all homomorphisms of $R$-modules $M\to N$
that can be written as a finite sum of homomorphisms of graded $R\ld$-modules $M\ld\to N\ld\langle k\rangle$. In
general, there might be homomorphisms of $R$-modules $M\to N$ which cannot be written in this way, but if $M\ld$ is
finitely generated over $R\ld$, the above map is an isomorphism. Homomorphisms of graded modules $M\p\ld\to M\ld$ and
$N\ld\to N\p\ld$ induce a homomorphism of graded modules $\hom_{R}(M\ld,N\ld)\ld\to\hom_{R}(M\p\ld,N\p\ld)\ld$, and in
this way $\hom_{R}(-,-)\ld$ becomes a biadditive functor $R\ld\Mod\op\times R\ld\Mod\to R\ld\Mod$. Let
$\ext\ua_{R}(-,-)\ld$ denote the family of derived functors. Since the functor $M\ld\to M_0$ from
$R\ld\Mod$ to $\ZZ\Mod$ is exact, we have that $\ext\ua_{R}(-,-)_0$ is the family of derived functors of $\hom_{R}(-,-)_0 = 
\hom_{R\ld}(-,-)$, so for any two $R\ld$-modules $M\ld$, $N\ld$ there is a natural identification
$\ext\ua_{R}(M\ld,N\ld)_0 = \ext\ua_{R\ld}(M\ld,N\ld)$. As above, for finitely generated $M\ld$ we have
$\ext\ua_{R}(M\ld,N\ld)\cong\ext\ua_{R}(M,N)$, but in general these two $R\ld$-modules may differ.

For graded $R\ld$-modules $M\ld$, $N\ld$ define the \textit{tensor product} $M\ld\otimes_{R\ld} N\ld$ by
$(M\ld\otimes_{R\ld} N\ld)_k := \left(\bigoplus\limits_{p+q=k} M_p\otimes_\ZZ N_q\right)/_\sim$, where $\sim$ is generated
by $x.m\otimes n\sim (-1)^{r(p-r)}m\otimes x.m$ for $x\in R_r$ and $m\in M_{p-r}$. This gives rise to an additive
bifunctor $R\ld\Mod\times R\ld\Mod\to R\ld\Mod$, and we denote by $\tor\ua_{R}(-,-)\ld$ the family of derived functors
of this functor. 

\subsection{Graded vs. Ungraded}

Most of the theorems on ungraded local Noetherian rings are true for local graded Noetherian rings. One reason for this
is that for a finitely generated graded $M\ld$ over a local Noetherian graded ring $(R\ld,\frm\ld)$ the map $M\ld\mapsto
M_\frm$ takes many numerical invariants like the Betti-numbers, the dimension, the depth or the projective dimension of
the graded $R\ld$-module $M\ld$ into the ones for the ungraded $R_\frm$-module $M_\frm$. This makes it 
 possible to carry over results from the ungraded case stating relations between these numerical invariants (the
 Auslander-Buchsbaum formula, for example) to the graded case without having to copy the proof verbatim. The material of
 this section is completely contained in \cite{BrunsHerzog}, but for the reader's 
 convenience we will reproduce it here and provide some details not contained in loc.cit.

To get a feeling why the essential information carried by $M\ld$ is already encoded in $M_\frm$, we think about why the
vanishing of $M_\frm$ implies the vanishing of $M\ld$.

\begin{fact}\label{fact_locdoesnotkillanything}
Let $M\ld$ be a graded module over the graded ring $R\ld$, and let $\frp$ be a (not necessarily homogeneous) prime ideal
in $R$. Then $M_\frp=0$ if and only if $M_{(\frp)}=0$. Here $M_\frp$ denotes the localization of $M$ with
respect to $R\setminus\frp$ (an ungraded module over the ungraded ring $R_{\frp}$), and $M_{(\frp)}$ denotes the
localization of $M$ with respect to $\bigcup_{k\in\ZZ} R_k\setminus\frp_k$. 

In particular, if $(R\ld,\frm\ld)$ is a local graded ring, then $M\ld=0$ if and only if $M_\frm=0$. 
\end{fact}
\begin{proof}
Let $S := \bigcup_{k\in\ZZ} R_k\setminus\frp_k$. We have $S\subset R\setminus\frp$, so $M_S=0$ implies $M_\frp =
M_{R\setminus\frp}=0$. Now assume $M_{R\setminus\frp}=0$. To show that $M_S=0$ it is sufficient to prove that any
homogeneous $m\in M_k$ vanishes in $M_S$. As it vanishes in $M_{R\setminus\frp}$, there is some $x\in R\setminus\frp$
such that $x.m=0$. By the homogeneity of $m$ it follows that $x_r.m=0$ for any homogeneous component $x_r$ of
$x$. However, since $x\notin\frp$, some homogeneous component $x_r\in R_r$ of $x$ has to lie in $R\setminus\frp$, and
thus in $R_r\setminus\frp_r\subset S$. Hence, $m$ is killed by an element in $S$, and therefore vanishes in $M_S$. 
\end{proof}

\begin{definition}\label{def_dim}
Let $(R\ld,\frm)$ be a local graded ring and let $M\ld$ be a finitely generated graded $R\ld$-module. The
\textit{dimension} of $M\ld$, denoted $\dim_{R\ld} M\ld$, is defined as the maximal $k$ such that there exists a chain
of homogeneous prime ideals $\frp_0\subsetneq\frp_1\subsetneq ...\subsetneq \frp_k$ such that $M_{(\frp_i)}\neq 0$ for
all $i=0,...,k$.  
\end{definition}
\begin{prop}\label{prop_dim}
In the situation of Definition \ref{def_dim}, we have $\dim_{R\ld} M\ld\ =\ \dim_{R_\frm} M_\frm$. 
\end{prop}
\begin{proof}
For a homogeneous prime ideal $\frp\subset R\ld$ we have $M_{(\frp)}=0$ if and only if $M_\frp=0$
(Fact \ref{fact_locdoesnotkillanything}), hence $\dim_{R\ld} M\ld\leq\dim_{R_\frm} M_\frm$. It is therefore sufficient
to show that for $d := \dim_{R_\frm} M_\frm$ there is a sequence $\frp_0\subsetneq ...\subsetneq\frp_d$ of
\textit{homogeneous} prime ideals such that $M_{\frp_i}\neq 0$ for all $i=0,...,d$, which is done in \cite[Theorem
1.5.8]{BrunsHerzog}. 
\end{proof}

\begin{definition}\label{def_depth}
Let $(R\ld,\frm)$ be a local graded ring and let $M\ld$ be a finitely generated graded $R\ld$-module. The
\textit{depth} of $M\ld$, denoted $\depth_{R\ld} M\ld$, is defined as the maximal length of a $M$-regular sequence of
homogeneous elements in $\frm$. If there is no chance of confusion, we will shortly write $\depth(M\ld)$ for
$\depth_{R\ld}(M\ld)$.  
\end{definition}

We have the following description of the depth in terms of the vanishing of $\ext$-groups:
\begin{prop}\label{prop_depth}
In the situation of Definition \ref{def_depth} we have
$$\depth_{R\ld} M\ld\ =\ \inf\{i\in\ZZ_{\geq 0}\ |\ \ext^i_{R}(k\ld,M\ld)\ld\neq\{0\}\},$$ and any maximal $M\ld$-regular 
sequence in $\frm$ has length $\depth_{R\ld} M\ld$. In particular, we have $\depth_{R\ld} M\ld<\infty$
and $$\depth_{R\ld} M\ld = \depth_{R_\frm} M_\frm.$$
\end{prop}
\begin{proof}
Since $R\ld$ is Noetherian, any $M\ld$-regular sequence must be finite. Thus, the second statement indeed implies that
$\depth(M\ld)<\infty$.  

Let $x_1,...,x_n\in\frm$ be an arbitrary $M\ld$-regular sequence of homogeneous elements, and let $d_1,...,d_n$ denote
the degrees of the $x_i$. By definition, $x_1$ is not a zero divisor in $M\ld$, so we have a short exact sequence $0\to 
M\langle -d_1\rangle\ld\xrightarrow{x_1\cdot} M\ld\to M\ld/x_1 M\ld\to 0$. Applying $\ext\ua_{R}(k\ld,-)\ld$ gives
$$\cdots\mor{x_1\cdot}\ext^i_{R}(k\ld,M\ld)\ld\to\ext^i_{R}(k\ld,M\ld/x_1
M\ld)\ld\to\ext^{i+1}_{R}(k\ld,M\ld)\ld\langle-d_1\rangle\mor{\cdot x_1}\cdots,$$ 
and since $k\ld$ is annihilated by $x_1$, this sequence decomposes into short exact sequences
$$0\to\ext^i_{R}(k\ld,M\ld)\ld\to\ext^i_{R}(k\ld,M\ld/x_1
M\ld)\ld\to\ext^{i+1}_{R}(k\ld,M\ld)\ld\langle-d_1\rangle\to 0$$ for all $i\in\ZZ$. For $i=-1$ and then for $i=0$, we get
$\ext^0_{R}(k\ld,M\ld)\ld=0$ and $\ext^1_{R}(k\ld,M\ld)\ld\cong\ext^0_{R}(k\ld,M\ld/x_1 M\ld)\langle
d_1\rangle\ld.$ Continuing in this way, we obtain $\ext^i_{R}(k\ld,M\ld)\ld=0$ for all $0\leq i<n$ and
\begin{align}\label{eq:extshift}\ext^n_{R}(k\ld,M\ld)\ld\cong\ext^0_{R}(k\ld,M\ld/(x_1,...,x_n)M\ld)\ld\langle
  d_1+...+d_n\rangle,\end{align} which shows that $$\depth(M\ld)\ \leq\ \min\{k\in\ZZ_{\geq 0}\ |\
\ext^k_{R}(k\ld,M\ld)\ld\neq\{0\}\}.$$ Next, take $x_1,...,x_n$ maximal (as noted at the beginning of the
proof, such a sequence must exist). We will show that $\ext^0_{R}(k\ld,M\ld/(x_1,...,x_n)M\ld)\ld\neq 0$, and this
will finish the proof because of \eqref{eq:extshift}. As $x_1,...,x_n$ is maximal, any element of $\frm$ is a
zero-divisor of 
$M\ld/(x_1,...,x_n)M\ld$. Hence $\frm$ is contained in the union of the associated primes of $M\ld/(x_1,...,x_n)M\ld$,
so $\frm\in\ass_{R\ld}(M\ld/(x_1,...,x_n)M\ld)$ by prime avoidance. Thus, there exists an embedding $$k\langle
d\rangle\ld=R\ld/\frm\langle d\rangle\hookrightarrow M\ld/(x_1,...,x_n)M\ld$$ for suitable $d\in\ZZ$, and therefore
$\hom_{R}(k\ld,M\ld/(x_1,...,x_n)M\ld)\ld\neq 0$ as claimed. 

The second statement can be seen as follows: First note that since $k\ld$ is finitely generated, we have a canonical
isomorphism of $R$-modules $\ext\ua_{R}(k\ld,M\ld) = \ext\ua_{R}(k,M)$. Hence, using Fact
\ref{fact_locdoesnotkillanything} and the compatibility of $\ext$ with localization, we
get
\begin{align*}
\depth_{R\ld} M\ld
   & = \inf\{i\in\ZZ_{\geq 0}\ |\ \ext^i_{R}(k\ld,M\ld)\ld\neq 0\}
\\ & = \inf\{i\in\ZZ_{\geq 0}\ |\ \ext^i_{R}(k,M)_\frm\neq 0\}
\\ & = \inf\{i\in\ZZ_{\geq 0}\ |\ \ext^i_{R_\frm}(k_\frm,M_\frm)\neq 0\}
\\ & = \inf\{i\in\ZZ_{\geq 0}\ |\ \ext^i_{R_\frm}(R_\frm/\frm R_\frm,M_\frm)\neq 0\}
\\ & = \depth_{R_\frm} M_\frm
\end{align*}
as claimed.
\end{proof}

Next we turn to injective dimensions.

\begin{definition}\label{def_injdim}
Let $(R\ld,\frm)$ be a local graded ring and $M\ld$ a finitely generated $R\ld$-module. Then the \textit{injective
  dimension} of $M\ld$, denoted $\injdim_{R\ld} M\ld$, is defined as the injective dimension of $M\ld$ in the abelian
category $R\ld\mod$. 
\end{definition}
\begin{rem}
For a finitely generated $R\ld$-module $M\ld$ we have
  $$\injdim_{R\ld\mod}(M\ld)=\injdim_{R\ld\Mod}(M\ld)$$ by Baer's criterion (which works for every
  generator in a Grothendieck category). 
\end{rem}
\begin{fact}\label{fact_onlyprimes}
In the situation of Definition \ref{def_injdim}, we have $$\injdim_{R\ld} M\ld\ =\ \sup\{n\in\ZZ_{\geq 0}\ |\ \text{ex. }
\frp\text{ homogeneous prime s.t.\ }\ext_R^n(R\ld/\frp,M\ld)\ld\neq 0\}.$$
\end{fact}
\begin{proof}
This follows immediately from the fact the any finitely generated $R\ld$-module has a finite filtration with filtration
quotients of the form $R\ld/\frp\langle d\rangle$ for homogeneous prime ideals $\frp$ and $d\in\ZZ$.
\end{proof}
\begin{prop}[see {\cite[Proposition 3.1.13]{BrunsHerzog}}]\label{prop_smallerprimes}
Let $(R\ld,\frm)$ be a local graded ring, $\frp\subsetneq\frm$ a homogeneous prime and $M\ld$ a finitely generated
graded $R\ld$-module. Assuming $\ext^{n+1}_R(R\ld/\frq,M\ld)\ld=0$ for all $\frq\supsetneq\frp$, then
$\ext^n_R(R\ld/\frp,M\ld)\ld = 0$.
\end{prop}
\begin{proof}
Pick a homogeneous $x\in\frm\setminus\frp$ of degree $d$. The exact sequence $$0\to R\ld/\frp\langle
-d\rangle\xrightarrow{\ x\ }R\ld/\frp\xrightarrow{\ \ \ }R\ld/(\frp,x)\to 0$$ induces an exact sequence
$$\ext_R^n(R\ld/\frp,M\ld)\ld\xrightarrow{\ x\ }\ext_R^n(R\ld/\frp,M\ld)\ld\langle d\rangle\xrightarrow{\ \ \
}\ext^{n+1}_R(R\ld/(\frp,x),M\ld)\ld$$ 
Any homogeneous prime $\frq$ in the support of $R\ld/(\frp,x)$ satisfies $\frp\subsetneq\frq$, and by assumption
$\ext^{n+1}_R(R\ld/\frq,M\ld)\ld=0$ for any such $\frq$. Hence $\ext^{n+1}_R(R\ld/(\frp,x),M\ld)\ld=0$, which in turn
implies $\ext_R^n(R\ld/\frp,M\ld)\ld=0$ by Nakayama (Lemma \ref{lem_gradednakayama}).
\end{proof}
\begin{prop}[see {\cite[Proposition 3.1.14]{BrunsHerzog}}]\label{prop_injdim}
Let $(R\ld,\frm)$ be a local graded ring and $M\ld$ a finitely generated graded $R\ld$-module. Then
$$\injdim_{R\ld} M\ld\ =\ \sup\{k\in\ZZ_{\geq 0}\ |\ \ext^k_R(k\ld,M\ld)\ld\neq 0\}.$$ In particular, we have
$\injdim_{R\ld} M\ld = \injdim_{R_\frm} M_\frm$. 
\end{prop}
\begin{proof}
The first statement follows from Fact \ref{fact_onlyprimes} and Proposition \ref{prop_smallerprimes}. For the second
statement, Fact \ref{fact_locdoesnotkillanything} and the compatibility of $\ext$ with localization yields
\begin{align*}
\injdim_{R\ld} M\ld & = \sup\{n\in\ZZ_{\geq 0}\ |\ \ext_R^n(k\ld, R\ld)\ld\neq 0\}\\
& = \sup\{n\in\ZZ_{\geq 0}\ |\ \ext_R^n(k\ld,M\ld)_\frm\neq 0\}\\
& = \sup\{n\in\ZZ_{\geq 0}\ |\ \ext_{R_\frm}^n(R_\frm/\frm R_\frm,M_\frm)_\frm\neq 0\}\\
& = \injdim_{R_\frm} M_\frm.
\end{align*}
\end{proof}

Next we show that the functor $M\ld\mapsto M_\frm$ preserves minimal free resolutions. We first recall the definition of
a minimal free resolution in the graded case. 

\begin{fact}\label{fact_freecover}
Let $f: F\ld\to M\ld$ be an epimorphism of finitely generated graded $R\ld$-modules, and let $F\ld$ be free. Then the
following are equivalent:
\begin{enumerate}
\item\label{item:fc1} Any homogeneous $R\ld$-basis of $F\ld$ is mapped by $f$ to a minimal generating system of $M\ld$.
\item\label{item:fc2} There exists a homogeneous $R\ld$-basis of $F\ld$ which is mapped by $f$ to a minimal generating
system of $M\ld$. 
\item\label{item:fc3} $\ker(f)\subset\frm F\ld$.
\end{enumerate}
\end{fact}
\begin{proof}
\eqref{item:fc1}$\Rightarrow$\eqref{item:fc2} is trivial. Next assume that $m_1,...,m_n$ is a homogeneous basis of
$F\ld$ mapping to a minimal generating system of $M\ld$ under $f$, and let $d_1,...,d_n$ be the degrees of the
$m_i$. Then, if $x_1 m_1 + ... + x_n 
m_n\in\ker(f)$ for homogeneous $x_i\in R\ld$, we must have $x_i\in\frm$ for all $i$, since otherwise we had $f(m_i) =
x_i\ui\sum_{j\neq i} x_j f(m_j)$, contradicting the minimality of $\{f(m_j)\}$. This shows
\eqref{item:fc2}$\Rightarrow$\eqref{item:fc3}. It remains to prove \eqref{item:fc3}$\Rightarrow$\eqref{item:fc1}, so
assume $\ker(f)\subset\frm F\ld$ and $m_1,...,m_n$ is a homogeneous 
$R\ld$-basis of $F\ld$. If $\{f(m_i)\}$ was not minimal, there would be some $i$ and homogeneous
$x_1,...,\hat{x_i},...,x_n\in R\ld$ such that $f(m_i) = \sum_{j\neq i} x_j f(m_j)$, so $m_i - \sum_{j\neq i} x_j
m_j\in\ker(f)\setminus\frm F\ld$, contrary to our assumption. 
\end{proof}
\begin{definition}\label{def_minimalres}
If $f$ fulfills the equivalent conditions of Fact \ref{fact_freecover}, then we call it a \textit{free cover} of
$M\ld$. A free resolution $$(F\ua,\delta):\ ...\to F^{-2}\ld\to F^{-1}\ld\to F^0\ld\to M\ld\to 0$$ of $M\ld$ is called 
\textit{minimal} if  $F^n\to\image(\delta_n)$ is a free cover for all $n\in\ZZ$. 
\end{definition}
\begin{fact}\label{fact_freecoveressential}
Let $(R\ld,\frm)$ be a local graded ring and let $f: F\ld\to M\ld$ be a free cover. Then the only submodule $U\ld\subset
F\ld$ such that $f|_{U\ld}: U\ld\to M\ld$ is still surjective is $F\ld$ itself.
\end{fact}
\begin{proof}
If $f|_{U\ld}:U\ld\to M\ld$ is surjective, then $F\ld = U\ld+\ker(f)\subset U\ld+\frm F\ld$. Hence $F\ld/U\ld = \frm
(F\ld / U\ld)$, and by the graded version of Nakayama's Lemma \ref{lem_gradednakayama} we get $U\ld = F\ld$ as required.
\end{proof}
\begin{lem}[Nakayama]\label{lem_gradednakayama}
Let $(R\ld,\frm)$ be a local graded ring. If $M\ld$ is a finitely generated graded $R\ld$-module such that $M\ld = \frm
M\ld$, then $M\ld = 0$.
\end{lem}
\begin{proof}
Suppose on the contrary that $M\ld\neq 0$ and choose a minimal system of homogeneous generators $m_1,...,m_n$ of
$M\ld$ of degrees $d_i\in\ZZ$. As $M\ld = \frm M\ld$ by assumption, we can find homogeneous $x_i\in\frm_{d_1-d_i}$ such
that $m_1 = x_1 m_1 + ... + x_n m_n$. However, $1-x_1\in R_0\setminus\frm_0$ is invertible, and so we get $m_1 =
-(1-x_1)\ui (x_2 m_2 + ... + x_n m_n)$, contradicting the minimality of $\{m_i\}$. 
\end{proof}
Without proof we recall the following standard result.
\begin{fact}
Let $R\ld$ be a local graded ring and let $M\ld$ be a finitely generated graded $R\ld$-module. Then a free cover/a
minimal resolution of $M\ld$ exists and is unique up to non-canonical isomorphism.
\end{fact}

\begin{definition}\label{def_betti}
Let $R\ld$ be a local graded ring an $M\ld$ be a finitely generated $R\ld$-module with minimal free resolution
$F\ua\ld\to M\ld$. The \textit{Betti-numbers} of $M\ld$, denoted $\beta^i(M\ld)$, are defined as the ranks of the
free $R\ld$-modules $F^{-i}\ld$.
\end{definition}

Exactness of localization implies the following:
\begin{prop}\label{prop_freeresolutionlocalizes}
Let $(R\ld,\frm)$ be a local graded ring and $M\ld$ be a finitely generated graded $R\ld$-module with minimal free
resolution $F\ua\ld\to M\ld$. Then $F\ua_\frm\to M_\frm$ is a minimal free resolution of the $R_\frm$-module $M_\frm$. 
\end{prop}

Summarizing, we get the following theorem which will allow us to carry over
results from the ungraded setting to the graded one.
\begin{samepage}\begin{prop}\label{prop_gradedvsungraded}
Let $(R\ld,\frm)$ be a local graded ring and $M\ld$ be a finitely generated graded $R\ld$-module. Then the following
hold:
\begin{enumerate}
\item\label{item:betti} $\beta_{R\ld}^i(M\ld) = \beta^i_{R_\frm}(M_\frm)$ for all $i\in\ZZ_{\geq 0}$. 
\item\label{item:vanish} $M\ld=0$ if and only if $M_\frm=0$.
\item\label{item:projective} $M\ld$ is projective in $R\ld\mod$ if and only if $M\ld$ is free.
\item\label{item:projectivedim} $\projdim_{R\ld}M\ld =\projdim_{R_\frm} M_\frm$.
\item\label{item:dim} $\dim_{R\ld} M\ld = \dim_{R_\frm} M_\frm$.
\item\label{item:injdim} $\injdim_{R\ld} M\ld = \injdim_{R_\frm} M_\frm$.
\item\label{item:depth} $\depth_{R\ld} M\ld = \depth_{R_\frm} M_\frm$. 
\end{enumerate}
\end{prop}\end{samepage}
\begin{proof}
\eqref{item:betti} follows from Proposition \ref{prop_freeresolutionlocalizes}. \eqref{item:vanish} follows from
\eqref{item:betti} and the fact that $M\ld=0$ resp. $M_\frm=0$ if and only if $\beta^0_{R\ld}(M\ld)=0$
resp. $\beta^0_{R_\frm}(M_\frm)=0$. \eqref{item:projective} If $M\ld$ is projective, then $\beta^1_{R\ld}(M\ld) =
\rk_{k\ld} \tor^1_{R}(k\ld,M\ld)\ld = 0$, hence $M\ld$ is free. \eqref{item:projectivedim} follows 
from $\projdim_{R\ld} M\ld = \max\{k\in\ZZ_{\geq 0}\ |\ \beta^k_{R\ld} M\ld\neq 0\}$ and the analogous equation for
$M_\frm$. \eqref{item:dim}, \eqref{item:injdim} and \eqref{item:depth} were already shown in Propositions
\ref{prop_dim}, \ref{prop_injdim} and \ref{prop_depth}, respectively. 
\end{proof}

As an example of how to apply Proposition \ref{prop_gradedvsungraded}, we note the graded version of the well-known
formula of Auslander and Buchsbaum.
\begin{theorem}\label{thm_auslanderbuchsbaum}[Auslander-Buchsbaum formula]
Let $R\ld$ be a local graded ring and $M\ld$ be a finitely generated $R\ld$-module of finite projective dimension. Then
we have $$\projdim_{R\ld} M\ld = \depth_{R\ld} R\ld - \depth_{R\ld} M\ld.$$
\end{theorem}
Finally we recall the definition of a (maximal) Cohen-Macaulay module.
\begin{definition}\label{def_cm}
Let $R\ld$ be a local graded ring and $M\ld$ be a finitely generated $R\ld$-module. Then $M\ld$ is called
\textit{Cohen-Macaulay} if $\depth_{R\ld} M\ld = \dim_{R\ld} M\ld$. It is called \textit{maximal Cohen-Macaulay} if 
$\depth_{R\ld} M\ld = \dim_{R\ld} M\ld = \depth_{R\ld} R\ld$. The ring $R\ld$ is called Cohen-Macaulay if it is
Cohen-Macaulay as a graded module over itself, i.e. if $\depth_{R\ld} R\ld = \dim_{R\ld} R\ld$. 
\end{definition}

From Proposition \ref{prop_gradedvsungraded} we immediately get:
\begin{fact}
Let $R\ld$ be a local graded ring and $M\ld$ be a finitely generated graded $R\ld$-module. Then $M\ld$ is a (maximal)
Cohen-Macaulay module over $R\ld$ if and only if $M_\frm$ is a (maximal) Cohen-Macaulay module over $R_\frm$. 
\end{fact}

\section{The stabilization functor and matrix factorizations}\label{sec_stabilization}

\subsection{Semiorthogonal decomposition of the stable category of a Gorenstein ring}\label{subsec_stabilization}

In this section, we introduce a special class of Cohen-Macaulay rings, called Gorenstein rings, and recall a
semi-orthogonal decomposition of the category $R\ld\mod$ of finitely generated graded $R\ld$-modules into the subcategory
$\MCM(R\ld)$ of maximal Cohen-Macaulay modules and the subcategory $\fpd(R\ld)$ of modules of finite
projective dimension. 

\begin{definition}\label{def_gorenstein}
A local graded ring $R\ld$ is called \textit{Gorenstein} if $\injdim_{R\ld} R\ld<\infty$.
\end{definition}
It's not obvious from this definition that any Gorenstein ring is Cohen-Macaulay, however note the following
proposition. 
\begin{prop}\label{prop_fininjdim}
Let $R\ld$ be a local graded ring and $M\ld$ be a finitely generated graded $R\ld$-module such that
$\injdim_{R\ld} M\ld<\infty$. Then we have $$\dim_{R\ld} M\ld\leq \injdim_{R\ld} M\ld = \depth_{R\ld} R\ld.$$
In particular, Gorenstein local graded rings are Cohen-Macaulay.
\end{prop}
\begin{proof} This follows from Proposition \ref{prop_gradedvsungraded} and the corresponding ungraded version, see
  \cite[Theorem 3.1.17]{BrunsHerzog}.\end{proof} 

In a Gorenstein ring, we have the following very useful characterization of maximal Cohen-Macaulay modules (note that
part (b) is actually taken as the \textit{definition} of maximal Cohen-Macaulayness in \cite{Buchweitz}). For the
convenience of the reader we will sketch its proof.

\begin{prop}\label{prop_charmcm}
Let $R\ld$ be a Gorenstein local graded ring and $M\ld$ be a finitely generated graded
  $R\ld$-module. Then the following are equivalent:
\begin{enumerate}
\item\label{item:charmcm1} $M\ld$ is maximal Cohen-Macaulay.
\item\label{item:charmcm2} $\ext\ua_{R}(M\ld,R\ld)\ld = 0$ for all $k>0$.
\item\label{item:charmcm3} $M\ld$ is an arbitrarily high syzygy, i.e. for all $n>0$ there exists a finitely generated
  graded $R\ld$-module $N\ld$ such that $M\ld$ is an $n$-th syzygy of $N\ld$.
\item\label{item:charmcm4} $M\ld$ admits a projective coresolution to the right.
\end{enumerate}
\end{prop}

\begin{proof}
For \eqref{item:charmcm1}$\Leftrightarrow$\eqref{item:charmcm2} see \cite[Theorem
  3.3.10]{BrunsHerzog}. Further, we have \eqref{item:charmcm3}$\implies$\eqref{item:charmcm2} because of
  $\injdim_{R\ld} R\ld<\infty$. As \eqref{item:charmcm4}$\implies$\eqref{item:charmcm3} is trivial, it remains to prove
  \eqref{item:charmcm1}$\implies$\eqref{item:charmcm4}. For this, choose a finitely generated projective resolution
  $P\ua\ld\to M\ld$. The assumption that $\ext^\ast_R(M\ld,R\ld)\ld=0$ for $\ast>0$ implies that $M\ld\us\to
  \left(P\ua\ld\right)\us$ is a projective coresolution of $M\ld\us$, where $(-)\us :=
  \hom_{R}(-,R\ld)\ld$. This implies that $M\ld\us$ satisfies \eqref{item:charmcm4}, hence \eqref{item:charmcm2},
  and so dualizing a finitely generated projective resolution of $M\ld\us$ yields a projective coresolution for
  $M\ld\uss$. Condition \eqref{item:charmcm2} for $M\ld$ and $M\us\ld$ implies that
  $M\uss\ld\cong M\ld$ (the duality $(-)\us$ is exact on modules satisfying \eqref{item:charmcm2}, and the
  transformation $\id\to (-)\uss$ is an isomorphism on free, finitely generated modules), hence $M\ld$ satisfies
  \eqref{item:charmcm4} as claimed. 
\end{proof}

\begin{definition}\label{def_gorensteinprojective}
Let $R\ld$ be a Gorenstein local graded ring. A (not necessarily finitely generated) graded
  $R\ld$-module $M\ld$ is called \textit{Gorenstein projective} if it admits a projective coresolution. In view of
  Proposition \ref{prop_charmcm}, we denote the category of Gorenstein projectives by $\MCM^\infty(R\ld)$.
\end{definition}

\begin{rem}\label{rem_gorensteinproj}
For a detailed treatment of Gorenstein projective modules, see \cite[Chapter 3]{Christensen}. There it
  is proved that over a Gorenstein ring every module becomes Gorenstein projective when taking high enough syzygies; in
  fact, this 
  property characterizes Gorenstein rings. In other words, a local Noetherian ring $R$ is Gorenstein if and only if every
  module $M$ has finite Gorenstein projective dimension $\gdim_R M$, in beautiful analogy to Serre's criterion for
  regularity. In this case we have $$\gdim_R M\ =\ \depth_R R -
  \depth_R M$$ for every $R$-module $M$, generalizing the Auslander-Buchsbaum formula \ref{thm_auslanderbuchsbaum}. 
\end{rem}

\begin{prop}\label{prop_mcmfrobenius}
Let $R\ld$ be a Gorenstein local graded ring. Then $\MCM^\infty(R\ld)$ (respectively $\MCM(R\ld)$), equipped with the class
of short exact sequences in the usual sense, is a Frobenius category (see \cite{KellerDGKats}), and the
projective-injectives are precisely the (finitely generated) projective $R\ld$-modules.
\end{prop}
\begin{proof}
We only do the infinite case. The finitely generated case is treated analogously.

Denote by $\scE^{\infty}$ the class of short exact sequences in $\MCM^{\infty}(R\ld)$. We only check that
$(\MCM^{\infty}(R\ld),\scE^{\infty})$ has enough projectives and injectives, that projectives and
injectives coincide and that the class of projective-injectives equals the class of projective
$R\ld$-modules.

As $\injdim_{R\ld} R\ld<\infty$ we have $\ext^{>0}_R(M\ld,R\ld)=0$ for all $M\la\in\MCM^{\infty}(R\ld)$. Thus every 
projective $R\ld$-module is projective-injective in
$(\MCM^{\infty}(R\ld),\scE^{\infty})$. Moreover, for $M\ld\in\MCM^{\infty}(R\ld)$ the existence of a projective
coresolutions of $M\ld$ shows that $M\ld$ admits an embedding into a projective 
$R\ld$-module. It follows that $(\MCM^{\infty}(R\ld),\scE^{\infty})$ has enough injectives, and
that any injective object is a summand of a projective $R\ld$-module. Thus, the injectives in
$(\MCM^{\infty}(R\ld),\scE^{\infty})$ are precisely the projective $R\ld$-modules, and hence coincide with the
projective-injectives in $(\MCM^{\infty}(R\ld),\scE^{\infty})$.  
\end{proof}

\begin{definition}
We denote $\uMCM^{\infty}(R\ld)$ the stable category of the Frobenius category\break$(\MCM^\infty(R\ld),\scE^\infty)$ from
Proposition \ref{prop_mcmfrobenius}. In plain terms, the objects of $\uMCM^\infty(R\ld)$ are the objects of
$\MCM^\infty(R\ld)$, and for $M\ld,N\ld\in\MCM^\infty(R\ld)$ we have $$\uMCM^\infty(M\ld,N\ld)=\hom_{R\ld}(M\ld,N\ld) /
P(M\ld,N\ld),$$ where $P(M\ld,N\ld)$ consists of the morphisms factoring through a projective $R\ld$-module. This is a
full subcategory of the stable category $\ul{R\ld\Mod}$, defined in the same way.

Similarly, we denote $\uMCM(R\ld)$ the stable category of $(\MCM(R\ld),\scE)$, which is a full subcategory of the stable
category $\ul{R\ld\mod}$.
\end{definition}

The following proposition is the main statement of this section. For convenience we give a detailed proof, although
everything apart from the explicit construction in part (e) is contained in \cite{Buchweitz}.

\begin{prop}\label{prop_semiorthogonaldecomposition}
Let $R\ld$ be a Gorenstein local graded ring. 
\begin{enumerate}
\item\label{item_sod2} For $M\ld\in R\ld\mod$ we have $M\ld\in\MCM(R\ld)$ if and only if $\ext^k_{R}(M\ld,N\ld)\ld=0$ 
  for all $k>0$ and all $N\ld\in\fpd(R\ld)$.
\item\label{item_sod3} For $N\ld\in R\ld\mod$ we have $N\ld\in\fpd(R\ld)$ if and only if $\ext^k_{R}(M\ld,N\ld)\ld=0$
  for all $k>0$ and all $M\ld\in\MCM(R\ld)$.
\item\label{item_sod4} If $M\ld\in\MCM(R\ld)$ and $N\ld\in\fpd(R\ld)$, then $\hom_{\ul{R\ld\mod}}(M\ld,N\ld)=0$.
\item\label{item_sod5} For any finitely generated graded $R\ld$-module $M\ld$ there is an exact sequence
$$0\to P\ld\to N\ld\to M\ld\to 0,$$ where $N\ld\in\MCM(R\ld)$ and $P\ld\in\fpd(R\ld)$. 
\item\label{item_sod6} The inclusion $\ul{\MCM}(R\ld)\to\ul{R\ld\mod}$ has a right adjoint
  $\bfM:\ul{R\ld\mod}\to\ul{\MCM}(R\ld)$.
\end{enumerate}
\end{prop}
\begin{proof}
Let $M\ld\in\MCM(R\ld)$. Since $\ext^k_{R}(M\ld,R\ld)\ld=0$ for all $k>0$, we have
$\ext^k_{R}(M\ld,N\ld)\cong\ext^{k+n}_{R}(M\ld,\Omega^{n}N\ld)\ld$ for all $n\geq 0$. Now, if $N\ld\in\fpd(R\ld)$,
we have $\Omega^n N\ld=0$ for $n\gg 0$. This shows \eqref{item_sod2}.

Next we do \eqref{item_sod3}. By \eqref{item_sod2} we only have to show that any $N\ld\in R\ld\mod$ with
$\ext^k_{R}(M\ld,N\ld)\ld=0$ for all $k>0$ and $M\ld\in\MCM(R\ld)$ has finite projective dimension. For this, take a
graded free resolution $F\ua\ld\to N\ld$ of $N\ld$, and note $\Omega^n_{F\ua\ld} N\ld\in\MCM(R\ld)$ for all $n\gg 0$. Hence
\begin{align*}
\hom_{\ul{R\ld\mod}}(\Omega^nN\ld,\Omega^nN\ld) & =
\coker\big(\hom_{R\ld\mod}(\Omega^nN\ld,F^{-n}\ld)\to\hom_{R\ld\mod}(\Omega^nN\ld,\Omega^nN\ld)\big)\\
& \cong\ext^1_{R\ld}(\Omega^nN\ld,\Omega^{n-1}N\ld)\cong
...\cong\ext^n_{R\ld}(\Omega^nN\ld,N\ld)\\
& = 0.
\end{align*}
Here the first isomorphism follows from the fact that, since $F^{-n}\ld$ maps surjectively to $\Omega^n N\ld$, a
homomorphism $\Omega^n M\ld\to\Omega^nN\ld$ factors through some projective if and only if it factors through $F^{-n}\ld$. 

Therefore $\Omega^n_{F\ua} N\ld=0$ in $\ul{R\ld\mod}$ and $\Omega^n N\ld$ is projective, hence free. Point
\eqref{item_sod4} is similar: as 
$M\ld\in\MCM(R\ld)$ there exists $\Sigma M\in\MCM(R\ld)$ such that $M\ld\cong\Omega\Sigma M\ld$, and therefore
$\hom_{\ul{R\ld\mod}}(M\ld,N\ld)\cong\ext^1_{R\ld}(\Sigma M\ld,N\ld)=0$ as claimed. 

We now show \eqref{item_sod5}. Let $F\ua\ld\to M\ld$ be a free resolution of $M\ld$ and take $n\gg 0$ such that
$\Omega_{F\ua}^n 
M\ld\in\MCM(R\ld)$. Further, let $\Omega_{F\ua}^n M\ld\hookrightarrow P^{-n+1}\ld\to P^{-n+2}\ld\to ...\to P^0\ld$ be
the beginning of a free coresolution of $\Omega_{F\ua}^n M\ld$ in $\MCM(R\ld)$. This exists because $\MCM(R\ld)$ is a
Frobenius category, see Proposition \ref{prop_mcmfrobenius}. A small diagram chase using the injectivity of $R\ld$ in
$\MCM(R\ld)$ gives the following commutative 
diagram: 
\begin{equation}\label{eq:chase}\begin{tikzpicture}[description/.style={fill=white,inner sep=2pt},baseline=(current
    bounding box.center)]
    \matrix (m) [matrix of math nodes, row sep=2.5em,
                 column sep=2.5em, text height=1.5ex, text depth=0.25ex,
                 inner sep=0pt, nodes={inner xsep=0.3333em, inner ysep=0.3333em}]
    {
       \Omega^n_{F\ua\ld}M\ld & P^{-n+1}\ld & P^{-n+2}\ld & \cdots & P^0\ld & \Sigma^n_{P\ua\ld}\Omega^n_{F\ua\ld} M\ld
       & 0 \\
       \Omega^n_{F\ua\ld}M\ld & F^{-n+1}\ld & F^{-n+2}\ld & \cdots & F^0\ld & M\ld & 0\\
    };
    \draw[double distance=0.7mm] (barycentric cs:m-1-1=0.25,m-2-1=0.75) -- (barycentric cs:m-1-1=0.75,m-2-1=0.25);
    \draw[->] (m-1-1) -- (m-1-2);
    \draw[->] (m-1-2) -- (m-1-3);
    \draw[->] (m-1-3) -- (m-1-4);
    \draw[->] (m-1-4) -- (m-1-5);
    \draw[->] (m-1-5) -- (m-1-6);
    \draw[->] (m-1-6) -- (m-1-7);
    \draw[->] (m-2-1) -- (m-2-2);
    \draw[->] (m-2-2) -- (m-2-3);
    \draw[->] (m-2-3) -- (m-2-4);
    \draw[->] (m-2-4) -- (m-2-5);
    \draw[->] (m-2-5) -- (m-2-6);
    \draw[->] (m-2-6) -- (m-2-7);

    \draw[dashed,->] (barycentric cs:m-1-2=0.75,m-2-2=0.25) -- (barycentric cs:m-1-2=0.25,m-2-2=0.75);
    \draw[dashed,->] (barycentric cs:m-1-3=0.75,m-2-3=0.25) -- (barycentric cs:m-1-3=0.25,m-2-3=0.75);
    \draw[dashed,->] (barycentric cs:m-1-5=0.75,m-2-5=0.25) -- (barycentric cs:m-1-5=0.25,m-2-5=0.75);
    \draw[dashed,->] (barycentric cs:m-1-6=0.75,m-2-6=0.25) -- (barycentric cs:m-1-6=0.25,m-2-6=0.75);
\end{tikzpicture}\end{equation}
With the leftmost terms $\Omega^n_{F\ua}M\ld$ removed, the rows become complexes, where we put the entries $M\ld$ and 
$\Sigma^n_{P\ua}\Omega^n_{F\ua} M\ld$ in cohomological degree $1$. Then, the vertical maps constitute a morphism of
complexes $f: P\ua\to F\ua$, inducing isomorphisms on cohomology in every degree. Looking at 
the long exact cohomology sequence of the triangle $P\ua\to F\ua\to \cone(f)\to F\ua[1]$ we deduce that $\cone(f)_{<0}$ is
a finite free resolution of $$\ker(\cone(f)^0\ld\to\cone(f)^{1}\ld) =
\ker(\Sigma_{P\ua}^n\Omega_{F\ua}^nM\ld\oplus F^0\ld\twoheadrightarrow M\ld).$$ Since
$\Sigma_{P\ua}^n\Omega_{F\ua}^nM\ld\oplus F^0\ld$ is maximal Cohen-Macaulay, the claim follows.

Finally, part \eqref{item_sod6} is a formal consequence \eqref{item_sod2}-\eqref{item_sod5}: For each finitely generated
graded $R\ld$-module $M\ld$ choose an 
exact sequence $$0\to M^{\nfpd}\ld\to M^{\nmcm}\ld\to M\ld\to 0$$ as in \eqref{item_sod5}, i.e. $M^{\nfpd}\ld$ is of
finite projective dimension and $M^{\nmcm}\ld$ is maximal Cohen-Macaulay. Further, given a homomorphism $f: M\ld\to
N\ld$ it is easy to check that there is an extension to a commutative diagram 
\begin{equation*}\begin{tikzpicture}[description/.style={fill=white,inner sep=2pt}]
    \matrix (m) [matrix of math nodes, row sep=3em,
                 column sep=2.5em, text height=1.5ex, text depth=0.25ex,
                 inner sep=0pt, nodes={inner xsep=0.3333em, inner ysep=0.3333em}]
    {
0 & M^{\nfpd}\ld & M^{\nmcm}\ld & M\ld  & 0 \\
0 & N^{\nfpd}\ld & N^{\nmcm}\ld & N\ld  & 0 \\
    };
    \draw[->] (m-1-1) -- (m-1-2);
    \draw[->] (m-1-2) -- (m-1-3);
    \draw[->] (m-1-3) -- (m-1-4);
    \draw[->] (m-1-4) -- (m-1-5);

    \draw[->] (m-2-1) -- (m-2-2);
    \draw[->] (m-2-2) -- (m-2-3);
    \draw[->] (m-2-3) -- (m-2-4);
    \draw[->] (m-2-4) -- (m-2-5);

    \draw[->] (m-1-4) -- node[right,scale=0.75]{$f$} ($(m-2-4) + (0,3mm)$);
    \draw[->] (m-1-3) -- node[right,scale=0.75]{$\wt{f}$} ($(m-2-3) + (0,3mm)$);
    \draw[->] (m-1-2) -- ($(m-2-2) + (0,3mm)$);
\end{tikzpicture}\end{equation*}
Here, the class of the extension $\tilde{f}$ in the stable category is uniquely determined by $f$, as the difference of any
two extensions factors through $N^{\nfpd}\ld$, and any homomorphism from a maximal Cohen-Macaulay module to a
module of finite projective dimension is stably trivial by part \eqref{item_sod4}. This defines a functor
$\ul{R\mod}\to\uMCM(R\ld)$ which we claim to be the right adjoint to the inclusion functor 
$\uMCM(R\ld)\to\ul{R\mod}$. Indeed, let $M\ld$ be an arbitrary finitely generated graded $R\ld$-module, $N\ld$ a
maximal Cohen-Macaulay module and $f: N\ld\to M\ld$ be a homomorphism of graded modules. Then we have to see that, up to
stable equivalence, there is precisely one lifting $\tilde{f}: N\ld\to M^{\nmcm}\ld$ such that
\begin{equation*}\begin{tikzpicture}[description/.style={fill=white,inner sep=2pt}]
    \matrix (m) [matrix of math nodes, row sep=3em,
                 column sep=2.5em, text height=1.5ex, text depth=0.25ex,
                 inner sep=0pt, nodes={inner xsep=0.3333em, inner ysep=0.3333em}]
    {
0 & M^{\nfpd}\ld & M^{\nmcm}\ld & M\ld  & 0 \\
&&& N\ld \\
    };
    \draw[->] (m-1-1) -- (m-1-2);
    \draw[->] (m-1-2) -- (m-1-3);
    \draw[->] (m-1-3) -- (m-1-4);
    \draw[->] (m-1-4) -- (m-1-5);
    \draw[->] (m-2-4) -- node[right,scale=0.75]{$f$} (m-1-4);
    \draw[dashed,->] (m-2-4) -- node[below,scale=0.75]{$\wt{f}$} (m-1-3);
\end{tikzpicture}\end{equation*}
commutes. The uniqueness is clear, since the difference of any two such liftings factors through $M^{\nfpd}\ld$, and
$\hom_{\ul{R\mod}}(N\ld,M^{\nfpd}\ld)=0$ by \eqref{item_sod4}. For the existence, note that the only obstruction against
the existence of $\tilde{f}$ lies in $\ext^1_{R\ld\mod}(N\ld,M^{\nfpd}\ld)$, and this group is trivial by \eqref{item_sod2}.
\end{proof}

The proof of Proposition \ref{prop_semiorthogonaldecomposition} actually shows the following:

\begin{cor}
Let $R\ld$ be a Gorenstein local graded ring, and let $n\gg 0$ such that $\Omega^n$ maps $\ul{R\mod}$ to
$\uMCM(R\mod)$. Then the functor $$\Sigma^n\circ\Omega^n:\quad\ul{R\mod}\xrightarrow{\ \ \Omega^n\ \
}\uMCM(R\mod)\xrightarrow{\ \ \Sigma^n\ \ }\ul{R\mod}$$ together with the canonical map $\Sigma^n\Omega^n M\ld\to M\ld$
constructed in the proof of Proposition \ref{prop_semiorthogonaldecomposition}.\eqref{item_sod6} is right adjoint to the
inclusion functor $\uMCM(R\ld)\to\ul{R\ld\mod}$. 
\end{cor}

\begin{rem}
As every $R\ld$-module has finite Gorenstein-projective dimension (see Remark \ref{rem_gorensteinproj}),
  the proof of 
Proposition \ref{prop_semiorthogonaldecomposition} applies to show that $R\ld\Mod$ admits a semi-orthogonal
decomposition into the full subcategory $\uMCM^\infty(R\ld)$ of Gorenstein projective modules and the full subcategory
$\ufpd^\infty(R\ld)$ of modules of finite projective dimension.
\end{rem}

\subsection{Maximal Cohen-Macaulay modules on a graded hypersurface}

Now we specialize the results of the preceding section to the case where $R\ld = S\ld/(w)$ for a regular local graded
ring $S\ld$ and some $w\in S\ld\setminus\{0\}$. In this case it will turn out that $\Omega^2\cong\langle -d\rangle$ for
$d := \deg(w)$, which we then use to simplify the construction of the Cohen-Macaulay approximation functor
$\ul{R\mod}\to\MCM(R\ld)$ in the case of hypersurfaces.
\begin{definition}
A local graded ring $S\ld$ is called \textit{regular} if $\gldim(S\ld\Mod)<\infty$.
\end{definition}
\begin{prop}\label{prop_regcrit}
Let $(S\ld,\frm)$ be a local graded ring. Then the following are equivalent:
\begin{enumerate}
\item\label{item:regbig} $S\ld$ is regular, i.e. $\gldim(S\ld\Mod)<\infty$.
\item\label{item:regsmall} $\gldim(S\ld\mod)<\infty$.
\item\label{item:ressmall} $\projdim_{S\ld}(k\ld)<\infty$. 
\end{enumerate}
In particular, if $S\ld$ is regular and $\frp$ is a homogeneous prime in $S\ld$, then $S_{(\frp)}$ is regular.
\end{prop}
\begin{proof} 
The implications \eqref{item:regbig}$\Rightarrow$\eqref{item:regsmall}$\Rightarrow$\eqref{item:ressmall} are clear, so
we have to show $\gldim(S\ld\Mod)<\infty$ if $\projdim_{S\ld}(k\ld)<\infty$. By Proposition \ref{prop_injdim}, we have
$\injdim_{S\ld}(M\ld)\leq\projdim_{S\ld}(k\ld)$ for each finitely generated graded $S\ld$-module $M\ld$, hence
$\gldim(S\ld\mod)=\projdim_{S\ld}(k\ld)<\infty$. Finally, Baer's criterion implies that $\injdim_{S\ld}\projlim{i\in I}
M_i\leq\sup\limits_{i\in I} \injdim_{S\ld} M_i$ for any directed system $\{M_i\}_{i\in I}$, and as any graded
$S\ld$-module is a direct limit of finitely generated graded $S\ld$-modules, it follows that
$\injdim_{S\ld}(M\ld)<\infty$ for every graded $S\ld$-module $M\ld$, hence $\gldim(S\ld\Mod)<\infty$.

The second statement follows from the first applied to $S_{(\frp)}$, noting that $$\projdim_{S_{(\frp)}}(S_{(\frp)}/\frp
S_{(\frp)}) = \projdim_{S_{(\frp)}} ((S\ld /  \frp)_{(\frp)}) \leq \projdim_{S\ld}(S\ld/\frp) < \infty.$$This concludes
the proof.\end{proof}
 
\begin{prop}\label{prop_regulargradedungraded}
Let $(R\ld,\frm)$ be a local graded ring. Then $R\ld$ regular $\Leftrightarrow$ $R_\frm$ regular.
\end{prop}
\begin{proof}
This follows from Proposition \ref{prop_regcrit} together with $\projdim_{R_\frm} k_\frm = \projdim_{R\ld} k\ld$
(Proposition \ref{prop_gradedvsungraded}). 
\end{proof}

Usually, a local ring $(R,\frm)$ with residue class field $k := R/\frm$ is called regular if $\dim(R) = \dim_{k}
(\frm/\frm^2)$. By a famous Theorem of Serre, this is equivalent to $R\Mod$ being of finite global dimension. Note,
however, that the following theorem becomes more difficult to prove with our definition.
\begin{prop}\label{prop_stayregular}
Let $(S\ld,\frm)$ be a regular local graded ring and $f\in\frm\setminus\frm^2$ homogeneous. Then $S\ld/(f)$ is regular. 
\end{prop}
\begin{proof}
By Proposition \ref{prop_regulargradedungraded} it suffices to prove the proposition in the ungraded case, and this is
done in \cite[Proposition 2.2.2]{Avramov}. For convenience of the reader, we recall the proof in the Appendix, see
Proposition \ref{prop_stayregularappendix}.
\end{proof}

In case $w\in\frm^2$, the quotient ring $S\ld/(w)$ is still a Gorenstein ring.
\begin{prop}\label{prop_hypersurfacegorenstein}
Let $(S\ld,\frm)$ be a regular local graded ring and $w\in\frm$ be homogeneous. Then $S\ld/(w)$ is Gorenstein.
\end{prop}
\begin{proof}
More generally, if $(R\ld,\frm)$ is Gorenstein and $w\in\frm$ is homogeneous and not a zero
divisor, then $R\ld/\frm$ is Gorenstein. This follows
from \begin{align}\label{eq:exthypersurface}\ext_{R\ld/(w)}^{\ast}(M\ld,R\ld/(w))\ld\ \cong\
  \ext_{R\ld}^{\ast+1}(M\ld,R\ld)\ld\langle -d\rangle\end{align} 
for each finitely generated graded $R\ld/(w)$-module $M\ld$, where $d$ is the degree of $w$. To prove
\eqref{eq:exthypersurface} it suffices to do the case $\ast=0$, i.e.
\begin{align}\label{eq:exthypersurface2}\hom_{R\ld/(w)}(M\ld,R\ld/(w))\ =\ \hom_{R\ld}(M\ld,R\ld/(w))\ \cong\
  \ext^1_{R\ld}(M\ld,R\ld)\langle -d\rangle,\end{align}
because both sides of \eqref{eq:exthypersurface} are effaceable $\delta$-functors on $R\ld\mod$. The isomorphism
\eqref{eq:exthypersurface2} follows from applying $\ext^\ast_{R\ld}(M\ld,-)\ld$ to the exact sequence $$0\to
R\ld\langle -d\rangle\xrightarrow{\ w\ } R\ld\xrightarrow{\ \ \ }R\ld/(w)\to 0.$$ It remains to show that any regular
graded ring is a domain, which is done in Fact \ref{fact_regularisadomain}. 
\end{proof}
\begin{fact}\label{fact_regularisadomain}
Let $(S\ld,\frm)$ be a regular local graded ring. Then $S\ld$ is a domain.
\end{fact}
\begin{proof}
We divide the proof into three steps:
\begin{enumerate}
\item\label{item:assmin} Show that any associated prime of $S\ld$ is minimal.
\item\label{item:uniqueminimalprime} Show that there is precisely one minimal prime in $S\ld$.
\item\label{item:conclude} Conclude the proof.
\end{enumerate}
\eqref{item:assmin}: If $\frp\in\Ass(S\ld)$, then $\frp S_{(\frp)}\in\Ass(S_{(\frp)})$. Replacing $S\ld$ by $S_{(\frp)}$
(which is again regular by Proposition \ref{prop_regcrit}), it is sufficient to show that for $(S\ld,\frm)$
regular and $\frm\in\Ass(S\ld)$ we have $\frm=0$. Pick $x\in S\ld\setminus\{0\}$ homogeneous with 
$\frm = \Ann_{S\ld}(x)$. Then $x M\ld = 0$ for each $M\ld$ which can be embedded into $\frm^{\oplus k}$ for some $k$, and in
particular no such $M\ld\neq 0$ can be projective. Any syzygy in a minimal free resolution embeds into some
$\frm^{\oplus k}$, so it follows that any non-free finitely generated module has infinite projective dimension. As
$\gldim(S\ld\mod)<\infty$, we conclude that any finitely generated module is free; in particular $\frm$ is 
free, contradicting the fact that $x$ acts trivially on $\frm$.

\eqref{item:uniqueminimalprime}: For any additive function on $\mu: K_0(S\ld\mod)\to\ZZ$ on $S\ld\mod$ and any $M\ld$ we
have $\mu([M\ld]) = \mu([S\ld])\cdot\chi([M\ld])$, where $\chi: K_0(S\ld\mod)\to\ZZ$ is the Euler characteristic. Hence
$\hom_{\ZZ}(K_0(S\ld\mod),\ZZ)=\ZZ\langle\chi\rangle$. On the other hand, let $\frp$ be a minimal homogeneous prime in
$S\ld$. Then the assignment $M\ld\mapsto \length_{S_{(\frp)}}(M_{(\frp)})$ defines an additive function $\mult_\frp:
K_0(S\ld\mod)\to\ZZ$ satisfying $\mult_\frp([S\ld/\frp])=\length_{S_{(\frp)}}(S_{(\frp)}/\frp S_{(\frp)})=1$ but
$\mult_\frp([S\ld/\frq])=0$ for any prime $\frq$ not containing $\frp$. This implies that there is precisely one minimal
prime $\frp$ in $S\ld$, as claimed. 

\eqref{item:conclude}: If $\frp$ denotes the unique associated prime in $S\ld$, then $\frp$ coincides with the ideal
of zero divisors, and hence the localization map $S\ld\to S_{(\frp)}$ is injective. As the proof of
\eqref{item:assmin} shows $\frp S_{(\frp)} = 0$, it follows that $\frp = 0$, and hence $S\ld$ is a domain.
\end{proof}

\begin{prop}\label{prop_trivsingcat}
Let $R\ld$ be a regular local graded ring. Then $\uMCM^{(\infty)}(R\ld)=0$. 
\end{prop}
\begin{proof}
By Proposition \ref{prop_charmcm} resp. Definition \ref{def_gorensteinprojective} any $M\ld\in\MCM^{(\infty)}(R\ld)$ 
can be written as an arbitrarily high syzygy, i.e. for all $n\in\NN$ there exists some $N\ld\in R\ld\Mod$ such that
$M\ld\cong\Omega^n M\ld$. By assumption, $\gldim(R\ld)<\infty$, and so taking $n>\gldim(R\ld)$ shows that $M\ld$ is
projective and hence vanishes in $\uMCM^{(\infty)}(R\ld)$. 
\end{proof}

Now, we fix $w\in\frm\setminus\{0\}$ (possibly in $\frm^2$) and consider maximal Cohen-Macaulay modules over the
quotient singularity $R\ld := S\ld / (w)$ which is Gorenstein by Proposition \ref{prop_hypersurfacegorenstein}. If
$M\ld\in\MCM(R\ld)$, then considering $M\ld$ as a module over 
$S\ld$ the Auslander-Buchsbaum formula \ref{thm_auslanderbuchsbaum} yields 
$$\projdim_{S\ld}(M\ld) = \depth_{S\ld}(S\ld) - \depth_{S\ld} (M\ld) = \dim(S\ld) - \depth_{R\ld}(M\ld)= 1.$$ Hence,
there is an exact sequence of $S\ld$-modules $0\to P\ld\mor{\alpha} Q\ld\to M\ld\to 0$, where 
$P\ld$ and $Q\ld$ are projective, hence free. As $w\cdot M\ld = \{0\}$, we get $w\cdot Q\ld\subset\image(\alpha)$, and
therefore we can choose $\beta\in\hom_{R\ld}(Q\ld,P\ld)_{-\deg(w)}$ such that $\alpha\beta = w\cdot
\id_{P\ld}$. Applying $\beta$ from the left yields $\beta\alpha\beta=w\beta$, so the injectivity of $\beta$ yields
$\beta\alpha = w\cdot \id_{P\ld}$. Hence, we end up with what is called a \textit{matrix factorization} of type
$(S\ld,w)$:  

\begin{definition}\label{def_gradedmf}
Let $S\ld$ be a regular local graded ring and $w\in\frm$ be homogeneous of degree $d>0$. 
\begin{enumerate}
\item A \textit{graded matrix factorization of type $(S\ld,w)$} is a sequence $M^0\ld\mor{f}M^{-1}\ld\mor{g} M^0\ld$ of the
  following form:
\begin{enumerate}
\item $M^0\ld$ and $M^{-1}\ld$ are free (not necessarily finitely generated) graded $S\ld$-modules.
\item $f$ is a homomorphism of graded $S\ld$-modules of degree $d$.
\item $g$ is a homomorphism of graded $S\ld$-modules of degree $0$.
\item $gf = w\cdot\id_{M^0\ld}$ and $fg = w\cdot\id_{M^{-1}\ld}$. 
\end{enumerate}
The element $w$ is called the \textit{potential} of the matrix factorization.
\item A \textit{morphism of graded matrix
    factorizations} $$\left(M^0\ld\mor{f}M^{-1}\ld\mor{g}M^0\ld\right)\quad\longrightarrow\quad 
\left(N^0\ld\mor{f\p}N^{-1}\ld\mor{g\p}N^0\ld\right)$$ is a pair $(\alpha,\beta)$ of morphisms of graded
$R\ld$-modules $\alpha: M^0\ld\to N^0\ld$ and $\beta: M^{-1}\ld\to N^{-1}\ld$ such that
\begin{equation*}\begin{tikzpicture}[description/.style={fill=white,inner sep=2pt}]
    \matrix (m) [matrix of math nodes, row sep=3em,
                 column sep=2.5em, text height=1.5ex, text depth=0.25ex,
                 inner sep=0pt, nodes={inner xsep=0.3333em, inner ysep=0.3333em}]
    {
      M^0\ld & M^{-1}\ld & M^0\ld \\
      N^0\ld & N^{-1}\ld & N^0\ld \\
    };
    \draw[->] (m-1-1) -- node[above,scale=0.75] {$f$} (m-1-2);
    \draw[->] (m-1-2) -- node[above,scale=0.75] {$g$} (m-1-3);
    \draw[->] (m-2-1) -- node[above,scale=0.75] {$f\p$} (m-2-2);
    \draw[->] (m-2-2) -- node[above,scale=0.75] {$g\p$} (m-2-3);

    \draw[->] (m-1-1) -- node[description,scale=0.75]{$\alpha$} (m-2-1);
    \draw[->] (m-1-2) -- node[description,scale=0.75]{$\beta$} (m-2-2);
    \draw[->] (m-1-3) -- node[description,scale=0.75]{$\alpha$} (m-2-3);
\end{tikzpicture}\end{equation*}
commutes.

The category of graded matrix factorizations of type $(S\ld,w)$ with morphisms of graded matrix factorizations is
denoted $\MF^\infty(S\ld,w)$. The full subcategory of graded matrix factorizations $M^0\ld\to M^{-1}\ld\to M^0\ld$ with
$M^0\ld, M^{-1}\ld$ finitely generated is denoted by $\MF(S\ld,w)$. 
\item A morphism $(\alpha,\beta)$ as above is called \textit{nullhomotopic}, if there are morphisms of graded
  $R\ld$-modules $D^0: M^0\ld\to N^{-1}\ld$ and $D^{-1}: M^{-1}\ld\to N^0\ld$ of degree $0$ and $-d$, respectively, such
  that $g\p D^0 + D^{-1} f = \alpha$ and $f\p D^{-1} + D^0 g = \beta$. Two morphisms of graded matrix factorizations are
  called \textit{homotopic} if their difference is nullhomotopic.

  The quotient of $\MF^\infty(S\ld,w)$ and $\MF(S\ld,w)$ with respect to the homotopy relation is called
  \textit{homotopy category of matrix factorizations of type $(S\ld,w)$} and is denoted $\HMF^\infty(S\ld,w)$ and
  $\HMF(S\ld,w)$, respectively.
\item If $S\ld\mor{\iota} T\ld$ is a local homomorphism of regular local graded rings such that $T\ld$ is free over
  $S\ld$, any matrix factorization of type $(T\ld,\iota(w))$ can be considered as a matrix factorization of type
  $(S\ld,w)$. This gives \textit{restriction functors} $\MF(T\ld,\iota(w))\to\MF(S\ld,w)$ and
  $\HMF^\infty(T\ld,\iota(w))\to\HMF^\infty(S\ld,w)$ which we will denote by $(-)\downarrow^{T\ld}_{S\ld}$ Note that in
  our application $T\ld$ will usually be of infinite 
  rank over $S\ld$, so that $\HMF(T\ld,\iota(w))$ is \textit{not} mapped to $\HMF(S\ld,w)$ under the restriction functor.
\end{enumerate}
\end{definition}

\begin{rem}\label{rem_pretriag}
The category of graded matrix factorizations has a natural pretriangulated dg-enrichment, giving rise to the homotopy
category just defined. We will describe this now; the reader may skip this on first reading. 

Given a matrix factorizations $M^0\ld\mor{f}M^{-1}\ld\mor{g}M^0\ld$, let us agree on writing $M\ua\ld$ for the
sequence $$...\to M^0\ld\langle -(k+1)d\rangle\xrightarrow{f} M^{-1}\ld\langle
-kd\rangle\xrightarrow{g} M^0\ld\langle -kd\rangle\xrightarrow{f} M^{-1}\ld\langle
-(k-1)d\rangle\to ...,$$ where the maps increase the cohomological grading and $M^0\ld = M^0\ld\langle 0\rangle$ is
placed in cohomological degree $0$. Note that this is compatible with the previous meaning of $M^0\ld$ and
$M^{-1}\ld$. Further, let us call the ``differential'' on $M\ua\ld$ simply by $\delta$, so that $\delta^2 =
w\cdot\id_{M\ua\ld}$. Now, given another graded matrix factorization $N^0\ld\mor{f\p}N^{-1}\ld\mor{g\p}N^0\ld$ with
corresponding complex $N\ua\ld$, a  
\textit{morphism of matrix factorizations of degree $k$} between $M$ and $N$ is a family $\{\alpha_n\}_{n\in\ZZ}$ of
homomorphisms of graded modules $\alpha_n: M^n\ld\to N^{n+k}\ld$ such that $\alpha_{n+2} = \alpha_n\langle d\rangle$
under the equalities $M^{n+2}\ld = M^n\langle d\rangle\ld$ and $N^{n+k+2}\ld = N^{n+k}\langle d\rangle\ld$. Given such a
morphism $\alpha$ of degree $k$, we can define its differential $\differential\alpha$ by $(\differential\alpha)_n := \delta_N\alpha_n +
(-1)^{k+1}\alpha_{n+1}\delta_M$. This is a homomorphism of degree $k+1$. Note that this construction is completely
analogous to the construction of the complex of graded homomorphisms between two complexes. What is remarkable is that
even though we only have $\delta^2=w\cdot\id$ instead of $\delta^2=0$, taking twice the differential of a graded
morphism between matrix factorizations still gives the zero map: 
\begin{align*}
(\differential^2\alpha)_n & = \delta_N(\differential\alpha)_n + (-1)^{k+2}(\differential\alpha)_{n+1}\delta_M\\
& = \delta_N(\delta_N\alpha_n + (-1)^{k+1} \alpha_{n+1}\delta_M) + (-1)^k(\delta_N\alpha_{n+1} +
(-1)^{k+1}\alpha_{n+2}\delta_M)\delta_M\\
& = (-1)^k\left(\delta_N^2\alpha_n - \alpha_{n+2}\delta_M^2\right)\\
& = 0.
\end{align*}
Thus, the essential thing is that there is a degree $d$ element of the center of $R\ld\mod$, namely the multiplication
by $w$, such that $\delta^2$ is equal to the action of this element. 

Summing up, we have constructed for each pair of graded matrix factorizations a complex of graded morphisms between
them. The $0$-cocycles in this complex are precisely the morphisms of graded matrix factorizations as defined in
Definition \ref{def_gradedmf}, and a morphism is a $0$-boundary if and only if it is nullhomotopic. Therefore, we obtain
a natural dg-enhancement $\MF^{\infty}\ldg(S\ld,w)$ of $\MF^\infty(S\ld,w)$, such that the associated homotopy category
$\Ho(\MF^{\infty}\ldg(S\ld,w))$ equals $\HMF^\infty(S\ld,w)$.  

The dg-category $\MF^\infty\ldg(S\ld,w)$ is particularly nice in the sense that for each object
$X\in\MF^\infty\ldg(S\ld,w)$ and each morphism $f\in\MF^\infty\ldg(S\ld,w)(X,Y)_0 = \MF^\infty(S\ld,w)(X,Y)$ the
functors $$\MF^\infty\ldg(S\ld,w)(-,X)[k]\quad\text{\ and\ }\quad\cone\left[\MF^\infty\ldg(S\ld,w)(-,f)\right]$$ are
representable by objects in $\MF^\infty\ldg(S\ld,w)$. This means that there are objects $X[k]\in\MF^\infty\ldg(S\ld,w)$
and $\cone(f)\in\MF^\infty\ldg(S\ld,w)$ such that for each $Z\in\MF^\infty\ldg(S\ld,w)$ there are natural isomorphisms of
complexes 
\begin{align}\label{eq:shift}\MF^\infty\ldg(S\ld,w)(Z,X[k])\cong\MF^\infty\ldg(S\ld,w)(Z,X)[k]\end{align} 
and \begin{align}\label{eq:cone}\MF^\infty\ldg(S\ld,w)(Z,\cone(f))\cong\cone\left[\MF^\infty\ldg(S\ld,w)
    (Z,X)\xrightarrow{f\circ -} \MF^\infty \ldg(S\ld,w)(Z,Y)\right]\end{align}
A dg-category satisfying these two representability conditions is called \textit{pretriangulated}, and the homotopy
category of a pretriangulated dg-category is canonically triangulated (see \cite[Section 2]{SchwedeTorsion}).

It remains to check that $\MF^\infty\ldg(S\ld,w)$ indeed satisfies the above representability conditions. Both the shift
and the cone can be constructed as for usual complexes, and the verification of \eqref{eq:shift} and \eqref{eq:cone} is
just a long and tedious computation. As we do not want to dig too deep into these things, we content ourself by giving
the definitions of shift and cone. For the shift, we put
\begin{align*}
\left(M^0\ld\mor{f} M^{-1}\ld\mor{g} M^0\ld\right)[1]\quad:=\quad M^{-1}\ld\langle d\rangle\xrightarrow{-g}
M^0\ld\xrightarrow{-f} M^{-1}\ld\langle d\rangle,
\end{align*}
and given a morphism
\begin{equation*}\begin{tikzpicture}[description/.style={fill=white,inner sep=2pt}]
    \matrix (m) [matrix of math nodes, row sep=3em,
                 column sep=2.5em, text height=1.5ex, text depth=0.25ex,
                 inner sep=0pt, nodes={inner xsep=0.3333em, inner ysep=0.3333em}]
    {
       M\ld && M^0\ld & M^{-1}\ld & M^0\ld\\ N\ld && N^0\ld & N^{-1}\ld & N^0\ld\\
    };
    \draw[->] (m-1-1) -- node[scale=0.75,description]{$(\alpha,\beta)$} (m-2-1);

    \draw[->] (m-1-3) -- node[scale=0.75,above]{$f$}   (m-1-4);
    \draw[->] (m-2-3) -- node[scale=0.75,below]{$f\p$} (m-2-4);
    \draw[->] (m-1-4) -- node[scale=0.75,above]{$g$}   (m-1-5);
    \draw[->] (m-2-4) -- node[scale=0.75,below]{$g\p$} (m-2-5);
    
    \draw[->] (m-1-3) -- node[scale=0.75,description]{$\alpha$} (m-2-3);
    \draw[->] (m-1-4) -- node[scale=0.75,description]{$\beta$}  (m-2-4);
    \draw[->] (m-1-5) -- node[scale=0.75,description]{$\alpha$} (m-2-5);
\end{tikzpicture}\end{equation*}
of graded matrix factorizations, we define $\cone(\alpha,\beta)$ as the factorization
\begin{align}\label{eq:defcone}
\begin{tikzpicture}[baseline=-1mm,description/.style={fill=white,inner sep=2pt}]
    \matrix (m) [matrix of math nodes, row sep=3em,
                 column sep=5em, text height=1.5ex, text depth=0.25ex,
                 inner sep=0pt, nodes={inner xsep=0.3333em, inner ysep=0.3333em}]
    {
      N^0\ld\oplus M^{-1}\ld\langle d\rangle \pgfmatrixnextcell N^{-1}\ld\oplus M^0\ld \pgfmatrixnextcell N^0\ld\oplus
      M^{-1}\ld\langle d\rangle\\ 
    };
    \draw[->] (m-1-1) -- node[above,scale=0.75]{$\begin{pmatrix} f\p & \beta \\ 0 & -g\end{pmatrix}$} (m-1-2);
    \draw[->] (m-1-2) -- node[above,scale=0.75]{$\begin{pmatrix} g\p & \alpha \\ 0 & -f\end{pmatrix}$} (m-1-3);
\end{tikzpicture}
\end{align}

Note that, in contrast to the situation in the ungraded case, neither the dg-category $\MF^\infty\ldg(S\ld,w)$ nor the
triangulated category $\HMF^\infty(S\ld,w)$ are $2$-periodic! Instead, we have $[2]\cong\langle d\rangle $ on 
$\MF^\infty\ldg(S\ld,w)$ and $\HMF^\infty(S\ld,w)$.
\end{rem}

Our motivation for studying matrix factorizations was that for any maximal Cohen-Macaulay module $M\ld$ over
$R\ld := S\ld/(w)$ we constructed a graded matrix factorization $M^0\ld\mor{f} M^{-1}\ld\mor{g} M^0\ld$ of type
$(S\ld,w)$ such that $M\ld\cong\coker(M^{-1}\ld\mor{g} M^0\ld)$. Indeed, this construction yields a very close relationship
between matrix factorizations of type $(S\ld,w)$ and maximal Cohen-Macaulay over $S\ld/(w)$, as we shall see now:

\begin{theorem}\label{thm_hmfmcm}
Let $S\ld$ be a regular local graded ring, $w\in\frm\setminus\{0\}$ homogeneous of degree $d$ and $R\ld :=
  S\ld / (w)$. Then the functor
\begin{equation*}\begin{tikzpicture}[description/.style={fill=white,inner sep=2pt}]
    \matrix (m) [matrix of math nodes, row sep=1em,
                 column sep=3.5em, text height=1.5ex, text depth=0.25ex,
                 inner sep=0pt, nodes={inner xsep=0.3333em, inner ysep=0.3333em}]
    {
       \MF^\infty(S\ld,w) && R\ld\Mod\\
       \left(M^0\ld\mor{f} M^{-1}\ld\mor{g} M^0\ld\right) && \coker(g)\\
    };
    \node (lau) at ($(m-2-1.east) +  (4mm,0)$) {};
    \node (rau) at ($(m-2-3.west) +  (-4mm,0)$) {};

    \node (la) at (lau |- m-1-1) {};
    \node (ra) at (rau |- m-1-1) {};

    \draw[->] (la) -- node[above,scale=0.75]{$\coker$} (ra);
    \draw[|->] (lau) -- (rau);
\end{tikzpicture}\end{equation*}
induces a fully faithful functor $\coker: \HMF^\infty(S\ld,w)\to\ul{R\ld\Mod}.$ The essential image of
$\HMF^{(\infty)}(S\ld,w)$ under $\coker$ equals $\uMCM^{(\infty)}(R\ld)$, and we get an equivalence of triangulated
categories $$\coker:\ \HMF^{(\infty)}(S\ld,w)\ \ \cong\ \ \uMCM^{(\infty)}(R\ld).$$
\end{theorem}
\begin{proof}
The proof that $\coker: \HMF^\infty(S\ld, w)\to\ul{R\Mod}$ is fully faithful is just some diagram chasing,
  so we skip it. See for example \cite{OrlovGrad}. It remains to show that the essential image of
  $\HMF^{(\infty)}(S\ld,w)$ 
is $\uMCM^\infty(R\ld)$. We will do the finitely generated case only, but the proof applies verbatim to the
Gorenstein-projective case as well. 

Let $M^0\ld\mor{f} M^{-1}\ld\mor{g} M^0\ld$ be a graded matrix factorization of type $(S\ld,w)$. Since $w\cdot M^0 =
\image(gf)\subset\image(g)$ the module $K\ld := \coker(g)$ is annihilated by $w$, and therefore can be considered as a
graded module over $R\ld := S\ld/(w)$. Furthermore, the sequence
\begin{align*}...\to M^0\ld/wM^0\ld\langle
  -d\rangle\xrightarrow{\ f\ } M^{-1}\ld/wM^{-1}\ld\xrightarrow{\ g\ } M^0\ld/wM^0\ld\xrightarrow{\ f\ }
  M^{-1}\ld/wM^{-1}\ld\langle d\rangle\to...\end{align*} is exact, and every second syzygy is isomorphic to $K\ld$, hence
$K\ld$ is maximal Cohen-Macaulay (see Proposition \ref{prop_charmcm}). This shows that $\coker$ restricts to a fully
faithful functor $\HMF(S\ld,w)\to\uMCM(R\ld)$. The proof of the essential surjectivity of this functor was already shown
in the beginning of this section; however, we will now describe a proof which doesn't use the Auslander-Buchsbaum
formula \ref{thm_auslanderbuchsbaum} and applies to the Gorenstein-projective case as well.

We already know from the beginning of the section that we only have to show that any $M\ld\in\MCM(R\ld)$
  satisfies $\projdim_{S\ld} M\ld \leq 1$. Choose a  projective coresolution $M\ld\to P^0\ld\to
P^1\ld\to ...$ of $M\ld$ and let $Q^n\ld$ be the $n$-th syzygy of $P\ua\ld$. Then, since  
$\projdim_{S\ld} R\ld = 1$, we have $\ext^k_{S}(M\ld,N\ld)\ld\cong\ext^{k+n}_{S} (Q^n\ld,N\ld)\ld$ for all
$S\ld$-modules $N\ld$, $k>1$ and $n>0$ by dimension shifting. Choosing $n\gg 0$ such that $k+n>\injdim_{S\ld} N\ld$ we
conclude that $\ext^k_{S}(M\ld,N\ld)\ld=0$ for all $k>1$ and all $N\ld$, hence $\projdim_{S\ld} M\ld\leq 1$ as claimed. 
\end{proof}

Now we can define the stabilization functor, which will be our main tool for studying Khovanov-Rozansky homology.

\begin{definition}
Let $S\ld$ be a regular local graded ring, $w\in\frm\setminus\{0\}$ homogeneous and $R\ld := S\ld / (w)$. The
\textit{stabilization functor} $(-)\stab{w}: R\ld\mod\to\HMF(S\ld,w)$ is defined as the composition
$$R\ld\mod\xrightarrow{\ \ \text{can}\ \ }\ul{R\ld\mod}\xrightarrow{\ \ \bfM\ \ }\ul{\MCM}(R\ld)\xrightarrow{\ \
  \coker\ui\ \ }\HMF(S\ld,w)\hookrightarrow\HMF^\infty(S\ld,w).$$ 
\end{definition}

\begin{rem}
One can extend the stabilization functor to a functor $$(-)\stab{w}:\ R\ld\Mod\ \xrightarrow{\ \ }\
\HMF^\infty(S\ld,w)$$ using the adjoint $\ul{R\ld\Mod}\to\uMCM^\infty(R\ld)$ constructed through the non finitely
generated analogue of Proposition \ref{prop_semiorthogonaldecomposition}.\end{rem}

Next we want to make the stabilization functor explicit. First, note the following fact which follows immediately from
Proposition \ref{prop_semiorthogonaldecomposition} and $\Omega^2\cong\langle -d\rangle$ on $\uMCM(R\ld)$.

\begin{fact}
For any finitely generated $R\ld$-module $M\ld$ we have 
$$\bfM(M\ld)\ \cong\ \Sigma^{2n}\Omega^{2n} M\ld\ \cong\ \Omega^{2n}M\ld\langle nd\rangle,$$ where $n\gg 0$ is
chosen in such a way that $\Omega^{2n}M\ld$ is maximal Cohen-Macaulay.
\end{fact}
\begin{rem}\label{rem_praecounit}
It is not clear (at least to the author) what the counit map $$\Omega^{2n}M\ld\langle
nd\rangle\cong\bfM(M\ld)\to M\ld$$ should look like. Later we will construct for each $M\ld$ a special $R\ld$-free
resolution with respect to which the map $\Omega^{2n}M\ld\langle nd\rangle\to M\ld$ can be made explicit. See Remark
\ref{rem_explicitcounit}.  
\end{rem}

The following proposition explains the name 'stabilization functor':

\begin{prop}\label{prop_getstab}
Let $M\ld$ be a finitely generated graded $R\ld$-module and $F\ua\ld\to M\ld$ a free resolution of $M\ld$ (not
necessarily of finite rank) with the following properties:
\begin{enumerate}
\item $F\ua\ld$ is eventually $2$-periodic: there exists $n\gg 0$ such that for each $k\leq -2n$ there is a commutative
  diagram
\begin{equation*}\begin{tikzpicture}[description/.style={fill=white,inner sep=2pt}]
    \matrix (m) [matrix of math nodes, row sep=3em,
                 column sep=6em, text height=1.5ex, text depth=0.25ex]
    {
       F^k\ld\langle-d\rangle & F^{k+1}\ld\langle -d\rangle & F^{k+2}\ld\langle -d\rangle\\
       F^{k-2}\ld & F^{k-1}\ld & F^k\ld\\
    };
    \draw[->] (m-1-1) -- node[scale=0.75,above]{$\partial_k$} (m-1-2);
    \draw[->] (m-1-2) -- node[scale=0.75,above]{$\partial_{k+1}$} (m-1-3);
    \draw[->] (m-2-1) -- node[scale=0.75,above]{$\partial_{k-2}$} (m-2-2);
    \draw[->] (m-2-2) -- node[scale=0.75,above]{$\partial_{k-1}$} (m-2-3);
    \draw[->] (m-1-1) -- node[scale=0.75,left]{$\cong$} (m-2-1);
    \draw[->] (m-1-2) -- node[scale=0.75,left]{$\cong$} (m-2-2);
    \draw[->] (m-1-3) -- node[scale=0.75,left]{$\cong$} (m-2-3);
\end{tikzpicture}\end{equation*}
where the vertical maps are isomorphisms.
\item The $2$-periodic part $F^{-2n}\ld\langle -d\rangle\to F^{-2n-1}\ld\to F^{-2n}\ld$ of $F\ua\ld$ can be lifted to a
  matrix factorization $M^0\ld\mor{f} M^{-1}\ld\mor{g} M^0\ld$, i.e. there is a commutative diagram
\begin{equation*}\begin{tikzpicture}[baseline=(current bounding  box.center), description/.style={fill=white,inner sep=2pt}]
    \matrix (m) [matrix of math nodes, row sep=3em,
                 column sep=6em, text height=1.5ex, text depth=0.25ex]
    {
       F^{-2n}\ld\langle -d\rangle & F^{-2n-1}\ld & F^{-2n}\ld \\
       M^0\ld\langle -d\rangle\otimes_{S\ld} R\ld & M^{-1}\ld\otimes_{S\ld} R\ld & M^0\ld\otimes_{S\ld} R\ld\\
    };
    \draw[->] (m-1-1) -- node[scale=0.75,left]{$\cong$} (m-2-1);
    \draw[->] (m-1-2) -- node[scale=0.75,left]{$\cong$} (m-2-2);
    \draw[->] (m-1-3) -- node[scale=0.75,left]{$\cong$} (m-2-3);
    \draw[->] (m-1-1) -- (m-1-2);
    \draw[->] (m-1-2) -- (m-1-3);
    \draw[->] (m-2-1) -- node[scale=0.75,above]{$f\otimes\id$} (m-2-2);
    \draw[->] (m-2-2) -- node[scale=0.75,above]{$g\otimes\id$} (m-2-3);
\end{tikzpicture}\end{equation*}
where the vertical maps are isomorphisms.
\end{enumerate}
Then there is an isomorphism in $\HMF$ $$M\ld\stab{w}\ \ \cong\ \ \left(M^0\ld\mor{f} M^{-1}\ld\mor{g}
  M^0\ld\right)\langle nd\rangle.$$ 
\end{prop}
\begin{proof}
We have the following diagram which commutes up to canonical natural isomorphisms
\begin{equation*}\begin{tikzpicture}[description/.style={fill=white,inner sep=2pt}]
    \matrix (m) [matrix of math nodes, row sep=3em,
                 column sep=2.5em, text height=1.5ex, text depth=0.25ex,
                 inner sep=0pt, nodes={inner xsep=0.3333em, inner ysep=0.3333em}]
    {
       R\ld\mod & \ul{R\ld\mod} & \uMCM(R\ld) & \HMF(S\ld,w)\\
       && \uMCM^\infty(R\ld) & \HMF^\infty(S\ld,w)\\
    };
    \draw[->] (m-1-1) -- (m-1-2);
    \draw[->] (m-1-2) -- node[scale=0.75,above]{$\Omega^{2n}$} (m-1-3);
    \draw[->] (m-1-4) -- node[scale=0.75,above]{$\coker$} node[scale=0.75,below]{$\cong$}(m-1-3);
    \draw[->] (m-2-4) -- node[scale=0.75,above]{$\coker$} node[scale=0.75,below]{$\cong$}(m-2-3);

    \draw[->] (m-2-4) -- node[scale=0.75,above]{$\coker$} node[scale=0.75,below]{$\cong$}(m-2-3);
    \draw[->] (m-2-4) -- node[scale=0.75,above]{$\coker$} node[scale=0.75,below]{$\cong$}(m-2-3);

    \draw[right hook->] (m-1-3) -- (m-2-3);
    \draw[right hook->] (m-1-4) -- (m-2-4);
    \draw[->] (m-1-2) -- node[scale=0.75,description]{$\Omega^{2n}$} ($(m-2-3.north) + (-1cm,0)$);
\end{tikzpicture}\end{equation*}
where the composition $R\ld\mod\to\HMF^\infty(S\ld,w)$ is isomorphic to $(-)\stab{w}\langle nd\rangle$. 

Now, the image of $M^0\ld\mor{f} M^{-1}\ld\mor{g} M^0\ld$ under $\HMF^\infty(S\ld,w)\to\uMCM^\infty(R\ld)$ is by
definition the $2n$-th syzygy of $M\ld$, computed using the resolution $F\ua\ld$, and therefore it is isomorphic to the
image of $M\ld$ under the composition $R\ld\mod\to\ul{R\mod}\to\uMCM(R\ld)\to\uMCM^\infty(R\ld)$. The claim follows.
\end{proof}
\begin{rem}
Using the stabilization functor $R\ld\Mod\to\HMF^\infty(S\ld,w)$ one can generalize Proposition \ref{prop_getstab} to
non finitely generated modules $M\ld$. The somewhat unnatural version of Proposition \ref{prop_getstab} (involving a
mixture of both finitely generated and non finitely generated modules) then follows from the commutative diagram
\begin{equation*}\begin{tikzpicture}[description/.style={fill=white,inner sep=2pt}]
    \matrix (m) [matrix of math nodes, row sep=3em,
                 column sep=2.5em, text height=1.5ex, text depth=0.25ex,
                 inner sep=0pt, nodes={inner xsep=0.3333em, inner ysep=0.3333em}]
    {
       R\ld\mod & \ul{R\ld\mod} & \uMCM(R\ld) & \HMF(S\ld,w)\\
       R\ld\Mod & \ul{R\ld\Mod} & \uMCM^\infty(R\ld) & \HMF^\infty(S\ld,w)\\
    };
    \draw[->] (m-1-1) -- (m-1-2);
    \draw[->] (m-1-2) -- node[scale=0.75,above]{$\Omega^{2n}$} (m-1-3);
    \draw[->] (m-1-4) -- node[scale=0.75,above]{$\coker$} node[scale=0.75,below]{$\cong$}(m-1-3);
    \draw[->] (m-2-4) -- node[scale=0.75,above]{$\coker$} node[scale=0.75,below]{$\cong$}(m-2-3);

    \draw[->] (m-2-1) -- (m-2-2);
    \draw[->] (m-2-2) -- node[scale=0.75,above]{$\Omega^{2n}$} (m-2-3);
    \draw[->] (m-2-4) -- node[scale=0.75,above]{$\coker$} node[scale=0.75,below]{$\cong$}(m-2-3);
    \draw[->] (m-2-4) -- node[scale=0.75,above]{$\coker$} node[scale=0.75,below]{$\cong$}(m-2-3);

    \draw[right hook->] (m-1-3) -- (m-2-3);
    \draw[right hook->] (m-1-4) -- (m-2-4);
    \draw[right hook->] (m-1-1) -- (m-2-1);
    \draw[right hook->] (m-1-2) -- (m-2-2);
\end{tikzpicture}\end{equation*}
\end{rem}

We will see later that any finitely generated $R\ld$-module $M\ld$ possesses a free resolution satisfying the assumptions
of Proposition \ref{prop_getstab}. 

Stabilization commutes with restriction in case of free ring extensions in the following sense:
\begin{cor}\label{cor_forget}
Let $S\ld\mor{\ \iota\ } T\ld$ be a local homomorphism of regular local graded rings, such that $T\ld$ is free over
$S\ld$ with respect to $\iota$. Further, let $w\in \frm\setminus\{0\}$ and $M\ld$ be a finitely generated
$T\ld/(\iota(w))$-module, which is also finitely generated over $S\ld$. Then there is a natural isomorphism in
$\HMF^\infty(T\ld,\iota(w))$ $$M\ld\stab{\iota(w)}\downarrow^{T\ld}_{S\ld}\quad\cong\quad
(M\ld\downarrow^{T\ld}_{S\ld})\stab{w}.$$ 
\end{cor}

\subsection{A method for computing the stabilization of a graded module}

Proposition \ref{prop_getstab} gives us a way to compute the stabilization of a graded module $M\ld$, provided we can
find an eventually $2$-periodic $R\ld$-free resolution of $M\ld$ together with a lifting of its $2$-periodic part to a
matrix factorization of type $(S\ld,w)$. In \cite{EisenbudMatrixFactorizations}, Eisenbud showed how such an
$R\ld$-free resolution can be constructed starting from a finite $S\ld$-free resolution of $M\ld$. We will recall his
results now. See also \cite{Avramov}.
\begin{lem}\label{lem_preparationlemma}
Let $M\ld$ be a finitely generated graded module over $R\ld = S\ld/(w)$ and $F\ua\ld\to M\ld$ a free resolution of
$M\ld$ as a module over $S\ld$. Then there exists a family of endomorphisms $s_n: F\ua\ld\to F^{\ast-(2n-1)}\ld$ of
respective internal degrees $nd$, such that the following holds:
\begin{enumerate}
\item\label{item:s0} $s_0: F\ua\ld\to F^{\ast+1}\ld$ equals the differential of $F\ua\ld$.
\item\label{item:s1} $s_1: F\ua\ld\to F^{\ast-1}\ld$ is a nullhomotopy for the multiplication by $w$.
\item\label{item:s2} For all $n\geq 2$ we have $\sum\limits_{p+q=n} s_p s_q = 0.$
\end{enumerate}
\end{lem}
\begin{proof}
The following proof is the same as the one in \cite{EisenbudMatrixFactorizations}, with the obvious modifications for
the graded case. We construct the $s_n$ inductively. First, we define $s_0$ as the differential of $F\ua\ld$ and $s_1$
as an arbitrary nullhomotopy for the multiplication by $w$. Such a map exists, as $F\ua\ld\mor{\cdot w} F\ua\ld$ lifts
the multiplication map $M\ld\mor{\cdot w} M\langle d\rangle\ld$ which is zero since $M\ld$ is an $S\ld/(w)$-module.

Next, let $n\geq 2$ and assume we already constructed maps $s_1,...,s_{n-1}$ satisfying
\eqref{item:s0}-\eqref{item:s2}. We then consider 
 $$t\ := \ \sum\limits_{\substack{p+q=n\\ p,q>0}} s_p s_q:\ F\ua\ld\to F^{\ast-(2n-2)}\langle nd\rangle\ld.$$ 
A computation shows that $s_0 t = t s_0$. Further, the map $$M\ld=\coker(F^{-1}\ld\to 
F^0\ld)\to\coker(F^{-(2n-1)}\ld\to F^{-(2n-2)}\ld)\subset F^{-(2n-3)}\ld$$ induced by $t$ is the zero map, because
$M\ld$ is annihilated by $w$ and $w$ is not a zero divisor in $F^{-(2n-3)}\ld$ (see Fact
\ref{fact_regularisadomain}). Consequently, $t: F\ua\ld\to F^{\ast-(2n-2)}\langle nd\rangle\ld$ is nullhomotopic,
i.e. there is some $s_n: F\ua\ld\to F^{\ast-(2n-1)}\langle nd\rangle$ such that $s_n s_0 + s_0 s_n = -t$. Then
$s_0,...,s_n$ satisfy (a)-(c) as well and the induction step is complete.
\end{proof}

We fix a family of morphism $s_n: F\ua\ld\to F^{\ast-(2n-1)}\langle nd\rangle$ with the properties (a)-(c) from
Lemma \ref{lem_preparationlemma}. Further, we define a $\ZZ$-graded family $D\ua\ld$ of graded $S\ld$-modules as
follows. Put $D\ld^{-2n} := S\ld\langle -nd\rangle$ for $n\geq 0$ and $D\ld^k := 0$ otherwise. Further, for $n\geq 0$ let
$t^n\in D^{-2n}_{nd}$ denote the unit element in $D^{-2n} = S\ld\langle -nd\rangle$, and denote by $t^n\la:
D\ld\ua\to D^{\ast+2n}\langle -nd\rangle\ld$ the canonical map. In other words, $D\ua\ld$ is a polynomial ring over
$S\ld$, where the indeterminate $t$ lives in cohomological degree $-2$ and internal degree $d$, and the
map $t^n\la$ is just the division by $t^n$, where we set $t^k/t^n := 0$ for $k<n$.
\begin{prop}\label{prop_resolution}
Assume the setup of Lemma \ref{lem_preparationlemma}. Then the reduction of $$\left(D\ua\ld\otimes_{S\ld} F\ua\ld,\
  \sum\limits_{n\geq 0} t^n\la\otimes s_n\right)$$ modulo $w$ is an $R\ld$-free resolution of $M\ld$. 
\end{prop}
\begin{proof}
See \cite{EisenbudMatrixFactorizations}, Theorem 7.2. There the statement is formulated and proved in the ungraded case,
but that's ok, as the grading does not matter if we want to show acyclicity of a given complex of graded
$R\ld$-modules. In case $s_n = 0$ for all $n\geq 2$ we will give a proof based on the Bar resolution in the
Appendix. See Remark \ref{rem_resolution}.  
\end{proof}

Now let us assume in addition that $F\ua\ld$ is bounded. Then the complex $$\left(D\ua\ld\otimes_{S\ld}
  F\ua\ld,\ \sum\limits_{n\geq 0} t^n\la\otimes s_n\right)$$ is, up to internal grading, eventually $2$-periodic, because  
\begin{align}
\notag(D\ua\ld\otimes_{S\ld} F\ua\ld)^{-2N} &\ \ \ \cong\ \  \bigoplus\limits_{n\geq 0} F^{-2n}\ld\langle
-d(N-n)\rangle, \\\label{eq:explicitiso1} 
t^{\frac{2N+|x|}{2}}\otimes x &\ \longmapsfrom\ \  x\end{align} and \begin{align}
(D\ua\ld\otimes_{S\ld} F\ua\ld)^{-(2N+1)} & \ \ \ \cong\ \ \bigoplus\limits_{n\geq 0} F^{-(2n+1)}\ld\langle
-d(N-n)\rangle\notag\\\label{eq:explicitiso2}  t^{\frac{2N+1+|x|}{2}}\otimes x & \ \longmapsfrom\ \  x
\end{align}
for all $N\gg 0$ such that $F\ua\ld$ vanishes in degrees below $-2N$. In particular, we get induced maps
\begin{equation*}\begin{tikzpicture}[description/.style={fill=white,inner sep=2pt}]
    \matrix (m) [matrix of math nodes, row sep=3em,
                 column sep=6em, text height=1.5ex, text depth=0.25ex,
                 inner sep=0pt, nodes={inner xsep=0.3333em, inner ysep=0.3333em}]
    {
      \left(\bigoplus\limits_{n\geq 0} F^{-2n}\ld\langle -d(N-n)\rangle\right)\langle -d\rangle & (D\ua\ld\otimes_{S\ld}
      F\ua\ld)^{-(2N+2)} \\
      \bigoplus\limits_{n\geq 0} F^{-(2n+1)}\ld\langle -d(N-n)\rangle & (D\ua\ld\otimes_{S\ld} F\ua\ld)^{-(2N+1)}\\
    };
    \draw[dashed,->] (barycentric cs:m-1-1=0.75,m-2-1=0.25) -- (barycentric cs:m-1-1=0.25,m-2-1=0.75);
    \draw[->]        (m-1-2) -- node[scale=0.76,right]{$\partial$} (m-2-2);
    \draw[->]        (m-1-1) -- node[scale=0.75,above]{$\cong$}    (m-1-2);
    \draw[->]        (m-2-1) -- node[scale=0.75,above]{$\cong$}    (m-2-2);
\end{tikzpicture}\end{equation*}
and
\begin{equation*}\begin{tikzpicture}[description/.style={fill=white,inner sep=2pt}]
    \matrix (m) [matrix of math nodes, row sep=3em,
                 column sep=6em, text height=1.5ex, text depth=0.25ex,
                 inner sep=0pt, nodes={inner xsep=0.3333em, inner ysep=0.3333em}]
    {
       \bigoplus\limits_{n\geq 0} F^{-(2n+1)}\ld\langle-d(N-n)\rangle & (D\ua\otimes_{S\ld} F\ua\ld)^{-(2N+1)}\\
        \bigoplus\limits_{n\geq 0} F^{-2n}\ld\langle -d(N-n)\rangle & (D\ua\otimes_{S\ld} F\ua\ld)^{-2N}\\
    };
    \draw[dashed,->] (barycentric cs:m-1-1=0.75,m-2-1=0.25) -- (barycentric cs:m-1-1=0.25,m-2-1=0.75);
    \draw[->]        (m-1-2) -- node[scale=0.76,right]{$\partial$} (m-2-2);
    \draw[->]        (m-1-1) -- node[scale=0.75,above]{$\cong$}    (m-1-2);
    \draw[->]        (m-2-1) -- node[scale=0.75,above]{$\cong$}    (m-2-2);
\end{tikzpicture}\end{equation*}
which, by the explicit definition of the isomorphisms \eqref{eq:explicitiso1}, \eqref{eq:explicitiso2} and the
differentials involved, are equal to 
$\sum_{n\geq 0} s_n$. Thus, applying Propositions \ref{prop_resolution} and \ref{prop_getstab}, we get the following
useful method to calculate the stabilization of a module:
\begin{prop}\label{prop_stabalgo}
Let $M\ld$ be a finitely generated graded $R\ld$-module and $F\ua\ld\to M\ld$ a bounded, free resolution of
$M\ld$ as a module over $S\ld$. Further, let $s_n$ be as in Lemma \ref{lem_preparationlemma}. Then there is an
isomorphism in $\HMF$
\begin{align*}
M\ld\stab{w}\quad\cong\quad \left(\bigoplus\limits_{n\geq 0} F^{-2n}\ld\langle dn\rangle,\ \bigoplus\limits_{n\geq 0}
  F^{-(2n+1)}\ld\langle dn\rangle,\ \sum\limits_{n\geq 0} s_n\right).
\end{align*}
\end{prop}

As an example, we use Proposition \ref{prop_getstab} to calculate the stabilization of $M\ld := S\ld / (x_1,...,x_l)$,
where $x_1,...,x_l$ is a regular sequence of homogeneous elements, and $w\in (x_1,...,x_l)$. According to
\ref{prop_getstab}, we have to go through the following steps:
\begin{enumerate}
\item[(1)] Construct a bounded $S\ld$-free resolution of $S\ld / (x_1,...,x_l)$.
\item[(2)] Explicitly construct homotopies $s_n$ as in Lemma \ref{lem_preparationlemma}.
\item[(3)] Put together (1) and (2) to get the stabilization as described in \ref{prop_stabalgo}.
\end{enumerate}
\noindent\textit{Step 1:} As $x_1,...,x_l$ is regular, its Koszul-complex $$K(x_1,...,x_l)\ua\ld\ \ :=
\ \ {\bigwedge}\ua \bigoplus\limits_{i=1}^{l} S\ld\langle-\deg(x_i)\rangle e_i\quad\text{ with differential }\quad\d e_i
:= x_i,$$ is an $S\ld$-free resolution of $S\ld/(x_1,...,x_l)$. Note that the Koszul complex carries a natural structure
of a dg-algebra, which we will use in the next step.

\noindent\textit{Step 2:} As $w\in (x_1,...,x_l)$ we can choose homogeneous $y_1,...,y_l$ such that $w = x_1 y_1 + ... +
x_l y_l$. Define $s_1$ as the multiplication $\mult(y_1 e_1+ ... + y_l e_l)$ (in the Koszul-complex) by $y_1 e_1 +
... + y_l e_l$. The Leibniz rule for differentiation shows that $s_1$ is indeed a nullhomotopy for the multiplication by
$w$. Further, we have $s_1^2=0$, so we can put $s_n := 0$ for $n\geq 2$ and (a)-(c) from Lemma \ref{lem_preparationlemma}
are satisfied. 

\noindent\textit{Step 3:} As $s_n=0$ for all $n\geq 2$, we get the following concrete description of
$S\ld/(x_1,...,x_l)\stab{w}$: 

\begin{cor}\label{cor_koszulstabilization}
Let $x_1,...,x_l$ be an $S\ld$-regular sequence of homogeneous elements and $w\in (x_1,...,x_l)$. Choose elements
$y_1,...,y_l$ satisfying $w = x_1 y_1 + ... + x_l y_l$. Then there is a canonical isomorphism in $\HMF(S\ld,w)$
$$\left(S\ld/(x_1,...,x_l)\right)\stab{w}\ \cong\ \left(\left.K(x_1,...,x_l)^{\text{even}}\ld,\
  K(x_1,...,x_l)^{\text{odd}}\ld,\ \d+\mult(e_1
  y_1 + ... + e_l y_l)\right.\right),$$
where $$K(x_1,...,x_l)^{\text{even}}\ = \ \bigoplus\limits_{n\geq 0} \left[{\bigwedge}^{2n} \bigoplus\limits_{i=1}^{l}
  S\ld\langle-\deg(x_i)\rangle e_i\right]\langle dn\rangle$$ and
$$K(x_1,...,x_l)^{\text{odd}}\ = \ \bigoplus\limits_{n\geq 0} \left[{\bigwedge}^{2n+1} \bigoplus\limits_{i=1}^{l}
  S\ld\langle-\deg(x_i)\rangle e_i\right]\langle dn\rangle$$ 
\end{cor}

\begin{rem}
In the next section we identify the matrix factorization from Corollary \ref{cor_koszulstabilization} as the tensor
product of the elementary Koszul factorizations $$S\ld\xrightarrow{y_i} S\ld\langle
-\deg(x_i)\rangle\xrightarrow{x_i} S\ld.$$
\end{rem}

\begin{rem}
Note that the the example of a complete intersection $S\ld / (x_1,...,x_l)$ was so easy to compute because
we could choose a nullhomotopy $s_1$ for the multiplication by $w$ which satisfied $s_1^2=0$. In general, such a
homotopy need not exist. More precisely, one has the following: there are modules $M\ld$ whose \textit{minimal} free
resolutions do not possess a nullhomotopy $s_1$ satisfying $s_1^2=0$, but one can always choose \textit{some}
(non-minimal) resolution where it does exist. For details, see \cite{Avramov}.  
\end{rem}  

For later use in Section \ref{sec_description} (see Example \ref{ex_chimorphisms}) we will now study how stabilizations
of morphisms between complete intersections can be computed explicitly in terms of Koszul factorizations. To keep things
simple, we restrict to the case of two variables.

\begin{ex}\label{ex_stabmorphism2els}
Let $x_1,x_2$ and $\tilde{x}_1, \tilde{x}_2$ be homogeneous regular sequences in $S\ld$ such that $w\in (x_1,x_2)\cap
(\tilde{x}_1,\tilde{x}_2)$. Fix homogeneous $y_1, y_2$ and $\tilde{y}_1, \tilde{y}_2$ such that $w = x_1 y_1 + x_2 y_2$ and $w = \tilde{x}_1
\tilde{y}_1 + \tilde{x}_2 \tilde{y}_2$. Finally, let $\varphi: S\ld/(x_1,x_2)\to S\ld/(\tilde{x}_1,\tilde{x}_2)$ be some nonzero morphism of
$S\ld$-modules. We want to describe explicitly a map $\{\bfx,\bfy\}\to \{\tilde{\bfx},\tilde{\bfy}\}$ making the following
square commutative in $\HMF(S\ld,w)$:
\begin{align*}\label{eq:stabilizationofmorphisms}
\begin{tikzpicture}
  \matrix (m) [matrix of math nodes, row sep=3em,
  column sep=6em, text height=1.5ex, text depth=0.25ex,
  inner sep=0pt, nodes={inner xsep=0.3333em, inner ysep=0.3333em}]
  {
    \{\bfx,\bfy\} \pgfmatrixnextcell \pgfmatrixnextcell \left(S\ld/(x_1,x_2)\right)\stab{w}\\
    \{\tilde{\bfx},\tilde{\bfy}\} \pgfmatrixnextcell \pgfmatrixnextcell \left(S\ld/(\tilde{x}_1,\tilde{x}_2)\right)\stab{w}\\
  };
  \draw[->] (m-1-1) -- node[above,scale=0.75]{$\cong$} (m-1-3);
  \draw[->] (m-2-1) -- node[below,scale=0.75]{$\cong$} (m-2-3);
  \draw[->] (m-1-1) -- (m-2-1);
  \draw[->] (m-1-3) -- node[right,scale=0.75]{$\varphi\stab{w}$} (m-2-3);
\end{tikzpicture}
\end{align*}
First, note that $\varphi$ is given by some element $\alpha\in S\ld\setminus\{0\}$ such that $\alpha (x_1,x_2)\subset
(\tilde{x}_1,\tilde{x}_2)$. Fix elements $\lambda_{ij}\in S\ld$ such that $x_i = \sum_j \lambda_{ij} \tilde{x}_j$. Then we have
$$\sum_j (\alpha \tilde{y}_j) \tilde{x}_j = \alpha w = \sum_i y_i (\alpha x_i) = \sum_j \left(\sum_i \lambda_{ij} y_i\right)
\tilde{x}_j,$$ which means that $\left(\alpha \tilde{y}_j - \sum_i \lambda_{ij} y_i\right)_j$ is a $1$-cycle in
$K(\tilde{x}_1,\tilde{x}_2)\ua\ld$. As $\tilde{\bfx}$ is regular, it follows that there exists some $\mu$ such that
$\sum_i \lambda_{i,1} y_i = \alpha \tilde{y}_1 - \mu \tilde{x}_2$ and $\sum_i \lambda_{i,2} y_i = \alpha \tilde{y}_2 + \mu
\tilde{x}_1$.

By definition of $(-)\stab{w}$, in order to compute $\varphi\stab{w}$ we have to extend $\varphi$ to an eventually
$2$-periodic morphism between eventually $2$-periodic $S\ld/(w)$-free resolutions of $S\ld/(x_1,x_2)$ and
$S\ld/(\tilde{x}_1,\tilde{x}_2)$. In our situation, we use the resolutions constructed in Proposition \ref{prop_getstab} from
the Koszul resolutions $K(x_1,x_2)\ua\ld\to S\ld/(x_1,x_2)$ and $K(\tilde{x}_1,\tilde{x}_2)\ua\ld\to S\ld/(\tilde{x}_1,\tilde{x}_2)$
together with their square zero nullhomotopies $\mult(y_1 e_1 + y_2 e_2)$ and $\mult(\tilde{y}_1 \tilde{e}_1 + \tilde{y}_2 \tilde{e}_2)$
for the multiplication by $w$. Patience and some calculation shows that such an extension is explicitly given by
\begin{align*}
\begin{tikzpicture}
  \matrix (m) [matrix of math nodes, row sep=3em,
  column sep=2em, text height=1.5ex, text depth=0.25ex,
  inner sep=0pt, nodes={inner xsep=0.3333em, inner ysep=0.3333em}]
  {
    ... \pgfmatrixnextcell S\ld (e_1)\oplus S\ld(e_2) \pgfmatrixnextcell\pgfmatrixnextcell S\ld\oplus S\ld(e_1 e_2)
    \pgfmatrixnextcell\pgfmatrixnextcell
    S\ld(e_1)\oplus S\ld(e_2) \pgfmatrixnextcell\pgfmatrixnextcell S\ld\\
    ... \pgfmatrixnextcell S\ld (\tilde{e}_1)\oplus S\ld(\tilde{e}_2) \pgfmatrixnextcell\pgfmatrixnextcell S\ld\oplus S\ld(\tilde{e}_1
    \tilde{e}_2) \pgfmatrixnextcell\pgfmatrixnextcell S\ld(\tilde{e}_1)\oplus S\ld(\tilde{e}_2)
    \pgfmatrixnextcell\pgfmatrixnextcell S\ld\\
  };
  \draw[->] (m-1-1) -- (m-1-2);
  \draw[->] (m-1-2) -- node[above,scale=0.75]{$\begin{pmatrix} x_1 & x_2\\  -y_2 & y_1\end{pmatrix}$} (m-1-4);
  \draw[->] (m-1-4) -- node[above,scale=0.75]{$\begin{pmatrix} y_1& -x_2 \\ y_2 & x_1\end{pmatrix}$} (m-1-6);
  \draw[->] (m-1-6) -- node[above,scale=0.75]{$\begin{pmatrix} x_1 & x_2\end{pmatrix}$} (m-1-8);
  \draw[->] (m-2-1) -- (m-2-2);
  \draw[->] (m-2-2) -- node[below,scale=0.75]{$\begin{pmatrix} \tilde{x}_1 & \tilde{x}_2\\  -\tilde{y}_2 & \tilde{y}_1\end{pmatrix}$}
  (m-2-4); 
  \draw[->] (m-2-4) -- node[below,scale=0.75]{$\begin{pmatrix} \tilde{y}_1& -\tilde{x}_2 \\ \tilde{y}_2 & \tilde{x}_1\end{pmatrix}$}
  (m-2-6); 
  \draw[->] (m-2-6) -- node[below,scale=0.75]{$\begin{pmatrix} \tilde{x}_1 & \tilde{x}_2\end{pmatrix}$} (m-2-8);
  \draw[->] (m-1-2) -- node[right,scale=0.75]{$\begin{pmatrix} \lambda_{11} & \lambda_{21} \\ \lambda_{12} &
      \lambda_{22}\end{pmatrix}$} (m-2-2);
  \draw[->] (m-1-4) -- node[right,scale=0.75]{$\begin{pmatrix} \alpha & 0 \\ \mu & \frac{\lambda_{11}\lambda_{22} -
      \lambda_{12}\lambda_{21}}{\alpha}\end{pmatrix}$} (m-2-4);
  \draw[->] (m-1-6) -- node[right,scale=0.75]{$\begin{pmatrix} \lambda_{11} & \lambda_{21} \\ \lambda_{12} &
      \lambda_{22}\end{pmatrix}$} (m-2-6);
  \draw[->] (m-1-8) -- node[right,scale=0.75]{$\alpha$} (m-2-8);
\end{tikzpicture}
\end{align*}
provided $\frac{\lambda_{11}\lambda_{22}-\lambda_{12}\lambda_{21}}{\alpha}$ exists, which will be clear in our
applications (see Example \ref{ex_chimorphisms}). Thus, a concrete realization of a map
$\{\bfx,\bfy\}\to\{\tilde{\bfx},\tilde{\bfy}\}$ making \eqref{eq:stabilizationofmorphisms} commute is given by
\begin{align*}
\begin{tikzpicture}
  \matrix (m) [matrix of math nodes, row sep=4em,
  column sep=4em, text height=1.5ex, text depth=0.25ex,
  inner sep=0pt, nodes={inner xsep=0.3333em, inner ysep=0.3333em}]
  {
    \{\bfx,\bfy\} \pgfmatrixnextcell S\ld (e_1)\oplus S\ld(e_2) \pgfmatrixnextcell\pgfmatrixnextcell
    S\ld\oplus S\ld(e_1 e_2)\\
    \{\tilde{\bfx},\tilde{\bfy}\} \pgfmatrixnextcell S\ld (\tilde{e}_1)\oplus S\ld(\tilde{e}_2)
    \pgfmatrixnextcell\pgfmatrixnextcell S\ld\oplus S\ld(\tilde{e}_1\tilde{e}_2)\\
  };
  \draw[->] ($(m-1-2.east) + (0,+0.8mm)$) -- node[above,scale=0.6]{$\begin{pmatrix} x_1 & x_2\\  -y_2 &
      y_1\end{pmatrix}$} ($(m-1-4.west) + (0,+0.8mm)$); 
  \draw[->] ($(m-1-4.west) + (0,-0.8mm)$) -- node[below,scale=0.6]{$\begin{pmatrix} y_1& -x_2 \\ y_2 &
      x_1\end{pmatrix}$} ($(m-1-2.east) + (0,-0.8mm)$); 
  \draw[->] ($(m-2-2.east) + (0,+0.8mm)$) -- node[above,scale=0.6]{$\begin{pmatrix} \tilde{x}_1 & \tilde{x}_2\\  -\tilde{y}_2 &
      \tilde{y}_1\end{pmatrix}$} ($(m-2-4.west) + (0,+0.8mm)$);
  \draw[->] ($(m-2-4.west) + (0,-0.8mm)$) -- node[below,scale=0.6]{$\begin{pmatrix} \tilde{y}_1& -\tilde{x}_2 \\ \tilde{y}_2 &
      \tilde{x}_1\end{pmatrix}$ } ($(m-2-2.east) + (0,-0.8mm)$); 
  \draw[->] (m-1-1) -- (m-2-1);
  \draw[->] (m-1-2) -- node[left,scale=0.75]{$\begin{pmatrix} \lambda_{11} & \lambda_{21} \\ \lambda_{12} &
      \lambda_{22}\end{pmatrix}$} (m-2-2);
  \draw[->] (m-1-4) -- node[right,scale=0.75]{$\begin{pmatrix} \alpha & 0 \\ \mu & \frac{\lambda_{11}\lambda_{22} -
      \lambda_{12}\lambda_{21}}{\alpha}\end{pmatrix}$} (m-2-4); 
\end{tikzpicture}
\end{align*}
\end{ex}

\begin{rem}\label{rem_explicitcounit}
Let us return to Remark \ref{rem_praecounit} where we asked how the counit map $$\Omega^{2n}M\ld\langle
nd\rangle\longrightarrow M\ld$$ looks like explicitly. In this remark, we answer this question in the case where
$\Omega$ is computed using a resolution constructed through \ref{prop_resolution}.  

Thus, fix an $S\ld$-free resolution $F\ua\ld$ of $M\ld$ together with a family of homotopies $s_n$ as in lemma
\ref{lem_preparationlemma}, and choose $N\gg 0$ such that $F^{n}\ld=0$ for all $n<-2N$. Then the diagram \eqref{eq:chase}
can be realized concretely as: 
\begin{gather*}\begin{tikzpicture}[description/.style={fill=white,inner sep=2pt}]
    \matrix (m) [matrix of math nodes, row sep=3em,
                 column sep=1em, text height=1.5ex, text depth=0.25ex,
                 inner sep=0pt, nodes={inner xsep=0.3333em, inner ysep=0.3333em}]
    {
       \bigoplus\limits_{n=0}^{N} F^{-(2n+1)}\ld\langle d(n-N)\rangle \pgfmatrixnextcell \bigoplus\limits_{n=0}^{N}
       F^{-2n}\ld\langle d(n-N)\rangle \pgfmatrixnextcell \bigoplus\limits_{n=0}^{N} F^{-(2n+1)}\ld\langle
       d(n-N+1)\rangle \pgfmatrixnextcell \cdots\\      
       \bigoplus\limits_{n=0}^{N} F^{-(2n+1)}\ld\langle d(n-N)\rangle \pgfmatrixnextcell \bigoplus\limits_{n=0}^{N}
       F^{-2n}\ld\langle d(n-N)\rangle \pgfmatrixnextcell \bigoplus\limits_{n=0}^{N-1} F^{-(2n+1)}\ld\langle
       d(n-N+1)\rangle \pgfmatrixnextcell \cdots\\ 
    };
    \draw[->] (m-1-1) -- (m-1-2);
    \draw[->] (m-1-2) -- (m-1-3);
    \draw[->] (m-1-3) -- (m-1-4);
    \draw[->] (m-2-1) -- (m-2-2);
    \draw[->] (m-2-2) -- (m-2-3);
    \draw[->] (m-2-3) -- (m-2-4);
    \draw[double distance=0.7mm] ($(m-1-1.south) + (0,-1mm)$) -- ($(m-2-1.north) + (0,+1mm)$);
    \draw[double distance=0.7mm] ($(m-1-2.south) + (0,-1mm)$) -- ($(m-2-2.north) + (0,+1mm)$);
    \draw[->] (m-1-3) -- (m-2-3);
\end{tikzpicture}\\[3ex]
\begin{tikzpicture}[description/.style={fill=white,inner sep=2pt},scale=0.5]
    \matrix (m) [matrix of math nodes, row sep=3em,
                 column sep=1em, text height=1.5ex, text depth=0.25ex,
                 inner sep=0pt, nodes={inner xsep=0.3333em, inner ysep=0.3333em}]
    {
       \cdots \pgfmatrixnextcell \bigoplus\limits_{n=0}^{N} F^{-(2n+1)}\ld\langle d(n-1)\rangle \pgfmatrixnextcell
       \bigoplus\limits_{n=0}^{N} F^{-2n}\ld\langle d(n-1)\rangle \pgfmatrixnextcell \bigoplus\limits_{n=0}^{N}
       F^{-(2n+1)}\ld\langle dn\rangle \pgfmatrixnextcell\bigoplus\limits_{n=0}^{N} F^{-2n}\ld\langle dn\rangle
       \pgfmatrixnextcell  \coker(\d)\\ 
       \cdots \pgfmatrixnextcell F^{-3}\ld\oplus F^{-1}\ld\langle -d\rangle\pgfmatrixnextcell F^{-2}\ld\oplus
       F^0\ld\langle -d\rangle\pgfmatrixnextcell F^{-1}\ld\pgfmatrixnextcell F^0\ld\pgfmatrixnextcell M\ld \\
    };
    \draw[->] (m-1-1) -- (m-1-2);
    \draw[->] (m-1-2) -- (m-1-3);
    \draw[->] (m-1-3) -- (m-1-4);
    \draw[->] (m-1-4) -- (m-1-5);
    \draw[->] (m-1-5) -- (m-1-6);
    \draw[->] (m-2-1) -- (m-2-2);
    \draw[->] (m-2-2) -- (m-2-3);
    \draw[->] (m-2-3) -- (m-2-4);
    \draw[->] (m-2-4) -- (m-2-5);
    \draw[->] (m-2-5) -- (m-2-6);
    \draw[->] (m-1-2) -- (m-2-2);
    \draw[->] (m-1-3) -- (m-2-3);
    \draw[->] (m-1-4) -- (m-2-4);
    \draw[->] (m-1-5) -- (m-2-5);
    \draw[dashed] (m-1-6) -- (m-2-6);
\end{tikzpicture}
\end{gather*}
where the vertical maps are the projection maps.
\end{rem}

\subsection{Tensor products of graded matrix factorizations}

In this section we define internal and external tensor products of graded matrix factorizations and study the crucial
question in which situations taking tensor products commutes with stabilization.

\begin{definition}\label{def_internaltensor}
Let $S\ld$ be a regular local graded ring, $w_0,w_1\in S\ld$ be homogeneous and $M := M^0\ld\mor{f} M^{-1}\ld\mor{g}
M^0\ld$, $N := N^0\ld\mor{f\p} N^{-1}\ld\mor{g\p} N^0\ld$ be graded matrix factorizations of type $(S\ld,w_0)$
and $(S\ld,w_1)$, respectively. The \textit{(internal) tensor product} $M\otimes_{S\ld} N$ is defined as the graded
matrix factorization of type $(S\ld,w_0+w_1)$
\begin{equation*}\begin{tikzpicture}[description/.style={fill=white,inner sep=2pt}]
    \matrix (m) [matrix of math nodes, row sep=3em,
                 column sep=4em, text height=1.5ex, text depth=0.25ex,
                 inner sep=0pt, nodes={inner xsep=0.3333em, inner ysep=0.3333em}]
    {
       M^0\ld\otimes_{S\ld}N^{-1}\ld\ \oplus\ M^{-1}\ld\otimes_{S\ld}N^0\ld && M^0\ld\otimes_{S\ld} N^0\ld\ \oplus\
       M^{-1}\ld\otimes_{S\ld}N^{-1}\ld\langle d\rangle \\
    };
    \draw[->] ($(m-1-1.east) + (0,+0.8mm)$) -- node[above,scale=0.75] {$\begin{pmatrix} \id\otimes g\p & g\otimes\id\\
        f\otimes\id & -\id\otimes f\p\end{pmatrix}$} ($(m-1-3.west) + (0,+0.8mm)$);
    \draw[->] ($(m-1-3.west) + (0,-0.8mm)$) -- node[below,scale=0.75]{$\begin{pmatrix} 
        \id\otimes f\p & g\otimes\id\\ f\otimes\id & -\id\otimes g\p\end{pmatrix}$} ($(m-1-1.east) + (0,-0.8mm)$);
\end{tikzpicture}\end{equation*}
\end{definition}
In order to be able to compute tensor products with more than two factors,
we note the following equivalent definition: Given $M$, consider it as a $\ZZ$-graded family $M\ua\la$ of 
$S\ld$-modules, concentrated in degrees $1$ and $0$, and similar for $N$. Then take the tensor product of $M\ua\ld$ and
$N\ua\ld$ as $\ZZ$-graded families of graded $S\ld$-modules, i.e. $(M\ua\ld\otimes_{S\ld} N\ua\ld)_n :=
\bigoplus_{p+q=n} M^p\ld\otimes_{S\ld} N^q\ld$, and equip $M\ua\ld\otimes_{S\ld} N\ua\ld$ with the two
differentials, one raising and the other lowering the cohomological degree by $1$, induced by the structure maps of $M$
and $N$; obey the Koszul sign rule. Then collapse the cohomological $\ZZ$-grading on $M\ua\ld\otimes_{S\ld} N\ua\ld$ to a
$\ZZ/2\ZZ$-grading, but whenever a cohomological degree shift by $-2$ occurs, we shift up the internal degree by $d$. The
resulting $\ZZ/2\ZZ$-graded family of $S\ld$-modules is now equipped with a degree $0$ differential from cohomological
degree $1$ to $0$ and a degree $d$ differential from homological degree $0$ to $1$. This description is valid also for
more than two tensor factors, and we will make use of it shortly. 

Next we discuss the most basic matrix factorizations, the Koszul factorizations.
\begin{definition}
Let $S\ld$ be a regular local graded ring and let $x,y\in S\ld$ be homogeneous. The \textit{Koszul factorization of
  $x,y$}, denoted by $\{x,y\}$, is defined as the graded matrix factorization 
$$\{x,y\}\ \ := \ \ \left(S\ld\xrightarrow{\ y\ } S\ld\langle -\deg(x)\rangle\xrightarrow{\ x\ }S\ld\right)$$
of type $(S\ld,xy)$. More generally, if $\bfx := (x_1,...,x_l)$ and $\bfy := (y_1,..,y_l)$ are sequences of homogeneous
elements in $S\ld$, we define the Koszul factorization $\{\bfx,\bfy\}$ of $\bfx$ and $\bfy$ as the matrix factorization
of type $\left(S\ld,\sum\limits_{i=1}^{l} x_i y_i\right)$
$$\{\bfx, \bfy\}\ \ := \ \ \bigotimes\limits_{i=1}^{l} \{x_i,y_i\}\ =\  \bigotimes_{i=1}^{l}
\left(S\ld\xrightarrow{\ y_i\ } S\ld\langle -\deg(x_i)\rangle\xrightarrow{\ x_i\ }S\ld\right).$$
\end{definition}
Koszul-factorizations play a very prominent role, because the matrix factorization occurring in Corollary
\ref{cor_koszulstabilization} is just the Koszul-factorization of $(x_1,...,x_l)$ and $(y_1,...,y_l)$:
\begin{prop}\label{prop_koszulstabilization2} 
Let $\bfx = (x_1,...,x_n)$ and $\bfy = (y_1,...,y_n)$ be sequences of homogeneous elements in $S\ld$, and assume
that $\ol{x}$ is regular. Further, set $w := x_1 y_1 + ... + x_n y_n$. Then there is an isomorphism in $\HMF(S\ld,w)$
 $$\left(S\ld/(x_1,...,x_n)\right)\stab{w}\cong\{\bfx,\bfy\}.$$  
\end{prop}
\begin{rem}Proposition \ref{prop_koszulstabilization2} has an interesting consequence. The left hand side in \ref{prop_koszulstabilization2} does only depend on
$x_1,...,x_n$ and $w$, but not on the particular choice of the $y_i$. Thus, any two choices of $y_1,..,y_n$ satisfying
$w = x_1 y_1 + ... + x_n y_n$ give homotopy equivalent Koszul factorizations. Our proof is rather indirect; for a direct
proof, see \cite[Lemma 2]{WuLinkHomology}.\end{rem}

\subsection{Compatibility of taking tensor product and stabilization}\label{subsec:compatibility}

In the application to Khovanov-Rozansky homology we will identify the matrix factorizations associated to
basic MOY-graphs as stabilizations of certain Soergel bimodules. As these matrix factorizations are glued together
by tensoring afterwards, we are naturally led to study the question whether taking tensor products commutes with taking
stabilizations. The following gives a first criterion:

\begin{cor}\label{cor_concatenateregularsequences}
Let $I\ld,J\ld\subset S\ld$ be homogeneous ideals in $S\ld$ such that there is a regular sequence $(x_1,...,x_n)$ of
homogeneous elements $S\ld$ and some $k$, $1\leq k\leq n$, such that $I\ld = (x_1,...,x_k)$ and $J\ld =
(x_{k+1},...,x_n)$. Further, let $w= w_0 + w_1$ for some $w_0\in I\ld$ and $w_1\in J\ld$. Then there is an isomorphism
in $\HMF(S\ld,w)$
$$(S\ld/I\ld)\stab{w_0}\otimes_{S\ld} (S\ld/J\ld)\stab{w_1}\ \cong\ (S\ld/(I\ld+J\ld))\stab{w}.$$  
\end{cor}
\begin{proof}
Choose homogeneous elements $y_1,...,y_n$ in $S\ld$ such that $w_0 = x_1 y_1 + ... + x_k y_k$ and $w_1 = x_{k+1} y_{k+1}
+ ... + y_n 
x_n$. Applying Proposition \ref{prop_koszulstabilization2} three times then gives
\begin{align*}
(S\ld/I\ld)\stab{w_0}\otimes_{S\ld} (S\ld/J\ld)\stab{w_1} & \cong
\{(x_1,...,x_k),(y_1,...,y_k)\}\otimes_{S\ld} \{(x_{k+1},...,x_n),(y_{k+1},...,y_n)\}\\
& = \bigotimes\limits_{i=1}^{n} \{x_i,y_i\} = \{(x_1,...,x_n),(y_1,...,y_n)\}\\
& \cong (S\ld/(I\ld+J\ld))\stab{w}.\end{align*}
\end{proof}

\noindent Corollary \ref{cor_concatenateregularsequences} looks somewhat unnatural as it leaves the following questions open:

\begin{enumerate}
\item[(1)] Given two finitely generated graded modules $M\ld$ and $N\ld$ over $R\ld := S\ld/(w)$ and $R\p\ld :=
  S\ld/(w\p)$, respectively, is there always a canonical morphism between $M\ld\stab{w}\otimes_{S\ld}
  N\ld\stab{w\p}$ and $(M\ld\otimes_{S\ld} N\ld)\stab{w+w\p}$? 
\item[(2)] Are there criteria like \ref{cor_concatenateregularsequences} which can be applied to \textit{noncyclic}
  $S\ld$-modules $M\ld$ and $N\ld$ to check if there is an isomorphism $M\ld\stab{w}\otimes_{S\ld}
  N\ld\stab{w\p}\cong (M\ld\otimes_{S\ld} N\ld)\stab{w+w\p}?$
\end{enumerate}
Question (1) will be answered in Theorem \ref{thm_naturalmorphism}: there \textit{is} a canonical morphism
$$M\ld\stab{w}\otimes_{S\ld}N\ld\stab{w\p}\to \left(M\ld\otimes_{S\ld} N\ld\right)\stab{w+w\p}$$ which is 
even natural in $M\ld$ and $N\ld$. Concerning question (2), again the full answer is contained in Theorem
\ref{thm_naturalmorphism}, but for now, the following generalization of Corollary \ref{cor_concatenateregularsequences}
is sufficient: 

\begin{prop}\label{prop_tensorstabgeneral}
Let $M\ld$ and $N\ld$ be finitely generated modules over $R\ld := S\ld/(w)$ and $R\p\ld := S\ld/(w\p)$, respectively,
such that $\tor^{S\ld}_k(M\ld,N\ld)\ld=0$ for all $k>0$. Then there is an isomorphism in $\HMF(S\ld,w+w\p)$
$$M\ld\stab{w}\otimes_{S\ld} N\ld\stab{w\p}\cong (M\ld\otimes_{S\ld} N\ld)\stab{w+w\p}$$
\end{prop}
\begin{proof}
We want to apply Proposition \ref{prop_stabalgo} to $M\ld\otimes_{S\ld} N\ld$. Let $P\ua\ld\to M\ld$ and $Q\ua\ld\to N\ld$
be free resolutions of $M\ld$ and $N\ld$ over $S\ld$, and let $s\p_n: P\ua\ld\to P^{\ast-(2n-1)}\ld$ and $s\pp_n:
Q\ua\ld\to Q^{\ast-(2n-1)}\ld$ be as in Lemma \ref{lem_preparationlemma}. From this data we will now construct explicitly
a free resolution $F\ua\ld\to M\ld\otimes_{S\ld} N\ld$ together with a family of higher homotopies for $F\ua\ld$ needed
for the application of Lemma \ref{lem_preparationlemma} to $M\ld\otimes_{S\ld} N\ld$.

As $\tor^{S\ld}_k(M\ld,N\ld)\ld=0$ for $k>0$, the complex $F\ua\ld := P\ua\ld\otimes_{S\ld} Q\ua\ld$ is an $S\ld$-free
resolution of $M\ld\otimes_{S\ld} N\ld$. We now define $s_n: F\ua\ld\to F^{\ast-(2n-1)}\ld$ as 
$$s_n(x\otimes y)\ \ :=\ \ s\p_n(x)\otimes y + (-1)^{\deg(x)} x\otimes s\pp_n(x).$$
It's clear that $s_0$ is just the differential of $F\ua\ld$, and since
\begin{align*}
(\differential_{F\ua\ld}s_1+s_1\differential_{F\ua\ld})(x\otimes y) &\makebox[5mm][c]{=}\differential_{F\ua\ld}(s\p_1(x)\otimes y + (-1)^{|x|}
x\otimes 
s\pp_1(y)) \\ & \makebox[5mm][c]{}+ s_1(\differential_{P\ua\ld}(x)\otimes y + (-1)^{|x|} x\otimes\differential_{Q\ua\ld}(y))\\
& \makebox[5mm][c]{=}(\differential_{P\ua\ld}s\p_1+s\p_1\differential_{P\ua\ld})(x)\otimes y + x\otimes
(\differential_{Q\ua\ld}s\pp_1+s\pp_1\differential_{Q\ua\ld})(y)\\ 
& \makebox[5mm][c]{}+(-1)^{|x|+1}s\p_1(x)\otimes\differential_{Q\ua\ld}(y) + (-1)^{|x|}\differential_{P\ua\ld}x\otimes s\pp_1(y) \\ 
&\makebox[5mm][c]{}+(-1)^{|x|+1}\differential_{P\ua\ld}(x)\otimes s\pp_1(y) + (-1)^{|x|} s\p_1(x)\otimes\differential_{Q\ua\ld}(y)\\
&\makebox[5mm][c]{=} wx\otimes y + x\otimes w\p y = (w+w\p) x\otimes y
\end{align*}
we see that $s_1$ is a nullhomotopy for the multiplication by $w+w\p$. Finally, we have to check that
$\sum\limits_{p+q=n} s_p s_q = 0$ für $n\geq 2$, which follows by direct calculation: 
\begin{align*}
\sum\limits_{p+q=n} s_p s_q & = \sum\limits_{p+q=n} (s\p_p\otimes\id + (-1)^{|x|} \id\otimes s\pp_p)(s\p_q\otimes\id +
(-1)^{|x|}\id\otimes s\pp_q)\\
& = \sum\limits_{p+q=n} s\p_p s\p_q\otimes\id + \id\otimes s\pp_p s\pp_q + (-1)^{|x|}s\p_p\otimes s\pp_q + (-1)^{|x|+1}
s\p_q\otimes s\pp_p\\
& = 0;
\end{align*}
where we used that $s\pp_q$ changes the parity of the degree, and therefore $$((-1)^{|x|}\id\otimes s\pp_p)\circ
(s\p_q\otimes\id) = (-1)^{|x|+1} s\p_q\otimes s\pp_p.$$ Thus the $s_n$ satisfy the conditions of
Lemma \ref{lem_preparationlemma} and therefore can be used to calculate\break$(M\ld\otimes_{S\ld} 
N\ld)\stab{w+w\p}$. By Proposition \ref{prop_stabalgo}, we get
\begin{align*}
\makebox[12mm][c]{}&(M\ld\otimes_{S\ld} N\ld)\stab{w+w\p}\\
\makebox[12mm][c]{$\cong$} & \left(\bigoplus\limits_{n\geq 0} F^{2n}\ld\langle
  nd\rangle,\bigoplus\limits_{n\geq 0} 
  F^{2n+1}\ld\langle nd\rangle,\sum\limits_{n\geq 0} s\p_n\otimes\id + (-1)^{|x|} \id\otimes s\pp_n\right)\\
\makebox[12mm][c]{$=$} & \left(\bigoplus\limits_{n\geq 0} P^{2n}\ld\langle nd\rangle,\bigoplus\limits_{n\geq
    0}P^{2n+1}\ld\langle nd\rangle,\sum\limits_{n\geq 0}  
  s\p_n\right)\makebox[2mm][c]{$\stackrel[\mathclap{S\ld}]{}{\otimes}$}\left(\bigoplus\limits_{n\geq 0} Q^{2n}\ld\langle
  nd\rangle,\bigoplus\limits_{n\geq 
    0}Q^{2n+1}\ld\langle nd\rangle,\sum\limits_{n\geq 0} s\pp_n\right)\\ 
\makebox[12mm][c]{$=$} & M\ld\stab{w}\otimes_{S\ld} N\ld\stab{w\p}.
\end{align*}
\end{proof}

\subsection{Scalar extension and external tensor products}

Next we define scalar extensions and external tensor products of graded matrix factorizations and study their
compatibility with the stabilization functor.

\begin{definition}\label{def_scalarextension}
Let $\varphi: T\ld\to S\ld$ be a local homomorphism of regular local graded rings. We consider $S\ld$ as a graded
$T\ld$-module via $\varphi$. If $M\ua\ld:=(M^0\ld\mor{f} M^{-1}\ld\mor{g} M^0\ld)$ is a graded matrix factorization
of type $(T\ld,w)$, we denote by $M\ua\ld\otimes_{T\ld} S\ld$ or $M\ua\ld\uparrow_{T\ld}^{S\ld}$ the \textit{scalar
  extension of $M\ua\ld$ along $T\ld\to S\ld$}, defined by $$M^0\ld\otimes_{T\ld} S\ld\xrightarrow{f\otimes\id}
M^{-1}\ld\otimes_{T\ld} S\ld\xrightarrow{g\otimes\id} M^0\ld\otimes_{T\ld} S\ld.$$ This is a graded matrix factorization of
type $(S\ld,\varphi(w))$.
\end{definition}

\begin{fact}\label{fact_stabscalar}
Let $T\ld\mor{\varphi}S\ld$ and $w\in \frm_{T\ld}\setminus\{0\}$ be as in definition \ref{def_scalarextension}, and
assume furthermore that $S\ld$ is free as a graded $T\ld$-module. Then, given a finitely generated graded
$T\ld/(w)$-module $M\ld$, we have an isomorphism in $\HMF(S\ld,\varphi(w))$ $$M\ld\stab{w}\uparrow_{T\ld}^{S\ld}\ \ \cong\ \
\left(M\ld\otimes_{T\ld} S\ld\right)\stab{\varphi(w)}.$$ 
\end{fact}

\begin{proof}
As $S\ld/(w)$ is free and in particular flat as a module over $T\ld/(\varphi(w))$, the scalar extension of a graded free
resolution of some graded $T\ld/(w)$-module $M\ld$ along $\varphi$ is a graded free resolution of $M\ld\otimes_{T\ld}
S\ld$ over $S\ld/(w)$. The same holds for a lifting of the $2$-periodic part of such a resolution to $T\ld$. The claim
follows from Proposition \ref{prop_getstab}.
\end{proof}

\begin{definition}\label{def_glueing}
Let $T\ld$, $S\ld$ and $S\p\ld$ be regular local graded rings, and let $T\ld\hookrightarrow S\ld$ and
$T\ld\hookrightarrow S\p\ld$ be homomorphisms of local graded rings. If $M := (M^0\ld\to M^{-1}\ld\to M^0\ld)$ and $N := 
(N^0\ld\to N^{-1}\ld\to N^0\ld)$ are graded matrix factorizations of type $(S\ld,w)$ and $(S\p\ld,w\p)$, respectively,
we define 
\begin{align*}
M\otimes_{T\ld} N\ :=\ M\uparrow_{S\ld}^{S\ld\otimes_{T\ld} S\p\ld}\otimes_{S\ld\otimes_{T\ld} S\ld\p}
N\uparrow_{S\p\ld}^{S\ld\otimes_{T\ld} S\p\ld}.
\end{align*}
This is a graded matrix factorization of type $(S\ld\otimes_{T\ld} S\p\ld, w\otimes 1+ 1\otimes w\p)$.  
\end{definition}

\begin{prop}\label{prop_externaltensorstab}
Let $S\ld$, $S\p\ld$, $T\ld$ and $w\in\frm_{S\ld}\setminus\{0\}$, $w\p\in \frm_{S\p\ld}\setminus\{0\}$ be as in definition
\ref{def_glueing}. Assume $S\ld\otimes_{T\ld} S\p\ld$ is again regular local, and $w\otimes 1 + 1\otimes
w\p\neq 0$ in $S\ld\otimes_{T\ld} S\p\ld$. Further, let $M\ld$ and $N\ld$ be finitely
generated graded modules over $S\ld/(w)$ and $S\p\ld/(w\p)$, respectively, such that $\tor^{T\ld}_k(M\ld,N\ld)=0$ for
all $k>0$. Then there is an  isomorphism in $\HMF(S\ld\otimes_{T\ld} S\p\ld, w\otimes 1+1\otimes
w\p)$ $$M\ld\stab{w}\otimes_{T\ld}N\ld\stab{w\p}\ \ \cong\ \ (N\ld\otimes_{T\ld} M\p\ld)\stab{w+w\p}$$  
\end{prop}
\begin{proof}
As $S\ld$ and $S\p\ld$ are free and in particular flat over $T\ld$, we have canonical
isomorphisms $$\tor^{S\ld\otimes_{T\ld}
  S\p\ld}_k(M\ld\uparrow_{S\ld}^{S\ld\otimes_{T\ld}S\p\ld},N\ld\uparrow_{S\p\ld}^{S\ld\otimes_{T\ld}S\p\ld})= 
\tor^{S\ld\otimes_{T\ld} S\p\ld}_k(M\ld\otimes_{T\ld} S\p\ld,N\ld\otimes_{T\ld} 
S\ld)\ \cong\ \tor^{T\ld}_k(M\ld,N\ld)$$ for all $n\geq 0$. Now the claim follows from Fact \ref{fact_stabscalar} and
Proposition \ref{prop_tensorstabgeneral}.
\end{proof}

In particular, we get the following:

\begin{cor}\label{prop_exttensorproductstab}
Let $S\ld$, $S\p\ld$, $T\ld$ and $w\in\frm_{S\ld}\setminus\{0\}$, $w\p\in \frm_{S\p\ld}\setminus\{0\}$ be as in
Definition \ref{def_glueing}. Assume $S\ld\otimes_{T\ld} S\p\ld$ is again regular local, and $w\otimes 1 + 1\otimes
w\p\neq 0$ in $S\ld\otimes_{T\ld} S\p\ld$. Further, let  
$I\ld\subset S\ld$ and $J\ld\subset S\p\ld$ be homogeneous ideals and assume that there exist regular sequences
$x_1,...,x_n\in S\ld$ and $y_1,...,y_m$ of homogeneous elements in $S\ld$ and $S\p\ld$, respectively, such that $I\ld =
(x_1,...,x_n)$, $J\ld = (y_1,...,y_m)$ and $x_1\otimes 1,...,x_n\otimes 1$,$1\otimes y_1,...,1\otimes y_m$ is regular in
$S\ld\otimes_{T\ld} S\p\ld$. Then there is an isomorphism in $\HMF(S\ld\otimes_{T\ld} S\p\ld, w\otimes 1 + 1\otimes w\p)$
$$(S\ld/I\ld)\stab{w}\otimes_{T\ld} (S\p\ld/J\ld)\stab{w\p}\ \ \cong\ \ (S\ld\otimes_{T\ld} S\p\ld/I\ld\otimes 1+1\otimes
J\ld)\stab{w\otimes 1 + 1\otimes w\p}.$$ 
\end{cor}
\begin{proof}
This is a special case of Proposition \ref{prop_externaltensorstab}. Alternatively, it can be deduced from Fact
\ref{fact_stabscalar} and Corollary \ref{cor_concatenateregularsequences}.   
\end{proof}

\section{Khovanov-Rozansky homology via maximal Cohen-Macaulay modules}\label{sec_description}

In this section we apply the algebraic methods established in the previous section to give an alternative description of
Khovanov-Rozansky homology. For the sake of completeness, we first recall the construction of the generalized
Khovanov-Rozansky homology described in \cite{WuLinkHomology} and \cite{Yonezawa}. We observe that the matrix
factorizations associated to basic MOY-graphs can be written as stabilizations of certain Soergel bimodules and prove
that the tensor products occurring when glueing these matrix factorizations together commute with the stabilization
functor as long as the MOY-graph under consideration is acyclic, i.e. does not possess any oriented cycles, in accordance
with Webster's description of the Khovanov-Rozansky complex of some acyclic MOY-graph, see \cite[Section
2.4]{WebsterCanopolis}. 

\subsection{The construction of generalized KR-homology}\label{sec_KRconstruction}

First, we recall the definition of a (marked) MOY-graph.

\begin{definition}
A \textit{MOY-graph} is a directed graph $\Gamma := (V,E,s,t)$ (with vertices $V$, edges $D$ and source-target functions
$s,t: E\to V$) together with a weight function on its edges $\nu: E\to\{0,1,2,...,n\}$, such that the following
properties hold:  
\begin{enumerate}
\item For all $v\in V$ we have $\eta(v) := |\{\alpha\in E\ |\ v\in\{s(\alpha),t(\alpha)\}\}|\geq 1$. 
\item For all $v\in V$ such that $\eta(v)\geq 2$, we have $$\sum\limits_{\substack{\alpha\in E\\ t(\alpha)=v}}
  \nu(\alpha) = \sum\limits_{\substack{\alpha\in E\\ s(\alpha)=v}} \nu(\alpha).$$
\end{enumerate}
If $(\Gamma,\nu)$ is a MOY-graph, a vertex $v\in V$ satisfying $\eta(v) = 1$ is called an \textit{outer point} of
$\Gamma$. Otherwise $v$ is called an \textit{inner point}.
\end{definition}

\begin{definition}
A \textit{marking} on a MOY-graph $(\Gamma,\nu)$ consists of the following data:
\begin{enumerate}
\item A subset $P\subset|\Gamma|$, whose elements we will call \textit{marked points}, with the following properties:
\begin{enumerate}
\item Every edge of $\Gamma$ contains at least one marked point.
\item Every outer point of $\Gamma$ is marked.
\item No inner point of $\Gamma$ is marked.
\end{enumerate}
\item For every marked point $p\in P$ a set of variables $\XX_p$ with the following properties:
\begin{enumerate}
\item For $p\neq q\in P$ the sets $\XX_p$ and $\XX_q$ are disjoint.
\item If $p$ lies in the interior of the edge $e$ of $\Gamma$, we have $|\XX_p|=\nu(e)$.
\item If $p$ is an outer point of $\Gamma$ and $e\in E$ is the unique edge of $\Gamma$ such that $p\in \{s(e),t(e)\}$,
  we have $|\XX_p|=\nu(e)$. 
\end{enumerate}
\noindent If $p\in P$ is a marked point, we call $|\XX_p|$ the \textit{value} of $p$. 
\end{enumerate}
\end{definition}

If one cuts the edges of a marked MOY-graph $\Gamma$ along their marked points, the graph decomposes into MOY-graphs of
the form $\Gamma^{m_1,...,m_k}_{n_1,...,n_l}$ or $\Gamma_n$ depicted in Figure \ref{fig:buildingblocks}. We will call
these elementary MOY-graphs \textit{building blocks}.  

\begin{figure}[h]\begin{center}
\begin{tikzpicture}
\draw[oredged] (0,0) -- node[txt]{${m_1}$}   (-2,1.2) node[fat]{} node[above]{$\XX_1$};
\draw[oredged] (0,0) -- node[txt]{$m_2$}     (-1,1.2) node[fat]{} node[above]{$\XX_2$};
\draw[oredged] (0,0) -- node[txt]{$m_{k-1}$} (+1,1.2) node[fat]{} node[above]{$\XX_{k-1}$};
\draw[oredged] (0,0) -- node[txt]{$m_k$}     (+2,1.2) node[fat]{} node[above]{$\XX_k$};
\draw[oredgeu] (-2,-1.2) node[fat]{} node[below]{$\YY_1$} -- node[txtp]{$n_1$} (0,0);
\draw[oredgeu] (-1,-1.2) node[fat]{} node[below]{$\YY_2$} -- node[txtp]{$n_2$}     (0,0) ;
\draw[oredgeu] (+1,-1.2) node[fat]{} node[below]{$\YY_{l-1}$} -- node[txtp]{$n_{l-1}$} (0,0);
\draw[oredgeu] (+2,-1.2) node[fat]{} node[below]{$\YY_l$} -- node[txtp]{$n_l$}  (0,0)node[big]{} ;
\node at (0,+1)[above] {$\cdots$};
\node at (0,-1)[below] {$\cdots$};
\node at (0,-2.25) {$\Gamma_{n_1,...,n_l}^{m_1,...,m_k}$};
\node at (6cm,-2.25) {$\Gamma_m^m$};
\begin{scope}[xshift=6cm]
\draw[oredgem] (0,-1.2) node[fat]{} node[below]{$\YY$} -- node[right=8pt,txt,pos=0.5,below=1pt]{$m$}(0,1.2) node[fat]{}
node[above]{$\XX$}; 
\end{scope}
\end{tikzpicture}
\end{center}
\caption{Building blocks}
\label{fig:buildingblocks}
\end{figure}
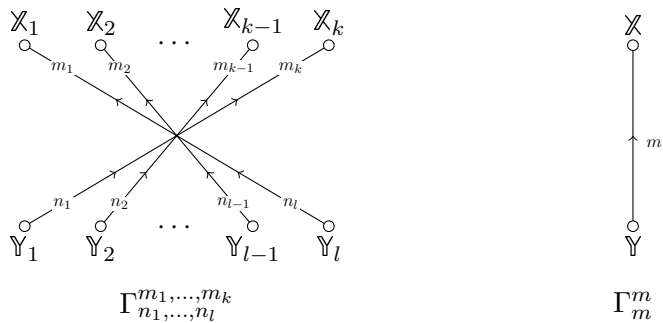

\noindent\textbf{Notation: }\label{notation} Given sets of variables $\XX_i = \{x_{i,1},...,x_{i,m_i}\}$ we denote by
$\Sym(\XX_1|...|\XX_n)$ the 
subring of $\CC[x_{i,j}|1\leq i\leq n,\ 1\leq j\leq m_i]$ consisting of those polynomials which are symmetric in the
variables $x_{i,1},...,x_{i,m_i}$ from $\XX_i$ for each $i=1,...,n$. We have $\Sym(\XX_1|...|\XX_n)\cong\Sym(\XX_1)
\otimes_\CC...\otimes_\CC\Sym(\XX_n)$, so that $\Sym(\XX_1|...|\XX_n)$ is a polynomial ring in the
elementary symmetric polynomials of the $\XX_i$. In particular, it is a regular local graded ring.  Given an arbitrary
set of variables $\XX$ we denote by $X_l\in\Sym(\XX)$ 
the $l$-th elementary symmetric polynomial in the variables contained in $\XX$. Given a variable set like $\XX_i$ which
itself carries in index, we denote the $l$-th elementary symmetric polynomial of the variables contained in $\XX_i$ by
$X_{i,l}$. If the set of variables under consideration is not denoted by a single letter, but for example is of the form
$\XX\cup\YY$, then we denote the $l$-th elementary symmetric polynomial of $\XX\cup\YY$ by $(\XX\cup\YY)_l$ etc. For sets
of variables $\XX_1,...,\XX_n$ we denote by $\XX$ the union $\bigcup\XX_i$, \textit{but} we set $\Sym(\XX) :=
\Sym(\XX_1|...|\XX_n)$, overloading the previous definition of $\Sym(\XX)$. If we want to talk about the ring of
completely symmetric polynomials in the variables of $\XX = \bigcup\XX_i$, we will write $\Sym(\XX^\cup)$ for this
instead. Finally we set $\Sigma\XX^{n+1} := \sum\limits_{x\in\XX} x^{n+1}\in\Sym(\XX)$ for an arbitrary set of variables
$\XX$.  

Given a marked MOY-graph $\Gamma$ we associate to it a matrix factorization $\KR(\Gamma)$
along the following steps, which we will look at more closely below:
\begin{enumerate}
\item Cut $\Gamma$ along its marked points to get the basic MOY-subgraphs $\Gamma_1,...,\Gamma_r$ of $\Gamma$.
\item To each of the $\Gamma_i$ with ingoing variables $\YY_1,...,\YY_l$ and outgoing variables $\XX_1,...,\XX_k$
  associate a graded matrix factorization $\KR(\Gamma_i)$ of type $(\Sym(\XX|\YY), \Sigma\XX^{n+1}-\Sigma\YY^{n+1})$.
\item Glue together the matrix factorizations associated to the basic MOY-subgraphs of $\Gamma$ along their common
  variables: $\KR(\Gamma) := \bigotimes\limits_{i=1}^{r} \KR(\Gamma_r)$.  
\end{enumerate}
We first explain step (2) in detail, beginning with the matrix factorization associated to the basic MOY-graph
$\Gamma_{n_1,...,n_l}^{m_1,...,m_k}$ (see Figure \ref{fig:buildingblocks}). Put $m := \sum\limits_{i=1}^{k}
|\XX_i|=\sum\limits_{j=1}^{l}|\YY_j|$. We define $$\KR(\Gamma_{n_1,...,n_l}^{m_1,...,m_k})\ \ := \ \
\bigotimes\limits_{i=1}^{m} \{\ast_i, X_i-Y_i\}\left\langle\sum\limits_{1\leq i<j\leq k} m_i m_j\right\rangle,$$
considered as a matrix factorization of type $(\Sym(\XX|\YY),\Sigma\XX^{n+1}-\Sigma\YY^{n+1})$. Here the $\ast_i$ are
homogeneous elements chosen in such a way that  $\sum\limits_{i=1}^{m} \ast_i (\XX_i -\YY_i) =
\Sigma\XX^{n+1}-\Sigma\YY^{n+1}$. As the sequence  $(X_i-Y_i)_{1\leq i\leq m}$ is regular in $\Sym(\XX|\YY)$, the
particular choice of the $\ast_i$ is irrelevant by Proposition \ref{prop_koszulstabilization2}. The case of the basic
MOY-graph $\Gamma_m^m$ (see \ref{fig:buildingblocks}) is similar; we set $$\KR(\Gamma_m^m)\ \ :=\ \ 
\bigotimes\limits_{i=1}^{m} \{\ast_i,X_i-Y_i\},$$ considered as a graded matrix factorization of type
$(\Sym(\XX|\YY),\Sigma\XX^{n+1}-\Sigma\YY^{n+1})$. Again, the $\ast_i$ are homogeneous elements chosen in such a way
that $\sum\limits_{i=1}^{m} \ast_i(X_i-Y_i)=\Sigma\XX^{n+1}-\Sigma\YY^{n+1}$.

Summing up, we have shown the following:
\begin{theorem}\label{thm_KRviaStab}
There is a canonical isomorphism in $\HMF(\Sym(\XX|\YY),\Sigma\XX^{n+1}-\Sigma\YY^{n+1})$:
$$\KR\left(\Gamma_{n_1,...,n_l}^{m_1,...,m_k}\right)=\KR\left(\ \begin{tikzpicture}[scale=0.5,baseline=-0.6mm]
\draw[oredged] (0,0) -- node[txt,scale=0.8]{$m_1$}   (-2,1.2) node[fat]{} node[above,scale=0.75]{$\XX_1$}; 
\draw[oredged] (0,0) -- node[txt,scale=0.8]{$m_2$}     (-1,1.2) node[fat]{} node[above,scale=0.75]{$\XX_2$};
\draw[oredged] (0,0) -- node[txt,scale=0.8]{$m_{k-1}$} (+1,1.2) node[fat]{} node[above,scale=0.75]{$\XX_{k-1}$};
\draw[oredged] (0,0) -- node[txt,scale=0.8]{$m_k$}     (+2,1.2) node[fat]{} node[above,scale=0.75]{$\XX_k$};
\draw[oredgeu] (-2,-1.2) node[fat]{} node[below,scale=0.75]{$\YY_1$} -- node[txtp,scale=0.8]{$n_1$} (0,0);
\draw[oredgeu] (-1,-1.2) node[fat]{} node[below,scale=0.75]{$\YY_2$} -- node[txtp,scale=0.8]{$n_2$}     (0,0) ;
\draw[oredgeu] (+1,-1.2) node[fat]{} node[below,scale=0.75]{$\YY_{l-1}$} -- node[txtp,scale=0.8]{$n_{l-1}$} (0,0);
\draw[oredgeu] (+2,-1.2) node[fat]{} node[below,scale=0.75]{$\YY_l$} -- node[txtp,scale=0.8]{$n_l$}  (0,0)node[big]{} ;
\node at (0,+1)[above,scale=0.75] {$\cdots$};
\node at (0,-1)[below,scale=0.75] {$\cdots$};\end{tikzpicture}\ \right)\simeq \left(\Sym(\XX|\YY)/(X_i-Y_i)
\langle r\rangle\right)^{\left\{\Sigma\XX^{n+1}-\Sigma\YY^{n+1}\right\}},$$ where $r := \sum\limits_{1\leq i<j\leq k}
m_i m_j$. 
\end{theorem}

\begin{rem} Note that according to our definitions $\Gamma^m_m$ denotes two different basic MOY-graphs. However, this
  causes no trouble, as the two factorizations associated to them are the same.\end{rem}
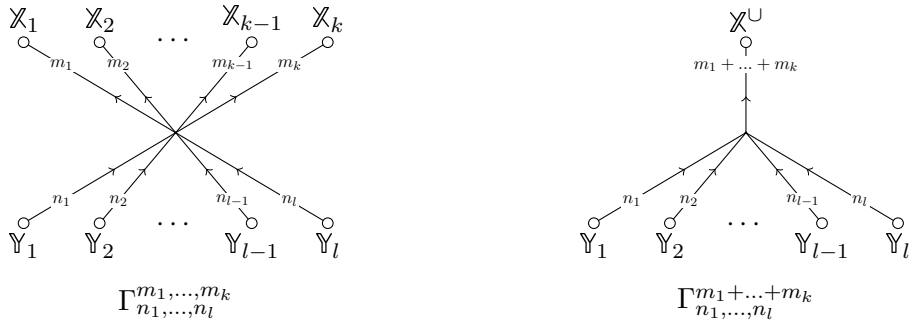
\begin{figure}[h]\begin{center}
\begin{tikzpicture}
\draw[oredged] (0,0) -- node[txt]{${m_1}$}   (-2,1.2) node[fat]{} node[above]{$\XX_1$};
\draw[oredged] (0,0) -- node[txt]{$m_2$}     (-1,1.2) node[fat]{} node[above]{$\XX_2$};
\draw[oredged] (0,0) -- node[txt]{$m_{k-1}$} (+1,1.2) node[fat]{} node[above]{$\XX_{k-1}$};
\draw[oredged] (0,0) -- node[txt]{$m_k$}     (+2,1.2) node[fat]{} node[above]{$\XX_k$};
\draw[oredgeu] (-2,-1.2) node[fat]{} node[below]{$\YY_1$} -- node[txtp]{$n_1$} (0,0);
\draw[oredgeu] (-1,-1.2) node[fat]{} node[below]{$\YY_2$} -- node[txtp]{$n_2$}     (0,0) ;
\draw[oredgeu] (+1,-1.2) node[fat]{} node[below]{$\YY_{l-1}$} -- node[txtp]{$n_{l-1}$} (0,0);
\draw[oredgeu] (+2,-1.2) node[fat]{} node[below]{$\YY_l$} -- node[txtp]{$n_l$}  (0,0)node[big]{} ;
\node at (0,+1)[above] {$\cdots$};
\node at (0,-1)[below] {$\cdots$};
\node at (0,-2.25) {$\Gamma_{n_1,...,n_l}^{m_1,...,m_k}$};
\node at (7.5cm,-2.25) {$\Gamma_{n_1,...,n_l}^{m_1+...+m_k}$};
\begin{scope}[xshift=7.5cm]
\draw[oredged] (0,0) -- node[txt]{${m_1+...+m_k}$}   (0,1.2) node[fat]{} node[above]{$\XX^\cup$};
\draw[oredgeu] (-2,-1.2) node[fat]{} node[below]{$\YY_1$} -- node[txtp]{$n_1$} (0,0);
\draw[oredgeu] (-1,-1.2) node[fat]{} node[below]{$\YY_2$} -- node[txtp]{$n_2$}     (0,0) ;
\draw[oredgeu] (+1,-1.2) node[fat]{} node[below]{$\YY_{l-1}$} -- node[txtp]{$n_{l-1}$} (0,0);
\draw[oredgeu] (+2,-1.2) node[fat]{} node[below]{$\YY_l$} -- node[txtp]{$n_l$}  (0,0)node[big]{} ;
\node at (0,-1)[below] {$\cdots$};
\end{scope}
\end{tikzpicture}
\end{center}
\caption{Comparison of two basic MOY-graphs}
\label{fig:notethering}
\end{figure}
\begin{rem} The matrix factorizations associated to $\Gamma_{n_1,...,n_l}^{m_1,...,m_k}$ and
  $\Gamma_{n_1,...,n_l}^{m_1+...+m_k}$ in Figure \ref{fig:notethering} only differ with respect to the choice of the
  base ring, but not in the choice of potential or the module which is stabilized: in the first case, the base ring is
  $\Sym(\XX|\YY)$, while in the second case it's the subring $\Sym(\XX^\cup|\YY)\subset\Sym(\XX|\YY)$.

In general, the base ring associated to a family of edges with a high value is a subring of the base ring associated to
the configuration of edges where some edges have been split into several 
  edges with smaller value. In this way we are naturally led to consider the graded ranks of rings of symmetric
  polynomials considered as graded modules over smaller rings of symmetric polynomials, and these admit
  interpretations as Poincar\'{e}-polynomials of certain flag-varieties: for example, the graded rank of $\Sym(\XX)$ over
  $\Sym(\XX^\cup)$ equals the Poincar\'{e} polynomial of the algebra $\Sym(\XX)/\langle\Sym(\XX^\cup)_+\rangle$, which is
  isomorphic to the complex cohomology ring of the flag variety $\Fl(|\XX_1|,...,|\XX_n|)$ of flags in $\CC^{|\XX|}$
  with dimension differences $|\XX_i|$. Thus, these graded ranks carry interesting information, and it is therefore very
  important to be aware of which ring we are working with. See also Example \ref{ex_splitmerge}\end{rem} 

Finally, we consider step (3), the glueing of all matrix factorizations associated to basic MOY-graphs according to
their common endpoints in more detail. Let $\Gamma_0,...,\Gamma_n$ be the basic MOY-graphs of $\Gamma$ and let
$\KR(\Gamma_i)$ the matrix factorizations associated to them. Then, we first consider the exterior tensor product
$\tilde{\KR}(\Gamma) := \KR(\Gamma_0)\otimes_{\CC} ...\otimes_{\CC} \KR(\Gamma_n)$ of these factorizations, graphically
corresponding to the disjoint union of the $\Gamma_i$. Now, $\KR(\Gamma)$ is a certain quotient of $\tilde{C}(\Gamma)$,
intuitively identifying common endpoints of the $\Gamma_i$. That is, given $i\neq j$ such that $\Gamma_i$ and $\Gamma_j$
share the sets of variables $\ZZ$, we let action of elements of $\Sym(\ZZ)$ pass from $\KR(\Gamma_i)$ to $\KR(\Gamma_j)$
and vice versa; thus, if we were only looking at $\Gamma_i$ and $\Gamma_j$, the resulting quotient of
$\KR(\Gamma_i)\otimes_\CC \KR(\Gamma_j)$ would just be $\KR(\Gamma_i)\otimes_{\Sym(\ZZ)} \KR(\Gamma_j)$. 

For a detailed example of how to calculate the value of the unknot, see Section \ref{subsec_ex}.

\subsection{The matrix factorization associated to an acyclic MOY-graph}

In the previous section we saw that the matrix factorization associated to a basic MOY-graph can be written as the
stabilization of a 'singular' Soergel bimodule (see \cite{Williamson} and \cite{StroppelGolod}), and we also know that
these matrix factorizations are tensored together in order to get the matrix factorization associated to more
complicated MOY-graphs. Further, in Proposition \ref{prop_externaltensorstab} we gave a sufficient condition for tensor
products of matrix factorizations and stabilizations to commute. In this section, we will see that the conditions for
Proposition \ref{prop_externaltensorstab} are satisfied as long as the MOY-graph under consideration is acyclic,
i.e. does not possess any oriented cycles. In particular, we see that the matrix factorization associated to a MOY-braid 
is isomorphic to the stabilization of the corresponding Soergel bimodule. 

\begin{ex}
We begin by discussing in full detail a very simple example, namely we determine the matrix factorization associated to
the MOY-graph $\Gamma_0$ in Figure \ref{fig:exampleglueing}. We have the following isomorphisms in
$\HMF(\CC[X,Z],X^{n+1}-Z^{n+1})$ (explanations are given below):   
\begin{align*}
\KR(\Gamma_0) &
\makebox[8mm][c]{$\stackrel{\text{def}}{=}$}\CC[X,Y]/(X-Y)\stab{X^{n+1}-Y^{n+1}}\uparrow_{\CC[X,Y]}^{\CC[X,Y,Z]}
    \\ & \makebox[8mm][c]{}\ \ \
    \left.\left.\makebox[8mm][c]{$\stackrel[\mathclap{\CC[X,Y,Z]}]{}{\otimes}$}\CC[Y,Z]/(Y-Z)\stab{Y^{n+1}-Z^{n+1}} 
    \uparrow^{\CC[X,Y,Z]}_{\CC[Y,Z]}\right.\right\downarrow^{\CC[X,Y,Z]}_{\CC[X,Z]} \\  
& \makebox[8mm][c]{$\stackrel{\ref{fact_stabscalar}}{\cong}$}
\CC[X,Y,Z]/(X-Y)\stab{X^{n+1}-Y^{n+1}}\\
& \left.\makebox[8mm][c]{}\ \ \ \makebox[8mm][c]{$\stackrel[\mathclap{\CC[X,Y,Z]}]{}{\otimes}$}
  \CC[X,Y,Z]/(Y-Z)\stab{Y^{n+1}-Z^{n+1}}\right\downarrow^{\CC[X,Y,Z]}_{\CC[X,Z]}\\ 
& \makebox[8mm][c]{$\stackrel{\ref{prop_exttensorproductstab}}{\cong}$}
\left.\CC[X,Y,Z]/(X-Y,Y-Z)\stab{X^{n+1}-Z^{n+1}}\right\downarrow^{\CC[X,Y,Z]}_{\CC[X,Z]}\\ 
& \makebox[8mm][c]{$\cong$} \left.\CC[X,Z]/(X-Z)\stab{X^{n+1}-Z^{n+1}}\right\downarrow^{\CC[X,Y,Z]}_{\CC[X,Z]}\\
& \makebox[8mm][c]{$\stackrel{\ref{cor_forget}}{\cong}$} \CC[X,Z]/(X-Z)\stab{X^{n+1}-Z^{n+1}}\end{align*}
Here, Fact \ref{fact_stabscalar} is applicable because $\CC[X,Y,Z]$ is free over $\CC[X,Y]$ and $\CC[Y,Z]$, and
Proposition \ref{prop_exttensorproductstab} can be applied because $X-Y$, $Y-Z$ is regular in $\CC[X,Y,Z]$. In the last
step, we may apply Corollary \ref{cor_forget} because $\CC[X,Y,Z]$ is a free $\CC[X,Z]$-module.  \end{ex}

The last example shows quite quell how the machinery established to far can be used to make graphically intuitive
relations between matrix factorizations rigorous, without forcing us to actually write down explicit homotopy
equivalences between them. In the above example, we had to 'compute' only one thing,
namely
\begin{align*}
\CC[X,Y,Z]/(X-Y)\stackrel[{\CC[X,Y,Z]}]{}{\otimes}\CC[X,Y,Z]/(Y-Z) & \makebox[8mm][c]{$\cong$} \CC[X,Y,Z]/(X-Y,Y-Z) \\
& \makebox[8mm][c]{$\cong$} \CC[X,Z]/(X-Z).\end{align*}

In much the same way we can handle more complicated glueings of MOY-graphs; only the application of
Proposition \ref{prop_exttensorproductstab} becomes more difficult. We consider another example.

\begin{figure}[h]
\begin{center}
\begin{tikzpicture}
\draw[oredgem] (0,+0.6) --node[txt,pos=0.5,left=8pt]{$1$} (0,+2.4) node[fat]{};
\draw[oredgem] (0,-1.2) node[fat]{} --node[txt,pos=0.5,left=8pt]{$1$} (0,0.6)node[fat]{};
\draw[decorate,decoration={brace,amplitude=5pt}] (-0.75,-1.2) -- node[left=6pt]{$\Gamma_0$}(-0.75,2.4);
\draw[decorate,decoration={brace,amplitude=5pt}] (+0.75,+0.5) -- node[right=6pt]{$\Gamma_1^1$}(+0.75,-1.2);
\draw[decorate,decoration={brace,amplitude=5pt}] (+0.75,+2.4) -- node[right=6pt]{$\Gamma_1^1$}(+0.75,0.7);
\begin{scope}[xshift=7cm]
\begin{scope}[xshift=1.3cm,yshift=1.8cm]
\draw[oredged] (0,0) -- node[left,below=5pt,txt]{$s_1$}      (-0.7,0.6) node[fat]{} node[above]{$\WW_1$};
\draw[oredged] (0,0) -- node[txt,right=3pt,below=9pt]{$s_q$} (+0.7,0.6) node[fat]{} node[above]{$\WW_q$};
\draw[oredgeu] (-0.7,-0.6) -- (0,0);
\draw[oredgeu] (+0.7,-0.6) -- (0,0) node[big]{};
\node at (0, +0.6) {$\cdots$};
\node at (0, -0.6) {$\cdots$};
\end{scope}
\draw[oredged] (0,0) -- node[txt]{$m_1$} (-2,1.2) node[fat]{} node[above]{$\XX_1$};
\draw[oredged] (0,0) -- node[txt]{$m_k$} (-0.6,1.2) node[fat]{} node[above]{$\XX_k$};
\draw[oredged] (0,0) -- node[txt]{$n_1$} (+0.6,1.2) node[fat]{} node[left=2pt,above]{$\YY_{1}$};
\draw[oredged] (0,0) -- node[txt]{$n_l$} (+2,1.2) node[fat]{} node[right=2pt,above]{$\YY_l$};
\draw[oredgeu] (-2,-1.2) node[fat]{} node[below]{$\ZZ_1$} -- node[txtp]{$r_1$} (0,0);
\draw[oredgeu] (-1,-1.2) node[fat]{} node[below]{$\ZZ_2$} -- node[txtp]{$r_2$}     (0,0) ;
\draw[oredgeu] (+1,-1.2) node[fat]{} node[below]{$\ZZ_{p-1}$} -- node[txtp]{$r_{p-1}$} (0,0);
\draw[oredgeu] (+2,-1.2) node[fat]{} node[below]{$\ZZ_p$} -- node[txtp]{$r_p$}  (0,0)node[big]{} ;
\node at (0, -1.2) {$\cdots$};
\node at (-1.3, +1.2) {$\cdots$};

\draw[decorate,decoration={brace,amplitude=5pt}] (-2.6,-1.2) -- node[left=6pt]{$\Gamma_1$}(-2.6,2.4);
\draw[decorate,decoration={brace,amplitude=5pt}] (2.6,+1) --
node[right=6pt]{$\Gamma_{r_1,...,r_p}^{m_1,...,m_k,n_1,...,n_l}$}(2.6,-1); 
\draw[decorate,decoration={brace,amplitude=5pt}] (2.6,+2.6) --
node[right=6pt]{$\Gamma_{n_1,...,n_l}^{s_1,...,s_q}$}(2.6,1.4); 
\end{scope}
\end{tikzpicture}
\end{center}
\caption{Gluing two basic MOY graphs}
\label{fig:exampleglueing}
\end{figure}
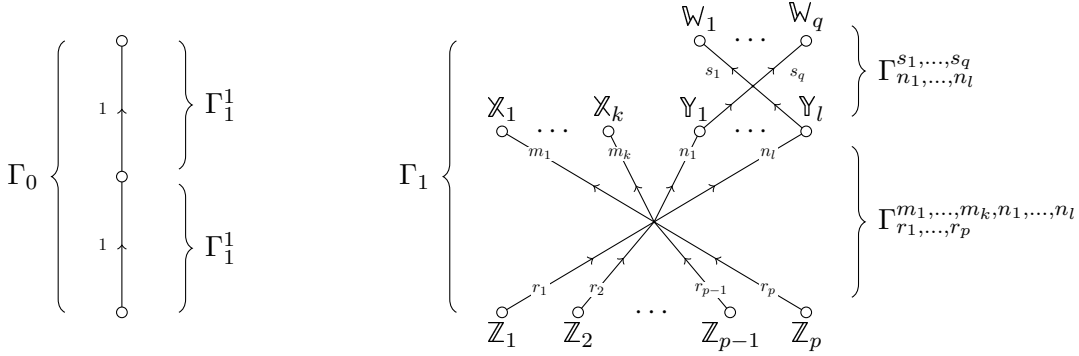
\begin{ex}
We compute the matrix factorization associated to $\Gamma_1$ in Figure \ref{fig:exampleglueing}. Put $u := r_1+...+r_p$
and $v := s_1+...+s_q$. To ease the notation, we will omit the internal degree shifts in our calculation. As above we
first get 
\begin{align*}
\KR(\Gamma_1) & \makebox[8mm][c]{$\stackrel{\text{def}}{=}$}\Sym(\XX|\YY|\ZZ)/((\XX\cup\YY)_i -  
Z_i)\stab{\Sigma\XX^{n+1}+\Sigma\YY^{n+1}-\Sigma\ZZ^{n+1}}\uparrow_{\Sym(\XX|\YY|\ZZ)}^{\Sym(\XX|\YY|\ZZ|\WW)}\\  
&\quad\quad\left.\makebox[8mm][c]{$\stackrel[\mathclap{\Sym[\XX|\YY|\ZZ|\WW]}]{}{\otimes}$}    
\Sym(\WW|\YY)/(W_i-Y_i)\stab{\Sigma\WW^{n+1}-\Sigma\YY^{n+1}}\uparrow_{\Sym(\WW|\YY)}^{\Sym(\XX|\YY|\ZZ|\WW)}\right
\downarrow^{\Sym(\XX|\YY|\ZZ|\WW)}_{\Sym(\XX|\ZZ|\WW)}\\
& \makebox[8mm][c]{$\stackrel{\ref{fact_stabscalar}}{\cong}$}\Sym(\XX|\YY|\ZZ|\WW)/
  ((\XX\cup\YY)_i-Z_i)\stab{\Sigma\XX^{n+1}+\Sigma\YY^{n+1}-\Sigma\ZZ^{n+1}}\\
&\quad\quad\left. \makebox[10mm][c]{$\stackrel[\mathclap{\Sym(\XX|\YY|\ZZ|\WW)}]{}{\otimes}$}
\Sym(\XX|\YY|\ZZ|\WW)/(W_i-Y_i)\stab{\Sigma\WW^{n+1}-\Sigma\YY^{n+1}} 
\right\downarrow^{\Sym(\XX|\YY|\ZZ|\WW)}_{\Sym(\XX|\ZZ|\WW)}
\end{align*}
In the next step we want to apply Proposition \ref{prop_exttensorproductstab} to exchange tensoring and the
stabilization functor. For this, we have to see that the concatenation of the sequences 
$(\XX\cup\YY)_i-Z_i, 1\leq i\leq u$ and $W_j-Y_j, 1\leq j\leq v$ is regular in $\Sym(\XX|\YY|\ZZ|\WW)$. Intuitively this 
is plausible, as attaching $\Gamma_{n_1,...,n_l}^{s_1,...,s_q}$ to $\Gamma_{r_1,...,r_p}^{m_1,...,m_k,n_1,...,n_k}$
introduces new variables from $\WW$, and the attached sequence $W_j-Y_j, 1\leq j\leq n$ becomes regular under the map
$\Sym(\WW|\YY)\to\Sym(\WW)$.  

We can make this intuition rigorous as follows: The sequences $(X_{i,j})$ and $(Y_{i,j})$ of elementary symmetric
polynomials in the variables of $\XX_i$ and $\YY_i$ are regular in $\Sym(\XX)$ and $\Sym(\YY)$, respectively, and hence
their concatenation is regular in $\Sym(\XX|\YY)\cong\Sym(\XX)\otimes_{\CC} \Sym(\YY)$ spanning the ideal
$\Sym(\XX|\YY)_+$. Further $Z_1,...,Z_m$ is regular in $\Sym(\ZZ)$, and we have $\langle Z_1,...,Z_m\rangle =
\langle\Sym(\ZZ^\cup)_+\rangle\subset\Sym(\ZZ).$ 
Since $\Sym(\XX|\YY|\ZZ)/\langle\Sym(\XX|\YY)_+\rangle\cong\Sym(\ZZ)$ we therefore see that
$X_{i,j},Y_{i,j},Z_k-(\XX\cup\YY)_k$ is regular in $\Sym(\XX|\YY|\ZZ)$ spanning the ideal $\Sym(\XX|\YY)_+\otimes_\CC
\Sym(\ZZ) + 
\Sym(\XX|\YY)\otimes_\CC\langle\Sym(\ZZ^\cup)_+\rangle\subset\Sym(\XX|\YY|\ZZ)$. Hence $$\Sym(\XX|\YY|\ZZ|\WW)
/(X_{i,j},Y_{i,j},Z_k-(\XX\cup\YY)_k)\cong\Sym(\WW)\otimes_\CC\Sym(\ZZ)/\langle\Sym(\ZZ^\cup)_+\rangle$$     
and so $X_{i,j},Y_{i,j},Z_k-(\XX|\YY)_k,W_k-Y_k$ is regular in $\Sym(\XX|\YY|\ZZ|\WW)$. As the regularity of a sequence
in a local graded ring is independent of its ordering, we deduce that $Z_k-(\XX|\YY)_k,W_k-Y_k$ is regular as claimed. 
We are therefore allowed to apply Proposition \ref{prop_exttensorproductstab} and get 
\begin{align*} 
  \KR(\Gamma_1) &
  \cong\left.\Sym(\XX|\YY|\ZZ|\WW)/((\XX\cup\YY)_i-Z_i,W_i-Y_i)\stab{\Sigma\WW^{n+1}+\Sigma\XX^{n+1}-\Sigma\ZZ^{n+1}}
  \right\downarrow^{\Sym(\XX|\YY|\WW|\ZZ)}_{\Sym(\XX|\ZZ|\WW)} 
\end{align*}
which is what we expected. Note that, in contrast to the first example, the identification of $Y_i$ and $W_i$ does
\textit{not} imply that $\Sym(\XX|\YY|\ZZ|\WW)/((\XX\cup\YY)_i-Z_i,W_i-Y_i)$ equals $\Sym(\XX|\ZZ|\WW)/ 
((\XX\cup\WW)_i-Z_i)$, just because the ideal spanned by the $Y_i$ in $\Sym(\YY)$ equals
$\langle\Sym(\YY^\cup)_+\rangle$, and this is strictly contained in $\Sym(\YY)_+$ if there was more than one
$\YY_i$. Note however that since $\Sym(\YY)$ is free of finite rank over $\Sym(\YY^\cup)$,
$\Sym(\XX|\YY|\ZZ|\WW)/((\XX\cup\YY)_i-Z_i,W_i-Y_i)$ is finitely generated over $\Sym(\XX|\ZZ|\WW)$. 
\end{ex}
In this example it was already a bit tricky to apply Proposition \ref{prop_exttensorproductstab}, because we to
convince ourself first that the sequence occurring was indeed regular. However, the real problem arises if we now want
to glue our matrix factorization with another one; namely, if we consider
$\Sym(\XX|\YY|\ZZ|\WW)/((\XX\cup\YY)_i-Z_i,W_i-Y_i)$ only as a module over the 'free' variables in $\XX,\ZZ,\WW$ we
loose the nice presentation of the module as a complete intersection, and Proposition \ref{prop_exttensorproductstab} is
no longer available. One could solve this problem by not forgetting the internal variables, but this seems
unnatural. Instead it is more natural to replace Proposition \ref{prop_exttensorproductstab} by Proposition
\ref{prop_externaltensorstab} and to argue through the freeness of the modules at the glueing points. 
\begin{definition}
Let $\XX_1,...,\XX_n$ be sets of variables and $M\ld$ be a graded module over $\Sym(\XX)$. We say that $M\ld$ is
\textit{free over $\XX_i$} if $M\ld$ is free as a graded module over $\Sym(\XX_i)\subset\Sym(\XX)$.
\end{definition}
\begin{theorem}\label{thm_glue}
Let $\XX_1,...,\XX_n,\YY,\ZZ_1,...,\ZZ_n$ be sets of variables (the case $\YY=\emptyset$ is explicitly allowed),
$w\in\Sym(\XX)$, $w\p\in\Sym(\ZZ)$ and $w_0\in\Sym(\YY)$. Further, let $M\ld$, $N\ld$ be graded modules over
$\Sym(\XX|\YY)$ and $\Sym(\YY|\ZZ)$, respectively, such that the following hold: 
\begin{enumerate}
\item $M\ld$ is finitely generated over $\Sym(\XX)$, and $N\ld$ is finitely generated over $\Sym(\ZZ)$
\item $(w+w_0)M\ld = \{0\}$ and $(w\p-w_0)N\ld=\{0\}$.
\item At least one of the modules $M\ld$ and $N\ld$ is free over $\YY$.
\end{enumerate}
Then the following hold:
\begin{enumerate}
\item[(1)]
$M\ld\otimes_{\Sym(\YY)}N\ld$ is finitely generated over $\Sym(\XX|\ZZ)/(w+w\p)$ and there is an isomorphism 
 $$M\ld\stab{w+w_0}\otimes_{\Sym(\YY)} N\ld\stab{w\p-w_0}\ \cong\ (M\ld\otimes_{\Sym(\YY)} N\ld)\stab{ w+w\p}.$$
\item[(2)] If $M\ld$ is free over $\YY$, then the graded $\Sym(\XX|\ZZ)$-module $M\ld\otimes_{\Sym(\YY)}N\ld$
is free over any subset of $\ZZ$ over which $M\ld$ is free.
\item[(3)] If $N\ld$ is free over $\YY$, then $M\ld\otimes_{\Sym(\YY)}N\ld$
is free over any subset of $\XX$ over which $M\ld$ is free.
\end{enumerate}
\end{theorem}
\begin{proof}
Since $M\ld$ is finitely generated over $\Sym(\XX)$ and $N\ld$ is finitely generated over $\Sym(\ZZ)$,\break
$M\ld\otimes_{\Sym(\YY)}N\ld$ is finitely generated over $\Sym(\XX|\ZZ)$. Further, $M\ld\otimes_{\Sym(\YY)} N\ld$ is
annihilated by the ideal $(w+w_0,w\p-w_0)$ an in particular by the element $w+w\p = (w+w_0)+(w\p-w_0)$. Therefore, the
expression $(M\ld\otimes_{\Sym(\YY)}N\ld)\stab{w+w\p}$ makes sense. Further, $\tor^{\Sym(\YY)}_k(M\ld,N\ld)=0$ for
all $k>0$, since by assumption either $M\ld$ or $N\ld$ is free $\Sym(\YY)$. Thus we can apply Proposition
\ref{prop_externaltensorstab} and Corollary \ref{cor_forget} to get
\begin{align*}
\left.M\ld\stab{w+w_0}\otimes_{\Sym(\YY)} N\ld\stab{w\p-w_0}\ 
\right\downarrow^{\Sym(\XX|\YY|\ZZ)}_{\Sym(\XX|\ZZ)} & \makebox[8mm][c]{$\stackrel{\ref{prop_externaltensorstab}}{\cong}$}
\left.(M\ld\otimes_{\Sym(\YY)} N\ld)\stab{w+w\p}\right\downarrow^{\Sym(\XX|\YY|\ZZ)}_{\Sym(\XX|\ZZ)}\\
& \makebox[8mm][c]{$\stackrel{\ref{cor_forget}}{\cong}$} (M\ld\otimes_{\Sym(\YY)} N\ld)\stab{ w+w\p}
\end{align*} as claimed. This shows statement (1), and statements (2) and (3) are obvious.
\end{proof}

\begin{ex} In the following examples, we again omit internal degree shifts.

(1) First we consider the basic MOY-graph $\Gamma_{n_1,...,n_l}^{m_1+...+m_k}$ in Figure
\ref{fig:notethering}. The matrix factorization associated to it is the stabilization if the $\Sym(\XX^\cup|\YY)$-module
$\Sym(\XX^\cup|\YY)/(X_i-Y_i)$ with respect to $\Sigma\XX^{n+1}-\Sigma\YY^{n+1}$. Since $\Sym(\XX^\cup)_+ = \langle
X_i\rangle$ we have $\Sym(\XX^\cup|\YY)/(X_i-Y_i)\cong\Sym(\YY)$ as $\Sym(\XX|\YY)$-modules, where $X_i\in \Sym(\XX^\cup)$
acts on $\Sym(\YY)$ through multiplication by $X_i$. Therefore $\Sym(\XX^\cup|\YY)/(X_i-Y_i)$ is free of rank $1$ over
$\YY$ and free of rank $\rg_{\Sym(\YY^\cup)}\Sym(\YY)$ over $\XX^\cup$. 

(2) Next we look at the matrix factorization associated to an arbitrary basic MOY-graph
$\Gamma_{n_1,...,n_l}^{m_1,...,m_k}$. We claim that this is just the glueing/the tensor product of the matrix
factorizations associated to the graphs $\Gamma_{n_1,...,n_l}^{n_1+...+n_k}$ and $\Gamma_{n_1+...+n_k}^{m_1,...,m_k}$
along their common end $\ZZ$; see Figure \ref{fig:mergesplit}. Indeed: example (1) shows that the assumptions of
Proposition \ref{thm_glue} are satisfied, and so we get
\begin{align*} &\ \ \ \Sym(\XX|\ZZ)/(X_i-Z_i)\stab{\Sigma\XX^{n+1}-\Sigma\ZZ^{n+1}}
\otimes_{\Sym(\ZZ)}\Sym(\YY|\ZZ)/(Z_i-Y_i)\stab{\Sigma\ZZ^{n+1}-\Sigma\YY^{n+1}}\\ \cong &\ \ \ 
\Sym(\XX|\YY)/(X_i-Y_i)\stab{\Sigma\XX^{n+1}-\Sigma\YY^{n+1}}.\end{align*}
$\Sym(\XX|\YY)/(X_i-Y_i)$ is free of rank $\rg_{\Sym(\YY^\cup)}\Sym(\YY)$ over $\XX$ and free of rank 
$\rg_{\Sym(\XX^\cup)}\Sym(\XX)$ over $\YY$. 

(3) Examples (1) and (2) show that we can apply Theorem \ref{thm_glue} to the examples which led us to Theorem
\ref{thm_glue}. Thus we also get the isomorphisms established there without the somewhat cumbersome use of regular
sequences.
\end{ex} 

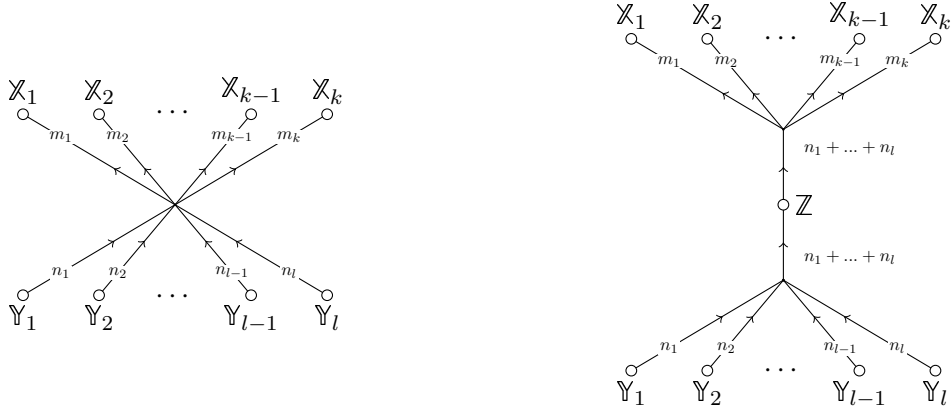
\begin{figure}[h]\begin{center}
\begin{tikzpicture}
\begin{scope}[yshift=-1cm]
\draw[oredged] (0,0) -- node[txt]{${m_1}$}   (-2,1.2) node[fat]{} node[above]{$\XX_1$};
\draw[oredged] (0,0) -- node[txt]{$m_2$}     (-1,1.2) node[fat]{} node[above]{$\XX_2$};
\draw[oredged] (0,0) -- node[txt]{$m_{k-1}$} (+1,1.2) node[fat]{} node[above]{$\XX_{k-1}$};
\draw[oredged] (0,0) -- node[txt]{$m_k$}     (+2,1.2) node[fat]{} node[above]{$\XX_k$};
\draw[oredgeu] (-2,-1.2) node[fat]{} node[below]{$\YY_1$} -- node[txtp]{$n_1$} (0,0);
\draw[oredgeu] (-1,-1.2) node[fat]{} node[below]{$\YY_2$} -- node[txtp]{$n_2$}     (0,0) ;
\draw[oredgeu] (+1,-1.2) node[fat]{} node[below]{$\YY_{l-1}$} -- node[txtp]{$n_{l-1}$} (0,0);
\draw[oredgeu] (+2,-1.2) node[fat]{} node[below]{$\YY_l$} -- node[txtp]{$n_l$}  (0,0)node[big]{} ;
\node at (0,+1)[above] {$\cdots$};
\node at (0,-1)[below] {$\cdots$};
\end{scope}
\begin{scope}[xshift=8cm]
\draw[oredged] (0,0) -- node[txt]{${m_1}$}   (-2,1.2) node[fat]{} node[above]{$\XX_1$};
\draw[oredged] (0,0) -- node[txt]{$m_2$}     (-1,1.2) node[fat]{} node[above]{$\XX_2$};
\draw[oredged] (0,0) -- node[txt]{$m_{k-1}$} (+1,1.2) node[fat]{} node[above]{$\XX_{k-1}$};
\draw[oredged] (0,0)node[big]{} -- node[txt]{$m_k$}     (+2,1.2) node[fat]{} node[above]{$\XX_k$};
\draw[oredgem] (0,-2)node[big]{} --node[right=7pt,txt,pos=0.3]{$n_1+...+n_l$} (0,-1);
\draw[oredgem] (0,-1)node[fat]{} node[right]{$\ZZ$} --node[right=7pt,txt]{$n_1+...+n_l$} (0,0);
\node at (0,1.2){$\cdots$};
\begin{scope}[yshift=-2cm]
\draw[oredgeu] (-2,-1.2) node[fat]{} node[below]{$\YY_1$} -- node[txtp]{$n_1$} (0,0);
\draw[oredgeu] (-1,-1.2) node[fat]{} node[below]{$\YY_2$} -- node[txtp]{$n_2$}     (0,0) ;
\draw[oredgeu] (+1,-1.2) node[fat]{} node[below]{$\YY_{l-1}$} -- node[txtp]{$n_{l-1}$} (0,0);
\draw[oredgeu] (+2,-1.2) node[fat]{} node[below]{$\YY_l$} -- node[txtp]{$n_l$}  (0,0)node[big]{} ;
\node at (0,-1.2) {$\cdots$};
\end{scope}
\end{scope}
\end{tikzpicture}
\end{center}
\caption{Decomposition of $\Gamma_{n_1,...,n_l}^{m_1,...,m_k}$ into $\Gamma_{n_1,...,n_l}^{n_1+...+n_l}$ and
  $\Gamma_{n_1+...+n_l}^{m_1,...,m_k}$.} 
\label{fig:mergesplit}
\end{figure}

These examples together with Theorem \ref{thm_glue} show the following theorem:

\begin{theorem}\label{thm_acyclicgraphs}
In the construction of the matrix factorization associated to an acyclic MOY-graph the stabilization functor commutes 
with tensor products. 
\end{theorem}

In the next section we will see how Theorem \ref{thm_acyclicgraphs} connects KR-homology to Soergel bimodules and how
we can get lots of relations on KR-homology from this, but first we look at a simpler example of how to apply Theorem
\ref{thm_acyclicgraphs}.

\begin{definition}
For a finite-dimensional, graded $\CC$-algebra $A\ld$ the \textit{Poincar\'{e} polynomial} of $A\ld$, denoted
$\PoincPoly(A\ld)$, is defined as $$\PoincPoly(A\ld)\ :=\ \sum\limits_{i\in\ZZ_{\geq 0}} \dim_{\CC}(A_i) q^i\ \in
\ZZ[q].$$ If $X$ is a complex manifold, we put $$\PoincPoly(X) := \PoincPoly(\H\ua(X;\CC)).$$
\end{definition}

\begin{definition}
Let $(\calA,\langle-\rangle)$ be an additive $\ZZ$-graded category. Then, given some object
$X\in\calA$ and a Laurent polynomial $p = \sum\limits_{i\in\ZZ} p_i q^{i}\in\ZZ[q^{\pm 1}]$ we put $$X\langle
p\rangle\ :=\ \sum\limits_{i\in\ZZ} X\langle i\rangle^{\oplus p_i}.$$ 
\end{definition}

\begin{fact}\label{fact_rankwelldefined}
Let $R\ld$ be a positively graded ring and $M\ld$ a finitely generated free $R\ld$-module. Then there exists a unique
sequence of natural numbers $\{n_i\}_{i\in\ZZ}$, almost all of which are zero, such that
$M\ld\cong\bigoplus\limits_{i\in\ZZ} R\ld\langle i\rangle^{n_i}$. 
\end{fact}
\begin{proof}
By definition of freeness of a module, we only have to show the uniqueness of the $n_i$. We have $$M\ld / R_+ M\ld\
\cong\ \bigoplus\limits_{i\in\ZZ} R_0\langle i\rangle^{n_i},$$ and so $n_i$ is uniquely determined as the
(ungraded) rank of $\left(M\ld/R_+ M\ld\right)_{-i}$ over the commutative base ring $R_0$. 
\end{proof}

\begin{definition}\label{def_rank}
In the situation of Definition \ref{fact_rankwelldefined}, define the \textit{rank} of $M\ld$ over $R\ld$, denoted
$\rank_{R\ld} M\ld$, as $$\rank_{R\ld} M\ld\ :=\ \sum\limits_{i\in\ZZ} n_i q^i\ \in\ZZ[q^{\pm 1}].$$
\end{definition}

\begin{fact}\label{fact_gotosmallerring}
Let $R\ld\subset S\ld$ be positively graded rings, and assume that $S\ld$ is free of finite rank as a module over
$R\ld$. Then there is an isomorphism of $R\ld$-$R\ld$-bimodules
$$S\ld\ \cong\ R\ld\langle\rank_{R\ld} S\ld\rangle.$$
\end{fact}

\begin{figure}[h]\begin{center}
\begin{tikzpicture}
\node at (0,-1)[below] {$\cdots$};
\draw[oredged] (0,0) -- node[txt]{${m_1+...+m_k}$}   (0,1.2) node[fat]{} node[above]{$\XX$};
\draw[oredged] (0,-3.6) node[fat]{} node[below]{$\ZZ$} -- node[txt]{${m_1+...+m_k}$} (0,-2.4);
\draw[oredgeu] (0,-2.4) -- node[txtp,pos=0.7]{$m_1$} (-2.4,-1.2);
\draw[oredgeu] (0,-2.4) -- node[txtp,pos=0.7]{$m_2$}  (-1.2,-1.2);
\draw[oredgeu] (0,-2.4) -- node[txtp,pos=0.7]{$m_{k-1}$} (+1.2,-1.2);
\draw[oredgeu] (0,-2.4) -- node[txtp,pos=0.7]{$m_k$} (+2.4,-1.2);
\draw[oredgeu] (-2.4,-1.2) node[fat]{} node[left,scale=0.9]{$\YY_1$} -- node[txtp,pos=0.3]{$m_1$} (0,0);
\draw[oredgeu] (-1.2,-1.2) node[fat]{} node[left,scale=0.9]{$\YY_2$} -- node[txtp,pos=0.3]{$m_2$}     (0,0) ;
\draw[oredgeu] (+1.2,-1.2) node[fat]{} node[right,scale=0.9]{$\YY_{k-1}$} -- node[txtp,pos=0.3]{$m_{k-1}$} (0,0);
\draw[oredgeu] (+2.4,-1.2) node[fat]{} node[right,scale=0.9]{$\YY_k$} -- node[txtp,pos=0.3]{$m_k$}  (0,0)node[big]{} ;

\draw[decorate,decoration={brace,amplitude=5pt}] (+3,1.3) -- node[right=10pt]{$\Gamma_0$}(+3,-1.1);
\draw[decorate,decoration={brace,amplitude=5pt}] (+3,-1.3) -- node[right=10pt]{$\Gamma_1$}(+3,-3.8);
\end{tikzpicture}
\end{center}
\caption{$\Gamma$}
\label{fig:fancyrelation}
\end{figure}

\begin{ex}\label{ex_splitmerge}
We want to prove the fancy looking relation
\begin{align}
\label{eq:splitmergerank}\KR(\Gamma)\left\langle -\sum\limits_{1\leq i<j\leq k} m_i m_j\right\rangle &
\makebox[8mm][c]{$\cong$} \KR(\Gamma_m^m)\left\langle\rank_{\Sym(m)} \Sym(m_1|...|m_k)\right\rangle\\
\label{eq:splitmergepoly} &  \makebox[8mm][c]{$=$} \KR(\Gamma_m^m)\left\langle
  \PoincPoly(\Fl(m_1|...|m_k))\right\rangle. 
\end{align}
where $\Gamma$ is as in Figure \ref{fig:fancyrelation}, $m := m_1+...+m_k$ and where $\Fl(m_1|...|m_k)$ is the complex
manifold of flags $\{0\} = V_0\subset V_1\subset ...\subset V_k=\CC^m$ in $\CC^m$ with $\dim_\CC(V_{i+1}) -
\dim_\CC(V_i) = m_i$ for all $i=0,1,...,k-1$. By definition, we have 
\begin{align*}
\KR(\Gamma)\langle -\sum\limits_{1\leq i<j\leq k} m_i m_j\rangle &
\makebox[8mm][c]{=} \KR(\Gamma_0)\otimes_{\Sym(\YY)}\KR(\Gamma_1) 
\\ & \makebox[8mm][c]{=}\Sym(\XX|\YY)/(X_i-Y_i)\stab{\Sigma\XX^{n+1}-\Sigma\YY^{n+1}}
\\ & \makebox[8mm][c]{}\ \ \ \makebox[5mm][c]{$\stackrel[\mathclap{\Sym(\YY)}]{}{\otimes}$}  
\Sym(\YY|\ZZ)/(Y_i-Z_i)\stab{\Sigma\YY^{n+1} - \Sigma\ZZ^{n+1}}.
\end{align*}
Since $\Sym(\XX)=\CC[X_1,...,X_m]$ we have $\Sym(\XX|\YY)/(X_i-Y_i)\cong\Sym(\YY)$
as $\Sym(\XX)$-$\Sym(\YY)$-bimodules, where $X_i$ acts by multiplication with $Y_i$. Similarly,
$\Sym(\YY|\ZZ)/(Y_i-Z_i)\cong\Sym(\YY)$ as $\Sym(\YY)$-$\Sym(\ZZ)$-bimodules, with $Z_i$ acting on $\Sym(\YY)$ by
multiplication with $Y_i$. Hence 
$$\Sym(\XX|\YY)/(X_i-Y_i)\makebox[5mm][c]{$\stackrel[\mathclap{\Sym(\YY)}]{}{\otimes}$}
\Sym(\YY|\ZZ)/(Y_i-Z_i)\cong\Sym(\YY)\cong\Sym(\XX|\ZZ)/(X_i-Z_i)\left\langle \rank_{\Sym(\YY^\cup)}
  \Sym(\YY)\right\rangle$$ as 
$\Sym(\XX|\ZZ)$-modules, where we applied Fact \ref{fact_gotosmallerring} in the last step. Applying Theorem
\ref{thm_acyclicgraphs}, we get
\begin{align*}
\KR(\Gamma)\langle -\sum\limits_{1\leq i<j\leq k} m_i m_j\rangle & \cong
\left(\Sym(\YY|\ZZ)/(X_i-Z_i)\langle\rank_{\Sym(\YY^\cup)}\Sym(\YY)\rangle\right)
\stab{\Sigma\XX^{n+1}-\Sigma\ZZ^{n+1}}\\
& = \KR(\Gamma_m^m)\langle\rank_{\Sym(\YY^\cup)}\Sym(\YY)\rangle, 
\end{align*}
yielding \eqref{eq:splitmergerank}. For \eqref{eq:splitmergepoly}, note that since $\Sym(\YY)$ is free as a module
over $\Sym(\YY^\cup)$, we have $\rank_{\Sym(\YY^\cup)}\Sym(\YY) = \PoincPoly(\Sym(\YY)/\langle\Sym(\YY^\cup)_+\rangle)$,
so the claim follows from the algebra isomorphism (see
\cite{Fulton}) $$\H\ua(\Fl(m_1|...|m_k);\CC)\cong\Sym(m_1|...|m_n)/\langle\Sym(m)_+\rangle.$$ 
To get a feeling of what \eqref{eq:splitmergerank} looks like explicitly, take $k=2$ and $m_1=m_2=1$. We have $\Fl(1|1)
= \PP_\CC^1$, hence $\calP(\Fl(1|1)) = 1+q^2$, and so \eqref{eq:splitmergerank} yields 
$$\KR\left(\begin{tikzpicture}[baseline=1.1cm,scale=0.6]
\draw[oredgefe] (0,0) node[fat]{}  -- node[txtp,pos=0.4,scale=0.9]{$2$} (0,1);
\draw[oredgefe] (0,1)  -- node[txtp,pos=0.4,scale=0.9]{$1$} (-1,2);
\draw[oredgefe] (0,1)  node[fat]{} -- node[txtp,pos=0.4,scale=0.9]{$1$} (+1,2);
\draw[oredgefa] (-1,2) node[fat]{} -- node[txt,pos=0.6,scale=0.9]{$1$} (0,3);
\draw[oredgefa] (+1,2) node[fat]{} -- node[txt,pos=0.6,scale=0.9]{$1$} (0,3);
\draw[oredgefe] (0,3)  node[fat]{} -- node[txt,pos=0.4,scale=0.9]{$2$} (0,4) node[fat]{};
\end{tikzpicture}\right)\ \cong\ \KR\left(\ \begin{tikzpicture}[baseline=0.55cm,scale=0.6]
\draw[oredgefe] (0,0) node[fat]{}  -- node[txtp,pos=0.4,scale=0.9]{$2$} (0,2) node[fat]{};
\end{tikzpicture}\ \right)\langle -1\rangle\oplus \KR\left(\ \begin{tikzpicture}[baseline=0.55cm,scale=0.6]
\draw[oredgefe] (0,0) node[fat]{}  -- node[txtp,pos=0.4,scale=0.9]{$2$} (0,2) node[fat]{};
\end{tikzpicture}\ \right)\langle 1\rangle$$
\end{ex}

\subsection{Connecting KR-homology to Soergel bimodules}\label{subsec_connectingkrtosoergel}

In this section we describe how Theorem \ref{thm_acyclicgraphs} connects Khovanov-Rozansky homology to Soergel
bimodules and look at a few examples in the construction of KR-homology where working on the level of Soergel bimodules
makes life a bit easier.

Suppose $\Gamma$ is an acyclic MOY-graph, and we aim to calculate its value $\KR(\Gamma)$ under
KR-homology. By definition (see Section \ref{sec_KRconstruction}), we have to go through the following
steps: \begin{enumerate}\item Decompose $\Gamma$ into basic MOY-graphs $\Gamma_1, ...,
  \Gamma_r$, \item\label{item_step1} take $\KR(\Gamma_1), ..., \KR(\Gamma_r)$ and finally \item\label{item_step2}
  calculate $\KR(\Gamma) := \bigotimes\limits_{i=1}^{r}\KR(\Gamma_i)$, 
  where the tensor product is over the common endpoints of the $\Gamma_i$.\end{enumerate} For step \eqref{item_step1}, we
know by Theorem \ref{thm_KRviaStab} that each $\KR(\Gamma_i)$ is given as
$$\left(\Sym(\XX|\YY)/(X_i-Y_i)\langle r\rangle\right)\stab{\Sigma\XX^{n+1}-\Sigma\YY^{n+1}}$$ for sets of variables
$\XX_1,...,\XX_k$, $\YY_1,...,\YY_l$ and some internal degree shift $r\in\ZZ$, and for step \eqref{item_step2} we know
from Theorem \ref{thm_acyclicgraphs} that we can exchange 
tensor products and stabilization when calculating $\bigotimes\limits_{i=1}^{r}\KR(\Gamma_i)$. We are therefore
naturally led to consider modules which can be ``built'' from the modules $\Sym(\XX|\YY)/(X_i-Y_i)$ by
tensoring, and these are examples of \textit{singular Soergel bimodules}. For simplicity, we will focus on the case
where all MOY graphs have labels $1$ and $2$ only; the modules occurring in this case where first studied by Soergel in
\cite{SoergelBimoduln}, and we recall some of his results now. The general case of singular Soergel bimodules was
recently studied by \cite{Williamson}.

\begin{definition}
Fix $m\in\NN$ and let $\XX_i := \{x_i\}$ be some variable for $i=1,...,m$. Denote $\CC[\XX] = \Sym(\XX)$ the polynomial
ring over $\{x_1,...,x_m\}$. The symmetric group $\frS_m$ acts on $\CC[\XX]$ by permutation of variables, and for
$I\subset\frS_m$ we denote $\CC[\XX]^I$ the subring of $\CC[\XX]$ consisting of those polynomials which are invariant
under the actions of all $w\in I$. For $w = (i,\ i+1)$ we abbreviate $\CC[\XX]^{\{e,w\}}$ by $\CC[\XX]^i$.  
\end{definition}

\begin{definition}\label{def_soergelbimodules}
The \textit{category of Soergel bimodules (for $\frS_m$)}, denoted $\calB^m$, is defined as the smallest full,
additive and idempotent-complete subcategory of the category $\CC[\XX]\bimod\CC[\XX]$ of graded $\CC[\XX]$-bimodules
containing all modules of the form $$\CC[\XX]\otimes_{\CC[\XX]^{i_1}} \CC[\XX]\otimes_{\CC[\XX]^{i_2}}
...\otimes_{\CC[\XX]^{i_r}} \CC[\XX]$$ for  $i_1,...,i_r\in\{1,2,...,m-1\}$.
\end{definition}
\begin{rem}
Definition \ref{def_soergelbimodules} seems to be weaker than the one given in \cite[Definition 5.11, Lemma
5.13]{SoergelBimoduln}. However, by \cite[Theorem 6.14(4)]{SoergelBimoduln} both definitions agree.
\end{rem}

One of the main results of \cite{SoergelBimoduln} is that the combinatorics of the category $\calB^m$ is captured by
the Hecke algebra $\Heck_m(q)$ of $\frS_m$, which we now recall.

\begin{definition}\label{def_heckealgebra}
Fix $m\in\NN$. The (generic) \textit{Hecke-algebra} $\Heck_m(q)$ (of $\frS_m$) is the associative $\ZZ[q^{\pm
  1}]$-algebra generated by elements $T_1,...,T_{m-1}$ and unit $T_e$ subject to the relations
\begin{alignat*}{2}
T_i T_j & =  T_j T_i &\qquad \text{ for all } & i,j=1,2, ...,m-1 \text{ s.t. } |i-j|>1\\
T_i T_{i+1} T_i & = T_{i+1} T_i T_{i+1} &\qquad \text{ for all } & i=1,2,...,m-2\\
T_i^2 & = (q^2-1) T_i + q^2 T_e &&
\end{alignat*}
We denote $\{\KL{w}\}_{w\in\frS_m}$ the Kazhdan-Lusztig basis of $\Heck_m(q)$ (see \cite[Section 7.9]{HumphreysCoxeter},
where these elements are denoted $C\p_w$). In particular, we have $\KL{i} := \KL{(i, i+1)} = q\ui(T_e + T_i)$. 
\end{definition}

\begin{definition}\label{def_splitgrothendieck}
Let $\calA$ be an essentially small additive category, and let $\Iso(\calA)$ be the set of isomorphism classes of
objects in $\calA$. Further, for $X\in\calA$ we denote $[X]$ the isomorphism class of $X$. The \textit{split
  Grothendieck group} of $\calA$, denoted $\grsplit(\calA)$, is defined as the free abelian group $\ZZ^{(\Iso(\calA))}$
subject to the relations $[X] + [Y] - [X\oplus Y] = 0$ for all $X,Y\in\calA$. 

If $\calA$ carries a $\ZZ$-grading (i.e. a strict, additive action of $\ZZ$), then $\grsplit(\calA)$ carries a natural
structure of a $\ZZ[q^{\pm 1}]$-module given by $q^n.[X] := [n.X]$, where $n.X$ is the action of $n$ on $X$. If $\calA$
carries an additive monoidal structure, then $\grsplit(\calA)$ admits a natural ring structure given by $[X]\cdot [Y] :=
[X\otimes Y]$.  

In particular, endowing $\calB^m$ with the monoidal structure given by the tensor product of graded
$\CC[X]\bimod\CC[\XX]$ bimodules and the $\ZZ$-grading given by $n.X := X\langle -n\rangle$, the split Grothendieck
group $\grsplit(\calB^m)$ becomes a $\ZZ[q^{\pm 1}]$-algebra.
\end{definition}

\begin{rem}
Note the reversed $\ZZ$-grading on $\calB^m$.
\end{rem}

\begin{fact}\label{fact_ggbyindecs}
Let $\calA$ be an essentially small Krull-Remak-Schmidt category, i.e. an additive category such that every object has a
finite decomposition into indecomposable objects which is unique up to permutation, and denote
$\Indec(\calA)\subset\Iso(\calA)$ the set of isomorphism classes of indecomposable objects in $\calA$. Then the natural map 
\begin{align*}
\ZZ^{(\Indec(\calA))}\ & \longrightarrow\ \grsplit(\calA)
\end{align*}
is an isomorphism. In particular, for all $X,Y\in\calA$ we have 
\begin{align}\label{eq:reflectiso}
[X] = [Y]\text{ in }\grsplit(\calA)\qquad\Longleftrightarrow\qquad X\cong Y.
\end{align}
\end{fact}

\begin{rem}
The equivalence \eqref{eq:reflectiso} allows to check relations in $\calA$ inside $\grsplit(\calA)$. Hence, to understand 
the combinatorics of $\calA$ it is therefore sufficient to calculate $\grsplit(\calA)$. 
\end{rem}

\begin{theorem}\label{thm_soergelbimodulesviahecke}
In the notation of Definitions \ref{def_soergelbimodules} and \ref{def_heckealgebra}, the following hold:
\begin{enumerate}
\item\label{item_thmsbmrels} The assignment $\KL{i}\mapsto \left[\CC[\XX]\otimes_{\CC[\XX]^i} \CC[\XX]\langle
    1\rangle\right]$ extends to an isomorphism of $\ZZ[q^{\pm 1}]$-algebras $$\calE:\
  \Heck_m(q)\quad\xrightarrow{\quad\cong\quad}\quad\grsplit(\calB^m).$$  
\item\label{item_thmsbmindecs} The set of isomorphism classes of indecomposable objects in $\calB^m$ (up to shift) is
  canonically parametrized by 
  $\frS_m$: $$\Indec(\calB^m)\ =\ \{[^m\SBM{w}]\langle r\rangle\}_{w\in\frS_m,\ r\in\ZZ}\ \cong\ \frS_m\times\ZZ$$
\item\label{item_thmsbmkl} For each $w\in\frS_m$ we have $\calE(\KL{w}) = \left[^m\SBM{w}\right]$. 
\item\label{item_thmsbmlongest} For a subset $I = \{(1,2), (2,3), ..., (m-1,m)\}$ of the simple transpositions in
  $\frS_m$ and $w^I$ the longest element in $W_I := \langle I\rangle$, we have $$^m\SBM{w^I}\ \cong\
  \CC[\XX]\otimes_{\CC[\XX]^I}\CC[\XX]\langle l(w^I)\rangle.$$ In particular, for $i=1,2,...,m-1$ we have 
\begin{align*}
^m\SBM{(i, i+1)} & \cong \CC[\XX]\otimes_{\CC[\XX]^i}\CC[\XX]\langle 1\rangle
\\ & \cong \CC[x_1,...,x_m,y_1,...,y_m]/(x_i + x_{i+1} - y_i - y_{i+1}, x_i x_{i+1} - y_i y_{i+1}).
\end{align*}
\end{enumerate}
\end{theorem}
\begin{rem}
Part \eqref{item_thmsbmindecs} of Theorem \ref{thm_soergelbimodulesviahecke} is intentionally kept a bit vague as we
will only need the explicit description of $^m\SBM{w^I}$ given in part \eqref{item_thmsbmlongest}. In
\cite{SoergelBimoduln}, explicit conditions characterizing all the indecomposable bimodules $^m\SBM{w}$ are given.  
\end{rem}
\begin{proof}
For \eqref{item_thmsbmrels}, \eqref{item_thmsbmindecs} and \eqref{item_thmsbmkl}, see Soergel's original article
\cite{SoergelBimoduln} or the recent work \cite{Williamson} on generalized ``singular'' Soergel bimodules by
Williamson. For \eqref{item_thmsbmlongest}, apply \cite[Theorem 7.4.3]{Williamson}: in the notation of loc.cit., one has
${^I B^I}= {^I\nabla^I}$ (apply \cite[Theorem 7.4.2]{Williamson} for $p = W_I e W_I \in W_I\setminus W/W_I$) and ${^I\nabla
^I}={^I R^I}\langle l(w^I)\rangle$ (see \cite[Section 6.1]{Williamson}), while ${^I R^I}$ is, in our notation,
given as $\CC[\XX]^I$ (see \cite[Definition 4.2.1]{Williamson}).
\end{proof}

\begin{figure}[h]\begin{center}
\begin{tikzpicture}
\begin{scope}[xshift=5cm]
\draw[oredged] (0,0) -- node[txt]{$1$} (-2,1.2) node[above=1mm,scale=0.75]{$i$} node[fat]{};
\draw[oredged] (0,0) -- node[txt]{$1$} (-1,1.2) node[above=1mm,scale=0.75]{$i+1$} node[fat]{};
\draw[oredged] (0,0) -- node[txt]{$1$} (+1,1.2) node[fat]{};
\node[above=1mm,scale=0.75] at (+0.7,+1.2){$i+k-2$};
\draw[oredged] (0,0) -- node[txt]{$1$} (+2,1.2) node[above=1mm,scale=0.75]{$i+k-1$} node[fat]{};
\draw[oredgeu] (-2,-1.2) node[fat]{} node[below=1mm,scale=0.75]{$i$}     -- node[txtp]{$1$} (0,0);
\draw[oredgeu] (-1,-1.2) node[fat]{} node[below=1mm,scale=0.75]{$i+1$}     -- node[txtp]{$1$} (0,0) ;
\draw[oredgeu] (+1,-1.2) node[fat]{} -- node[txtp]{$1$} (0,0);
\node[below=1mm,scale=0.75] at (+0.7,-1.2){$i+k-2$};
\draw[oredgeu] (+2,-1.2) node[fat]{} node[below=1mm,scale=0.75]{$i+k-1$}     -- node[txtp]{$1$} (0,0)    node[big]{} ;
\node at (0,1.2){$\cdots$};
\node at (0,-1.2){$\cdots$};
\end{scope} 
\draw[oredgeu] (0,-1.2) node[fat]{} node[below=1mm,scale=0.75]{$1$} -- node[txtq]{$1$} (0,1.2)
node[above=1mm,scale=0.75]{$1$} node[fat]{};  
\draw[oredgeu] (1,-1.2) node[fat]{} node[below=1mm,scale=0.75]{$2$} -- node[txtq]{$1$} (1,1.2)
node[above=1mm,scale=0.75]{$2$} node[fat]{};  
\draw[oredgeu] (2,-1.2) node[fat]{} node[below=1mm,scale=0.75]{$i-1$} -- node[txtq]{$1$} (2,1.2)
node[above=1mm,scale=0.75]{$i-1$} node[fat]{};   
\draw[oredgeu] (8,-1.2) node[fat]{} node[below=1mm,scale=0.75]{$i+k$} -- node[txtq]{$1$} (8,1.2)
node[above=1mm,scale=0.75]{$i+k$} node[fat]{};  
\draw[oredgeu] (9,-1.2) node[fat]{} node[below=1mm,scale=0.75]{$m-1$} -- node[txtq]{$1$} (9,1.2)
node[above=1mm,scale=0.75]{$m-1$} node[fat]{};  
\draw[oredgeu] (10,-1.2) node[fat]{} node[below=1.3mm,scale=0.75]{$m$} -- node[txtq]{$1$} (10,1.2)
node[above=1.3mm,scale=0.75]{$m$} node[fat]{};   

\node at (1.5,0){$\cdots$};
\node at (8.5,0){$\cdots$};

\draw[decorate,decoration={brace,amplitude=5pt}] (11,+1.8) -- node[right=10pt]{$\Gamma$}(11,-1.8);
\end{tikzpicture}
\end{center}
\caption{Another basic MOY-graph}
\label{fig:mfasstabsbm}
\end{figure}
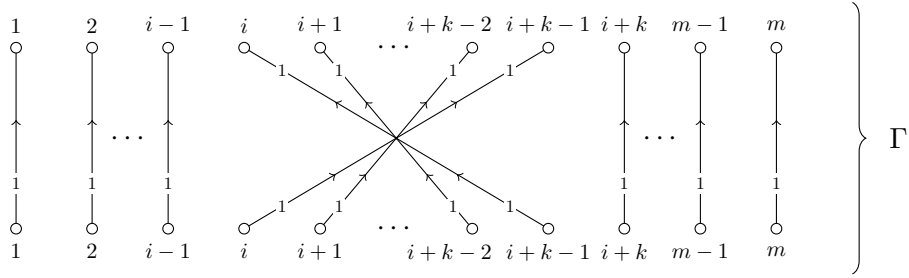

With the notation of Theorem \ref{thm_soergelbimodulesviahecke}, Theorem \ref{thm_KRviaStab} becomes:
\begin{theorem}\label{thm_mfasstabsbm}
Fix $n\geq 2$ and let $\Gamma$ be the MOY-graph depicted in Figure \ref{fig:mfasstabsbm}. Define
$$w := \begin{pmatrix} 1 & 2 & \cdots & i-1 & i & i+1 & \cdots & i+k-1 & i+k & \cdots & m\\
1 & 2 & \cdots & i-1 & i+k-1 & i+k-2 & \cdots & i & i+k & \cdots & m\end{pmatrix}.$$
Then, considering $^m\SBM{w}$ as a module over $\CC[\XX,\YY]$ with $|\XX|=|\YY|=m$, there is a canonical homotopy
equivalence 
$$\KR(\Gamma)\ \simeq\ {^m\SBM{w}}\stab{\Sigma\XX^{n+1}-\Sigma\YY^{n+1}}.$$
\end{theorem}
\begin{proof}
Considering $\frS_k$ a subgroup of $\frS_m$ via 
$$\sigma\mapsto\begin{pmatrix} 1 & 2 & \cdots & i-1 & i & i+1 & \cdots & i+k-1 & i+k & \cdots & m\\
1 & 2 & \cdots & i-1 & i-1+\sigma(1) & i-1+\sigma(2) & \cdots & i-1+\sigma(k) & i+k & \cdots & m\end{pmatrix},$$
Theorem \ref{thm_KRviaStab} implies that
$$\KR(\Gamma)\ \simeq\ \left(\CC[\XX]\otimes_{\frS_k}\CC[\YY]\left\langle
  \frac{k(k-1)}{2}\right\rangle\right)\stab{\Sigma\XX^{n+1}-\Sigma\YY^{n+1}}.$$ On the other hand, Theorem
\ref{thm_soergelbimodulesviahecke} shows that $$\CC[\XX]\otimes_{\frS_k}\CC[\YY]\left\langle
  \frac{k(k-1)}{2}\right\rangle\ \cong\ {^m\SBM{w}},$$ where we use the fact that the length of the longest element $w$ in
  $\frS_k$ is $\frac{k(k-1)}{2}$. 
\end{proof}
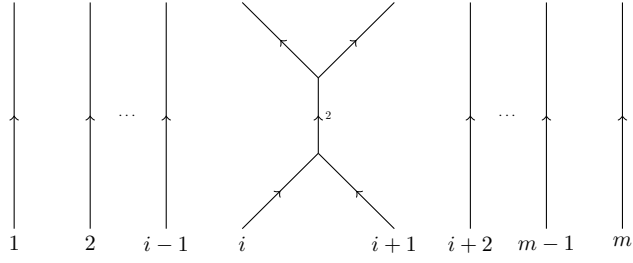
\begin{figure}[h]\begin{center}
\begin{tikzpicture}
\draw[oredgem] (0cm,0cm) node[scale=0.75,below]{$1$} -- (0cm,3cm);
\draw[oredgem] (1cm,0cm) node[scale=0.75,below]{$2$} -- (1cm,3cm);
\draw[oredgem] (2cm,0cm) node[scale=0.75,below]{$i-1$} -- (2cm,3cm);
\begin{scope}[xshift=3cm]
\draw[oredgem] (0cm,0cm) node[scale=0.75,below]{$i$} -- (1cm,1cm);
\draw[oredgem] (2cm,0cm) node[scale=0.75,below]{$i+1$} -- (1cm,1cm);
\draw[oredgem] (1cm,1cm) -- node[scale=0.4, right=3pt]{$2$} (1cm,2cm); % wide edge
\draw[oredgem] (1cm,2cm) -- (0cm,3cm);
\draw[oredgem] (1cm,2cm) -- (2cm,3cm);
\end{scope}
\draw[oredgem] (6cm,0cm) node[scale=0.75,below]{$i+2$} -- (6cm,3cm);
\draw[oredgem] (7cm,0cm) node[scale=0.75,below]{$m-1$} -- (7cm,3cm);
\draw[oredgem] (8cm,0cm) node[scale=0.75,below]{$m$} -- (8cm,3cm);
\node[scale=0.5] at (1.5cm,1.5cm) {$\cdots$};
\node[scale=0.5] at (6.5cm,1.5cm) {$\cdots$};
\end{tikzpicture}
\end{center}
\caption{\small{Basic MOY-braid $\sigma_i$}}
\label{fig:basicmoybraid}
\end{figure}
\begin{cor}\label{cor:krofamoybraid}
Let $\gamma = s_{i_1} s_{i_2} ... s_{i_k}$ be a MOY-braid on $m$ strands, with $i_1,...,i_k\in\{1,2,...,m-1\}$ and
$s_i$ as depicted in Figure \ref{fig:basicmoybraid}. There is canonical homotopy equivalence
$$\KR(\gamma)\ \simeq\ \left({^m\SBM{i_1}}\otimes_{\CC[\XX]}{^m\SBM{i_2}}\otimes_{\CC[\XX]}
  ...\otimes_{\CC[\XX]}{^m\SBM{i_k}}\right)\stab{\Sigma\XX^{n+1}-\Sigma\YY^{n+1}}.$$
\end{cor}

Theorem \ref{thm_soergelbimodulesviahecke}.\eqref{item_thmsbmrels} completely explains the relations that hold between
Soergel bimodules in terms of the Hecke algebra. Applying the stabilization functor, we see that the same relations hold
on the level of matrix factorizations. Let us pause to see an example for that.

\begin{ex}
We want to prove the relation
\begin{align}\label{eq:musthaveiso}
\KR(\Gamma_0)\oplus \KR(\Gamma_1)\ \simeq\ \KR(\Gamma_2)\oplus \KR(\Gamma_3),
\end{align}
 where $\Gamma_0,...,\Gamma_3$ are as in Figure \ref{fig:examplerelation2}. 
\begin{figure}[h]\begin{center}
\begin{tikzpicture}[scale=0.5]
\draw[oredgem] (0,0) -- (1,1);
\draw[oredgem] (2,0) -- (1,1);
\draw[oredgem] (1,1) -- node[scale=0.6, right=3pt]{$2$} (1,2); % wide edge
\draw[oredgem] (1,2) -- (0,3);
\draw[oredgee] (1,2) -- (2,3);
\draw[oredgem] (4,0) -- (4,3);
\begin{scope}[xshift=2cm, yshift=3cm]
\draw[oredgem] (-2,0) -- (-2,3);
\draw (0,0) -- (1,1);
\draw[oredgem] (2,0) -- (1,1);
\draw[oredgem] (1,1) -- node[scale=0.6, right=3pt]{$2$} (1,2); % wide edge
\draw[oredgee] (1,2) -- (0,3);
\draw[oredgem] (1,2) -- (2,3);
\end{scope}
\begin{scope}[yshift=6cm]
\draw[oredgem] (0,0) -- (1,1);
\draw (2,0) -- (1,1);
\draw[oredgem] (1,1) -- node[scale=0.6, right=3pt]{$2$} (1,2); % wide edge
\draw[oredgem] (1,2) -- (0,3);
\draw[oredgem] (1,2) -- (2,3);
\draw[oredgem] (4,0) -- (4,3); 
\end{scope}
\node at (2,-1){$\Gamma_0$};
\node at (5.5,4.5){$\oplus$};
\begin{scope}[xshift=7cm]
\draw[oredgem] (1,0) -- (2,1);
\draw[oredgem] (3,0) -- (2,1);
\draw[oredgem] (2,1) -- node[scale=0.6, right=3pt]{$2$} (2,8);
\draw[oredgem] (2,8) -- (1,9);
\draw[oredgem] (2,8) -- (3,9);
\draw[oredgem] (0,0) -- (0,9);
\end{scope}
\node at (8.5,-1){$\Gamma_1$};
\node at (12,4.5) {$=$};
\begin{scope}[xshift=14cm]
\draw[oredgem] (2,0) -- (3,1);
\draw[oredgem] (4,0) -- (3,1);
\draw[oredgem] (3,1) -- node[scale=0.6, right=3pt]{$2$} (3,2); % wide edge
\draw[oredgee] (3,2) -- (2,3);
\draw[oredgem] (3,2) -- (4,3);
\draw[oredgem] (0,0) -- (0,3);
\begin{scope}[yshift=3cm]
\draw[oredgem] (0,0) -- (1,1);
\draw (2,0) -- (1,1);
\draw[oredgem] (1,1) -- node[scale=0.6, right=3pt]{$2$} (1,2); % wide edge
\draw[oredgem] (1,2) -- (0,3);
\draw[oredgee] (1,2) -- (2,3);
\draw[oredgem] (4,0) -- (4,3);
\end{scope}
\begin{scope}[yshift=6cm]
\draw (2,0) -- (3,1);
\draw[oredgem] (4,0) -- (3,1);
\draw[oredgem] (3,1) -- node[scale=0.6, right=3pt]{$2$} (3,2); % wide edge
\draw[oredgem] (3,2) -- (2,3);
\draw[oredgem] (3,2) -- (4,3);
\draw[oredgem] (0,0) -- (0,3);
\end{scope}
\node at (2,-1){$\Gamma_2$};
\node at (5.5,4.5){$\oplus$};
\begin{scope}[xshift=7cm]
\draw[oredgem] (0,0) -- (1,1);
\draw[oredgem] (2,0) -- (1,1);
\draw[oredgem] (1,1) -- node[scale=0.6, right=3pt]{$2$} (1,8);
\draw[oredgem] (1,8) -- (0,9);
\draw[oredgem] (1,8) -- (2,9);
\draw[oredgem] (3,0) -- (3,9);
\end{scope}
\node at (8.5,-1){$\Gamma_3$};
\end{scope}
\end{tikzpicture}
\end{center}
\caption{Basic MOY-relation}
\label{fig:examplerelation2}
\end{figure}
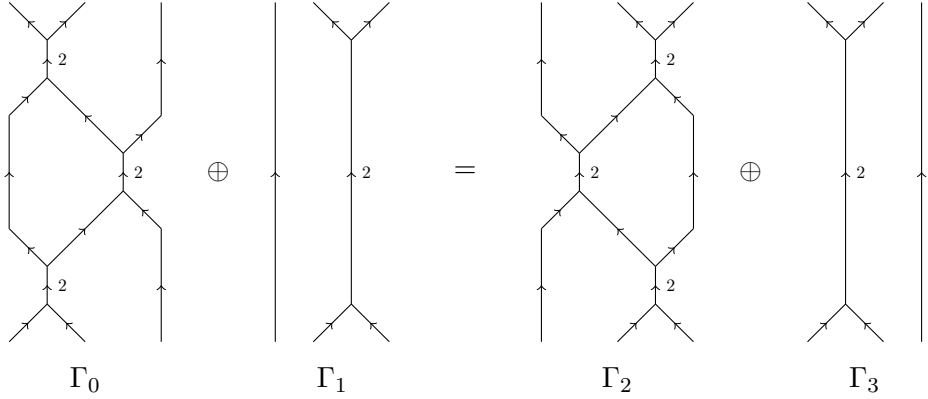
A short calculation in the Hecke algebra shows that for $i=1,2,...,m-1$ we have $$\KL{i}\KL{i+1}\KL{i} + \KL{i}\ =\
\KL{i+1}\KL{i}\KL{i+1} + \KL{i+1}.$$ Applying Theorem \ref{thm_soergelbimodulesviahecke} implies that the
corresponding relation also holds on the level of Soergel bimodules, i.e. that we have
an isomorphism of $\CC[\XX]$ bimodules (abbreviating $\CC[\XX]$ by $S$)
\begin{align}\label{eq:simplerelationinheckealgebra}
S\makebox[5mm][c]{$\stackrel[\mathclap{S^i}]{}{\otimes}$}S
\makebox[5mm][c]{$\stackrel[\mathclap{S^{i+1}}]{}{\otimes}$}S
\makebox[5mm][c]{$\stackrel[\mathclap{S^i}]{}{\otimes}$}S\langle
3\rangle\ \oplus\ S\makebox[5mm][c]{$\stackrel[\mathclap{S^i}]{}{\otimes}$}S\langle 1\rangle\ \ \cong\ \ 
S\makebox[5mm][c]{$\stackrel[\mathclap{S^{i+1}}]{}{\otimes}$}S
\makebox[5mm][c]{$\stackrel[\mathclap{S^{i}}]{}{\otimes}$}S
\makebox[5mm][c]{$\stackrel[\mathclap{S^{i+1}}]{}{\otimes}$}S\langle
3\rangle\ \oplus\ S\makebox[5mm][c]{$\stackrel[\mathclap{S^{i+1}}]{}{\otimes}$}S\langle 1\rangle.
\end{align}
Considering $\CC[\XX]$-bimodules as modules over $\Sym(\XX|\YY)$, where $\YY_i := \{y_i\}$ for variables $y_1,...,y_m$,
we can now apply the stabilization functor 
$$\Sym(\XX|\YY)/(\Sigma\XX^{n+1}-\Sigma\YY^{n+1})\mod\xrightarrow{(-)\stab{\Sigma\XX^{n+1}-\Sigma\YY^{n+1}}}
\HMF(\Sym(\XX|\YY),\Sigma\XX^{n+1}-\Sigma\YY^{n+1})$$ 
to \eqref{eq:simplerelationinheckealgebra}; using Theorem \ref{thm_acyclicgraphs} to exchange stabilization and tensor 
products, we get the desired isomorphism \eqref{eq:musthaveiso}. Note that it is nontrivial to prove
\eqref{eq:musthaveiso} directly from the definitions.
\end{ex}

\begin{figure}[h]\begin{center}
\begin{tikzpicture}[description/.style={fill=white,inner sep=2pt}]
    \matrix (m) [matrix of math nodes, row sep=2.5em,
                 column sep=2.5em, text height=1.5ex, text depth=0.25ex,
                 inner sep=0pt, nodes={inner xsep=0.3333em, inner ysep=0.3333em}]
    {
      \KR
\left(\ \begin{tikzpicture}[baseline=0.65cm,scale=0.5]
\draw[oredgefe] (0,0) node[scale=0.75,below]{$\YY_1$}-- (1,1);
\draw[oredgefe] (2,0) node[scale=0.75,below]{$\YY_2$} -- (1,1);
\draw[oredgem]  (1,1) -- node[pos=0.5,right,scale=0.6]{$2$} (1,2); % wide edge
\draw[oredgefe] (1,2) -- (0,3) node[scale=0.75,above]{$\XX_1$};
\draw[oredgefe] (1,2) -- (2,3) node[scale=0.75,above]{$\XX_2$};
\end{tikzpicture}\ \right) &&
      \KR
\left(\ \begin{tikzpicture}[baseline=0.65cm,scale=0.5]
\draw[oredgefe] (0,0) node[scale=0.75,below]{$\YY_1$} -- node[txtp]{$1$} (0,3) node[scale=0.75,above]{$\XX_1$};
\draw[oredgefe] (2,0) node[scale=0.75,below]{$\YY_2$} -- node[txtp]{$1$} (2,3) node[scale=0.75,above]{$\XX_2$};
\end{tikzpicture}\ \right)\\
    };
    \draw[->] ($(m-1-1.east) + (0,+2mm)$) -- node[above]{$\chi_1$} ($(m-1-3.west) + (0,+2mm)$);
    \draw[->] ($(m-1-3.west) + (0,-1mm)$) -- node[below]{$\chi_0$} ($(m-1-1.east) + (0,-1mm)$);
    \node[scale=0.9] at ($(m-1-1.south) + (4mm,-12mm)$) {$\Gamma_1$};
    \node[scale=0.9] at ($(m-1-3.south) + (4mm,-12mm)$) {$\Gamma_0$};
\end{tikzpicture}
\end{center}
\caption{$\chi$-morphisms}
\label{fig:chimorphisms}
\end{figure}

\begin{ex}\label{ex_chimorphisms}
In this example, we describe the $\chi$-morphisms (see Figure \ref{fig:chimorphisms}) of \cite[Section
6]{KR1} using the stabilization functor and Soergel bimodules. It would be very interesting to do the same
thing for the more general $\chi$-morphisms in \cite[Section 7.6]{WuLinkHomology}.

First we recall the definition of $\chi_0$ and $\chi_1$ given in \cite{KR1}, beginning with $\chi_1$. Abbreviating
$\Sym(\XX|\YY) = \CC[x_1,x_2,y_1,y_2]$ by $S\ld$, by the very definition we have
\begin{align*}
\KR(\Gamma_0)\langle -1\rangle & = \{x_1+x_2-y_1-y_2, u_1\}\otimes_{S\ld} \{x_1 x_2-y_1 y_2, u_2\}\\
& = \left(S\ld\xrightarrow{u_1} S\ld(e_1)\langle -2\rangle\xrightarrow{x_1+x_2-y_1-y_2}
  S\ld\right)\makebox[3mm][c]{$\stackrel[\mathclap{S\ld}]{}{\otimes}$}
\left(S\ld\xrightarrow{u_2} S\ld(e_2)\langle -4\rangle\xrightarrow{x_1 x_2-y_1 y_2} S\ld\right)\\
& = \begin{tikzpicture}[baseline=-1mm]
    \matrix (m) [matrix of math nodes, row sep=2.5em,
                 column sep=5em, text height=1.5ex, text depth=0.25ex,
                 inner sep=0pt, nodes={inner xsep=0.3333em, inner ysep=0.3333em}]
    {
      S\ld(e_1)\oplus S\ld(e_2) \pgfmatrixnextcell\pgfmatrixnextcell S\ld(\emptyset)\oplus S\ld(e_1 e_2)\\
    };
    \draw[->] ($(m-1-1.east) + (0mm,+1mm)$) -- node[above,scale=0.75]{$\begin{pmatrix} x_1+x_2-y_1-y_2 & x_1 x_2 - y_1
        y_2\\ -u_2 & u_1\end{pmatrix}$} ($(m-1-3.west) + (0mm,+1mm)$);
    \draw[->] ($(m-1-3.west) + (0mm,-1mm)$) -- node[below,scale=0.75]{$\begin{pmatrix} u_1 & y_1 y_2 - x_1 x_2\\ u_2 &
        x_1 + x_2 - y_1 - y_2\end{pmatrix}$} ($(m-1-1.east) + (0mm,-1mm)$);
\end{tikzpicture}
\end{align*}
Here $\emptyset$, $e_1$, $e_2$ and $e_1 e_2$ are names for the generators of the several copies of $S\ld$, and 
$u_1$ and $u_2$ are as usual chosen in such a way that the potential is $x_1^{n+1} + x_2^{n+1} - y_1^{n+1} -
y_2^{n+1}$. Similarly,
\begin{align*}
\KR(\Gamma_1)\langle -1\rangle & = \{x_1-y_1,\pi_1\}\otimes_{S\ld} \{x_2-y_2,\pi_2\}\\
& = \left(S\ld\xrightarrow{\pi_1} S\ld(\tilde{e}_1)\langle -2\rangle\xrightarrow{x_1-y_1} S\ld\right)\otimes_{S\ld}
\left(S\ld\xrightarrow{\pi_2} S\ld(\tilde{e}_2)\langle -2\rangle\xrightarrow{x_2-y_2} S\ld\right)\\
& = \begin{tikzpicture}[baseline=-1mm]
    \matrix (m) [matrix of math nodes, row sep=2.5em,
                 column sep=5em, text height=1.5ex, text depth=0.25ex,
                 inner sep=0pt, nodes={inner xsep=0.3333em, inner ysep=0.3333em}]
    {
      S\ld(\tilde{e}_1)\oplus S\ld(\tilde{e}_2) \pgfmatrixnextcell\pgfmatrixnextcell S\ld(\emptyset)\oplus S\ld(\tilde{e}_1 \tilde{e}_2)\\
    };
    \draw[->] ($(m-1-1.east) + (0mm,+1mm)$) -- node[above,scale=0.75]{$\begin{pmatrix} x_1-y_1 & x_2-y_2\\ -\pi_2 &
        \pi_1\end{pmatrix}$} ($(m-1-3.west) + (0mm,+1mm)$); 
    \draw[->] ($(m-1-3.west) + (0mm,-1mm)$) -- node[below,scale=0.75]{$\begin{pmatrix} \pi_1 & y_2-x_2\\ \pi_2 &
        x_1 -y_1\end{pmatrix}$} ($(m-1-1.east) + (0mm,-1mm)$);
\end{tikzpicture}
\end{align*}
where $\pi_1 := \frac{x_1^{n+1}-y_1^{n+1}}{x_1-y_1}$ and $\pi_2 := \frac{x_2^{n+1}-y_2^{n+1}}{x_2-y_2}$. In this
explicit form of $\KR(\Gamma_0)$ and $\KR(\Gamma_1)$, the map $\chi_1$ is given as
\begin{align}\label{eq:morphismfromKR}
\begin{tikzpicture}
\matrix (m) [matrix of math nodes, row sep=5em,
                 column sep=7em, text height=1.5ex, text depth=0.25ex,
                 inner sep=0pt, nodes={inner xsep=0.3333em, inner ysep=0.3333em}]
    {
      S\ld(e_1)\oplus S\ld(e_2) \pgfmatrixnextcell\pgfmatrixnextcell S\ld(\emptyset)\oplus S\ld(e_1 e_2)\\
      S\ld(\tilde{e}_1)\oplus S\ld(\tilde{e}_2) \pgfmatrixnextcell\pgfmatrixnextcell S\ld(\emptyset)\oplus S\ld(\tilde{e}_1
      \tilde{e}_2)\\
    };
    \draw[->] ($(m-2-1.east) + (0mm,+1mm)$) -- node[above,scale=0.75]{$\begin{pmatrix} x_1-y_1 & x_2 - y_2\\ -\pi_2 &
        \pi_1\end{pmatrix}$} ($(m-2-3.west) + (0mm,+1mm)$);
    \draw[->] ($(m-2-3.west) + (0mm,-1mm)$) -- node[below,scale=0.75]{$\begin{pmatrix} \pi_1 & y_2-x_2\\ \pi_2 &
        x_1-y_1\end{pmatrix}$} ($(m-2-1.east) + (0mm,-1mm)$); 
    \draw[->] ($(m-1-1.east) + (0mm,+1mm)$) -- node[above,scale=0.75]{$\begin{pmatrix} x_1+x_2-y_1-y_2 & x_1 x_2 - y_1
        y_2\\ -u_2 & u_1\end{pmatrix}$} ($(m-1-3.west) + (0mm,+1mm)$);
    \draw[->] ($(m-1-3.west) + (0mm,-1mm)$) -- node[below,scale=0.75]{$\begin{pmatrix} u_1 & y_1 y_2 - x_1 x_2\\ u_2 &
        x_1 + x_2 - y_1 - y_2\end{pmatrix}$} ($(m-1-1.east) + (0mm,-1mm)$); 
    \draw[->] (m-1-1) -- node[scale=0.75,left]{$\begin{pmatrix} 1 & y_2 + \lambda(x_2-y_2)\\ 1 & x_1 +
        \lambda(y_1-x_1) \end{pmatrix}$} (m-2-1); 
    \draw[->] (m-1-3) -- node[scale=0.75,right]{$\begin{pmatrix} 1 & 0 \\ a & b\end{pmatrix}$} (m-2-3);
\end{tikzpicture}
\end{align}
for some $\lambda\in\ZZ$ and $$a = -\lambda u_2 + \frac{u_1 + x_1 u_2 - \pi_2}{x_1-y_1}\quad\text{and}\quad b = x_1-y_2 +
\lambda(y_1+y_2 - x_1-x_2).$$ Note that in \cite{KR1} the sign of $a$ differs, but this is just because of a different
convention in the differential of the Koszul complex. We will now use Example \ref{ex_stabmorphism2els} to show that
under the canonical isomorphisms 
\begin{align*}
\KR(\Gamma_0) & \cong\left(S\ld/(x_1 + x_2 - y_1 - y_2, x_1 x_2 - y_1
  y_2)\langle 1\rangle\right)\stab{x_1^{n+1}+x_2^{n+1}-y_1^{n+1}-y_2^{n+1}}\\
\KR(\Gamma_1) & \cong\left(S\ld/(x_1-y_1, x_2-y_2)\right)\stab{x_1^{n+1}+x_2^{n+1}-y_1^{n+1}-y_2^{n+1}}
\end{align*}
the morphism $\chi_1$ corresponds to the stabilization of the canonical quotient map
\begin{align}\label{eq:canprojmap}
S\ld/(x_1 + x_2 - y_1 - y_2, x_1 x_2 - y_1 y_2)\xrightarrow{\ \ \can\ \ } S\ld/(x_1-y_1, x_2-y_2).
\end{align}
in the sense that the following diagram commutes ($w := x_1^{n+1}+x_2^{n+1}-y_1^{n+1}-y_2^{n+1}$):
\begin{align}\label{eq:stabilizationofmorphisms}
\begin{tikzpicture}
  \matrix (m) [matrix of math nodes, row sep=3em,
  column sep=1.2em, text height=1.5ex, text depth=0.25ex,
  inner sep=0pt, nodes={inner xsep=0.3333em, inner ysep=0.3333em}]
  {
    \{(x_1+x_2-y_1-y_2, x_1 x_2 - y_1 y_2),(u_1,u_2)\} \pgfmatrixnextcell \pgfmatrixnextcell
    S\ld/(x_1+x_2-y_1-y_2, x_1 x_2 - y_1 y_2)\stab{w}\\
   \{(x_1 - y_1, x_2 - y_2),(\pi_1,\pi_2)\} \pgfmatrixnextcell \pgfmatrixnextcell
   S\ld/(x_1-y_1,x_2-y_2)\stab{w}\\
  };
  \draw[->] (m-1-1) -- node[above,scale=0.75]{$\cong$} (m-1-3);
  \draw[->] (m-2-1) -- node[below,scale=0.75]{$\cong$} (m-2-3);
  \draw[->] (m-1-1) -- (m-2-1);
  \draw[->] (m-1-3) -- node[right,scale=0.75]{$\can\stab{w}$} (m-2-3);
\end{tikzpicture}
\end{align}
In particular, we see that the homotopy class of $\chi_1$ does not depend on the choice of $\lambda$.

For the proof, we use the method of Example \ref{ex_stabmorphism2els} with $\alpha=1$. To avoid confusion with the
notation, let us denote the variables $x,y,\tilde{x},\tilde{y}$ from there by $a,b,\tilde{a},\tilde{b}$. Hence, $(a_1,a_2)$ and
$(\tilde{a}_1,\tilde{a}_2)$ are given by the regular sequences $(x_1+x_2-y_1-y_2, x_1 x_2 - y_1 y_2)$ and $(x_1-y_1,x_2-y_2)$,
respectively, and $(b_1,b_2)$ and $(\tilde{b}_1,\tilde{b}_2)$ are given by $(u_1,u_2)$ and $(\pi_1,\pi_2)$, respectively. We
have 
\begin{align*}
x_1 + x_2 - y_1 - y_2 & = (x_1 + y_1) - (x_2 - y_2)\\
 x_1 x_2 - y_1 y_2 &  = (x_1 - y_1)(y_2 + \lambda(x_2 - y_2)) + (x_2 - y_2)(x_1 + \lambda(y_1 - x_1)),
\end{align*}
hence
$$\begin{pmatrix}\lambda_{11} & \lambda_{21} \\ \lambda_{12} & \lambda_{22}\end{pmatrix}\ =\ \begin{pmatrix} 1 & y_2 +
  \lambda(x_2-y_2)\\ 1 & x_1 + \lambda(y_1-x_1)\end{pmatrix},$$ and in particular 
$$\lambda_{11}\lambda_{22} - \lambda_{12} \lambda_{21} = x_1-y_2 + \lambda(y_1 + y_2 - x_1 - x_2)$$ Finally, we compute
\begin{align*}
\mu = \frac{\lambda_{12} b_1 + \lambda_{22} b_2 - \tilde{b}_2}{\tilde{a}_1}= \frac{u_1 + (x_1 + \lambda(y_1-x_1))u_2 -
  \pi_2}{x_1-y_1} = -\lambda u_2 + \frac{u_1 + x_1 u_2 - \pi_2}{x_1 - y_1}.
\end{align*}
Putting everything together, we see that the morphism constructed in \ref{ex_stabmorphism2els} coincides with
\eqref{eq:morphismfromKR}, as claimed.

Similarly, we can handle the map $\chi_0$. Originally, it is defined as 
\begin{align}\label{eq:morphismfromKR2}
\begin{tikzpicture}
\matrix (m) [matrix of math nodes, row sep=7em,
                 column sep=7em, text height=1.5ex, text depth=0.25ex,
                 inner sep=0pt, nodes={inner xsep=0.3333em, inner ysep=0.3333em}]
    {
      S\ld(\tilde{e}_1)\oplus S\ld(\tilde{e}_2) \pgfmatrixnextcell\pgfmatrixnextcell S\ld(\emptyset)\oplus S\ld(\tilde{e}_1
      \tilde{e}_2)\\
      S\ld(e_1)\oplus S\ld(e_2) \pgfmatrixnextcell\pgfmatrixnextcell S\ld(\emptyset)\oplus S\ld(e_1 e_2)\\
    };
    \draw[->] ($(m-2-1.east) + (0mm,+1mm)$) -- node[above,scale=0.75]{$\begin{pmatrix} x_1+x_2-y_1-y_2 & x_1 x_2 - y_1
        y_2\\ -u_2 & u_1\end{pmatrix}$} ($(m-2-3.west) + (0mm,+1mm)$);
    \draw[->] ($(m-2-3.west) + (0mm,-1mm)$) -- node[below,scale=0.75]{$\begin{pmatrix} u_1 & y_1 y_2 - x_1 x_2\\ u_2 &
        x_1 + x_2 - y_1 - y_2\end{pmatrix}$} ($(m-2-1.east) + (0mm,-1mm)$); 
    \draw[->] ($(m-1-1.east) + (0mm,+1mm)$) -- node[above,scale=0.75]{$\begin{pmatrix} x_1-y_1 & x_2 - y_2\\ -\pi_2 &
        \pi_1\end{pmatrix}$} ($(m-1-3.west) + (0mm,+1mm)$);
    \draw[->] ($(m-1-3.west) + (0mm,-1mm)$) -- node[below,scale=0.75]{$\begin{pmatrix} \pi_1 & y_2-x_2\\ \pi_2 &
        x_1-y_1\end{pmatrix}$} ($(m-1-1.east) + (0mm,-1mm)$); 
    \draw[->] (m-1-1) -- node[fill=white,scale=0.75]{$\begin{pmatrix} y_1 + \lambda(x_1-y_1) & \lambda(x_2-y_2) - x_2 \\ -1 &
        1\end{pmatrix}$} (m-2-1);  
    \draw[->] (m-1-3) -- node[fill=white,scale=0.75]{$\begin{pmatrix} y_1 - x_2 + \mu(x_1+x_2-y_1-y_2) & 0 \\ a &
        1\end{pmatrix}$} (m-2-3); 
\end{tikzpicture}
\end{align}
for some $\lambda\in\ZZ$ and
\begin{align}\label{eq:longexpr}
a = (1-\lambda)u_2 + \frac{u_1 + x_1 u_2 - \pi_2}{y_1-x_1}.
\end{align}
We claim that $\chi_0$ is the stabilization of 
\begin{align}
\label{eq:canmaptostabilize}S\ld/(x_1-y_1, x_2-y_2) \xrightarrow{\mult(x_1 - y_2) = \mult(x_2 - y_1)} S\ld/(x_1 + x_2 -
y_1 - y_2, x_1 x_2 - y_1 y_2).
\end{align}
For the proof, we again use the method of Example \ref{ex_stabmorphism2els} with $\alpha = y_1 - x_2 +
\lambda(x_1+x_2-y_1-y_2)$. First, we compute the $\lambda_{ij}$; we have
\begin{align*}
(x_1-y_1)\alpha & = (x_1-y_1)(y_1 - x_2 + \lambda(x_1+x_2-y_1-y_2))\\
& = (x_1 + x_2 - y_1 - y_2)(y_1 + \lambda(x_1-y_1)) - (x_1 x_2 - y_1 y_2)\\
(x_2-y_2)\alpha & = (x_2-y_2)(y_1 - x_2 + \lambda(x_1+x_2-y_1-y_2))\\
& = (x_1 + x_2 - y_1 - y_2)(\lambda(x_2-y_2)-x_2) + x_1 x_2 - y_1 y_2,
\end{align*}
and hence
$$\begin{pmatrix}\lambda_{11} & \lambda_{21} \\ \lambda_{12} & \lambda_{22}\end{pmatrix}\ =\ \begin{pmatrix} y_1 +
  \lambda(x_1-y_1) & \lambda(x_2-y_2) - x_2 \\ -1 & 1\end{pmatrix}.$$
In particular, we get $\frac{\lambda_{11}\lambda_{22}-\lambda_{12}\lambda_{21}}{\alpha} = 1$. Finally, 
\begin{align*}
\mu & = \frac{\lambda_{12} b_1 + \lambda_{22} b_2 - (y_1 - x_2 +
\lambda(x_1+x_2-y_1-y_2))\tilde{b}_2}{\tilde{a}_1}
\\ &  = \frac{-\pi_1 + \pi_2 - (y_1 - x_2 + \lambda(x_1+x_2-y_1-y_2)) u_2}{x_1+x_2-y_1-y_2},
\end{align*}
and it is a tedious but straightforward computation to show that this equals \eqref{eq:longexpr}. Applying the result of
example \ref{ex_stabmorphism2els}, we see that indeed $\chi_0$ is the stabilization of \eqref{eq:canmaptostabilize}. 

Summing up, we have seen in this example that the morphisms $\chi_0$ and $\chi_1$ from \cite{KR1} are stabilizations of
canonical morphisms between Soergel bimodules. In \cite{EliasKhovanov}, these morphisms are depicted by
$\begin{tikzpicture}[scale=0.25]\draw (0,0) -- (0,1) node[inner sep=1pt, fill=black]{}; \end{tikzpicture}$ and
$\begin{tikzpicture}[scale=0.25]\draw (0,0) node[inner sep=1pt, fill=black]{} -- (0,1); \end{tikzpicture}$\ . It would
be interesting to see if the stabilizations of other canonical morphisms from \cite{EliasKhovanov} play a role in the
construction of Khovanov-Rozansky homology, too.  
\end{ex}

\subsection{The effect of stabilization on Soergel bimodules}\label{sec_badbimodule}

Until now, we showed how the image $\KR(\gamma)$ of a MOY-braid $\gamma$ under the Khovanov-Rozansky construction can be
expressed as the stabilization of the Soergel bimodule corresponding to $\gamma$. Though this gives us a
bunch of relations between the $\KR(\gamma)$ for free -- those which are already true on the level of Soergel bimodules
-- we didn't investigate the effect and use of stabilization yet. 

By the big picture \ref{fig:bigpicture} from the introduction, the following theorem meets our expectations:
\begin{theorem}\label{thm_killbadbimodule}
Let $w\in\frS_m$ be such that the Robinson-Schensted shape of $w$ has more than $n$ rows. Then we have
$$\projdim_{\Sym(\XX|\YY)/(\Sigma\XX^{n+1}-\Sigma\YY^{n+1})}\ \SBM{w}\ <\ \infty,\quad\text{ i.e. }
^m\SBM{w}\stab{\Sigma\XX^{n+1}-\Sigma\YY^{n+1}}\ =\ 0.$$
\end{theorem}
The proof will be divided in two steps:
\begin{enumerate}
\item\label{item:klstep1} Using the theory of (two-sided) Kazhdan-Lusztig cells we reduce to $w= (k,\ k-1,\
  ...,\ 1)$ for $k>n$, in which case $^m\SBM{w}$ is the Soergel bimodule associated to the MOY graph in Figure
  \ref{fig:badbimodule} (Theorem \ref{thm_soergelbimodulesviahecke}).
\item\label{item:klstep2} In this case, we prove that $^m\SBM{(k,\ k-1,\ ...,\ 1)}$ somehow involves the trivial category
$$\HMF(\Sym(x_1,...,x_k),\Sigma x_i^{n+1})=0$$ and deduce the triviality of its stabilization.
\end{enumerate}
We begin with step \ref{item:klstep1}. Recall the following definition of Kazhdan-Lusztig cells of the Hecke algebra (see
\cite[Exercise 6.11]{BjornerBrenti}). Note that
thinking of the Kazhdan-Lusztig elements as functors and products of them as compositions of these functors, it
essentially formalizes what should be meant by saying that one such functor factors through another. 
\begin{definition}
For elements $w,w\p\in\frS_m$ we write $w\leq_{\LR} w\p$ if there exist $s,t\in\frS_m$ such that the coefficient of
$\KL{w}$ in the product $\KL{s}\KL{w\p}\KL{t}$ is nonzero. This defines a preorder on $\frS_m$, and we say that $w$ and
$w\p$ are $\leq_{\LR}$-equivalent, written as $w\sim_{\LR} w\p$, if both $w\leq_{\LR} w\p$ and $w\p\leq_{\LR} w$.  
\end{definition}
The following proposition completely characterizes $\leq_{\LR}$-equivalence in terms of the Robinson-Schensted
correspondence (see \cite{Fulton}):
\begin{prop}[see {\cite[Exercise 6.11(b)]{BjornerBrenti}}]\label{prop_rscharacoflreq}
For $w,w\p\in\frS_m$ the following are equivalent:
\begin{enumerate}
\item $w\sim_{\LR} w\p$, i.e. $w$ and $w\p$ are $\leq_{\LR}$-equivalent.
\item The Robinson-Schensted shapes of $w$ and $w\p$ are the same.
\end{enumerate}
In particular, any $w\in\frS_m$ whose Robinson-Schensted shape has columns of length $$d_1\ \geq\ d_2\ \geq\ ...\ \geq\
d_k$$ is $\leq_{\LR}$-equivalent to the permutation  
$$
(d_1, d_1-1,\cdots, 2, 1)\ (d_1 + d_2, d_1 + d_2 - 1, ..., d_1+1)\cdots (d_1+...+d_k, ..., d_1+...+d_{k-1}+1)
$$
\end{prop}
\begin{cor}\label{cor_crucial}
Let $w\in\frS_m$ have a Robinson-Schensted shape with $k$ rows. Consider
\begin{align*}
^m\SBM{w}\stab{\Sigma\WW^{n+1}-\Sigma\YY^{n+1}} &\ \in\ \HMF(\Sym(\WW|\ZZ),\Sigma\WW^{n+1}-\Sigma\ZZ^{n+1})\\ 
^m\SBM{(k,\ k-1,\ ...,\ 1)}\stab{\Sigma\XX^{n+1}-\Sigma\YY^{n+1}}&\ \in\
\HMF(\Sym(\XX|\YY),\Sigma\XX^{n+1}-\Sigma\YY^{n+1}) 
\end{align*}
for $|\WW|=|\XX|=|\YY|=|\ZZ|=m$. Then there exist matrix 
factorizations $$A\in\HMF(\Sym(\WW|\XX),\Sigma\WW^{n+1}-\Sigma\XX^{n+1}),\quad
B\in\HMF(\Sym(\YY|\ZZ),\Sigma\YY^{n+1}-\Sigma\ZZ^{n+1})$$ such that  
$^m\SBM{w}\stab{\Sigma\WW^{n+1}-\Sigma\ZZ^{n+1}}$ is a summand of $$A\otimes_{\XX} {^m\SBM{(k,\ k-1,\ ...,\
  1)}\stab{\Sigma\XX^{n+1}-\Sigma\YY^{n+1}}}\otimes_{\YY} B.$$ 
In particular, we have 
$$^m\SBM{(k,\ k-1,\ ...,\ 1)}\stab{\Sigma\XX^{n+1}-\Sigma\YY^{n+1}}\ =\ 0\quad\Longrightarrow\quad
^m\SBM{w}\stab{\Sigma\WW^{n+1}-\Sigma\ZZ^{n+1}}\ =\ 0.$$ 
\end{cor}
\begin{proof}
Proposition \ref{prop_rscharacoflreq} and Fact \ref{fact_ggbyindecs} show that the statement is true for Soergel 
bimodules, hence by applying the stabilization functor we get the result from Theorem \ref{thm_acyclicgraphs}.
\end{proof}
This finishes step \ref{item:klstep1}. For step \ref{item:klstep2}, the following proposition is crucial:
\begin{prop}\label{prop_trivialcat}
Let $\XX$ and $\YY$ be sets of variables such that $|\XX|=|\YY|>n$. Then we have
$\HMF(\Sym(\XX|\YY),\Sigma\XX^{n+1}-\Sigma\YY^{n+1})=0$.
\end{prop}
\begin{proof}
Let $k := |\XX|= |\YY|$ and recall that we denoted $X_1,...,X_k$ the elementary symmetric polynomials in $\XX$, while
$X_l := 0$ for $l>k$.  By \cite[Formula 4.4]{WuLinkHomology}) we have $\Sigma\XX^{n+1} = P(X_1,...,X_d)$, where 
$$P := \begin{vmatrix} X_1 & X_2 & X_3 & \cdots & X_{n}   & (n+1)X_{n+1}\\ 
                       1   & X_1 & X_2 & \cdots & X_{n-1} &    n X_n\\
                       0   &   1 & X_1 & \cdots & X_{n-2} & (n-1)X_{n-1}\\
                       \cdots & \cdots & \cdots & \cdots & \cdots & \cdots\\
                       0   &   0 &   0 & \cdots & X_1     & 2 X_2\\
                       0   &   0 &   0 & \cdots &   1     & X_1\end{vmatrix}$$
and similar for $\Sigma\YY^{n+1}$. In particular, we conclude that $\Sigma\XX^{n+1}-\Sigma\YY^{n+1}\in\frm^2$, for\break 
$\frm := (X_1,...,X_k,Y_1,...,Y_k)$ the maximal ideal in $\Sym(\XX|\YY)$, if and only if $k\leq n$. Thus, in case
$k>n$ we see that $\Sym(\XX|\YY)/(\Sigma\XX^{n+1}-\Sigma\YY^{n+1})$ is regular (Proposition \ref{prop_stayregular})
and hence its singularity category $\HMF(\Sym(\XX|\YY),\Sigma\XX^{n+1}-\Sigma\YY^{n+1})$ is trivial (Proposition
\ref{prop_trivsingcat}).  
\end{proof}

\begin{figure}[h]\begin{center}
\begin{tikzpicture}
\draw[oredged] (0,0) -- node[txt]{$1$} (-2,1.2) node[above]{$\YY_1$} node[fat]{};
\draw[oredged] (0,0) -- node[txt]{$1$} (-1,1.2) node[above]{$\YY_2$} node[fat]{};
\draw[oredged] (0,0) -- node[txt]{$1$} (+1,1.2) node[above]{$\YY_{k-1}$} node[fat]{};
\draw[oredged] (0,0) -- node[txt]{$1$} (+2,1.2) node[above]{$\YY_k$} node[fat]{};
\draw[oredgeu] (-2,-1.2) node[fat]{} node[below]{$\XX_1$}     -- node[txtp]{$1$} (0,0);
\draw[oredgeu] (-1,-1.2) node[fat]{} node[below]{$\XX_2$}     -- node[txtp]{$1$} (0,0) ;
\draw[oredgeu] (+1,-1.2) node[fat]{} node[below]{$\XX_{k-1}$} -- node[txtp]{$1$} (0,0);
\draw[oredgeu] (+2,-1.2) node[fat]{} node[below]{$\XX_k$}     -- node[txtp]{$1$} (0,0)    node[big]{} ;
\draw[oredgeu] (+3,-1.2) node[fat]{} node[below]{$\ol{\XX}_{1}$} -- node[txtq]{$1$} (+3,1.2) node[above]{$\ol{\YY}_{1}$}
node[fat]{}; 
\draw[oredgeu] (+4,-1.2) node[fat]{} node[below]{$\ol{\XX}_{m-k-1}$} -- node[txtq]{$1$} (+4,1.2)
node[above]{$\ol{\YY}_{m-l-1}$} 
node[fat]{}; 
\node at (5,0){$\cdots$};
\draw[oredgeu] (+6,-1.2) node[fat]{} node[below]{$\ol{\XX}_{m-k}$}-- node[txtq]{$1$} (+6,1.2) node[above]{$\ol{\YY}_{m-k}$}
node[fat]{}; 
\node at (0,+1)[above] {$\cdots$};
\node at (0,-1)[below] {$\cdots$};
\draw[decorate,decoration={brace,amplitude=5pt}] (2,-1.8) -- node[txt,below=10pt,pos=0.5]{$k>n$\text{
    strands}}(-2,-1.8);
\draw[decorate,decoration={brace,amplitude=5pt}] (-2,1.8) -- node[txt,above=10pt,pos=0.5]{$m$\text{ strands}}(6,1.8);
\end{tikzpicture}
\end{center}
\caption{MOY-graph with trivial matrix factorization}
\label{fig:badbimodule}
\end{figure}
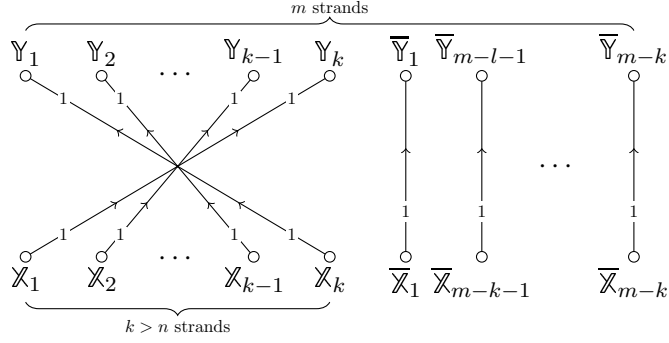

\begin{proof}[of Theorem \ref{thm_killbadbimodule}] We know from Corollary \ref{cor_crucial} that it suffices to
 show that $$^m\SBM{(k,\ k-1,\ ...,\ 1)}\stab{\Sigma\XX^{n+1}+\Sigma\ol{\XX}^{n+1}-\Sigma\YY^{n+1}-\Sigma\ol{\YY}^{n+1}}\
 =\ 0.$$ for $k>n$, which is the matrix factorization associated to the MOY-graph in Figure
 \ref{fig:badbimodule}. This graph can be decomposed into $\Gamma_{1^k}^{1^k}$ involving the variables from $\XX$ and $\YY$
 and into $m-k$ copies of $\Gamma_1^1$ involving the variables from $\ol{\XX}$ and $\ol{\YY}$, and so Theorem 
 \ref{thm_glue} yields $$\KR(\Gamma_{1^k}^{1^k}\sqcup\Gamma_{1}^1\sqcup\cdots\sqcup\Gamma_1^1)\ \simeq\
 \KR(\Gamma_{1^k}^{1^k})\otimes_{\CC} \KR(\Gamma_1^1)\otimes_\CC ...\otimes_\CC \KR(\Gamma_1^1),$$ whence we may assume
 $m=k$. Applying Theorem \ref{thm_glue} again, this time to the presentation of $\Gamma_{1^k}^{1^k}$ as the concatenation of
 $\Gamma_{1^k}^{k}$, $\Gamma_k^k$ and $\Gamma_k^{1^k}$ (similar to Figure \ref{fig:mergesplit}), shows that it suffices
 to prove $\KR(\Gamma_k^k)\simeq 0$. This follows from Proposition \ref{prop_trivialcat}.
\end{proof}

\begin{rem}
The proof of Theorem \ref{thm_killbadbimodule} suggests that one should rather think of a matrix factorization
$X\in\HMF(\Sym(\XX|\YY),\Sigma\XX^{n+1}-\Sigma\YY^{n+1})$ as a
functor $$\left.X\makebox[8mm][c]{$\stackrel[\mathclap{\Sym(\XX|\YY)}]{}{\otimes}$}
  -\uparrow_{\Sym(\YY)}^{\Sym(\XX|\YY)}\right\downarrow^{\Sym(\XX|\YY)}_{\Sym(\XX)}:\quad 
\HMF^\infty(\Sym(\YY),\Sigma\YY^{n+1})\longrightarrow\HMF^\infty(\Sym(\XX),\Sigma\XX^{n+1})$$ from the homotopy category 
of graded matrix factorizations of type $(\Sym(\YY),\Sigma\YY^{n+1})$ to those of type $(\Sym(\XX),\Sigma\XX^{n+1})$
. In some sense, this can be thought of as some kind of Fourier Mukai transform with kernel $X$, noting the striking
similarity to the usual formula $$\bfR \pr^X\la\circ \left(X\stackrel{\LL}{\otimes} -\right)\circ
\pr_Y\ua:\quad\bfD^b(Y)\longrightarrow\bfD^b(X)$$ for the Fourier-Mukai transform associated to some $X\in\bfD^b(X\times
Y)$.  However, though this viewpoint seems to be the most natural in our context, the author does not know to what
extend $X$ is determined by the Fourier Mukai transform attached to it, which is why we sticked working with $X$ instead
of its Fourier Mukai transform the proof of Theorem \ref{thm_killbadbimodule}.
\end{rem}

\section{Duality on graded matrix factorizations}\label{sec_duality}

In this section, we define for graded matrix factorizations $M$ and $N$ of type $(S\ld,w_0)$ and $(S\ld,w_1)$,
respectively, a \textit{homomorphism factorization} $\hom_{S\ld}(M,N)$, which is a graded matrix factorization of type
$(S\ld,w_1-w_0)$. In particular, in case $w_0 = w_1$ we get a homomorphism complex which in fact coincides with the
homomorphism complex in the canonical differential-graded enrichment of $\MF^\infty(S\ld,w_0)$. As a special case, we will
define for each graded matrix factorization $M$ its \textit{dual} $M\us$ to be $\hom_{S\ld}(M,S)$, where $S$ is the
trivial matrix factorization $S\ld\to 0\to S\ld$ of type $(S\ld,0)$, and we will check that the usual isomorphism
$M\us\otimes_{S\ld} N\cong\hom_{S\ld}(M,N)$ holds for finitely generated $M$. We will compare this duality with the usual
duality on $\MCM(S\ld/(w))$. 

\subsection{Homomorphism factorizations and Duality}

For easier reference, recall the definition of the shift functor:
\begin{align*}
[1]:\ \MF^{(\infty)}(S\ld,w) & \longrightarrow \MF^{(\infty)}(S\ld,w)\\
\left(M^0\ld\mor{f}M^{-1}\ld\mor{g}M^0\ld\right) & \longmapsto \left(M^{-1}\ld\langle
  d\rangle\mor{-g}M^0\ld\mor{-f}M^{-1}\ld\langle d\rangle\right) 
\end{align*}

\begin{definition}
Let $M := M^0\ld\mor{f} M^{-1}\ld\mor{g} M^0\ld$, $N := N^0\ld\mor{f\p} N^{-1}\ld\mor{g\p} N^0\ld$ be graded matrix
factorizations of type $(S\ld,w_0)$ and $(S\ld,w_1)$, respectively. The \textit{homomorphism factorization}
$\hom_{S\ld}(M,N)$ is defined as  
\begin{equation*}\begin{tikzpicture}[description/.style={fill=white,inner sep=2pt}]
    \matrix (m) [matrix of math nodes, row sep=3em,
                 column sep=2.5em, text height=1.5ex, text depth=0.25ex,
                 inner sep=0pt, nodes={inner xsep=0.3333em, inner ysep=0.3333em}]
    {
\hom_{S\ld}(M,N)^0 & := & \hom_S(M\ld^0,N\ld^0)\ld\ \oplus\ \hom_S(M^{-1}\ld,N^{-1}\ld)\ld \\
\hom_{S\ld}(M,N)^{-1} & := & \hom_S(M\ld^0,N^{-1}\ld)\ld\ \oplus\ \hom_S(M^{-1}\ld,N^0\ld)\ld\langle -d\rangle\\
\hom_{S\ld}(M,N)^0 & := & \hom_S(M\ld^0,N\ld^0)\ld\ \oplus\ \hom_S(M^{-1}\ld,N^{-1}\ld)\ld\\
    };
    \draw[->] (m-1-1) -- (m-2-1);
    \draw[->] (m-2-1) -- (m-3-1);

    \draw[->] (m-1-3) -- node[right,scale=0.75]{$\begin{pmatrix} f\p\circ\sqm & -(\sqm\circ f)\\ -(\sqm\circ g) &
        g\p\circ\sqm\end{pmatrix}$} (m-2-3); 
    \draw[->] (m-2-3) -- node[right,scale=0.75]{$\begin{pmatrix} g\p\circ\sqm & \sqm\circ f\\ \sqm\circ g & f\p\circ
        \sqm\end{pmatrix}$} (m-3-3);
\end{tikzpicture}\end{equation*}
This is a graded matrix factorization of type $(S\ld,w_1-w_0)$.
\end{definition}

\begin{definition}
Let $M := M^0\ld\mor{f} M^{-1}\ld\mor{g} M^0\ld$ be a graded matrix factorization of type $(S\ld,w)$ and denote by $S$ the
trivial matrix factorization $S\ld\to 0\to S\ld$ of type $(S\ld,0)$. The \textit{dual} of $M$, denoted $M\us$, is
defined as $$M\us\ :=\ \hom_{S\ld}(M,S)\ =\
\left(M^0\ld\right)\us\xrightarrow{-g\ua}\left(M^{-1}\ld\right)\us\langle -d\rangle\xrightarrow{f\ua}
\left(M^0\ld\right)\us.$$ $M\us$ is a graded matrix factorization of type $(S\ld,-w)$. 
\end{definition}

\begin{fact}\label{fact_doubledual}
For a finitely generated graded matrix factorization $M := M^0\ld\mor{f} M^{-1}\ld\mor{g} M^0\ld$, the
double dual $M\uss$ is canonically isomorphic to $M$ via
\begin{equation*}\begin{tikzpicture}[description/.style={fill=white,inner sep=2pt}]
    \matrix (m) [matrix of math nodes, row sep=3em,
                 column sep=2.5em, text height=1.5ex, text depth=0.25ex,
                 inner sep=0pt, nodes={inner xsep=0.3333em, inner ysep=0.3333em}]
    {
M\uss && \left(M^0\ld\right)\uss & \left(M^{-1}\ld\right)\uss & \left(M^0\ld\right)\uss \\ 
M && M^0\ld & M^{-1}\ld & M^0\ld\\       
    };
    \draw[->] (m-2-1) -- node[left,scale=0.75]{$\cong$} (m-1-1);
    \draw[->] (m-2-3) -- node[left,scale=0.75]{$-\ev$} node[right,scale=0.75]{$\cong$} (m-1-3);
    \draw[->] (m-2-4) -- node[left,scale=0.75]{$-\ev$} node[right,scale=0.75]{$\cong$} (m-1-4);
    \draw[->] (m-2-5) -- node[left,scale=0.75]{$-\ev$} node[right,scale=0.75]{$\cong$} (m-1-5);

    \draw[->] (m-1-3) -- node[above,scale=0.75]{$-f^{\ast\ast}$} (m-1-4);
    \draw[->] (m-1-4) -- node[above,scale=0.75]{$-g^{\ast\ast}$} (m-1-5);

    \draw[->] (m-2-3) -- node[above,scale=0.75]{$f$} (m-2-4);
    \draw[->] (m-2-4) -- node[above,scale=0.75]{$g$} (m-2-5);
\end{tikzpicture}\end{equation*}
\end{fact}

\begin{fact}\label{fact_tensorhomiso}
Let $M := M^0\ld\mor{f} M^{-1}\ld\mor{g} M^0\ld$, $N := N^0\ld\mor{f\p} N^{-1}\ld\mor{g\p} N^0\ld$ be graded matrix
factorizations of type $(S\ld,w_0)$ and $(S\ld,w_1)$, respectively, and assume that $M$ is finitely generated. Then
there is a canonical isomorphism of matrix factorizations of type $(S\ld,w_1-w_0)$
$$M\us\otimes_{S\ld} N\ \ \stackrel{\cong}{\longrightarrow}\ \ \hom_{S\ld}(M,N).$$
\end{fact}

There is another kind of duality on graded matrix factorizations which preserves the type, namely the one corresponding
to the usual duality $\hom_{R}(-,R\ld)\ld$ on $\uMCM(R\ld)$ for a Gorenstein graded ring $R\ld$. It can be described
explicitly as follows:

\begin{definition}\label{def_wdual}
Let $M := M^0\ld\mor{f} M^{-1}\ld\mor{g} M^0\ld$ be a graded matrix factorization of type $(S\ld,w)$. We define its
\textit{$w$-dual} $M\uc$ as
$$M\uc\ :=\ \left(M^{-1}\ld\right)\us\langle -d\rangle\xrightarrow{f\us\langle -d\rangle} \left(M^0\ld\right)\us\langle
-d\rangle\xrightarrow{g\us\langle -d\rangle} \left(M^{-1}\ld\right)\us\langle -d\rangle.$$
\end{definition}

The shift in the internal degree is explained in the proof of the following fact.

\begin{fact}\label{fact_wdualcompatible}
The following diagram is commutative up to canonical isomorphism
\begin{equation*}\begin{tikzpicture}[description/.style={fill=white,inner sep=2pt}]
    \matrix (m) [matrix of math nodes, row sep=3em,
                 column sep=2.5em, text height=1.5ex, text depth=0.25ex,
                 inner sep=0pt, nodes={inner xsep=0.3333em, inner ysep=0.3333em}]
    {
      \MF(S\ld,w) && \MCM(R\ld)\\
      \MF(S\ld,w) && \MCM(R\ld)\\
    };
    \draw[->] (m-1-1) -- node[scale=0.75,above]{$\coker$} (m-1-3);
    \draw[->] (m-1-1) -- node[scale=0.75,left]{$(-)\uc$} (m-2-1);
    \draw[->] (m-2-1) -- node[scale=0.75,below]{$\coker$} (m-2-3);
    \draw[->] (m-1-3) -- node[scale=0.75,right]{$\hom_R(-,R\ld)\ld$} (m-2-3);
\end{tikzpicture}\end{equation*}
\end{fact}
\begin{proof}
For a graded matrix factorization $M := M^0\ld\mor{f} M^{-1}\ld\mor{g} M^0\ld$ of type $(S\ld,w)$ we have the usual
exact sequence of graded $R\ld$-modules $$...\to M^{-1}\ld/wM^{-1}\ld\langle -d\rangle\xrightarrow{g} M^0\ld/wM^0\ld\langle 
-d\rangle\xrightarrow{\ f\ } M^{-1}\ld/wM^{-1}\ld\xrightarrow{\ g\ }  M^0\ld/wM^0\ld\longrightarrow \coker(M).$$ Applying
$\hom_{R}(-,R\ld)\ld$ to this sequence yields the exact sequence
$$...\leftarrow\hom_R(M^{-1}\ld/wM^{-1}\ld,R\ld)\ld \leftarrow\hom_R(M^0\ld/wM^0\ld,R\ld)\ld\leftarrow
\hom_R(\coker(M),R\ld)\ld\leftarrow 0$$ 
which is canonically isomorphic to
$$...\xleftarrow{g\us} \left(M^0\ld\right)\mathclap{\us}/w\left(M^0\ld\right)\mathclap{\us}\langle d\rangle\leftarrow  
\left(M^{-1}\ld\right)\mathclap{\us}/w\left(M^{-1}\ld\right)\mathclap{\us}\stackrel{\ g\us}{\leftarrow}
\left(M^0\ld\right)\mathclap{\us}/w\left(M^0\ld\right)\mathclap{\us}\leftarrow\hom_R(\coker(M),R\ld)\ld.$$
We conclude that $$\hom_R(\coker(M),R\ld)\ld\cong\ker(g\us)\cong\coker(g\us)\langle -d\rangle\cong\coker(M\uc)$$ as
claimed. 
\end{proof}

\begin{fact}\label{fact_wdualtensorcompatibility}
Let $M := M^0\ld\mor{f} M^{-1}\ld\mor{g} M^0\ld$, $N := N^0\ld\mor{f\p} N^{-1}\ld\mor{g\p} N^0\ld$ be graded matrix
factorizations of type $(S\ld,w_0)$ and $(S\ld,w_1)$, respectively. Then there is a canonical isomorphism 
$$\left(M\otimes_{S\ld} N\right)\uc\quad\cong\quad M\uc\otimes_{S\ld} N\uc [1].$$
\end{fact}
\begin{proof}
This follows from direct calculation.
\end{proof}
The above two dualities are related by a sign change.
\begin{definition}
Let $M := M^0\ld\mor{f} M^{-1}\ld\mor{g} M^0\ld$ be a graded matrix factorization of type $(S\ld,w)$. We define the 
\textit{sign change} $\sigma(M)$ of $M$ as $$\sigma(M)\ :=\ M^0\ld\mor{-f} M^{-1}\ld\mor{g} M^0\ld.$$
This is a matrix factorization of type $(S\ld,-w)$, and $\sigma$ defines isomorphisms of categories
$$\MF^{(\infty)}(S\ld,w)\cong\MF^{(\infty)}(S\ld,-w)\ \ \quad\text{and}\quad\ \ 
\HMF^{(\infty)}(S\ld,w)\cong\HMF^{(\infty)}(S\ld,-w).$$ 
\end{definition}
\begin{fact}\label{fact_dualwdualrelation}
The following diagram is commutative up to canonical isomorphism:
\begin{equation*}\begin{tikzpicture}[description/.style={fill=white,inner sep=2pt}]
    \matrix (m) [matrix of math nodes, row sep=3em,
                 column sep=2.5em, text height=1.5ex, text depth=0.25ex,
                 inner sep=0pt, nodes={inner xsep=0.3333em, inner ysep=0.3333em}]
    {
       \MF(S\ld,w) && \MF(S\ld,-w)\\
       \MF(S\ld,-w) && \MF(S\ld,-w)\\
    };
    \draw[->] (m-1-1) -- node[left,scale=0.75]{$(-)\us$} (m-2-1);
    \draw[->] (m-1-1) -- node[above,scale=0.75]{$\sigma$} (m-1-3);
    \draw[->] (m-1-3) -- node[right,scale=0.75]{$(-)\uc$} (m-2-3);
    \draw[->] (m-2-3) -- node[below,scale=0.75]{$[1]$} (m-2-1);
\end{tikzpicture}\end{equation*}
\end{fact}
\begin{proof}
This is clear from the definitions.
\end{proof}

\subsection{Compatibility of Duality and Stabilization}

Next we study the compatibility of the stabilization functor with duality. 

\begin{prop}\label{prop_wdual}
Let $M\ld$ be a Cohen-Macaulay module over $R\ld$, and let $n := \dim(S\ld)-\depth(M\ld)$. Then there is a canonical
isomorphism in $\HMF(S\ld,w)$ $$\left(M\ld\stab{w}\right)\uc[n+1]\ \cong\ \ext_{S}^{n}(M\ld,S\ld)\ld\stab{w}.$$ 
\end{prop}
\begin{proof}
By a theorem of Grothendieck (\cite[Theorem 3.5.7]{BrunsHerzog}), the local cohomology $\H_{\frm}\ua(M\ld)$ is
concentrated in the interval $[\depth(M),\dim(M)]$ and nonzero on its boundary. Since $M\ld$ is Cohen-Macaulay of
dimension $n$, it follows that $\H_{\frm}\ua(M\ld)$ is concentrated in degree $n$. Next, by local duality
(\cite{BrunsHerzog}, Theorem 3.6.19) we have $$\ext_S^{\dim(S)-i}(M\ld,S\ld)\ld\ \cong\
\H_{\frm}^i(M\ld)^{\vee}\quad\text{ for all }i\in\ZZ,$$ where $(-)^{\vee}$ 
denotes Matlis duality. We conclude that $\ext\ua_{S}(M\ld,S\ld)\ld$ is concentrated in degree $n$ which equals
$\projdim_{S\ld}(M\ld)$ by the Auslander-Buchsbaum formula \ref{thm_auslanderbuchsbaum}. Therefore 
we can choose a finite free resolution $F\ua\ld\to M\ld$ of $M\ld$ such that $F^i\ld=0$ for $i<-n$, and as
$\ext^k_{S}(M\ld,S\ld)\ld=0$ for all $k<n$, applying $\hom_S(-,S\ld)\ld$ gives an $S\ld$-free resolution
\begin{align}\label{eq:dualseq}0\to \left(F^0\ld\right)\us\longrightarrow \left(F^{-1}\ld\right)\us\to ...\to
\left(F^{-(n-1)}\ld\right)\us\longrightarrow\left(F^{-n}\ld\right)\us\longrightarrow\ext^n_S(M\ld,S\ld)\ld\to
0.\end{align} Furthermore, if $s_n$ are higher homotopies for $F\ua\ld\to M\ld$ as in Lemma \ref{lem_preparationlemma},
their duals $s_n\us$ give higher homotopies for \eqref{eq:dualseq}, so we can compute $\ext^n_S(M\ld,S\ld)\ld\stab{w}$
using \eqref{eq:dualseq} and the $s_n\us$. 

We distinguish the cases $n$ even and odd. Carefully going through the constructions, in case $n = 2n\p$ we get
\begin{align*}
\ext^n_S(M\ld,S\ld)\ld\stab{w} & \makebox[8mm][c]{$\stackrel{\ref{prop_stabalgo}}{\cong}$}
\begin{tikzpicture}[baseline=-1mm,description/.style={fill=white,inner sep=2pt}]
    \matrix (m) [matrix of math nodes, row sep=3em,
                 column sep=3.2em, text height=1.5ex, text depth=0.25ex,
                 inner sep=0pt, nodes={inner xsep=0.3333em, inner ysep=0.3333em}]
    {
      \bigoplus\limits_{i\in\ZZ} \left(F^{-2i+1}\ld\right)\us\langle (n\p-i)d\rangle \pgfmatrixnextcell
      \bigoplus\limits_{i\in\ZZ} \left(F^{-2i}\ld\right)\us\langle (n\p-i)d\rangle\\
    };
    \draw[->] ($(m-1-1.east) + (0,0.8mm)$) -- node[above,scale=0.75]{$\sum\limits_{i\geq 0} s_i\us$} ($(m-1-2.west) +
    (0,0.8mm)$); 
    \draw[->] ($(m-1-2.west) - (0,0.8mm)$) -- node[below,scale=0.75]{$\sum\limits_{i\geq 0} s_i\us$} ($(m-1-1.east) -
    (0,0.8mm)$); 
\end{tikzpicture}\\
& \makebox[8mm][c]{$\cong$} 
\begin{tikzpicture}[baseline=-1mm,description/.style={fill=white,inner sep=2pt}]
    \matrix (m) [matrix of math nodes, row sep=3em,
                 column sep=3.2em, text height=1.5ex, text depth=0.25ex,
                 inner sep=0pt, nodes={inner xsep=0.3333em, inner ysep=0.3333em}]
    {
      \left(\bigoplus\limits_{i\in\ZZ}
  F^{-(2i+1)}\ld\langle id\rangle\right)^{\mathclap{\star}}\langle (n\p-1) d\rangle \pgfmatrixnextcell
      \left(\bigoplus\limits_{i\in\ZZ} F^{-2i}\ld\langle
  id\rangle\right)^{\mathclap{\star}}\langle n\p d\rangle\\
    };
    \draw[->] ($(m-1-1.east) + (0,0.8mm)$) -- node[above,scale=0.75]{$\left(\sum\limits_{i\geq 0}
    s_i\right)^{\mathclap{\star}}$} ($(m-1-2.west) + (0,0.8mm)$); 
    \draw[->] ($(m-1-2.west) - (0,0.8mm)$) -- node[below,scale=0.75]{$\left(\sum\limits_{i\geq 0}
        s_i\right)^{\mathclap{\star}}$} ($(m-1-1.east) - (0,0.8mm)$); 
\end{tikzpicture}\\
& \makebox[8mm][c]{$\cong$} \left(M\ld\stab{w}\right)\uc [n+1].
\end{align*}
Similarly, for $n = 2n\p+1$ we have
\begin{align*}
\ext^n_S(M\ld,S\ld)\ld\stab{w} & \makebox[8mm][c]{$\stackrel{\ref{prop_stabalgo}}{\cong}$}
\begin{tikzpicture}[baseline=-1mm,description/.style={fill=white,inner sep=2pt}]
    \matrix (m) [matrix of math nodes, row sep=3em,
                 column sep=3.2em, text height=1.5ex, text depth=0.25ex,
                 inner sep=0pt, nodes={inner xsep=0.3333em, inner ysep=0.3333em}]
    {
      \bigoplus\limits_{i\in\ZZ} \left(F^{-2i}\ld\right)\us\langle (n\p-i)d\rangle \pgfmatrixnextcell
      \bigoplus\limits_{i\in\ZZ} \left(F^{-(2i+1)}\ld\right)\us\langle (n\p-i)d\rangle\\
    };
    \draw[->] ($(m-1-1.east) + (0,0.8mm)$) -- node[above,scale=0.75]{$\sum\limits_{i\geq 0} s_i\us$} ($(m-1-2.west) +
    (0,0.8mm)$); 
    \draw[->] ($(m-1-2.west) - (0,0.8mm)$) -- node[below,scale=0.75]{$\sum\limits_{i\geq 0} s_i\us$} ($(m-1-1.east) -
    (0,0.8mm)$); 
\end{tikzpicture}\\
& \makebox[8mm][c]{$\cong$} 
\begin{tikzpicture}[baseline=-1mm,description/.style={fill=white,inner sep=2pt}]
    \matrix (m) [matrix of math nodes, row sep=3em,
                 column sep=3.2em, text height=1.5ex, text depth=0.25ex,
                 inner sep=0pt, nodes={inner xsep=0.3333em, inner ysep=0.3333em}]
    {
      \left(\bigoplus\limits_{i\in\ZZ}
      F^{-2i}\ld\langle id\rangle\right)^{\mathclap{\star}}\langle n\p d\rangle \pgfmatrixnextcell
      \left(\bigoplus\limits_{i\in\ZZ} F^{-(2i+1)}\ld\langle id\rangle\right)^{\mathclap{\star}}\langle
      n\p d\rangle\\
    };
    \draw[->] ($(m-1-1.east) + (0,0.8mm)$) -- node[above,scale=0.75]{$\left(\sum\limits_{i\geq 0}
    s_i\right)^{\mathclap{\star}}$} ($(m-1-2.west) + (0,0.8mm)$); 
    \draw[->] ($(m-1-2.west) - (0,0.8mm)$) -- node[below,scale=0.75]{$\left(\sum\limits_{i\geq 0}
        s_i\right)^{\mathclap{\star}}$} ($(m-1-1.east) - (0,0.8mm)$);  
\end{tikzpicture}\\
& \makebox[8mm][c]{$\cong$} \left(M\ld\stab{w}\right)\uc [n+1].
\end{align*}
which finishes the proof.
\end{proof}

Proposition \ref{prop_wdual} has the following interesting special case.

\begin{cor}\label{cor_wdual}
Let $M\ld = S\ld / (x_1,...,x_n)$ for a regular sequence $x_1,...,x_n$ of homogeneous elements such that $w\in
(x_1,...,x_n)$. Then there is a canonical isomorphism in $\HMF(S\ld,w)$ $$\left(M\ld\stab{w}\right)\uc\quad\cong\quad
M\ld\stab{w}\langle |x_1|+...+|x_n|\rangle[-(n+1)].$$ 
\end{cor}
\begin{proof}
As $(x_i)$ is regular, $M\ld$ is Cohen-Macaulay. Further, computing $\ext\ua_{S}(M\ld,S\ld)\ld$ using the self-dual
Koszul-complex of the $x_i$, we get $\ext^n_S(M\ld,S\ld)\ld\cong M\ld\langle |x_1|+...+|x_n|\rangle$. Now the claim
follows from Proposition \ref{prop_wdual}. 
\end{proof}

\begin{ex}
Let us pause for a moment to check the statement of Corollary \ref{cor_wdual} directly using Proposition
\ref{prop_koszulstabilization2}. For two homogeneous elements $x,y\in S\ld$ we have
\begin{align*}
\{x,y\}\uc & = \left(S\ld\langle |x|-d\rangle\mor{y} S\ld\langle -d\rangle\mor{x}S\ld\langle |x|-d\rangle\right) =
\{x,y\}\langle |x|\rangle[-2]\end{align*} 
which agrees with Corollary \ref{cor_wdual}. In general, for a homogeneous regular sequence $x_1,...,x_n$ and
homogeneous elements $y_1,...,y_n$ satisfying $w = x_1 y_1 + ... + x_n y_n$, we get from Fact
\ref{fact_wdualtensorcompatibility} that
$$\{\bfx,\bfy\}\uc = \left(\bigotimes_{i=1}^{n} \{x_i,y_i\}\right)\uc \cong
\left(\bigotimes_{i=1}^{n}\{x_i,y_i\}\uc\right)[n-1] \cong \{\bfx,\bfy\}\langle |x_1|+...+|x_n|\rangle [-(n+1)],$$
which agrees again with Corollary \ref{cor_wdual}.
\end{ex}

\begin{cor}\label{cor_dual}
Let $M\ld$ be a Cohen-Macaulay module over $R\ld$, and put $n := \dim(S)-\depth(M)$. Then there is a canonical
isomorphism in $\HMF(S\ld,-w)$ $$\left(M\ld\stab{w}\right)\ua\ \cong\ \ext_{S}^{n}(M\ld,S\ld)\ld\stab{-w}[-n].$$ 
\end{cor}
\begin{proof}
By Fact \ref{fact_dualwdualrelation} and Proposition \ref{prop_wdual} we have
$$\left(M\ld\stab{w}\right)\ua\cong \left(M\ld\stab{-w}\right)\uc[1] \cong \ext_{S}^{n}(M\ld,S\ld)\ld\stab{-w}[-n].$$ as
claimed. 
\end{proof}
\begin{cor}\label{cor_dualci}
Let $M\ld = S\ld / (x_1,...,x_n)$ for a regular sequence $x_1,...,x_n$ of homogeneous elements such that $w\in
(x_1,...,x_n)$. Then there is a canonical isomorphism in $\HMF(S\ld,-w)$ $$\left(M\ld\stab{w}\right)\ua\quad\cong\quad
M\ld\stab{-w}\langle |x_1|+...+|x_n|\rangle[-n].$$ 
\end{cor}
\begin{ex}
Again let us check explicitly that everything works for Koszul factorizations. If $x,y$ are homogeneous and regular, 
then 
\begin{align*}
\{x,y\}\ua & \makebox[8mm][c]{=} \left(S\ld\mor{-x} S\ld\langle |x|-d\rangle\mor{y} S\ld\right) =
\left(S\ld\mor{-y}S\ld\langle -|x|\rangle\mor{x}S\ld\right)\langle |x|\rangle [-1]\\
& \makebox[8mm][c]{=} \{x,-y\}\langle |x|\rangle [-1].\end{align*}
as claimed.
\end{ex}

\section{Closing a MOY-braid}\label{sec_closing}

\subsection{Braid closure as stabilized Hochschild cohomology}

We now apply the results about duality from the preceding section to Khovanov-Rozansky homology. Suppose we want to
calculate the value of Khovanov-Rozansky homology on the closure $\ol{\gamma}$ of a MOY-braid $\gamma$ with strands of
type $(i_1,...,i_r)$ and sets of variables $\XX$ and $\YY$; see Figure \ref{fig:closure}. Intuitively, we expect that
passing from $\gamma$ to $\ol{\gamma}$ should involve some categorical trace or Hochschild 
(co)homology; see \cite{KhovanovHochschild} or \cite[Section 2.4]{WebsterCanopolis}. This is indeed the case. We will see
that there is some ``identity'' matrix factorization $\id$ such that $\H\ua(\KR(\ol{\gamma}))$ is given by
$\HMF(\id[\ast],\KR(\gamma))$ (for the precise statement, see Theorem \ref{theorem_closethebraid}) which can be
interpreted as some kind of ``stabilized'' Hochschild cohomology or as generalized Tate cohomology; see Remark
\ref{rem_fancywords}. But now let's stop playing with fancy words and instead dig into the somewhat technical details.  
\begin{figure}[h]\begin{center}
\begin{tikzpicture}[rounded corners=2pt, scale=0.75]
\draw[oredgem, thin] (-2,4.5) -- (-2,-0.5);
\draw[oredgem, thin] (-2.3,4.7) -- (-2.3,-0.7);
\draw[oredgem, thin] (-2.6,4.9) -- (-2.6,-0.9);
\draw[rounded corners=20pt] (-2,-1.5) rectangle (0.5,5.5);
\draw[rounded corners=20pt] (-2.3,-1.7) rectangle (1.5,5.7);
\draw[rounded corners=20pt] (-2.6,-1.9) rectangle (3.5,5.9);
\draw (0.5,-0.5) node[fat]{} node[right=10pt,above=-8pt]{$\YY_1$} -- node[scale=0.7,pos=0.86,left]{$i_1$} node[scale=0.7,pos=0.14,left]{$i_1$} (0.5,4.5) node[fat]{} node[right=10pt,below=-8pt]{$\XX_1$};  
\draw (1.5,-0.5) node[fat]{} node[right=10pt,above=-8pt]{$\YY_2$} -- node[scale=0.7,pos=0.86,left]{$i_2$} node[scale=0.7,pos=0.14,left]{$i_2$} (1.5,4.5) node[fat]{} node[right=10pt,below=-8pt]{$\XX_2$};  
\draw (3.5,-0.5) node[fat]{} node[right=10pt,above=-8pt]{$\YY_r$} -- node[scale=0.7,pos=0.86,left]{$i_r$} node[scale=0.7,pos=0.14,left]{$i_r$} (3.5,4.5) node[fat]{} node[right=10pt,below=-8pt]{$\XX_r$};  
\node[scale=0.7] at (2.5,0.2) {$\cdots$};
\node[scale=0.7] at (2.5,3.8) {$\cdots$};
\draw[fill=white, rounded corners=2pt] (0,0.5) rectangle (4,3.5);
\node[scale=1.25] at (2,2) {$\gamma$};
\end{tikzpicture}
\end{center}
\caption{\small{Closure of a braid}}
\label{fig:closure}
\end{figure}
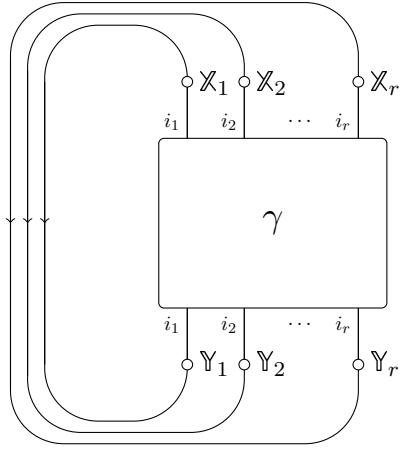

We already know that $\KR(\gamma)\cong \B(\gamma)\ld\stab{\Sigma\XX^{n+1}-\Sigma\YY^{n+1}}$, where $\B(\gamma)\ld$
denotes the singular Soergel bimodule corresponding to $\gamma$. Thus, we have  
$$\KR(\ol{\gamma}) = \Id\ld\stab{\Sigma\YY^{n+1}-\Sigma\XX^{n+1}}
\makebox[8mm][c]{$\stackrel[\mathclap{\Sym(\XX|\YY)}]{}{\otimes}$}  
\B(\gamma)\ld\stab{\Sigma\XX^{n+1}-\Sigma\YY^{n+1}},$$   
where the first factor $\Id\ld$ is defined as
\begin{align*}
\Id\ld\ :=\  \bigotimes_{j=1}^{r} \Sym(\XX_{j}|\YY_{j})/\langle X_{j,l}-Y_{j,l}\ |\
  l=1,...,i_j\rangle.
\end{align*} Clearly, it depends on the configuration of strands and the sets of variables; however, dropping these data
will hopefully not cause any confusion. The module $\Id\ld$ should be thought of as the 'identity' between $\XX$ and $\YY$. 

\begin{prop}\label{prop_dualidentity}
Put $N := i_1+...+i_r$ and $k := \sum\limits_{j=1}^{r}\frac{i_j(i_j+1)}{2}$. Then we have an isomorphism
$$\left(\Id\ld\stab{\Sigma\XX^{n+1}-\Sigma\YY^{n+1}}\right)\us\cong\Id\ld\stab{\Sigma\YY^{n+1}-\Sigma\XX^{n+1}}\langle
k\rangle [-N].$$ 
\end{prop}
\begin{proof} 
This follows from Corollary \ref{cor_dualci} on the dual of the stabilization of a complete intersection. Here, for each
$j=1,...,r$ we have elementary symmetric polynomials $X_{j,1},...,X_{j,i_j}$ of degrees $1,2,...,i_j$; hence, the sum of
their degrees is $k$. 
\end{proof}
Now we can rewrite the cohomology of $\KR(\ol{\gamma})$ purely in terms of maximal Cohen-Macaulay modules.
\begin{theorem}\label{theorem_closethebraid}
Let $\gamma$ be as above, and put $N := i_1+...+i_r$ and $k := \sum\limits_{j=1}^{r}\frac{i_j(i_j+1)}{2}$. Then there is
a homotopy equivalence of matrix factorizations  
\begin{align}
\label{eq:closethebraid}\KR(\ol{\gamma})\cong\hom_{\Sym(\XX|\YY)}\left(\Id\ld\stab{\Sigma\XX^{n+1}-\Sigma\YY^{n+1}}\langle 
  k\rangle [-N],\B(\gamma)\ld\stab{\Sigma\XX^{n+1}-\Sigma\YY^{n+1}}\right).\end{align} 
\end{theorem}
\begin{proof}
Now on with our fancy little proof. We have $$\Id\ld\stab{\Sigma\YY^{n+1}-\Sigma\XX^{n+1}} \cong
\left(\Id\ld\stab{\Sigma\YY^{n+1}-\Sigma\XX^{n+1}}\right)\uss\cong 
\left(\Id\ld\stab{\Sigma\XX^{n+1}-\Sigma\YY^{n+1}}\langle k\rangle [-N]\right)\us$$
by Fact \ref{fact_doubledual} and Proposition \ref{prop_dualidentity}. Applying Fact \ref{fact_tensorhomiso} yields
\eqref{eq:closethebraid}.
\end{proof}
\begin{cor}
There are isomorphisms (abbreviating $R\ld := \Sym(\XX|\YY)/(\Sigma\XX^{n+1}-\Sigma\YY^{n+1})$)
\begin{align}
\label{eq:closeisohmf}\H^l(\KR(\ol{\gamma})) & \makebox[8mm][c]{$\cong$}
\HMF(R\ld)\left(\Id\ld\stab{\Sigma\XX^{n+1}-\Sigma\YY^{n+1}}\langle k\rangle 
  [l-N],\B(\gamma)\ld\stab{\Sigma\XX^{n+1}-\Sigma\YY^{n+1}}\right)\\ \notag&\makebox[8mm][c]{$\cong$}
\uMCM(R\ld)\left(\Id\ld\stab{\Sigma\XX^{n+1}-\Sigma\YY^{n+1}}\langle k\rangle 
  [l-N],\B(\gamma)\ld\stab{\Sigma\XX^{n+1}-\Sigma\YY^{n+1}}\right)\\\label{eq:closeisoadj} &
\makebox[8mm][c]{$\cong$}\ul{R\ld\mod}\ \ \  
\left(\Id\ld\stab{\Sigma\XX^{n+1}-\Sigma\YY^{n+1}}\langle k\rangle[l-N],\B(\gamma)\ld\right). 
\end{align}
\end{cor}
\begin{proof}
This follows from Theorem \ref{theorem_closethebraid} by taking cohomology and using the adjointness of stabilization
and inclusion; to get the grading right, note that $\H^l(M)\cong\H^0(M[-l])$.  
\end{proof}

\begin{rem}\label{rem_fancywords}
Theorem \ref{theorem_closethebraid} has at least two interesting interpretations. Firstly, the expression
\eqref{eq:closeisoadj} can be interpreted from the bimodule point of view as some kind of Hochschild cohomology, by
analogy with the usual formula $$\HoH\ua(M) = \ext\ua_{A\otimes A}(A,M) = \hom_{\bfD(A\otimes A)}(R,M[\ast])$$ for a
ring $A$ and an $A$-bimodule $M$. Secondly, \eqref{eq:closeisoadj} also equals the \textit{Tate-cohomology} between
$\Id\ld$ and $\B(\gamma)\ld$ in the sense of \cite[Definition 6.1.1 and Lemma 6.1.2]{Buchweitz}.

It would be interesting to study whether the usual properties of Hochschild cohomology and/or Tate-cohomology can be 
applied here to calculate Khovanov-Rozansky homology. As it stands for now, Theorem \ref{theorem_closethebraid} is
unfortunately not very useful in practice.
\end{rem}

\subsection{A detailed example: The value of the unknot}\label{subsec_ex}

Let's work out the statement of Theorem \ref{theorem_closethebraid} explicitly in the simplest case of the unknot with label
$1$. In this case, $\gamma$ is a single strand of label $1$, hence $\B(\gamma)\ld = \CC[x,y]/(x-y)$ and
$\B(\gamma)\ld\stab{x^{n+1}-y^{n+1}} = \{x-y,\pi_{xy}\}$, where $\pi_{xy} = \frac{x^{n+1}-y^{n+1}}{x-y}$. Further, we
have $\Id\ld\stab{y^{n+1}-x^{n+1}} = \{y-x,\pi_{xy}\}$, so we get (see \ref{def_internaltensor})
\begin{align*} 
\KR(\text{unknot}) & = \Id\ld\stab{y^{n+1}-x^{n+1}}\otimes_{\CC[x,y]}\B(\gamma)\ld\stab{x^{n+1}-y^{n+1}}\\
& = {\begin{tikzpicture}[baseline,description/.style={fill=white,inner sep=2pt}]
    \matrix (m) [matrix of math nodes, row sep=3em,
                 column sep=2.5em, text height=1.5ex, text depth=0.25ex,
                 inner sep=0pt, nodes={inner xsep=0.3333em, inner ysep=0.3333em}]
    {
       \CC[x,y]\langle -1\rangle\oplus\CC[x,y]\langle -1\rangle\pgfmatrixnextcell\pgfmatrixnextcell
       \CC[x,y]\oplus\CC[x,y]\langle k-1\rangle  \\  
    };
    \draw[->] ($(m-1-1.east) + (0,+0.8mm)$) -- node[above,scale=0.75] {$\begin{pmatrix}y-x & x-y \\ -\pi_{xy} &
        \pi_{xy}\end{pmatrix}$} ($(m-1-3.west) + (0,+0.8mm)$);
    \draw[->] ($(m-1-3.west) + (0,-0.8mm)$) -- node[below,scale=0.75]{$\begin{pmatrix}\pi_{xy} & y-x \\ \pi_{xy} &
        y-x\end{pmatrix}$} ($(m-1-1.east) + (0,-0.8mm)$);
\end{tikzpicture}}
\end{align*}
which we consider as a $2$-periodic complex of graded $\CC[x,y]$-modules. The cohomology at the left is isomorphic to
$\CC[x,y] / (x-y,\pi_{xy})\langle -1\rangle\cong\CC[z]/(z^n)\langle -1\rangle$, while the cohomology at the right is (up 
to shift) the middle cohomology of the Koszul-complex $\bfK\ua(y-x,\pi_{xy})$, namely
$$0\to \CC[x,y]\langle -(n+1)\rangle\xrightarrow{\scriptsize\begin{pmatrix} y-x \\ -\pi_{xy}\end{pmatrix}}
\CC[x,y]\langle -n\rangle\oplus\CC[x,y]\langle -1\rangle\xrightarrow{\scriptsize\begin{pmatrix} \pi_{xy} &
    y-x\end{pmatrix}} \CC[x,y]\to 0$$ of the 
sequence $(y-x,\pi_{xy})$. Since the latter sequence is regular, this cohomology vanishes. The reader who prefers a
direct proof could consider the following chain map:
\begin{equation*}\begin{tikzpicture}[description/.style={fill=white,inner sep=2pt}]
    \matrix (m) [matrix of math nodes, row sep=3em,
                 column sep=2.5em, text height=1.5ex, text depth=0.25ex,
                 inner sep=0pt, nodes={inner xsep=0.3333em, inner ysep=0.3333em}]
    {
      0 & \CC[x,y] && \CC[x,y]\oplus\CC[x,y] && \CC[x,y] & 0\\
      0 & 0 && \CC[x,y]/(x-y) && \CC[x,y]/(x-y) & 0\\
    };
    \draw[->] (m-1-1) -- (m-1-2);
    \draw[->] (m-1-6) -- (m-1-7);
    \draw[->] (m-1-2) -- node[above,scale=0.75]{$\begin{pmatrix} y-x \\ -\pi_{xy}\end{pmatrix}$} (m-1-4);
    \draw[->] (m-1-4) -- node[above,scale=0.75]{$\begin{pmatrix} \pi_{xy} & y-x\end{pmatrix}$} (m-1-6);

    \draw[->] (m-2-1) -- (m-2-2);
    \draw[->] (m-2-2) -- (m-2-4);
    \draw[->] (m-2-4) -- node[above,scale=0.75]{$\pi_{xy}$} (m-2-6);
    \draw[->] (m-2-6) -- (m-2-7);

    \draw[->] (m-1-2) -- (m-2-2);
    \draw[->] (m-1-4) -- node[right,scale=0.75]{$\begin{pmatrix}\can & 0\end{pmatrix}$} (m-2-4);
    \draw[->] (m-1-6) -- node[right,scale=0.75]{$\can$} (m-2-6);
\end{tikzpicture}\end{equation*}
A direct check shows that it induces an isomorphism on cohomology. Further, $\pi_{xy}$ is not a zero divisor in
$\CC[x,y] / (x-y)$ as it corresponds to $(n+1)z^n$ under the canonical isomorphism $\CC[x,y] / (x-y) \cong
\CC[z]$, and hence the cohomology is concentrated in degree $0$ as claimed. Note that this is just the usual
reduction argument in showing that the Koszul complex of a regular sequence is acyclic.

Let us now check that this fits together with Theorem \ref{theorem_closethebraid}. According to \eqref{eq:closeisohmf}
we have $$\H^l(\KR(\ol{\gamma}))\ \cong\ \HMF(\{x-y,\pi_{xy}\}\langle 1\rangle[l-1],\{x-y,\pi_{xy}\}).$$ As
$\HMF(\{x,y\})\cong S\ld/(x,y)$ and ${\HMF}(\{x,y\},\{x,y\}[1])=0$ 
for $(x,y)$ regular, we conclude that $\H^0(\KR(\ol{\gamma})) = 0$ and
\begin{align}\label{eq:unknotfirstcohom}\H^1(\KR(\ol{\gamma}))\
\cong\ \CC[x,y] / (\pi_{xy},x-y)\langle -1\rangle\ \cong\ \CC[z]/(z^n)\langle -1\rangle,\end{align}
 in agreement with our explicit calculation above.

Finally, let us also check \eqref{eq:closeisoadj}. The module $\CC[x,y]/(x-y)$ is already maximal Cohen-Macaulay over
$R\ld := \CC[x,y]/(x^{n+1}-y^{n+1})$ as it possesses the $2$-periodic resolution
$$...\to R\ld\langle -(n+2)\rangle \xrightarrow{x-y} R\ld\langle -(n+1)\rangle \xrightarrow{\pi_{xy}}R\ld\langle
-1\rangle \xrightarrow{x-y}R\ld\to\CC[x,y] / (x-y)\to 0.$$ 
From this it is also clear that $\CC[x,y]/(x-y)[-1] = \CC[x,y]/(\pi_{xy})\langle -1\rangle$. By
\eqref{eq:closeisoadj} we therefore have 
$$\H^0(\KR(\ol{\gamma}))\cong\ul{\hom}_{R\ld}(\CC[x,y]/(\pi_{xy}),\CC[x,y]/(x-y))$$
and $$\H^1(\KR(\ol{\gamma}))\cong\ul{\End}_{R\ld}(\CC[x,y]/(x-y))\langle -1\rangle.$$  
As even $\hom_{\CC[x,y]}(\CC[x,y]/(\pi_{xy}),\CC[x,y]/(x-y))=0$ ($\pi_{xy}$ is not a zero divisor in\break$\CC[x,y]/(x-y)$),
we get $\H^0(\KR(\ol{\gamma}))=0$ as before. For the first cohomology, note that
$\End_{\CC[x,y]}(\CC[x,y]/(x-y))\cong\CC[x,y]/(x-y)$, and that the multiplication with a 
polynomial $p\in\CC[x,y]$ on $\CC[x,y]/(x-y)$ is stably trivial if and only if $p\in(\pi_{xy})$. Consequently, we get
$$\ul{\End}_{R\ld}(\CC[x,y]/(x-y))\cong\CC[x,y]/(\pi_{xy},x-y)$$ and therefore once again \eqref{eq:unknotfirstcohom}.

%%%%%%%%%%%%%%%%% ANHANG %%%%%%%%%%%%%%%%%%%

\newpage

\begin{appendix}
\rhead[Appendix]{\thepage}       \rfoot{} 

\theoremstyle      {plain}
\theoremheaderfont {\normalfont\bfseries}
\theorembodyfont   {\normalfont}
\theoremseparator  {}
\theoremindent0cm
\theoremnumbering  {arabic}
\theoremsymbol {}  

\renewtheorem{prop}{Proposition}[subsection]
\renewcommand{\thesubsection}{\Alph{subsection}}
\numberwithin{equation}{subsection}
\renewcommand{\theequation}{\thesubsection-\arabic{equation}}
\appendixpage
\addappheadtotoc

\subsection{Outline}\label{app_outline}
\markright{\ref{app_outline} Outline}

In this appendix we study in more detail the tensor product of matrix factorizations and attempt to extend
Corollary \ref{cor_concatenateregularsequences} and Proposition \ref{prop_tensorstabgeneral} to more general and -- most
importantly -- more natural statements about the compatibility of tensor products and stabilization. 

The rough outline is as follows. The reader might have observed that whenever we applied the stabilization functor to a
module $M\ld$ we first replaced $M\ld$ by an $S\ld$-free resolution of $M\ld$ together some 'enrichment' in the form of
higher nullhomotopies for the multiplication by $w$ on $S\ld$. In other words, the object we really worked with was the
$S\ld$-free resolution of $M\ld$ instead of the module $M\ld$ itself. Further, Proposition 
\ref{prop_tensorstabgeneral} only worked because the condition $\tor_S^k(M\ld,N\ld)=0$ for $k>0$ ensured that taking
$S\ld$-free resolutions respected tensor products of modules and derived tensor products of complexes, respectively, so
that one could forget about tensor products of $S\ld/(w)$-modules and instead work with tensor products of complexes of
$S\ld$-modules. These examples naturally lead us to the impression that we shouldn't work with $S\ld/(w)$-modules but
instead with enriched complexes of free $S\ld$-modules. 

We will see that for an arbitrary $w\in S_d$ the category $\HMF(S\ld,w)$ can be described as 
the singularity category $\D^b_\fg(K\ua_w)/\perf$ of the Koszul dg-$S\ld$-algebra $$K\ua_w\ :=\ ...\to 0\to 
S\ld\langle -d\rangle\xrightarrow{\ \ w\ \ }S\ld\to 0\to ...$$ concentrated in degrees $-1$ and $0$ (Theorem
\ref{thm_bigtheorem}). Modules over $K\ua_w$ are the same as complexes of $S\ld$-modules with the extra datum of a
nullhomotopy $s$ for the multiplication with $w$ satisfying $s^2=0$, so they are precisely the enriched complexes of
$S\ld$-modules we were looking for. Moreover, we will see that the derived tensor product for $K\ua_w$-modules is
compatible with the tensor product of matrix factorizations, i.e. that
\begin{equation*}\begin{tikzpicture}[description/.style={fill=white,inner sep=2pt}]
    \matrix (m) [matrix of math nodes, row sep=2.5em,
                 column sep=2.5em, text height=1.5ex, text depth=0.25ex,
                 inner sep=0pt, nodes={inner xsep=0.3333em, inner ysep=0.3333em}]
    {
       \D^b_\fg(K\ua_w)\times\D^b_\fg(K\ua_{w\p}) && \D^b_\fg(K\ua_{w+w\p})\\
       \D^b_\fg(K\ua_w)/\perf\times\D^b_\fg(K\ua_{w\p})/\perf && \D^b_\fg(K\ua_{w+w\p})/\perf\\
       \HMF(S\ld,w)\times\HMF(S\ld,w\p) && \HMF(S\ld,w+w\p)\\
    };
    \draw[->] (m-1-1) -- node[above,scale=0.75]{$-\stackrel{\LL}{\otimes}_{S\ld}-$} (m-1-3);
    \draw[->] (m-3-1) -- node[above,scale=0.75]{$-\otimes_{S\ld}-$} (m-3-3);
    \draw[->] (m-1-1) -- node[scale=0.75,left]{$\can$}(m-2-1);
    \draw[->] (m-1-3) -- node[scale=0.75,right]{$\can$}(m-2-3);
    \draw[->] (m-2-1) -- node[scale=0.75,left]{$\cong$}(m-3-1);
    \draw[->] (m-2-3) -- node[scale=0.75,right]{$\cong$}(m-3-3);
\end{tikzpicture}\end{equation*}
commutes (Theorem \ref{thm_compatible}). The question about the compatibility of the stabilization functor and tensor
products of matrix factorizations is therefore only a question about when the derived tensor product of two
$S\ld/(w)$-modules, considered as $K\ua_w$-modules, can be computed naively, and this immediately yields Proposition
\ref{prop_tensorstabgeneral}.

Finally, we will describe in Appendix \ref{sec_koszulduality} how the duality $(-)\uc$ on $\HMF(S\ld,w)$ looks like for 
$K\ua_w$-modules, in particular generalizing Proposition \ref{prop_wdual}.

Another nice feature of our new description of $\HMF(S\ld,w)$ is that it extends to the case $w=0$. There, the
classical identification $\HMF(S\ld,w)\cong\D^b_\fg(S\ld/(w))/\perf$ heavily breaks down, simply because $S\ld/(w) = S\ld$ is
regular in this case. However, the equivalence $\HMF(S\ld,w)\cong\D^b_\fg(K\ua_w)/\perf$ still holds in this case. In
view of possible applications to Khovanov-Rozansky homology, this seems reasonable, as the case of vanishing potential
is by no means forbidden there; on the contrary, it occurs each time a link gets closed. Moreover, the construction of
Khovanov-Rozansky is local -- if a knot gets closed it doesn't matter if there are some open link components somewhere
around or not -- while the condition that the potential vanishes is not, so it shouldn't play a role.

\subsection{The derived category of modules over a dg-algebra}\label{app_derivedcat}
\markright{\ref{app_derivedcat} The derived category of modules over a dg-algebra}

First we recall the definition of a dg-algebras and dg-modules over graded rings.

\begin{definition}
Let $S\ld$ be a graded ring.
\begin{enumerate}
\item A \textit{dg-$S\ld$-algebra} is an algebra object in the monoidal category of complexes of graded $S\ld$-modules. In
other words, a dg-$S\ld$-algebra is a complex of graded $S\ld$-modules $(A\ua\ld,\d)$ together with morphisms of
complexes of graded $S\ld$-modules
$$A\ua\ld\otimes_{S\ld} A\ua\ld\ \stackrel{\mu}{\longrightarrow}\ A\ua\ld\quad\text{and}\quad S\ld\
\stackrel{\eta}{\longrightarrow}\ A\ua\ld,$$ 
(here we consider $S\ld$ as a complex concentrated in degree $0$) satisfying the associativity and unit axiom. We will
usually abbreviate $(A\ua\ld,\d,\mu,\eta)$ as $A\ua\ld$ and write $ab := \mu(a\otimes b)$ for $a,b\in A\ua\ld$. We call 
$A\ua\ld$ \textit{graded-commutative} if $ab = (-1)^{|a|\cdot |b|} ba$ for all (cohomologically) homogeneous $a,b\in
A\ua\ld$. It is called \textit{connected} if $A^k\ld=0$ for $k>0$ and if the structure morphism $\eta: S\ld\to A^0\ld$
is an isomorphism.
\item Let $A\ua\ld$ be a dg-$S\ld$-algebra. A left \textit{$A\ua\ld$-module} is a module object over the algebra object
  $A\ua\ld$ in the monoidal category of complexes of graded $S\ld$-modules. In other words, it consists of a complex of
  graded $S\ld$-modules $(M\ua\ld,\d)$ together with a morphism $A\ua\ld\otimes_{S\ld}
  M\ua\ld\stackrel{\rho}{\longrightarrow} M\ua\ld$ of complexes of graded $S\ld$-modules satisfying the
  associativity and unit axiom. We will usually abbreviate $(M\ua\ld,\d,\rho)$ as $M\ua\ld$ and write $a.m :=
  \rho(a\otimes m)$ for $a\in A\ua\ld$ and $m\in M\ua\ld$. Right $A\ua\ld$-modules are defined similarly. If we
  say ``$A\ua\ld$-module'' we will always mean left a $A\ua\ld$-module. If $A\ua\ld$ is graded commutative, any left
  $A\ua\ld$-module structure on a complex of graded $S\ld$-modules $M\ua\ld$ yields a right $A\ua\ld$-module
  structure via $m.a := (-1)^{|m|\cdot |a|} a.m$ for cohomologically homogeneous $a\in A\ua\ld$ and $m\in M\ua\ld$. 
\end{enumerate}
\end{definition}

Now we turn to the definition of morphisms of $A\ua\ld$-modules.

\begin{definition}\label{def_homomorphismcomplex}
Let $M\ua\ld$ and $N\ua\ld$ be two $A\ua\ld$-modules. 
\begin{enumerate}
\item The \textit{homomorphism complex}
$$\left(\hom_{A\ua\ld}(M\ua\ld,N\ua\ld)\ua,\partial\right)\ \subset\ \hom_{S\ld}(M\ua\ld,N\ua\ld)\ua$$ is the subcomplex
of $\hom_{S\ld}(M\ua\ld,N\ua\ld)\ua$ such that $f = (f^k)_{k\in\ZZ}\in\hom_{S\ua\ld}(M\ua\ld,N\ua\ld)^n$ is in
$\hom_{A\ua\ld}(M\ua\ld,N\ua\ld)^n$ if and only if for each $a\in A^r\ld$ we have $a\circ f^k = (-1)^{nr} f^{k+r}\circ
a$ (here we identify $a$ with its action on $M\ua\ld$ and $N\ua\ld$). Note that $\hom_{S\ld}(M\ua\ld,N\ua\ld)\ua$ and
$\hom_{A\ua\ld}(M\ua\ld,N\ua\ld)$ are only complexes of abelian groups. 
\item A \textit{homomorphism of $A\ua\ld$-modules} $f: M\ua\ld\to N\ua\ld$ is a $0$-cocycle in
$\hom_{A\ua\ld}(M\ua\ld,N\ua\ld)\ua$. In other words, $f$ consists of a family of (internal) degree preserving maps $f^k:
M^k\ld\to N^k\ld$ such that $a\circ f^k = f^{k+r}\circ a$ for all $a\in A^r\ld$. The category of $A\ua\ld$-modules
is denoted $A\ua\ld\Mod$. 
\item Two homomorphisms of $A\ua\ld$-modules $f,g: M\ua\ld\to N\ua\ld$ are called \textit{homotopic} if $f-g$ is a
  $0$-coboundary in $\hom_{A\ua\ld}(M\ua\ld,N\ua\ld)\ua$. This amounts to the existence of a family of internal degree
  preserving maps $D^k: M^k\ld\to N^{k-1}\ld$ such that $a\circ D^k = (-1)^r D^{k+r}\circ a$ for all $a\in A^r\ld$
and $$f^k - g^k\ =\ \differential^{k-1}_{N\ua\ld}\circ D^k + D^{k+1}\circ\differential^k_{M\ua\ld}$$ for all $k\in\ZZ$. The \textit{homotopy
  category} 
  of $A\ua\ld$-modules is denoted $\Ho(A\ua\ld)$. 
\item For later use in Section \ref{sec_koszulduality} we define $\hom_{A\ua\ld}(M\ua\ld,N\ua\ld)\ua_k$ as
  $\hom_{A\ua\ld}(M\ua\ld,N\ua\ld\langle k\rangle)\ua$. We denote $\hom_{A\ua\ld}(M\ua\ld,N\ua\ld)\ua\ld :=
  \bigoplus\limits_{k\in\ZZ}\hom_{A\ua\ld}(M\ua\ld,N\ua\ld)\ua_k$. This is naturally a complex of graded
  $S\ld$-modules. Moreover, if $N\ua\ld$ is an $A\ua\ld$-bimodule, $\hom_{A\ua\ld}(M\ua\ld,N\ua\ld)\ua\ld$ carries a
  natural structure of a right $A\ua\ld$-module; similarly, if $M\ua\ld$ is an $A\ua\ld$-bimodule,
  $\hom_{A\ua\ld}(M\ua\ld,N\ua\ld)\ua\ld$ carries a natural structure of a (left) $A\ua\ld$-module. For example, the
  natural $A\ua\ld$-bimodule structure on $A\ua\ld$ yields a right $A\ua\ld$-module structure on
  $\hom_{A\ua\ld}(M\ua\ld,A\ua\ld)\ua\ld$, and for $A\ua\ld$ graded commutative we can (and will!) regard this as a left
  $A\ua\ld$-module. Explicitly, the left action of $A\ua\ld$ on $\hom_{A\ua\ld}(M\ua\ld,A\ua\ld)\ua\ld$ is given by $a.f
  := (-1)^{|a|\cdot n} f\circ a$ for (cohomologically) homogeneous $a\in A\ua\ld$ and
  $f\in\hom_{A\ua\ld}(M\ua\ld,A\ua\ld)^n\ld$. 
\end{enumerate}
\end{definition}
\begin{fact}
The dg-category of $A\ua\ld$-modules is pretriangulated (see Remark \ref{rem_pretriag}). In particular, the homotopy
category $\Ho(A\ua\ld)$ is naturally triangulated.
\end{fact}
\begin{proof}
This follows from the fact that the cone of a morphism of $A\ua\ld$-modules and the shift of a $A\ua\ld$-module
both carry natural structures of $A\ua\ld$-modules.
\end{proof}

\begin{definition}
Denote by $\Acyc(A\ua\ld)$ the class of acyclic $A\ua\ld$-modules. A homomorphism of $A\ua\ld$-modules $f:
M\ua\ld\to N\ua\ld$ is called a \textit{quasi-isomorphism} if it induces an isomorphism in cohomology, i.e. if
$\cone(f)\in\Acyc(A\ua\ld)$. The \textit{derived category} of $A\ua\ld$, denoted $\D(A\ua\ld)$, is defined as the Verdier
quotient $\Ho(A\ua\ld)/\Acyc(A\ua\ld)$.  
\end{definition}

We will mainly work with the description of $\D(A\ua\ld)$ as $\Ho(A\ua\ld)/\Acyc(A\ua\ld)$. However, to understand how
one can associate a derived base extension functor $\D(A\ua\ld)\to\D(B\ua\ld)$ to a morphism of dg-$S\ld$-algebras
$A\ua\ld\to B\ua\ld$ (what is the correct replacement for projective complexes in the classical derived category of
a ring?), it is more natural to view $\D(A\ua\ld)$ as the homotopy category of a particular model category
structure on $A\ua\ld\Mod$. For introductions to the theory model categories we refer see e.g. \cite{DwyerSpalinski},
\cite{GoerssSchemmerhorn}, \cite{Hovey}.

\begin{definition}\label{def_modelstructuredgmodules}
A morphism $f: M\ua\ld\to N\ua\ld$ of $A\ua\ld$-modules is called a 
\begin{enumerate}
\item \textit{weak equivalence}, if it is a quasi-isomorphism,
\item \textit{fibration}, if it is surjective,
\item \textit{cofibration}, if for each diagram 
\begin{equation*}\begin{tikzpicture}[description/.style={fill=white,inner sep=2pt}]
    \matrix (m) [matrix of math nodes, row sep=3em,
                 column sep=2.5em, text height=1.5ex, text depth=0.25ex,
                 inner sep=0pt, nodes={inner xsep=0.3333em, inner ysep=0.3333em}]
    {
       M\ua\ld & P\ua\ld \\ N\ua\ld & Q\ua\ld\\
    };
    \draw[->]
    (m-1-1) edge node[auto,scale=0.75]{$\alpha$} (m-1-2)
    (m-1-1) edge node[left,scale=0.75]{$f$}      (m-2-1)
    (m-1-2) edge node[auto,scale=0.75]{$g$}      (m-2-2)
    (m-2-1) edge[dashed] (m-1-2)
    (m-2-1) edge node[auto,scale=0.75]{$\beta$}  (m-2-2); 
\end{tikzpicture}\end{equation*}
with $g$ a trivial fibration (i.e. surjective quasi-isomorphism) a dotted map making the whole diagram commute exists. A
$A\ua\ld$-module $M\ua\ld$ is called \textit{cofibrant} (resp. \textit{fibrant}) if the canonical map $0\to M\ua\ld$
(resp. $M\ld\to 0$) is a cofibration (resp. fibration). The class of cofibrant objects is denoted $\Cof(A\ua\ld)$. Any
module is fibrant.
\end{enumerate}
\end{definition}

The reader who is not familiar with model categories shouldn't panic; for our purposes it is perfectly sufficient to
think of the cofibrant $A\ua\ld$-modules as generalizations of complexes of projectives in the classical homological
algebra over a ring. Evidence for this will be given soon. 

\begin{definition}
Let $A\ua\ld$ be a dg-algebra. For $n,k\in\ZZ$, define $D(n,k)\ua\ld\in A\ua\ld\Mod$ as $$D(n,k)\ua\ld\ :=\ e_n A\ua\ld
[n]\langle k\rangle\oplus e_{n-1} A\ua\ld[n-1]\langle k\rangle$$ with differential given by $\d(e_n) :=
\d(e_{n-1})$. Further, put $S(n,k)\ua\ld := A\ua\ld[n]\langle k\rangle$ and denote $\iota_{n,k}:
S(n,k)\ua\ld\hookrightarrow D(n,k)\ua\ld$ the canonical inclusion.
\end{definition}

\begin{theorem}\label{thm_modelstructure}
With the classes of weak equivalences, fibrations and cofibrations the category of $A\ua\ld$-modules is a cofibrantly
generated model category whose homotopy category is equivalent to the derived category $\D(A\ua\ld)$. The set $\calI :=
\{0\to S(n,k)\ |\ n,k\in\ZZ\}$ is a generating set for the cofibrations in $A\ua\ld\Mod$, and the set $\calJ :=
\{S(n-1,k)\xrightarrow{\ \iota_{n,k}\ }D(n,l)\ |\ k,l\in\ZZ\}$ is a generating set for the acyclic cofibrations.
\end{theorem}
\begin{proof}
It is shown in \cite[Theorem 2.3.11]{Hovey} that the category $\Ch(R)$ of (unbounded) cochain complexes over a ring $R$
equipped with quasi-isomorphism as weak equivalences and degree-wise surjections as fibrations is a model
category. Moreover, it is shown there that $\Ch(R)$ is cofibrantly generated, where the set $\{S(n-1)\hookrightarrow
D(n)\ |\ n\in\ZZ\}$ is a generating set for the cofibrations, and the set $\{0\to D(n)\ |\ n\in\ZZ\}$ is a generating
set of acyclic cofibrations. Here, $S(n) := \ZZ[n]$, and $$D(n)\ :=\ ... \to 0\to \ZZ\xrightarrow{\ 1\ }\ZZ\to 0\to ...$$
concentrated in degrees $-n$ and $-n+1$. In particular, this applies to $\Ch(\ZZ)$, and we provide $\Ch(\ZZ)^{\ZZ}$
with the product model structure.

The model structure on $\Ch(\ZZ)^\ZZ$ is cofibrantly generated, and generating sets for the cofibrations and acyclic
cofibrations can be described as follows: For $n,k\in\ZZ$ define $D(n,k)\in\Ch(\ZZ)^\ZZ$ by 
$D(n,k)^k := D(n)$ and $D(n,k)^l := 0$ for $l\neq k$. Here, $(-)^l$ denotes the $l$-th component of an object in
$\Ch(\ZZ)^\ZZ$. Similarly, define $S(n,k)\in\Ch(\ZZ)^\ZZ$ by $S(n,k)^k := S(n)$ and $S(n,k)^l := 0$ for $k\neq l$. Then
the sets $\calI_{\Ch(\ZZ)^\ZZ} := \{S(n-1,k)\hookrightarrow D(n,k)\ |\ n,k\in\ZZ\}$ and $\calJ_{\Ch(\ZZ)^\ZZ} :=  \{0\to 
D(n,k)\ |\ n,k\in\ZZ\}$ generate the cofibrations and acyclic cofibrations in $\Ch(\ZZ)^\ZZ$, respectively. 

This is enough preparation. To get the desired model structure on $A\ua\ld\Mod$, we use \cite[Theorem
3.6]{GoerssSchemmerhorn} to pull back the model structure on $\Ch(\ZZ)^\ZZ$ along the adjunctions 
\begin{equation*}\begin{tikzpicture}[description/.style={fill=white,inner sep=2pt}]
    \matrix (m) [matrix of math nodes, row sep=3em,
                 column sep=5em, text height=1.5ex, text depth=0.25ex,
                 inner sep=0pt, nodes={inner xsep=0.3333em, inner ysep=0.3333em}]
    {
      A\ua\ld\Mod & \Ch(S\ld) & \Ch(\ZZ)^{\ZZ}\\
    };
    \draw[->] ($(m-1-1.east) + (0,+0.5mm)$) -- node[above,scale=0.75]{forget} ($(m-1-2.west) + (0,+0.5mm)$);
    \draw[->] ($(m-1-2.east) + (0,+0.5mm)$) -- node[above,scale=0.75]{forget} ($(m-1-3.west) + (0,+0.5mm)$);

    \draw[->] ($(m-1-2.west) + (0,-0.5mm)$) -- node[below,scale=0.75]{$-\otimes_{S\ld} A\ua\ld$} ($(m-1-1.east) +
  (0,-0.5mm)$); 
    \draw[->] ($(m-1-3.west) + (0,-0.5mm)$) -- node[below,scale=0.75]{$-\otimes_{\ZZ} S\ld$} ($(m-1-2.east) +
  (0,-0.5mm)$); 
\end{tikzpicture}\end{equation*}
This yields a model structure on $A\ua\ld\Mod$ with the correct classes of weak equivalences,
cofibrations and fibrations. Further, \cite[Theorem 3.6]{GoerssSchemmerhorn} states that this model structure is
cofibrantly generated, and that the images of $\calI_{\Ch(\ZZ)^\ZZ}$ and $\calJ_{\Ch(\ZZ)^\ZZ}$ along the left adjoints
$-\otimes_{S\ld} A\ua\ld$ and $-\otimes_{\ZZ} S\ld$ form generating sets for the cofibrations and acyclic cofibrations,
respectively. It is clear that these images are precisely the sets $\calI$ and $\calJ$ from above, and so we're done.
\end{proof}

The crucial point is that Theorem \ref{thm_modelstructure} implies that the canonical functor from the homotopy category
$\Ho(\Cof(A\ua\ld))$ of cofibrant $A\ua\ld$-modules to $\D(A\ua\ld)$ is an equivalence of categories (see
\cite[Theorem 1.2.10]{Hovey}); this is analogous to classical equivalences like $\D^+(R)\cong\Ho^+(\pro(R))$ for a ring
$R$. It allows to define (left) derived functors by first replacing an arbitrary $A\ua\ld$-module by a
quasi-isomorphic, cofibrant $A\ua\ld$-module and then applying the functor which is to be derived; in analogy to
the classical situation where one has to take projective resolutions to compute derived tensor products for complexes
over a ring, for example. More generally, we have the following recipe for deriving a pair of adjoint functors between
two model categories.

\begin{definition}[see {\cite[Definition 1.3.1]{Hovey}}]
Let $\calC$, $\calD$ be model categories. An adjunction $\FF:\calC\leftrightarrows\calD:\GG$ (with $\FF$ left adjoint to
$\GG$) is called a \textit{Quillen adjunction} if the following equivalent conditions are satisfied:
\begin{enumerate}
\item $\FF$ preserves cofibrations and trivial cofibrations.
\item $\GG$ preserves fibrations and trivial fibrations.
\end{enumerate}
In this case, define the \textit{total left derived functor} $\bfL\FF$ as the composition
$$\bfL\FF:\ \ \Ho(\calC)\ \cong\ \Ho(\calC_{c})\ \xrightarrow{\FF}\ \Ho(\calD)$$
and the \textit{total right derived functor} $\bfR\GG$ as the composition
$$\bfR\GG:\ \ \Ho(\calD)\ \cong\ \Ho(\calD_{f})\ \xrightarrow{\GG}\ \Ho(\calC)$$
Here $\calC_{c}$ (resp. $\calD_{f}$) denotes the subcategory of $\calC$ (resp. $\calD$) consisting of cofibrant
(resp. fibrant) objects.
\end{definition}

Explicitly, $\bfL\FF$ can be described as follows: First, we choose for each $X\in\calC$ a \textit{cofibrant
  replacement}, i.e. a trivial fibration $q_X: QX\to X$ such that $QX$ is cofibrant. Sending $X$ to $QX$ extends to a
functor $\Ho(\calC)\to\Ho(\calC_c)$ which is quasi-inverse to the canonical functor $\Ho(\calC_c)\to\Ho(\calC)$, and
thus we have $\bfL\FF X\cong \FF QX$. This is precisely the recipe we sketched above: in order to calculate a left derived
functor on an object, we first have to replace it by some weakly equivalent cofibrant object, and then we
can apply the functor naively; in analogy to the calculation of, say, $-\stackrel{\LL}{\otimes}_R -$ in the derived
bounded above category $\D^+(R)$ of a commutative ring $R$ through projective resolutions.

Analogously, the right derived $\bfR\GG$ can be described as follows: We choose for each $Y\in\calD$ a \textit{fibrant
  resolution}, i.e. a trivial cofibration $r_X: X\to RX$ such that $RX$ is fibrant. Then, $X\mapsto RX$ extends to a
quasi-inverse $\Ho(\calD)\to\Ho(\calD_f)$ to the canonical functor $\Ho(\calD_f)\to\Ho(\calD)$, and hence we get
$\bfR\GG X\cong \GG RX$ for $X\in\calD$. 

Different choices of $Q$ and $R$ yield canonically isomorphic derived functors. In the following, we will fix some
particular choice for $Q$ and $R$. 

\begin{fact}[see {\cite[Lemma 1.3.10]{Hovey}}]
Let $\calC$, $\calD$ be model categories and let $\FF:\calC\leftrightarrows\calD:\GG$ be a Quillen adjunction with unit
$\veps: \id_{\calC}\to\GG\FF$ and counit $\eta: \FF\GG\to\id_{\calD}$. Then there is a \textit{derived adjunction}
\begin{align}\label{eq:quillenadjderiv}\bfL\FF:\Ho(\calC)\leftrightarrows\Ho(\calD):\bfR\GG.\end{align}
For cofibrant $X\in\calC$, its unit is given as the 
composition
\begin{align}\label{eq:quillenequivcof} X\xrightarrow{\ \veps_X\ }\GG\FF X\xrightarrow{\ \GG r_{\FF X}\ } \GG R\FF X\ =\
  (\bfR\GG\circ\bfL\FF)(X),\end{align}
 and for fibrant $Y\in\calD$, its counit is given as
\begin{align}\label{eq:quillenequivfib}(\bfL\FF\circ\bfR\GG)(Y)\ =\ \FF Q \GG Y\ \xrightarrow{\ \FF q_{\GG Y}\ }
 \FF\GG Y\xrightarrow{\ \eta_{Y}\ } Y.\end{align}
\end{fact}

\begin{definition}[see {\cite[Definition 1.3.12 and Proposition 1.3.13]{Hovey}}]
The adjunction $\FF:\calC\leftrightarrows\calD:\GG$ is called a \textit{Quillen equivalence} if it is a Quillen
adjunction and, in addition, if for all cofibrant $X\in\calC$ and all fibrant $Y\in\calD$ the morphisms
\eqref{eq:quillenequivcof} and \eqref{eq:quillenequivfib} are weak equivalences. 
\end{definition}
\begin{fact}
If $\FF:\calC\leftrightarrows\calD:\GG$ is a Quillen equivalence, then the derived adjunction \eqref{eq:quillenadjderiv}
is an adjoint equivalence of categories.
\end{fact}
\begin{proof}
By assumption, the morphisms \eqref{eq:quillenequivcof} and \eqref{eq:quillenequivfib} are isomorphisms in $\Ho(\calC)$
and $\Ho(\calD)$, respectively. On the other hand, they are unit and counit, respectively, of the derived adjunction
\eqref{eq:quillenadjderiv}, and the claim follows.
\end{proof}

\begin{fact}[see {\cite[Corollary 1.3.16]{Hovey}}]\label{fact_simplerdef}
Let $\FF:\calC\leftrightarrows\calD:\GG$ be a Quillen adjunction such that the following hold:
\begin{enumerate}
\item For all cofibrant $X\in\calC$ the morphism \eqref{eq:quillenequivcof} is a weak equivalence.
\item If $Y\mor{f} Y\p$ is a morphism of fibrant objects in $\calD$ such that $\GG f$ is a weak equivalence, then $f$ is
  a weak equivalence.
\end{enumerate}
Then $\FF\dashv \GG$ is a Quillen equivalence.
\end{fact}

We now return to the the model category of modules over a given dg-algebra described in
\ref{def_modelstructuredgmodules}. There, the cofibrant objects can be described explicitly as follows:

\begin{definition}
A $A\ua\ld$-module $M\ua\ld$ is called \textit{free} if it is isomorphic to a sum of shifted copies of $A\ua\ld$. It
is called \textit{semi-free} if it possesses a filtration $$0\ =\ ^0 M\ua\ld\ \subset\ ^1 M\ua\ld\ \subset\ ^2 M\ua\ld\
\subset\ ...$$ such that each filtration quotient $^{n+1} M\ua\ld / ^n M\ua\ld$ is a free $A\ua\ld$-module.
\end{definition}

\begin{fact}\label{fact_semifreecofibrant}
The following hold:
\begin{enumerate}
\item Any semi-free $A\ua\ld$-module is cofibrant. 
\item Any cofibrant module is a summand of a semi-free module.
\end{enumerate}
\end{fact}
\begin{proof}
This is a formal consequence of Theorem \ref{thm_modelstructure}. Using the notation of \cite[Section 2.1.2]{Hovey} and
the definition of $\calI$ from Theorem \ref{thm_modelstructure}, any semi-free $A\ua\ld$-module lies in $\calI\cell$
(see \cite[Lemma 2.1.13]{Hovey}), hence is cofibrant (\cite[Lemma 2.1.10]{Hovey}). On the other hand, the small object
argument and the finiteness of the objects involved in $\calI$ imply that any cofibrant object is a retract, hence a
summand, of a semi-free $A\ua\ld$-module (see \cite[Theorem 2.1.14]{Hovey} for the general statement, or \cite[Theorem
3.5]{GoerssSchemmerhorn} for the case of a finitely generated model category).
\end{proof}

As an example of a Quillen adjunction, we look at the base change adjunction associated to a morphism of dg-algebras. 

\begin{prop}\label{prop_dgadjointequivalence}
Let $\varphi: A\ua\ld\to B\ua\ld$ be a morphism of graded dg-$S\ld$-algebras. Then $\varphi$ defines a Quillen
adjunction  
\begin{equation}\label{eq:quilladj}\begin{tikzpicture}[description/.style={fill=white,inner sep=2pt},
    baseline=(m-1-1.base)] 
    \matrix (m) [matrix of math nodes, row sep=3em,
                 column sep=10em, text height=1.5ex, text depth=0.25ex,
                 inner sep=0pt, nodes={inner xsep=0.3333em, inner ysep=0.3333em}]
    {
       A\ua\ld\Mod & B\ua\ld\Mod\\
    };
    \draw[->] ($(m-1-1.east) + (0,+0.5mm)$) -- node[above,scale=0.75]{$\varphi\ua := -\otimes_{A\ua\ld} B\ua\ld$}
    ($(m-1-2.west) + (0,+0.5mm)$); 
    \draw[->] ($(m-1-2.west) + (0,-0.5mm)$) -- node[below,scale=0.75]{$\varphi\la$}
    ($(m-1-1.east) + (0,-0.5mm)$); 
\end{tikzpicture}\end{equation}
with induced derived adjunction
\begin{equation}\label{eq:adjderived}\begin{tikzpicture}[description/.style={fill=white,inner
      sep=2pt},baseline=(m-1-1.base)]  
    \matrix (m) [matrix of math nodes, row sep=3em,
                 column sep=10em, text height=1.5ex, text depth=0.25ex,
                 inner sep=0pt, nodes={inner xsep=0.3333em, inner ysep=0.3333em}]
    {
       \D(A\ua\ld) & \D(B\ua\ld)\\
    };
    \draw[->] ($(m-1-1.east) + (0,+0.5mm)$) -- node[above,scale=0.75]{$-\stackrel{\LL}{\otimes}_{A\ua\ld} B\ua\ld$}
    ($(m-1-2.west) + (0,+0.5mm)$); 
    \draw[->] ($(m-1-2.west) + (0,-0.5mm)$) -- node[below,scale=0.75]{$\varphi\la$}
    ($(m-1-1.east) + (0,-0.5mm)$); 
\end{tikzpicture}\end{equation}
on the derived categories. The adjunction \eqref{eq:quilladj} is a Quillen equivalence if and only if $\varphi$ is a
quasi-isomorphism. In this case, \eqref{eq:adjderived} is an adjoint equivalence of categories.\end{prop}  
\begin{proof}
The existence of the adjunction \eqref{eq:quilladj} is clear. To see that it is a Quillen adjunction, it suffices to see
that the forgetful functor preserves fibrations and trivial fibrations, which is obvious. 

Next, assume that \eqref{eq:quilladj} is a Quillen equivalence. Then, taking $X\ua\ld := A\ua\ld$ in
\eqref{eq:quillenequivcof} yields $\varphi: A\ua\ld\to B\ua\ld$, and hence $\varphi$ is a quasi-isomorphism. Conversely,
assume that $\varphi$ is a quasi-isomorphism. We have to show that \eqref{eq:quillenequivcof} is a weak equivalence for
all cofibrant modules $X\ua\ld$. By Fact \ref{fact_semifreecofibrant} we may assume that $X\ua\ld$ is a semi-free
$A\ua\ld$-module. Then \eqref{eq:quillenequivcof} is given by the canonical morphism of $A\ua\ld$-modules $$X\ua\ld\ \cong\
X\ua\ld\otimes_{A\ua\ld} A\ua\ld\ \xrightarrow{\ \id\otimes\varphi\ }\ X\ua\ld\otimes_{A\ua\ld} B\ua\ld.$$ The cone of
this morphism is isomorphic to $X\ua\ld\otimes_{A\ua\ld} \cone(\varphi)$. As $X\ua\ld$ is semi-free, the complex of
$S\ld$-modules underlying $X\ua\ld\otimes_{A\ua\ld}\cone(\varphi)$ has a bounded below increasing filtration by iterated
cones of sums of shifted copies of $\cone(\varphi)$, hence is acyclic. As $\varphi\ld$ clearly reflects
quasi-isomorphisms, Fact \ref{fact_simplerdef} yields that \eqref{eq:quilladj} indeed is a Quillen equivalence.
\end{proof}

As an application, we present the proof given in \cite[Proposition 2.2.2]{Avramov} of the following fact which is
crucial in section \ref{sec_badbimodule}.
\begin{prop}\label{prop_stayregularappendix}
Let $S\ld$ be a local graded ring with maximal homogeneous ideal $\frm$, and let $w\in\frm\setminus\frm^2$ be
homogeneous. Then the following hold:
\begin{enumerate}
\item\label{item:crucial1} If $M\ld$ is a finitely generated graded $S\ld/(w)$-module with minimal $S\ld$-free resolution
$F\ua\ld\to M\ld$, then $F\ua\ld$ admits the structure of a semi-free $K\ua_w$-module. 
\item\label{item:crucial2} If $S\ld$ is regular, then so is $S\ld/(w)$.
\end{enumerate}
\end{prop}
\begin{lem}\label{lem_splitinjection}
Let $S\ld$ be a local graded ring with maximal homogeneous ideal $\frm$ and suppose $f: P\ld\to Q\ld$ is a homomorphism
of finitely generated projective $S\ld$-modules. Further, assume that $$f\otimes_{S\ld} S\ld/\frm:\ 
P\ld / \frm P\ld\ \longrightarrow\ Q\ld / \frm Q\ld$$ is injective. Then $f$ is a split injection.
\end{lem}
\begin{proof}[of Lemma \ref{lem_splitinjection}]
Set $K\ld := \coker(f)$. Assume for a moment that $f$ is injective. Applying $-\otimes_{S\ld} S\ld/\frm$ to the short
exact sequence \begin{align}\label{eq:split}0\to P\ld\mor{f} Q\ld\to K\ld\to 0\end{align} yields the exact sequence
$$0 = \tor^1_S(Q\ld,S\ld/\frm)\ld\to \tor^1_S(K\ld,S\ld/\frm)\ld\to P\ld / \frm P\ld \xrightarrow{f\otimes_{S\ld}
  S\ld/\frm}  Q\ld / \frm Q\ld \to K\ld / \frm K\ld \to 0.$$
As $f\otimes_{S\ld} S\ld/\frm$ is injective by assumption, we conclude that $\tor^1_S(K\ld,S\ld/\frm)\ld=0$. Hence
$\beta^0_{S\ld}(K\ld) = 0$ (see Definition \ref{def_betti}), so $K\ld$ is projective, and \eqref{eq:split}
splits. Note that we did not use the projectivity of $P\ld$.

Now we treat the general case. The assumption that $f\otimes_{S\ld} S\ld/\frm$ is injective implies that
$\ker(f)\subset\frm P\ld$, so the projection $P\ld\to P\ld/\ker(f)$ becomes an isomorphism when applying
$-\otimes_{S\ld} S\ld/\frm$. Hence we can apply the case of injective $f$ to the map $\ol{f}: P\ld/\ker(f)\to
Q\ld$, proving that $P\ld/\ker(f)$ is projective. This implies that $\ker(f)$ is a summand of $P\ld$, which together
with $\ker(f)\subset\frm P\ld$ yields $\ker(f)=0$. Thus, $f$ is injective, hence split injective by the first part.
\end{proof}
\begin{proof}[of Proposition \ref{prop_stayregularappendix}]
\eqref{item:crucial1}: We have to construct a map $s: F\ua\ld\to F^{\ast-1}\ld\langle d\rangle$ with the
following properties: 
\begin{enumerate}
\item For each $n\in\NN_{>0}$, $\image(s^{n-1}) = \ker(s^n)$, and this $S\ld$-module has a complement in $F^n\ld$. 
\item We have $\partial^0 s^0 = w \id_{F^0\ld}$ and $\partial^{n+1} s^n + s^{n-1} \partial^n = w
  \id_{P^n\ld}$ for all $n\in\NN_{>0}$. 
\end{enumerate}
The existence of maps $s^n$ satisfying the second condition follows from the embedding
$S\ld\Mod\hookrightarrow\Ho^{-}_{\fr}(S\ld\Mod)$. We will now go through the usual inductive construction of the $s^n$,
additionally taking care of the first condition, for which we will need the assumption $w\in\frm\setminus\frm^2$. 

First, by projectivity of $F^0\ld$ there is a map $s^0: F^0\ld\to F^1\ld$ of internal degree $d$ 
such that $\partial^0 s^0 = w$. For homogeneous $x\in F^0\ld\setminus\frm F^0\ld$, we have $\partial^1 s^0 = wx\in\frm 
F^0\ld\setminus \frm^2 F^0\ld$, and as $\image(\partial^1)\subset\frm F^0\ld$ by the minimality of $F\ua\ld$ (see
Definition \ref{def_minimalres}), it follows that $s^0 x\in F^1\ld\setminus\frm F^1\ld$, hence $s^0\otimes_{S\ld}
S\ld/\frm$ is injective. By Lemma \ref{lem_splitinjection}, $s^0$ is a split injection (note that a priori it is not
even clear that $s^0$ is injective, as $w$ might be a zero divisor). Let $U^1\ld := \image(s^0)$ and
let $V^1\ld$ be some complement of $U^1\ld$ in $F^1\ld$. Next, a small calculation shows that $w - s^0\partial^1$
vanishes on $U^1\ld$ and has image in $\ker(\partial^1) = \image(\partial^2)$. Hence there exists some $s^1: F^1\ld\to
F^2\ld$ of degree $d$ such that $\partial^2 s^1 + s^0\partial^1 = w$ and $s^1|_{U^1\ld} = 0$. Now if $x\in
V^1\ld\setminus\frm V^1\ld$, we have $\partial^2 s^1 (x) = (w - s^0\partial^1)(x)\in \frm
F^1\ld\setminus\frm^2 F^1\ld$, since both summands on the right hand side live in different summands of $F^1\ld =
V^1\ld\oplus U^1\ld$, and $wx\in \frm V^1\ld\setminus\frm^2 V^1\ld$. As before, the minimality of $F\ua\ld$ implies that
$s^1 x\in F^2\ld\setminus\frm F^2\ld$. Applying Lemma \ref{lem_splitinjection} again shows that $s^1|_{V^1\ld}:
V^1\ld\to F^2\ld$ is a split injection, and we put $U^2\ld := \image(s^1)$. Continuing in this way, one can construct
the maps $s^n$ satisfying the above conditions.  

\eqref{item:crucial2}: It suffices to prove that any finitely generated graded $S\ld/(w)$-module $M\ld$ has finite
projective dimension over $S\ld/(w)$. By assumption we have $\projdim_{S\ld}(M\ld)<\infty$, so the minimal $S\ld$-free
resolution $F\ua\ld\to M\ld$ of $M\ld$ is finite. Now, applying \eqref{item:crucial2}, $F\ua\ld$ admits the structure of
a semi-free $K\ua_w$-module, and so we have $$M\ld\ \cong\ M\ld\stackrel{\LL}{\otimes}_{K\ua_w} S\ld/(w)\ \cong\
F\ua\ld\otimes_{K\ua_w} S\ld/(w)$$ in $\D(S\ld/(w))$. As $F\ua\ld\otimes_{K\ua_w} S\ld/(w)$ is bounded and
$S\ld/(w)$-free, it follows that $M\ld$ is perfect in $\D(S\ld/(w)))$, hence $\projdim_{S\ld/(w)}(M\ld)<\infty$ as claimed.
\end{proof}

\subsection{Boundedness conditions}\label{app_bounded}
\markright{\ref{app_bounded} Boundedness conditions}

Again fix a commutative graded Noetherian ring $S\ld$ and a dg-$S\ld$-algebra $(A\ua\ld,\d)$. In this section, we will
define several subcategories of $\D(A\ua\ld)$ imposing various boundedness conditions on (the cohomology of) the
dg-$A\ua\ld$-modules.

\begin{definition}
For $\ast,\ast\p\in\{+,-,b,\emptyset\}$, let $\D^{\ast,\ast\p}(A\ua\ld)$ denote the full subcategory of
$\D(A\ua\ld)$ consisting of those $A\ua\ld$-modules that are cohomologically bounded according to $\ast$ and bounded
according to $\ast\p$. For example, $\D^{+,\emptyset}(A\ua\ld)$ contains all (potentially unbounded) $A\ua\ld$-modules
with bounded below cohomology. Further, we will abbreviate $\D^{\ast,\emptyset}(A\ua\ld)$ by $\D^{\ast}(A\ua\ld)$.  

The full subcategories $\Ho^{\ast,\ast\p}(A\ua\ld)$ and $\Acyc^{\ast,\ast\p}(A\ua\ld)$ of $\Ho(A\ua\ld)$ are defined
analogously. 
\end{definition}

\begin{fact}\label{fact_truncation}
For $\ast\in\{+,-,b\}$, the subcategories $\D^{\ast,\emptyset}(A\ua\ld)$ are triangulated subcategories of
$\D(A\ua\ld)$, and $\D^{\ast,\ast}(A\ua\ld)\subset\D^{\ast,\emptyset}(A\ua\ld)$. If $A^k\ld=0$ for $k>0$, then this
inclusion is an equivalence, and in particular $\D^{\ast,\ast}(A\ua\ld)$ is a triangulated subcategory of $\D(A\ua\ld)$.  
\end{fact}
\begin{proof}
The subcategories $\D^{\ast,\emptyset}(A\ua\ld)$ are triangulated because an exact triangle induces a long exact
sequence in cohomology. It is clear the we have an inclusion
$\D^{\ast,\ast}(A\ua\ld)\subset\D^{\ast,\emptyset}(A\ua\ld)$. 

Now assume $A^k\ld=0$ for all $k>0$. We have to show that each object of $\D^{\ast,\emptyset}(A\ua\ld)$ is isomorphic to
an object of $\D^{\ast,\ast}(A\ua\ld)$. For this, first assume $M\ua\ld\in\D^{-,\emptyset}(A\ua\ld)$, and choose $n\gg
0$ such that $\H^k(M\ua\ld)=0$ for $k>n$. Then the truncation 
$$\tau_{\leq n}:\quad ... \longrightarrow M^{n-2}\ld\xrightarrow{\differential_{M\ua\ld}^{n-2}}
M^{n-1}\ld\xrightarrow{\differential_{M\ua\ld}^{n-1}} \ker\left(\differential_{M\ua\ld}^n\right)\longrightarrow 0\to
...$$ 
is an $A\ua\ld$-submodule of $M\ua\ld$ (here we need our assumption that $A^k\ld=0$ for $k>0$), and the inclusion
$\tau_{\leq n}M\ua\ld\to M\ua\ld$ is a quasi-isomorphism. Hence, $M\ua\ld\cong \tau_{\leq n} M\ua\ld$ in $\D(A\ua\ld)$,
and so the inclusion $\D^{-,-}(A\ua\ld)\hookrightarrow\D^{-,\emptyset}(A\ua\ld)$ is an equivalence. 

Analogously, for $M\ua\ld\in\D^{+,\emptyset}(A\ua\ld)$ we choose $n\ll 0$ such that $\H^k(M\ua\ld)=0$ for $k\leq n$ and
consider the truncation
\begin{align}\label{eq:truncationgeq}
\tau_{\geq n}:\quad ...\to 0\longrightarrow M^n\ld / \image\left(\differential_{M\ua\ld}^{n-1}\right)\xrightarrow{\differential_{M\ua\ld}^n} 
M^{n+1}\ld\xrightarrow{\differential_{M\ua\ld}^{n+1}} M^{n+2}\ld\to ...\end{align}
As $\tau_{\geq n}M\ua\ld = M\ua\ld / \tau_{\leq n}M\ua\ld$, this is a quotient $A\ua\ld$-module of $M\ua\ld$, and
the choice of $n$ implies that $M\ua\ld\to\tau_{\geq n}M\ua\ld$ is a quasi-isomorphism. Hence, $M\ua\ld\cong\tau_{\geq
  n}M\ua\ld$ in $\D(A\ua\ld)$, and so the inclusion $\D^{+,+}(A\ua\ld)\hookrightarrow\D^{+,\emptyset}(A\ua\ld)$ is an
equivalence. If moreover $M\ua\ld\in\D^{-,-}(A\ua\ld)$, we have $\tau_{\geq n}M\ua\ld\in\D^{-,-}(A\ua\ld)$ as well,
showing that $\D^{b,b}(A\ua\ld)\hookrightarrow\D^{b,\emptyset}(A\ua\ld)$ is an equivalence.
\end{proof}

\begin{definition}
A complex $M\ua\ld$ of graded $S\ld$-modules is called $S\ld$-free (resp. $S\ld$-finite) if each component $M^k\ld$ is a
free (resp. finitely generated) graded $S\ld$-module.
\end{definition}

\begin{definition}
For a category $\C$ of $A\ua\ld$-modules (e.g. $\D(A\ua\ld)$ or $\Ho(A\ua\ld)$) we denote $\C_{\fg}$
the full subcategory of $S\ld$-finite objects in $\C$. By $\C_{\fr}$ we denote the full subcategory of $S\ld$-free
objects in $\C$. If more than one condition is to be applied, the subscripts are separated by commata. For example,
$\C_{\fr,\fg}$ denotes the full subcategory of those objects in $\C$ which are both $S\ld$-free and $S\ld$-finite 
(note the difference with the meaning of, say, $\D^{+,b}(A\ua\ld)$, where the first supscript refers to the
cohomology).  
\end{definition}

\begin{prop}\label{prop_freetruncation}
If $A\ua\ld$ is $S\ld$-free and $A^k\ld=0$ for $k>0$, then the inclusions
$\D^{-,-}_{\fr}(A\ua\ld)\subset\D^{-,-}(A\ua\ld)$ is an equivalence. If in addition $S\ld$ is regular (i.e. $S\ld\Mod$
is of finite global dimension), the inclusion $\D^{b,b}_{\fr}(A\ua\ld)\subset\D^{b,-}_{\fr}(A\ua\ld)$ is an
equivalence. Analogous statements are true for the inclusions
$\D^{-,-}_{\fr,\fg}(A\ua\ld)\subset\D^{-,-}_{\fg}(A\ua\ld)$ and
$\D^{b,b}_{\fr,\fg}(A\ua\ld)\subset\D^{b,-}_{\fr,\fg}(A\ua\ld)$ if $A\ua\ld$ is $S\ld$-finite. 
\end{prop}
\begin{proof}
The small object argument yields for each $A\ua\ld$-module $M\ua\ld$ a functorial surjective quasi-isomorphism 
$\wt{M}\ua\ld\to M\ua\ld$ with $\wt{M}\ua\ld$ semi-free. Since $A\ua\ld$ is $S\ld$-free, any semi-free $A\ua\ld$-module
is $S\ld$-free, and so the inclusion $\D_{\fr}(A\ua\ld)\subset\D(A\ua\ld)$ is an equivalence. Moreover, a look into the
proofs of the small object argument in \cite[Theorem 2.1.14]{Hovey} and \cite[Theorem 3.5]{GoerssSchemmerhorn} shows
that the construction given there produces a bounded above $\wt{M}\ua\ld$ if $M\ua\ld$ was bounded above, proving that
$\D^{-,-}_{\fr}(A\ua\ld)\subset\D^{-,-}(A\ua\ld)$ is an equivalence. However, the unmodified small object argument
yields very large $A\ua\ld$-modules $\wt{M}\ua\ld$ even if $M\ua\ld$ is $S\ld$-finite, and hence it cannot be used to
establish the equivalence $\D^{-,-}_\fg(A\ua\ld)\cong\D^{-,-}_{\fg,\fr}(A\ua\ld)$. What we will do now is to give a
construction of a quasi-isomorphism $\wt{M}\ua\ld\to M\ua\ld$ based on the one given by the small object argument
such that the output $\wt{M}\ua\ld$ is $S\ld$-finite if $M\ua\ld$ and $A\ua\ld$ are $S\ld$-finite; however, this
construction will no longer be functorial. The topologically minded reader will note that the construction
below very much resembles the usual construction of CW-approximations for topological spaces (see \cite[Proposition
4.13]{Hatcher}).   

To prove that $\D^{-,-}_{\fr,\fg}(A\ua\ld)\subset\D^{-,-}_{\fg}(A\ua\ld)$ is an equivalence, it suffices to
construct for each bounded above, $S\ld$-finitely $A\ua\ld$-module $M\ua\ld$ a quasi-isomorphism
$\wt{M}\ua\ld\to M\ua\ld$ with $\wt{M}\ua\ld$ bounded above $S\ld$-free and $S\ld$-finite. Assume without loss of
generality that $M^k\ld = 0$ for $k>0$. Let $\{m^0_i\}_{i\in I^0}$ be a finite set of homogeneous elements generating
$M^0\ld$ as a graded $S\ld$-module and put $$^0\wt{M}\ua\ld := \bigoplus\limits_{i\in I^0} e^0_i S(0,|m^0_i|) =
\bigoplus\limits_{i\in I^0} e^0_i A\ua\ld\langle -|m^0_i|\rangle,$$ where the $e^0_i$ are just names for the units in
the respective copies of $A\ua\ld$. By construction, we have a morphism $\varphi^0: {^0\wt{M}}\ua\ld\to M\ua\ld$ defined 
by $e_i\mapsto m^0_i$ for $i\in I^0$, inducing an epimorphism in the $0$-th cohomology. This is our first approximation
to the desired quasi-isomorphism $\wt{M}\ua\ld\to M\ua\ld$. 

Next, we try to find a better approximation $^1\wt{M}\ua\ld\to M\ua\ld$ correcting the failure of injectivity of
$\H^0(\varphi)$ and surjectivity of $\H^{-1}(\varphi)$. For this, pick a finite set $\{z^0_j\}_{j\in J^0}\subset
Z^0(^0\wt{M}\ua\ld)$ of homogeneous elements such that $\{\ol{z^0_j}\}_{j\in J^0}$ is a generating set of
$$\ker\left(\H^0(^0\wt{M}\ua\ld)\xrightarrow{\H^0(\varphi^0)}\H^0(M\ua\ld)\right).$$ Further, pick a finite set
$\{m^1_i\}_{i\in I^1}\subset Z^1(M\ua\ld)$ of homogeneous elements such that $\{\ol{m^1_i}\}_{i\in I^1}$ generates
$\H^{-1}(M\ua\ld)$ as a graded $S\ld$-module. Then, we take $^1\wt{M}\ua\ld$ to be the pushout
\begin{equation}\label{eq:smallargpushout}\begin{tikzpicture}[description/.style={fill=white,inner sep=2pt},baseline =
    (current bounding box.center)]
    \matrix (m) [matrix of math nodes, row sep=4em,
                 column sep=3.5em, text height=1.5ex, text depth=0.25ex,
                 inner sep=0pt, nodes={inner xsep=0.3333em, inner ysep=0.3333em}]
    {
       \bigoplus\limits_{j\in J^0} \tilde{e}^0_j S(0,|z^0_j|) &&
       ^0\wt{M}\ua\ld\\
       \bigoplus\limits_{j\in J^0} f^1_j D(1,|z^0_j|)\oplus\bigoplus\limits_{i\in I^1} e^1_i S(1,|e^1_i|) && ^1
       \wt{M}\ua\ld\\ 
    };
    \draw[->] (m-1-1) -- (m-1-3);
    \draw[->] (m-1-1) -- (m-2-1);
    \draw[->,dashed] (m-2-1) -- (m-2-3);
    \draw[->,dashed] (m-1-3) -- ($(m-2-3.north) + (0,2mm)$);
\end{tikzpicture}\end{equation}
\vskip2mm\noindent Here, the left vertical map comes from the inclusions $S(n-1,k)\hookrightarrow D(n,k)$, and the upper
horizontal map is given by $\tilde{e}^0_j\mapsto z^0_j$. Next, by definition of the $z^0_j$ there are elements $b^0_j\in
M^{-1}_{|z^0_j|}$ such that $\d(b^1_j) = \varphi^0(z^0_j)$, and  $f^1_j\mapsto b^1_j$, $e^1_i\mapsto m^1_i$ defines a
morphism $\bigoplus\limits_{j\in J^0} f^1_j D(1,|z^0_j|)\oplus\bigoplus\limits_{i\in I^1} e^1_i S(1,|e^1_i|)\to M\ua\ld$
giving rise to a commutative outer square in the diagram
\begin{equation}\label{eq:smallargpushout}\begin{tikzpicture}[description/.style={fill=white,inner sep=2pt},baseline =
    (current bounding box.center)]
    \matrix (m) [matrix of math nodes, row sep=4em,
                 column sep=3.5em, text height=1.5ex, text depth=0.25ex,
                 inner sep=0pt, nodes={inner xsep=0.3333em, inner ysep=0.3333em}]
    {
       \bigoplus\limits_{j\in J^0} \tilde{e}^0_j S(0,|z^0_j|) &&
       ^0\wt{M}\ua\ld\\
       \bigoplus\limits_{j\in J^0} f^1_j D(1,|z^0_j|)\oplus\bigoplus\limits_{i\in I^1} e^1_i S(1,|e^1_i|) && ^1
       \wt{M}\ua\ld\\ 
&&& M\ua\ld\\
    };
    \draw[->] (m-1-1) -- (m-1-3);
    \draw[->] (m-1-1) -- (m-2-1);
    \draw[->] (m-2-1) -- (m-2-3);
    \draw[->] (m-1-3) -- ($(m-2-3.north) + (0,2mm)$);

    \draw[->] (m-1-3) edge [bend left=5] node[description,scale=0.75]{$\varphi^0$} (m-3-4);
    \draw[->] ($(m-2-1.south) + (1cm,-3mm)$) edge [bend right=5] (m-3-4);
    
    \draw[->,dashed] (m-2-3) -- node[description,scale=0.75]{$\varphi^1$}(m-3-4); 
\end{tikzpicture}\end{equation}
By the universal property of the pushout, we get a unique morphism $\varphi^1: {^1\wt{M}}\ua\ld\to M\ua\ld$ making the
whole diagram commute. 

Intuitively, taking the pushout \eqref{eq:smallargpushout} amounts to killing
the cohomology classes associated to the $z^0_j$ by making them 
boundaries of formally adjoint ``cells'', causing $\H^0(^1\wt{M}\ua\ld\to M\ua\ld)$ to become injective, while at the
same time glueing in new ``spheres'' to make $\H^{-1}(^1\wt{M}\ua\ld\to M\ua\ld)$ surjective.

Rigorously, the pushout \eqref{eq:smallargpushout} comes from a short exact sequence of $A\ua\ld$-modules 
\begin{align}\label{eq:smallargses}
0\to \bigoplus\limits_{j\in J^0} \tilde{e}^0_j S(0,|z^0_j|)\to
\bigoplus\limits_{j\in J^0} f^1_j D(1,|z^0_j|)\oplus\bigoplus\limits_{i\in I^1} e^1_i S(1,|e^1_i|)\oplus
{^0\wt{M}}\ua\ld\to {^1\wt{M}}\ua\ld\to 0.
\end{align}
which induces a long exact sequence in cohomology. Since $S(n,k)\cong A\ua\ld[n]\langle
-k\rangle$ has no cohomology in degrees above $-n$ and $D(n,k)$ is contractible, we see that the canonical map
${^0\wt{M}}\ua\ld\to {^1\wt{M}}\ua\ld$ induces an isomorphism in the cohomology in degrees above $0$. In degree $0$ and
$-1$, the long exact cohomology sequence induced by \eqref{eq:smallargses} degenerates to exact sequences
\begin{gather*}
\bigoplus\limits_{j\in J^0}\H^0(\tilde{e}^0_j S(0,|z^0_j|))\to\H^0({^0\wt{M}}\ua\ld)\to\H^0({^1\wt{M}}\ua\ld)\to 0\\
\bigoplus\limits_{i\in I^1}\H^{-1}(e^1_i
S(1,|e^1_i|))\oplus\H^{-1}({^0\wt{M}}\ua\ld)\to\H^0({^1\wt{M}}\ua\ld)\to\bigoplus\limits_{j\in J^0} \H^0(\tilde{e}^0_j
S(0,|z^0_j|))\to 0
\end{gather*}
By definition of $\varphi^1$, this implies that $\H^0(\varphi^1)$ is an isomorphism, while $\H^{-1}(\varphi^1)$ is
surjective, as claimed, finishing the construction of the second approximation ${^1\wt{M}}\ua\ld\mor{\varphi^1} M\ua\ld$
to the desired quasi-isomorphism $\wt{M}\ua\ld\to M\ua\ld$. 

The method by which we constructed ${^0\wt{M}}\ua\ld$ from $0\to M\ua\ld$ and ${^1\wt{M}}\ua\ld\to M\ua\ld$ from
${^0\wt{M}}\ua\ld\to M\ua\ld$ can be used again and again to find a commutative diagram
\begin{equation}\label{eq:smallargpushout}\begin{tikzpicture}[description/.style={fill=white,inner sep=2pt},baseline =
    (current bounding box.center)]
    \matrix (m) [matrix of math nodes, row sep=4em,
                 column sep=3.5em, text height=1.5ex, text depth=0.25ex,
                 inner sep=0pt, nodes={inner xsep=0.3333em, inner ysep=0.3333em}]
    {
      {^0\wt{M}}\ua\ld & {^1\wt{M}}\ua\ld & {^2\wt{M}}\ua\ld & ... & \wt{M\ua\ld} = \bigcup\limits_{n\geq 0}
      {^n\wt{M}}\ua\ld\\ 
      &&&& M\ua\ld\\
    };
    \draw[right hook->] (m-1-1) -- node[above,scale=0.75]{$\iota^0$} (m-1-2);
    \draw[right hook->] (m-1-2) -- node[above,scale=0.75]{$\iota^1$} (m-1-3);
    \draw[right hook->] (m-1-3) -- node[above,scale=0.75]{$\iota^2$} (m-1-4);
    \draw[right hook->] (m-1-4) --  (m-1-5);

    \draw[->] (m-1-1) edge [bend right=15] node[description,scale=0.75]{$\varphi^0$} (m-2-5);
    \draw[->] (m-1-2) edge [bend right=10] node[description,scale=0.75]{$\varphi^1$}(m-2-5);
    \draw[->] (m-1-3) edge [bend right=5] node[description,scale=0.75]{$\varphi^2$}(m-2-5);
    \draw[->] (m-1-5) -- node[description,scale=0.75]{$\varphi$}(m-2-5);
\end{tikzpicture}\end{equation}
such that for each $n\geq 0$ the following properties are satisfied:
\begin{enumerate}
\item\label{item:inductiveiso} $\H^k(\varphi^n): \H^k({^n\wt{M}}\ua\ld)\to\H^k(M\ua\ld)$ is an isomorphism for $k>-n$ and
  an epimorphism for $k=-n$. 
\item\label{item:goodsemifreefiltration} $\coker(\iota^n: {^n\wt{M}}\ua\ld\to {^{n+1}\wt{M}}\ua\ld)$ is a finite direct
  sum of modules of the form $A\ua\ld[k]\langle l\rangle$ with $l\in\ZZ$ and $k\geq n$. 
\end{enumerate}
As cohomology commutes with filtered colimits, \eqref{item:inductiveiso} implies that the induced map $\varphi:
\wt{M}\ua\ld\to M\ua\ld$ is a quasi-isomorphism. Finally, \eqref{item:goodsemifreefiltration} implies that
$S\ld$-finite, and so we're done.

For the second statement, assume $S\ld$ is regular local and let $M\ua\ld\in\D^{b,-}_{\fr}(A\ua\ld)$. Choose $n\gg 0$
such that $\H^k(M\ua\ld)=0$ for $k<-n$. Then $\ker(\differential_{M\ua\ld}^k)\hookrightarrow M^k\ld$ splits for all
$k<-n-\gldim(S\ld\Mod)$, and hence $\tau_{\geq k} M\ua\ld\in\D^{b,b}_{\fr}(A\ua\ld)$ (note that by Kaplansky's theorem,
every projective $S\ld$-module is $S\ld$-free). Moreover, if $M\ua\ld\in\D^{b,-}_{\fr,\fg}(A\ua\ld)$, then $\tau_{\geq
  k} M\ua\ld\in\D^{b,b}_{\fr,\fg}(A\ud\ld)$. This proves the second statement. 
\end{proof}

Note that we defined $\D^{b,b}_{\fr}(A\ua\ld)$ as a full subcategory of $\D(A\ua\ld)$, and hence a priori morphisms in
$\D^{b,b}_{\fr}(A\ua\ld)$ may involve unbounded $A\ua\ld$-modules. However, for regular local $S\ld$ we can avoid
unbounded modules in the description of the morphism spaces:
\begin{prop}\label{prop_boundedasverdierquotient}
If $A^k\ld=0$ for $k>0$, the canonical triangulated
functor 
\begin{align}\label{eq:pleasebebounded}\Ho^{b,b}_{(\fg)}(A\ua\ld)/\Acyc^{b,b}_{(\fg)}(A\ua\ld)\
  \longrightarrow\ \D^{b,b}_{(\fg)}(A\ua\ld)\end{align} is an equivalence. If, in addition, $S\ld$ is regular local, the 
same is true for $$\Ho^{b,b}_{\fr,(\fg)}(A\ua\ld)/\Acyc^{b,b}_{\fr,(\fg)}(A\ua\ld)\
  \longrightarrow\ \D^{b,b}_{\fr,(\fg)}(A\ua\ld).$$
\end{prop}
\begin{proof}
As \eqref{eq:pleasebebounded} is the identity on objects, we only have to check that it is fully faithful. We restrict
to the free, non-finitely generated 
case; the other cases are proved along the same lines, noting that truncation preserves the property of being
finitely generated.

We will use the description of morphisms in $\D(A\ua\ld)$ through upper roofs. Thus, assume that $M\ua\ld, 
N\ua\ld\in\D^{b,b}_{\fr}(A\ua\ld)$ and that we have a morphism $M\ua\ld\to N\ua\ld$ in $\D^{b,b}_{\fr}(A\ua\ld)$
represented by the upper roof
\begin{equation}\label{eq:upperroofunbounded}
\begin{tikzpicture}[description/.style={fill=white,inner sep=2pt},baseline=(current
    bounding box.center) ]
    \matrix (m) [matrix of math nodes, row sep=3em,
                 column sep=2.5em, text height=1.5ex, text depth=0.25ex,
                 inner sep=0pt, nodes={inner xsep=0.3333em, inner ysep=0.3333em}]
    {
       & X\ua\ld & \\ M\ua\ld && N\ua\ld\\
    };
    \draw[->] (m-1-2) -- node[description,scale=0.75]{$\alpha$} (m-2-1);
    \draw[->] (m-1-2) -- node[description,scale=0.75]{$\beta$} (m-2-3);
\end{tikzpicture}
\end{equation}
where $\alpha$ is a quasi-isomorphism. Then $X\ua\ld\in\D^{b,\emptyset}(A\ua\ld)$, and for $k\gg 0$ the inclusion
$\tau_{\leq k} X\ua\ld\hookrightarrow X\ua\ld$ is a quasi-isomorphism. Further, there exists a quasi-isomorphism
$Y\ua\ld\to \tau_{\leq k} X\ua\ld$ with $Y\ua\ld\in\D^{b,-}_{\fr}(A\ua\ld)$. Thus, expanding the above roof with the
resulting composition $Y\ua\ld\to X\ua\ld$ we may assume that a priori $X\ua\ld\in\D^{b,-}_{\fr}(A\ua\ld)$. Next, as
$M\ua\ld$ and $N\ua\ld$ are bounded, there exists $k\ll 0$ such that the following hold:
\begin{enumerate}
\item $\tau_{\geq k} X\ua\ld\in\D^{b,b}_{\fr}(A\ua\ld)$ (possible since $S\ld$ is regular)
\item $X\ua\ld\to\tau_{\geq k}X\ua\ld$ is a quasi-isomorphism.
\item $\alpha$ and $\beta$ factor as $X\ua\ld\to\tau_{\geq k} X\ua\ld\mor{\wt{\alpha}} M\ua\ld$ and
  $X\ua\ld\to\tau_{\geq k} X\ua\ld\mor{\wt{\beta}} N\ua\ld$, respectively.
\end{enumerate}
Under these assumptions, the roof \eqref{eq:upperroofunbounded} is equivalent to the roof
\begin{equation*}\begin{tikzpicture}[description/.style={fill=white,inner sep=2pt}]
    \matrix (m) [matrix of math nodes, row sep=3em,
                 column sep=2.5em, text height=1.5ex, text depth=0.25ex,
                 inner sep=0pt, nodes={inner xsep=0.3333em, inner ysep=0.3333em}]
    {
       & \tau_{\geq k}X\ua\ld & \\ M\ua\ld && N\ua\ld\\
    };
    \draw[->] (m-1-2) -- node[description,scale=0.75]{$\wt{\alpha}$} (m-2-1);
    \draw[->] (m-1-2) -- node[description,scale=0.75]{$\wt{\beta}$} (m-2-3);
\end{tikzpicture}\end{equation*}
proving that \eqref{eq:pleasebebounded} is full. The faithfulness is proved similarly.
\end{proof}

\begin{rem}
Let $\varphi: A\ua\ld\to B\ua\ld$ be a morphism of dg-$S\ld$-algebras concentrated in non-positive degrees. Then note
that even though we have a description of $\D^{b,b}_{(\fr),(\fg)}(A\ua\ld)$ not involving unbounded $A\ua\ld$-modules,
the calculation of the derived tensor product functor $$\D(A\ua\ld)\xrightarrow{\quad-\stackrel{\LL}{\otimes}_{A\ua\ld}
  B\ua\ld\quad }\D(B\ua\ld)$$ \textit{does} involve unbounded modules, even if we restrict it to bounded
$A\ua\ld$-modules and regular $S\ld$, because there might be bounded $A\ua\ld$-modules which do not possess bounded
semi-free resolutions.
\end{rem}

Next we discuss to what extend the derived adjunction corresponding corresponding to a morphism $\varphi: A\ua\ld\to
B\ua\ld$ respects the subcategories of $\D(A\ua\ld)$ and $\D(B\ua\ld)$ we just introduced.
\begin{fact}
Let $A\ua\ld$ and $B\ua\ld$ be dg-$S\ld$-algebras, and assume $A^k\ld=0$ for $k>0$. Further, let $\varphi: A\ua\ld\to
B\ua\ld$ be a homomorphism of dg-$S\ld$-algebras. Then the functor $$\varphi\la:\ \D(B\ua\ld)\
\longrightarrow\ \D(A\ua\ld)$$ takes $\D^{\ast,\ast\p}(B\ua\ld)$ to $\D^{\ast,\ast\p}(A\ua\ld)$ for all
$\ast,\ast\p\in\{\emptyset,b,+,-\}$. Its adjoint $$-\stackrel{\LL}{\otimes}_{A\ua\ld} B\ua\ld:\ \D(A\ua\ld)\
\longrightarrow\ \D(B\ua\ld)$$ takes $\D^{-,\emptyset}(A\ua\ld)\to\D^{-,\emptyset}(B\ua\ld)$. If $\varphi$ is a
quasi-isomorphism, the adjoint equivalence \eqref{eq:adjderived}
between $\D(A\ua\ld)$ and $\D(B\ua\ld)$ restricts to an adjoint equivalence between $\D^{\ast,\emptyset}(A\ua\ld)$ and
$\D^{\ast,\emptyset}(B\ua\ld)$ for all $\ast\in\{\emptyset,b,+,-\}$. 
\end{fact}

\subsection{The Koszul resolution of $S\ld/(w)$}\label{app_koszul}
\markright{\ref{app_koszul} The Koszul resolution of $S\ld/(w)$}

Consider $S\ld/(w)$ as a dg-$S\ld$-algebra concentrated in degree $0$. Then a dg-$S\ld/(w)$-module is just a
complex of $S\ld/(w)$-modules, and so the derived category of the dg-$S\ld$-algebra $S\ld/(w)$ equals the derived
category of the abelian category $S\ld/(w)\mod$. Hence, there is no unambiguity when talking about \textit{the} derived
category $\bfD(S\ld/(w))$. 

Our strategy is to apply Proposition \ref{prop_dgadjointequivalence} to certain $S\ld$-free dg-$S\ld$-algebras
quasi-isomorphic to $S\ld/(w)$, thereby converting our intuition '$S\ld/(w)$-modules should be replaced by enriched
complexes of free $S\ld$-modules' into a precise statement.

\begin{definition}\label{def_resolutions}
Let as usual $S\ld$ be a regular local graded ring and $w$ be homogeneous of degree $d$ (possibly zero). The
\textit{Koszul-resolution} $K\ua_w$ of $S\ld/(w)$ is defined as the dg-$S\ld$-algebra 
$$K\ua_w\ := \ ...\to 0\to S\ld\langle -d\rangle\xrightarrow{\cdot w} S\ld\to 0\to ...,$$
where $S\ld$ is concentrated in cohomological degree $0$. In other words, it is the free graded-commutative dg-algebra with
 generator $s$ of cohomological degree $-1$ and internal degree $d$ and differential given by $\d(s) := w\cdot 1$. 
\end{definition}

\begin{prop}
There is a  natural morphism of dg-$S\ld$-algebras $\kappa_w: K\ua_w\to S\ld/(w)$ which is a quasi-isomorphisms if and
only if $w\neq 0$. In particular, we have a derived adjunction
\begin{equation*}\begin{tikzpicture}[description/.style={fill=white,inner
      sep=2pt},baseline=(m-1-1.base)]  
    \matrix (m) [matrix of math nodes, row sep=3em,
                 column sep=10em, text height=1.5ex, text depth=0.25ex,
                 inner sep=0pt, nodes={inner xsep=0.3333em, inner ysep=0.3333em}]
    {
       \D(K\ua_w\Mod) & \D(S\ld/(w)\Mod) \\
    };
    \draw[->] ($(m-1-1.east) + (0,+0.5mm)$) -- node[above,scale=0.75]{$-\stackrel{\LL}{\otimes}_{K\ua_w} S\ld/(w)$}
    ($(m-1-2.west) + (0,+0.5mm)$); 
    \draw[->] ($(m-1-2.west) + (0,-0.5mm)$) -- node[below,scale=0.75]{$(\kappa_w)\la$}
    ($(m-1-1.east) + (0,-0.5mm)$); 
\end{tikzpicture}\end{equation*}
which is a derived equivalence if and only if $w\neq 0$.
\end{prop}
\begin{proof}
The morphism $\kappa_w$ is given by the diagram
\begin{equation*}\begin{tikzpicture}[baseline=(current bounding box.center), description/.style={fill=white,inner sep=2pt}]
    \matrix (m) [matrix of math nodes, row sep=2.25em,
                 column sep=2.5em, text height=1.5ex, text depth=0.25ex]
    {
       \cdots & 0 & S\ld\langle -d\rangle & S\ld & 0 & \cdots\\
       \cdots & 0 & 0 & S\ld/(w) & 0 & \cdots\\
    };
    \path[->,font=\scriptsize]
    (m-1-1) edge (m-1-2)
    (m-1-2) edge (m-1-3)
    (m-1-3) edge node[auto]{$\cdot w$} (m-1-4)
    (m-1-4) edge (m-1-5)
    (m-1-5) edge (m-1-6)
    (m-2-1) edge (m-2-2)
    (m-2-2) edge (m-2-3)
    (m-2-3) edge (m-2-4)
    (m-2-4) edge (m-2-5)
    (m-2-5) edge (m-2-6)
    (m-1-2) edge (m-2-2)
    (m-1-3) edge (m-2-3)
    (m-1-4) edge (m-2-4)
    (m-1-5) edge (m-2-5);
\end{tikzpicture}\end{equation*}
and it is clear that $\kappa_w$ is a quasi-isomorphism if and only if $w\neq 0$. The second statement follows from
Proposition \ref{prop_dgadjointequivalence} applied to $\kappa_w$. 
\end{proof}

\begin{fact}\label{fact_higherhomotopiesasmodulestructure}
Let $M\ua\ld$ be a complex of $S\ld$-modules. Giving $M\ua\ld$ the structure of a dg-$K\ua_w$-module is equivalent to
giving a nullhomotopy $s$ for the multiplication by $w$ on $M\ua\ld$ such that $s^2=0$. 
\end{fact}

\noindent\textbf{Convention:} We will often write a $K\ua_w$-module $M\ua\ld$ in the form
\begin{equation*}\begin{tikzpicture}[baseline=(current bounding  box.center), description/.style={fill=white,inner sep=2pt}]
    \matrix (m) [matrix of math nodes, row sep=4em,
                 column sep=4em, text height=1.5ex, text depth=0.25ex]
    {
       ... & M^{n-1}\ld & M^n\ld & M^{n+1}\ld & ... \\
     };
     \draw[->] ($(m-1-1.east) + (0,+0.5mm)$) -- node[above,scale=0.75]{$\differential$} ($(m-1-2.west) + (0,+0.5mm)$);
    \draw[->] ($(m-1-2.west) + (0,-0.5mm)$) -- node[below,scale=0.75]{$s$}  ($(m-1-1.east) + (0,-0.5mm)$);
    \draw[->] ($(m-1-2.east) + (0,+0.5mm)$) -- node[above,scale=0.75]{$\differential$} ($(m-1-3.west) + (0,+0.5mm)$);
    \draw[->] ($(m-1-3.west) + (0,-0.5mm)$) -- node[below,scale=0.75]{$s$}  ($(m-1-2.east) + (0,-0.5mm)$);
    \draw[->] ($(m-1-3.east) + (0,+0.5mm)$) -- node[above,scale=0.75]{$\differential$} ($(m-1-4.west) + (0,+0.5mm)$);
    \draw[->] ($(m-1-4.west) + (0,-0.5mm)$) -- node[below,scale=0.75]{$s$}  ($(m-1-3.east) + (0,-0.5mm)$);
    \draw[->] ($(m-1-4.east) + (0,+0.5mm)$) -- node[above,scale=0.75]{$\differential$} ($(m-1-5.west) + (0,+0.5mm)$);
    \draw[->] ($(m-1-5.west) + (0,-0.5mm)$) -- node[below,scale=0.75]{$s$}  ($(m-1-4.east) + (0,-0.5mm)$);
\end{tikzpicture}\end{equation*}
where the arrows pointing to the left denote the action of $s$ on $M\ua\ld$. 

\subsection{The Bar resolution}\label{app_bar}
\markright{\ref{app_bar} The Bar resolution}

In order to calculate the image of a dg-module $M\ua\ld$ under a derived functor like
$-\stackrel{\LL}{\otimes}_{K\ua_w} S\ld/(w)$ we need to know an explicit cofibrant   
resolution of $M\ua\ld$. The goal of this section is to describe one particular such resolution for $S\ld$-free
$A\ua\ld$-modules, which is even functorial in $M\ua\ld$: the Bar resolution. Throughout we fix an arbitrary 
commutative graded ring $S\ld$.  

The results of this and the following section are taken from \cite[Section 3.1]{Avramov}.

\begin{definition}\label{def_bar}
Let $A\ua\ld$ be a connected, $S\ld$-free dg-$S\ld$-algebra with unit $\eta: S\ld\to A\ua\ld$, and set
$\wt{A}\ua\ld := \coker(\eta) = A^{>0}\ld$. Further, let $M\ua\ld$ be an $A\ua\ld$-module. The \textit{Bar resolution}
$Q(A\ua\ld,M\ua\ld)$ of $M\ua\ld$ over $A\ua\ld$ is defined as follows:
\begin{enumerate}
\item The underlying $\ZZ$-graded graded $S\ld$-module is given by 
\begin{align}\label{eq:bardef}Q(A\ua\ld,M\ua\ld)^n\ld\ \ := \ \
  \bigoplus\limits_{h-p+i_1+...+i_p+j=n} A^h\ld\otimes_{S\ld}
  \wt{A}^{i_1}\ld\otimes_{S\ld}\cdots\otimes_{S\ld}\wt{A}\ld^{i_p}\otimes_{S\ld} M^j\ld,
\end{align}
and the action of $A\ua\ld$ on $Q(A\ua\ld,M\ua\ld)$ is given by left multiplication on the first tensor factor. 
\item The differential is given by $\partial := \partial\p + \partial\pp$, where
\begin{align}\label{eq:bardiff1}
\partial\p(a\otimes \wt{a}_1\otimes...\otimes\wt{a}_p\otimes m) := &\ 
\partial(a)\otimes\wt{a}_1\otimes...\otimes\wt{a}_p\otimes m \\\notag
& + \sum\limits_{r=1}^{p} (-1)^{r+h+i_1+...+i_{r-1}} a\otimes \wt{a}_1\otimes...\otimes\partial(\wt{a}_r)\otimes
...\otimes\wt{a}_p\otimes m\\\notag
& + (-1)^{h+p+i_1+...+i_p} a\otimes\wt{a}_1\otimes...\otimes\wt{a}_p\otimes \partial(m)
\end{align}
and
\begin{align}\label{eq:bardiff2}
\partial\pp(a\otimes \wt{a}_1\otimes...\otimes\wt{a}_p\otimes m) := &\ 
(-1)^h (aa_1)\otimes \wt{a}_2\otimes ...\otimes\wt{a}_p\otimes m \\\notag
& + \sum\limits_{r=1}^{p-1} (-1)^{r+h+i_1+...+i_{r}} a\otimes \wt{a}_1\otimes...\otimes\wt{a_ra_{r+1}}\otimes
...\otimes\wt{a}_p\otimes m\\\notag
& + (-1)^{h+p+i_1+...+i_{p-1}} a\otimes\wt{a}_1\otimes...\otimes\wt{a}_{p-1}\otimes a_p m
\end{align}
\item The structure map $Q(A\ua\ld,M\ua\ld)\to M\ua\ld$ is defined by
$$a\otimes\wt{a}_1\otimes ...\otimes\wt{a}_p\otimes m\ \ \longmapsto\ \ \begin{cases} 0 & \text{if}\ p>0\\ am &
  \text{if}\ p=0.\end{cases}.$$
\end{enumerate}
\end{definition}
\begin{prop}\label{prop_bar} 
The following hold:
\begin{enumerate}
\item $(Q(A\ua\ld,M\ua\ld),\partial)\to M\ua\ld$ is a quasi-isomorphism.
\item If $M\ua\ld$ is $S\ld$-free, $(Q(A\ua\ld,M\ua\ld),\partial)$ is a semi-free $A\ua\ld$-module.
\end{enumerate}
\end{prop}
\begin{proof}
The first statement follows from the results \cite[Construction 3.1.4]{Avramov} applied to the ungraded dg-algebra
and ring underlying $A\ua\ld$ and $S\ld$, respectively. For the second statement, note that if $M\ua\ld$ is
$S\ld$-free then the submodules $$U_{M\ua\ld}(n)\ :=\ \bigoplus\limits_{p-i_1-...-i_p-j\leq n} A\ua\ld\otimes_{S\ld}
\wt{A}^{i_1}\ld\otimes_{S\ld}\cdots\otimes_{S\ld}\wt{A}^{i_p}\ld\otimes_{S\ld} M^j\ld$$ form a semi-free
filtration of $Q(A\ua\ld, M\ua\ld)$. 
\end{proof}
\subsection{The Bar resolution for the Koszul-resolution of $S\ld/(w)$}\label{app_barres}
\markright{\ref{app_barres} The Bar resolution for the Koszul-resolution of $S\ld/(w)$}

Next, we make the Bar resolution explicit in the case where $A\ua\ld = K\ua_w$ is the Koszul-resolution of $S\ld/(w)$
(see Definition \ref{def_resolutions}). In this case $\wt{K\ua_w}\cong S\ld\langle -d\rangle [1]$, so we get the
following isomorphism of graded $S\ld$-modules, where we consider $S\ld[t]$ as a $\ZZ$-graded graded $S\ld$-module with
$t$ sitting in cohomological degree $-2$ and internal degree $d$. 
\begin{align}\label{eq:simplebar}
\begin{tikzpicture}[description/.style={fill=white,inner sep=2pt},baseline=(current
    bounding box.center)]
    \matrix (m) [matrix of math nodes, row sep=0.3em,
                 column sep=2.5em, text height=1.5ex, text depth=0.25ex]
    {
       Q(K\ua_w,M\ua\ld) \pgfmatrixnextcell \ \cong \pgfmatrixnextcell K\ua_w\otimes_{S\ld} S\ld[t]\otimes_{S\ld} M\ua\ld\\
       (-1)^n a\otimes \underbrace{\wt{s}\otimes...\otimes\wt{s}}_{n\text{ times}}\otimes m\pgfmatrixnextcell
       \longmapsfrom \pgfmatrixnextcell 
       a\otimes t^n\otimes m\\
    };
\end{tikzpicture}
\end{align}
(Note that the left hand side has cohomological degree $|a|-2n+|m|$; see the indexing in \eqref{eq:bardef}) This
isomorphism induces a differential on $K\ua_w\otimes_{S\ld} S\ld[t]\otimes_{S\ld} M\ua\ld$, yielding the following: 
\begin{prop}\label{prop_barreskoszul}
Let $M\ua\ld$ be an $S\ld$-free $K\ua_w$-module, and denote by $S\ld[t]$ a polynomial ring with the indeterminate $t$
sitting in cohomological degree $-2$ and internal degree $d$. Then, the $K\ua_w$-module $K\ua_w\otimes_{S\ld}
S\ld[t]\otimes_{S\ld} M\ua\ld$ with differential given by
\begin{align*}
a\otimes t^n\otimes m\ \ \longmapsto &\ \ \partial(a)\otimes t^n\otimes m + (-1)^{|a|}a\otimes t^n\otimes\partial(m)\\
&\ \ + (-1)^{|a|+1} a s\otimes t^{n-1}\otimes m + (-1)^{|a|}a\otimes t^{n-1}\otimes s m.
\end{align*}
is a semi-free resolution of $M\ua\ld$.
\end{prop}
\begin{proof}
This follows immediately from the isomorphism \eqref{eq:simplebar} and the explicit formula \eqref{eq:bardiff1} and
\eqref{eq:bardiff2} for the differential on the Bar resolution. Note that both the differential and the multiplication
on $\wt{K\ua_w}$ are trivial. 
\end{proof}
Proposition \ref{prop_barreskoszul} allows us to explicitly compute the image of some dg-$K\ua_w$-module under the
derived tensor functor $\D^b(K\ua_w)\to\D^-(S\ld/(w))$:
\begin{cor}\label{cor_barresmod}
Let $M\ua\ld$ be an $S\ld$-free $K\ua_w$-module, and denote by $S\ld/(w)[t]$ a polynomial ring over $S\ld/(w)$ with 
the indeterminate $t$ sitting in cohomological degree $-2$ and internal degree $d$. Then there is a canonical
isomorphism in $\D^b(S\ld/(w))$: $$M\ua\ld\stackrel{\LL}{\otimes}_{K\ua_w} S\ld/(w)\ \ \cong\ \
\left(S\ld/(w)[t]\otimes_{S\ld} M\ua\ld,\partial\right)$$ where $\partial$ is given by
$$\partial(t^n\otimes m)\ \ :=\ \ t^n\otimes\partial(m) + t^{n-1}\otimes s m.$$
\end{cor}
\begin{rem}\label{rem_resolution}
Corollary \ref{cor_barresmod} yields a proof of Proposition \ref{prop_resolution} in case $s_n=0$ for all $n\geq 2$ as 
follows. We start with an $S\ld/(w)$-module $M\ld$ and assume that we have chosen an $S\ld$-free resolution
$F\ua\ld\to M\ld$ together with a homotopy $s$ for the multiplication by $w$ on $F\ua\ld$ such that $s^2=0$. In this
case, the claim of \ref{prop_resolution} is that $S\ld/(w)[t]\otimes_{S\ld} F\ua\ld$ together with the differential
$\id\otimes\partial + t\ua\otimes s$ is an $S\ld/(w)$-free resolution of
$M\ld$. Now, this follows from Corollary \ref{cor_barresmod} by tracing $M\ld$ along the adjoint equivalence
\begin{align}\label{eq:compi}\D^-(S\ld/(w))\xrightarrow{\ \ \kappa_w\ua\ \ }\D^-(K\ua_w)\ \
  \xrightarrow{-\stackrel{\LL}{\otimes}_{K\ua_w}S\ld/(w)}\ \ \D^-(S\ld/(w)).\end{align}
Indeed, considering $F\ua\ld$ together with the homotopy $s$ as a module over $K\ua_w$ (see Fact
\ref{fact_higherhomotopiesasmodulestructure}), it is isomorphic to the image of $M\ld$ under $\kappa_w\ua$. By Corollary  
\ref{cor_barresmod}, $F\ua\ld$ is sent to $S\ld/(w)[t]\otimes_{S\ld} F\ua\ld$ with differential $\id\otimes\partial +
t\ua\otimes s$ under $-\stackrel{\LL}{\otimes}_{K\ua_w} S\ld/(w)$. However, we know a priori that the result has of
this computation has to be an $S\ld/(w)$-free resolution of $M\ld$, as the composition \eqref{eq:compi} is isomorphic to
the identity. 
\end{rem}

\subsection{Connecting the Koszul-resolution to matrix factorizations}\label{sec_bigtheorem}
\markright{\ref{sec_bigtheorem} Connecting the Koszul-resolution to matrix factorizations}

Let $S\ld$ be regular local, $w\in\frm\setminus\{0\}$. We know that
$\D^b_\fg(S\ld/(w))/\perf\cong\HMF(S\ld,w)$ and $\D^b_\fg(S\ld/(w))\cong\D^b_\fg(K\ua_w)$, so we ask what the composed
functor $$\D^b_\fg(K\ua_w)\longrightarrow\D^b_\fg(S\ld/(w))\longrightarrow \HMF(S\ld,w)$$ looks like. It turns out that
it has a nice description which even makes sense for arbitrary $w=0$.
\begin{definition}
We denote $\Perf^\infty\subset\D^{b,b}_\fr(K\ua_w)$ the smallest thick triangulated subcategory of
$\D^{b,b}_\fr(K\ua_w)$ which contains all free $K\ua_w$-modules.
\end{definition}
\begin{prop}\label{prop_fold}Let $S\ld$ be a regular local graded ring and $w\in S\ld$ be
  homogeneous of degree $d$. Then the assignment 
\begin{align}\label{eq:fold}
(F\ua\ld,\differential)\ & \longmapsto\ \left(\bigoplus\limits_{n\in\ZZ} F^{2n}\ld\langle -nd\rangle\xrightarrow{\ \
    \differential + s\ \ }\bigoplus\limits_{n\in\ZZ} F^{2n-1}\ld\langle -nd\rangle\xrightarrow{\ \ \differential +
    s\ \ } \bigoplus\limits_{n\in\ZZ} F^{2n}\ld\langle -nd\rangle\right)
\end{align}
for an $S\ld$-free $K\ua_w$-module $F\ua\ld$ induces triangulated functor
$$\fold:\quad \D^{b,b}_{\fr}(K\ua_w)/\Perf^\infty\ \ \longrightarrow\ \ \HMF^{\infty}(S\ld,w).$$
\end{prop}
\begin{proof}
We use the description of $\D^{b,b}_{\fr}(A\ua\ld)$ given in Proposition
\ref{prop_boundedasverdierquotient}. Thus, to make \eqref{eq:fold} into a triangulated functor
$\D^{b,b}_{\fr}(K\ua_w)/\Perf^\infty\to\HMF^\infty(S\ld,w)$, we have to work through the following steps:  
\begin{enumerate}
\item\label{item:fold1} Define $\fold$ as a functor $K\ua_w\Mod^{b,b}_\fr\to\MF^\infty(S\ld,w)$, i.e. say what
  happens to morphisms. 
\item\label{item:fold2} Check that homotopic morphisms of $K\ua_w$-modules yield homotopic morphisms of matrix
  factorizations. 
\item\label{item:fold3} Verify that the resulting functor $\Ho^{b,b}_{\fr}(K\ua_w)\to\HMF^\infty(S\ld,w)$ carries the
  structure of a triangulated functor, i.e. check that it naturally commutes with the shift and preserves exact
  triangles. 
\item\label{item:fold4} Prove that $\fold$ takes quasi-isomorphisms of $K\ua_w$-modules into homotopy equivalences of
  matrix factorizations. This will define $\fold$ as a functor $\D^{b,b}_{\fr}(K\ua_w)\to\HMF^\infty(S\ld,w)$. 
\item\label{item:fold5} Prove that $\fold$ vanishes on the subcategory $\Perf^\infty(K\ua_w)$ of perfect $K\ua_w$-modules,
  yielding the desired functor $\D^{b,b}_{\fr}(K\ua_w)/\Perf^\infty\to\HMF^\infty(S\ld,w)$.
\end{enumerate}
\eqref{item:fold1}: If $\varphi := (\varphi_n)_{n\in\ZZ}: P\ua\ld\to Q\ua\ld$ is a homomorphism of $S\ld$-free
$K\ua_w$-modules, define $\fold(\varphi)$ by
\begin{equation*}\begin{tikzpicture}[baseline=(current bounding  box.center), description/.style={fill=white,inner sep=2pt}]
    \matrix (m) [matrix of math nodes, row sep=5em,
                 column sep=4.5em, text height=1.5ex, text depth=0.25ex]
    {
       \bigoplus\limits_{n\in\ZZ} P^{2n}\ld\langle -nd\rangle & \bigoplus\limits_{n\in\ZZ}
           P^{2n-1}\ld\langle -nd\rangle & \bigoplus\limits_{n\in\ZZ} P^{2n}\ld\langle -nd\rangle\\
       \bigoplus\limits_{n\in\ZZ} P^{2n}\ld\langle -nd\rangle & \bigoplus\limits_{n\in\ZZ}
           Q^{2n-1}\ld\langle -nd\rangle & \bigoplus\limits_{n\in\ZZ} Q^{2n}\ld\langle -nd\rangle\\
    };
    \draw[->] (m-1-1) -- node[scale=0.75,description]{$\left(\varphi_{2n}\right)_{n\in\ZZ}$}   (m-2-1);
    \draw[->] (m-1-2) -- node[scale=0.75,description]{$\left(\varphi_{2n-1}\right)_{n\in\ZZ}$} (m-2-2);
    \draw[->] (m-1-3) -- node[scale=0.75,description]{$\left(\varphi_{2n}\right)_{n\in\ZZ}$}   (m-2-3);
    \draw[->] (m-1-1) -- node[scale=0.75,above] {$\differential+s$} (m-1-2);
    \draw[->] (m-1-2) -- node[scale=0.75,above] {$\differential+s$} (m-1-3);
    \draw[->] (m-2-1) -- node[scale=0.75,above] {$\differential+s$} (m-2-2);
    \draw[->] (m-2-2) -- node[scale=0.75,above] {$\differential+s$} (m-2-3);
\end{tikzpicture}\end{equation*}
This is a morphism of matrix factorizations since the $\varphi_n$ are degree-preserving and we have $\differential\varphi_n =
\varphi_{n+1}\differential$ and $s\varphi_n = \varphi_{n-1} s$ by definition of a morphism of dg-$K\ua_w$-modules.

\eqref{item:fold2}: Assume $\psi := (\psi_n)_{n\in\ZZ}$ is another morphism of $K\ua_w$-modules homotopic to
$\varphi$. Then, by definition of the homotopy relation, there exists a family of degree-preserving maps $D_n:
P\ua\ld\to P^{\ast-1}\ld$ such that $s D_n = - D_{n-1} s$ and $\differential D_n + D_{n+1} \differential = \varphi_n - \psi_n$ for all
$n\in\ZZ$. This yields a homotopy between $\fold(\varphi)$ and $\fold(\psi)$ as follows:
\begin{equation*}\begin{tikzpicture}[baseline=(current bounding  box.center), description/.style={fill=white,inner sep=2pt}]
    \matrix (m) [matrix of math nodes, row sep=5em,
                 column sep=4.5em, text height=1.5ex, text depth=0.25ex]
    {
       \bigoplus\limits_{n\in\ZZ} P^{2n}\ld\langle -nd\rangle & \bigoplus\limits_{n\in\ZZ}
           P^{2n-1}\ld\langle -nd\rangle & \bigoplus\limits_{n\in\ZZ} P^{2n}\ld\langle -nd\rangle\\
       \bigoplus\limits_{n\in\ZZ} P^{2n}\ld\langle -nd\rangle & \bigoplus\limits_{n\in\ZZ}
           Q^{2n-1}\ld\langle -nd\rangle & \bigoplus\limits_{n\in\ZZ} Q^{2n}\ld\langle -nd\rangle\\
    };
    \draw[->] ($(m-1-1.south) + (-1.5mm,0)$) -- node[scale=0.75,description]{$\varphi$} ($(m-2-1.north) + (-1.5mm,0)$);
    \draw[->] ($(m-1-1.south) + (+1.5mm,0)$) -- node[scale=0.75,description]{$\psi$}    ($(m-2-1.north) + (+1.5mm,0)$);
    \draw[->] ($(m-1-2.south) + (-1.5mm,0)$) -- node[scale=0.75,description]{$\varphi$} ($(m-2-2.north) + (-1.5mm,0)$);
    \draw[->] ($(m-1-2.south) + (+1.5mm,0)$) -- node[scale=0.75,description]{$\psi$}    ($(m-2-2.north) + (+1.5mm,0)$);
    \draw[->] ($(m-1-3.south) + (-1.5mm,0)$) -- node[scale=0.75,description]{$\varphi$} ($(m-2-3.north) + (-1.5mm,0)$);
    \draw[->] ($(m-1-3.south) + (+1.5mm,0)$) -- node[scale=0.75,description]{$\psi$}    ($(m-2-3.north) + (+1.5mm,0)$);
    \draw[->] (m-1-1) -- node[scale=0.75,above] {$\differential+s$} (m-1-2);
    \draw[->] (m-1-2) -- node[scale=0.75,above] {$\differential+s$} (m-1-3);
    \draw[->] (m-2-1) -- node[scale=0.75,below] {$\differential+s$} (m-2-2);
    \draw[->] (m-2-2) -- node[scale=0.75,below] {$\differential+s$} (m-2-3);
    \draw[dashed,->] (barycentric cs:m-1-2=0.75,m-2-1=0.25) -- node[scale=0.75,description]{$(D_{2n-1})_{n\in\ZZ}$}
    (barycentric cs:m-1-2=0.2,m-2-1=0.8);
    \draw[dashed,->] (barycentric cs:m-1-3=0.75,m-2-2=0.25) -- node[scale=0.75,description]{$(D_{2n})_{n\in\ZZ}$}
    (barycentric cs:m-1-3=0.2,m-2-2=0.8);
\end{tikzpicture}\end{equation*}
Note that the degree shifts in $\fold(P\ua\ld)$ and $\fold(Q\ua\ld)$ cause $(D_{2n})_{n\in\ZZ}$ to preserve and
$(D_{2n-1})_{n\in\ZZ}$ to raise the degree by $d$. 

\eqref{item:fold3}: For some bounded, $S\ld$-free $K\ua_w$-module $F\ua\ld$ we have the following:
\begin{align*}
\fold(F\ua\ld[1]) & = \left(\bigoplus\limits_{n\in\ZZ} F^{2n+1}\ld\langle -nd\rangle\xrightarrow{-\partial -
    s} \bigoplus\limits_{n\in\ZZ} F^{2n}\ld\langle -nd\rangle\xrightarrow{-\partial -
    s} \bigoplus\limits_{n\in\ZZ} F^{2n+1}\ld\langle -nd\rangle\right) \\
& = \left(\bigoplus\limits_{n\in\ZZ} F^{2n}\ld\langle -nd\rangle\xrightarrow{\partial +
    s} \bigoplus\limits_{n\in\ZZ} F^{2n-1}\ld\langle -nd\rangle\xrightarrow{\partial +
    s} \bigoplus\limits_{n\in\ZZ} F^{2n}\ld\langle -nd\rangle\right)[1]\\
& = \fold(F\ua\ld)[1].
\end{align*}
Note that the action of $s$ on $F\ua\ld[1]$ is the negative of the action on $F\ua\ld$, because the $K\ua_w$-module
structure on $F\ua\ld[1]$ is given by the composition $$K\ua_w\otimes_{S\ld} F\ua\ld[1]\xrightarrow{\ \ \cong\ \ }
\left(K\ua_w\otimes_{S\ld} F\ua\ld\right)[1]\xrightarrow{\quad\quad} K\ua_w[1],$$
where the first isomorphism is given by $a\otimes b\mapsto (-1)^{|a|}a\otimes b$, hence involves the required
sign. This shows that $\fold$ commutes with the shift functor. It remains to be checked that it preserves exact
triangles (see \eqref{eq:defcone} for the definition of the cone of a morphism between matrix factorizations). Given a
morphism $\varphi: P\ua\ld\to Q\ua\ld$ of bounded, $S\ld$-free $K\ua_w$-modules, we have
\begin{align*}
\fold(\cone(\varphi))^0\ld\makebox[8mm][c]{$=$} & \bigoplus\limits_{n\in\ZZ} \cone(\varphi)^{2n}\ld\langle -nd\rangle\\
\makebox[8mm][c]{$=$} & \bigoplus\limits_{n\in\ZZ} \left(Q^{2n}\ld\oplus P^{2n+1}\ld\right)\langle -nd\rangle\\
\makebox[8mm][c]{$=$} & \left(\bigoplus\limits_{n\in\ZZ} Q^{2n}\ld\langle
  -nd\rangle\right)\oplus\left(\bigoplus\limits_{n\in\ZZ} P^{2n+1}\ld\langle -(n+1)d\rangle\right)\langle d\rangle\\
\makebox[8mm][c]{$=$} & \cone(\fold(\varphi))^0\ld\\
\fold(\cone(\varphi))^{-1}\ld\makebox[8mm][c]{$=$} & \bigoplus\limits_{n\in\ZZ} \cone(\varphi)^{2n-1}\ld\langle -nd\rangle\\
\makebox[8mm][c]{$=$} & \bigoplus\limits_{n\in\ZZ} \left(Q^{2n-1}\ld\oplus P^{2n}\ld\right)\langle -nd\rangle\\
\makebox[8mm][c]{$=$} & \left(\bigoplus\limits_{n\in\ZZ} Q^{2n-1}\ld\langle
  -nd\rangle\right)\oplus\left(\bigoplus\limits_{n\in\ZZ} P^{2n}\ld\langle -nd\rangle\right)\\
\makebox[8mm][c]{$=$} & \cone(\fold(\varphi))^{-1}\ld,
\end{align*}
which shows that $\fold(\cone(\varphi))\cong\cone(\fold(\varphi))$ as $\ZZ/2\ZZ$-graded $S\ld$-modules. The proof that
this identification is compatible with the differentials on both sides is omitted.

\eqref{item:fold4}: By \eqref{item:fold3} we know that $\fold$ is a triangulated functor
$\Ho^{b,b}_\fr(K\ua_w)\to\HMF^\infty(S\ld,w)$. As the cone of a quasi-isomorphism is acyclic, the claim that
quasi-isomorphisms are mapped to  homotopy equivalences is therefore equivalent to the following: If $F\ua\ld$ is a
bounded, $S\ld$-free $K\ua_w$-module with vanishing cohomology, then $\fold(P\ua\ld)$ is contractible. This will follow
from direct calculation; it would be nice to have a more conceptual proof at hand. 

To prove that $\fold(F\ua\ld)$ is zero in $\HMF^\infty(S\ld,w)$, we have to construct a null-homotopy for the identity
on $\fold(F\ua\ld)$:  
\begin{equation*}\begin{tikzpicture}[baseline=(current bounding  box.center), description/.style={fill=white,inner sep=2pt}]
    \matrix (m) [matrix of math nodes, row sep=5em,
                 column sep=4.5em, text height=1.5ex, text depth=0.25ex]
    {
       \bigoplus\limits_{n\in\ZZ} F^{2n}\ld\langle -nd\rangle & \bigoplus\limits_{n\in\ZZ}
           F^{2n-1}\ld\langle -nd\rangle & \bigoplus\limits_{n\in\ZZ} F^{2n}\ld\langle -nd\rangle\\
       \bigoplus\limits_{n\in\ZZ} F^{2n}\ld\langle -nd\rangle & \bigoplus\limits_{n\in\ZZ}
           F^{2n-1}\ld\langle -nd\rangle & \bigoplus\limits_{n\in\ZZ} F^{2n}\ld\langle -nd\rangle\\
    };
    \draw[double distance=0.7mm] (m-1-1) -- (m-2-1.north);
    \draw[double distance=0.7mm] (m-1-2) -- (m-2-2.north);
    \draw[double distance=0.7mm] (m-1-3) -- (m-2-3.north);
    \draw[->] (m-1-1) -- node[scale=0.75,above] {$\differential+s$} (m-1-2);
    \draw[->] (m-1-2) -- node[scale=0.75,above] {$\differential+s$} (m-1-3);
    \draw[->] (m-2-1) -- node[scale=0.75,below] {$\differential+s$} (m-2-2);
    \draw[->] (m-2-2) -- node[scale=0.75,below] {$\differential+s$} (m-2-3);
    \draw[dashed,->] (barycentric cs:m-1-2=0.75,m-2-1=0.25) -- node[scale=0.75,description]{$\tilde{s}$}
    (barycentric cs:m-1-2=0.2,m-2-1=0.8);
    \draw[dashed,->] (barycentric cs:m-1-3=0.75,m-2-2=0.25) -- node[scale=0.75,description]{$\tilde{s}$}
    (barycentric cs:m-1-3=0.2,m-2-2=0.8);
\end{tikzpicture}\end{equation*}
The condition that $\tilde{s}$ is a null-homotopy for the identity means that
\begin{equation}\label{eq:hmftriv}
\id_{\fold(F\ua\ld)} \ \ = \ \ \tilde{s}\differential + \tilde{s}s + s\tilde{s} + \differential\tilde{s}.
\end{equation}
Let us pause for a moment and compare this to the datum of a contraction of the $K\ua_w$-module $F\ua\ld$ before being
folded. There, a contraction is a diagram
\begin{equation*}\begin{tikzpicture}[baseline=(current bounding  box.center), description/.style={fill=white,inner sep=2pt}]
    \matrix (m) [matrix of math nodes, row sep=4em,
                 column sep=4em, text height=1.5ex, text depth=0.25ex]
    {
       ... & F^{n-1}\ld & F^n\ld & F^{n+1}\ld & ... \\
       ... & F^{n-1}\ld & F^n\ld & F^{n+1}\ld & ... \\
    };
    \draw[double distance=0.7mm] (m-1-2) -- (m-2-2);
    \draw[double distance=0.7mm] (m-1-3) -- (m-2-3);
    \draw[double distance=0.7mm] (m-1-4) -- (m-2-4);
 
    \draw[dashed,->] (barycentric cs:m-1-3=0.75,m-2-2=0.25) -- node[scale=0.75,description]{$\tilde{s}$} (barycentric
    cs:m-1-3=0.25,m-2-2=0.75); 
    \draw[dashed,->] (barycentric cs:m-1-2=0.75,m-2-1=0.25) -- node[scale=0.75,description]{$\tilde{s}$} (barycentric
    cs:m-1-2=0.25,m-2-1=0.75); 
    \draw[dashed,->] (barycentric cs:m-1-5=0.75,m-2-4=0.25) -- node[scale=0.75,description]{$\tilde{s}$} (barycentric
    cs:m-1-5=0.25,m-2-4=0.75); 
    \draw[dashed,->] (barycentric cs:m-1-4=0.75,m-2-3=0.25) -- node[scale=0.75,description]{$\tilde{s}$} (barycentric
    cs:m-1-4=0.25,m-2-3=0.75);

    \draw[->] ($(m-1-1.east) + (0,+0.5mm)$) -- node[above,scale=0.75]{$\differential$} ($(m-1-2.west) + (0,+0.5mm)$);
    \draw[->] ($(m-1-2.west) + (0,-0.5mm)$) -- node[below,scale=0.75]{$s$}  ($(m-1-1.east) + (0,-0.5mm)$);
    \draw[->] ($(m-1-2.east) + (0,+0.5mm)$) -- node[above,scale=0.75]{$\differential$} ($(m-1-3.west) + (0,+0.5mm)$);
    \draw[->] ($(m-1-3.west) + (0,-0.5mm)$) -- node[below,scale=0.75]{$s$}  ($(m-1-2.east) + (0,-0.5mm)$);
    \draw[->] ($(m-1-3.east) + (0,+0.5mm)$) -- node[above,scale=0.75]{$\differential$} ($(m-1-4.west) + (0,+0.5mm)$);
    \draw[->] ($(m-1-4.west) + (0,-0.5mm)$) -- node[below,scale=0.75]{$s$}  ($(m-1-3.east) + (0,-0.5mm)$);
    \draw[->] ($(m-1-4.east) + (0,+0.5mm)$) -- node[above,scale=0.75]{$\differential$} ($(m-1-5.west) + (0,+0.5mm)$);
    \draw[->] ($(m-1-5.west) + (0,-0.5mm)$) -- node[below,scale=0.75]{$s$}  ($(m-1-4.east) + (0,-0.5mm)$);

    \draw[->] ($(m-2-1.east) + (0,+0.5mm)$) -- node[above,scale=0.75]{$\differential$} ($(m-2-2.west) + (0,+0.5mm)$);
    \draw[->] ($(m-2-2.west) + (0,-0.5mm)$) -- node[below,scale=0.75]{$s$}  ($(m-2-1.east) + (0,-0.5mm)$);
    \draw[->] ($(m-2-2.east) + (0,+0.5mm)$) -- node[above,scale=0.75]{$\differential$} ($(m-2-3.west) + (0,+0.5mm)$);
    \draw[->] ($(m-2-3.west) + (0,-0.5mm)$) -- node[below,scale=0.75]{$s$}  ($(m-2-2.east) + (0,-0.5mm)$);
    \draw[->] ($(m-2-3.east) + (0,+0.5mm)$) -- node[above,scale=0.75]{$\differential$} ($(m-2-4.west) + (0,+0.5mm)$);
    \draw[->] ($(m-2-4.west) + (0,-0.5mm)$) -- node[below,scale=0.75]{$s$}  ($(m-2-3.east) + (0,-0.5mm)$);
    \draw[->] ($(m-2-4.east) + (0,+0.5mm)$) -- node[above,scale=0.75]{$\differential$} ($(m-2-5.west) + (0,+0.5mm)$);
    \draw[->] ($(m-2-5.west) + (0,-0.5mm)$) -- node[below,scale=0.75]{$s$}  ($(m-2-4.east) + (0,-0.5mm)$);
\end{tikzpicture}\end{equation*}
such that
\begin{equation}\label{eq:dgmodtriv}
\differential\tilde{s} + \tilde{s}\differential = \id_{F\ua\ld}\quad\quad\text{and}\quad\quad\tilde{s} s = -s\tilde{s}.
\end{equation}
Thus, condition \eqref{eq:hmftriv} is weaker than \eqref{eq:dgmodtriv} in two respects: Firstly, in \eqref{eq:hmftriv} we
only impose a condition on the sum $\tilde{s}\differential + \tilde{s}s + s\tilde{s} + \differential\tilde{s}$, while in \eqref{eq:dgmodtriv} we impose
conditions on the summands $\tilde{s}\differential+\differential\tilde{s}$ and $s\tilde{s}+\tilde{s}s$. Secondly, in \eqref{eq:hmftriv} the map $\tilde{s}$ is
allows to consist of a whole family of maps $\tilde{s}_n: F\ua\ld\to F^{\ast-(2n-1)}\ld$ for all $n\in\ZZ$, while in
\eqref{eq:dgmodtriv} the map $\tilde{s}$ is of fixed cohomological degree $-1$. 

Back to the proof of \eqref{item:fold4}. The homotopy $\tilde{s}$ we construct will only involve $\tilde{s}_n: F\ua\ld\to
F^{\ast-(2n-1)}\ld$ for $n\geq 1$. Condition \eqref{eq:hmftriv} can then be rewritten as
\begin{equation}
\differential\tilde{s}_1 + \tilde{s}_1\differential\ =\ \id_{F\ua\ld},
\end{equation}
i.e. $\tilde{s}_1$ is a contraction of $F\ua\ld$ as a complex of $S\ld$-modules, and
\begin{equation}\label{eq:zerorel}
s\tilde{s}_n + \tilde{s}_n s + \differential\tilde{s}_{n+1} + \tilde{s}_{n+1}\differential\ =\ 0
\end{equation}
for all $n\geq 1$. We construct the maps $\tilde{s}_n$ inductively. Start by taking $\tilde{s}_1$ to be an arbitrary contraction
of $F\ua\ld$ as a complex of $S\ld$-modules; such a contraction exists since $F\ua\ld$ is $S\ld$-free, bounded above
and acyclic. Now assume that we already constructed $\tilde{s}_1,...,\tilde{s}_n$ satisfying \eqref{eq:zerorel}. Since
$F\ua\ld$ is contractible, the morphism complex $\hom_{S\ld}\ua(F\ua\ld,F\ua\ld)$ is acyclic, and so the existence of
$\tilde{s}_{n+1}$ satisfying \eqref{eq:zerorel} is equivalent the fact that $s\tilde{s}_n + \tilde{s}_n s$ is a cycle in
$\hom_{S\ld}\ua(F\ua\ld,F\ua\ld)$. Denoting by $\partial$ the differential of this complex, this follows from a direct
calculation, using $\partial(s) = w$, $\partial(\tilde{s}_n) = -(s\tilde{s}_{n-1} + \tilde{s}_{n-1} s)$ and the fact that $\partial$
satisfies the Leibniz-rule:
\begin{equation*}
\partial(s\tilde{s}_n + \tilde{s}_n s) = w\tilde{s}_n + s(s\tilde{s}_{n-1} + \tilde{s}_{n-1}s) - (\tilde{s}_{n-1} s + s\tilde{s}_{n-1} s) s -
w\tilde{s}_n = 0.
\end{equation*}
This finishes the inductive construction of the $\tilde{s}_n$ and establishes, in the whole, a contraction $\tilde{s}$ of
$\fold(F\ua\ld)$.

\eqref{item:fold5}: Finally we check that $\fold$ vanishes in perfect $K\ua_w$-modules. By definition, $\Perf(K\ua_w)$
is the smallest thick triangulated subcategory of $\D^{b,b}_{\fr}(K\ua_w)$ which containing all free
$K\ua_w$-modules. As we already know that $\fold$ is triangulated and commutes with internal degree shifts, it is
therefore sufficient to show that $\fold(K\ua_w)=0$. However, $$\fold(K\ua_w)\ \ =\ \ \left(S\ld\xrightarrow{\id}
  S\ld\langle 
  -d\rangle\xrightarrow{w}S\ld\right)$$ 
which vanishes since the dashed arrows in 
\begin{equation*}\begin{tikzpicture}[description/.style={fill=white,inner sep=2pt}]
    \matrix (m) [matrix of math nodes, row sep=3em,
                 column sep=2.5em, text height=1.5ex, text depth=0.25ex]
    {
       S\ld& S\ld\langle -d\rangle &S\ld\\
       S\ld& S\ld\langle -d\rangle &S\ld\\
    };
    \draw[double distance=0.7mm] (barycentric cs:m-1-1=0.7,m-2-1=0.3) -- (barycentric cs:m-1-1=0.3,m-2-1=0.7);
    \draw[double distance=0.7mm] (barycentric cs:m-1-2=0.7,m-2-2=0.3) -- (barycentric cs:m-1-2=0.3,m-2-2=0.7);
    \draw[double distance=0.7mm] (barycentric cs:m-1-3=0.7,m-2-3=0.3) -- (barycentric cs:m-1-3=0.3,m-2-3=0.7);
    \draw[->] (m-1-1) -- node[scale=0.75,above]{\id} (m-1-2);
    \draw[->] (m-1-2) -- node[scale=0.75,above]{w} (m-1-3);
    \draw[->] (m-2-1) -- node[scale=0.75,below]{\id} (m-2-2);
    \draw[->] (m-2-2) -- node[scale=0.75,below]{w} (m-2-3);
    \draw[->,dashed] (barycentric cs:m-1-2=0.8,m-2-1=0.2) -- node[scale=0.75,description]{$\id$} (barycentric
    cs:m-1-2=0.2,m-2-1=0.8); 
    \draw[->,dashed] (barycentric cs:m-1-3=0.8,m-2-2=0.2) -- node[scale=0.75,description]{$0$} (barycentric
    cs:m-1-3=0.2,m-2-2=0.8); 
\end{tikzpicture}\end{equation*}
constitute a nullhomotopy for the identity on $\left(S\ld\xrightarrow{\id} S\ld\langle
  -d\rangle\xrightarrow{w}S\ld\right)$. 
\end{proof}

\begin{rem}
Note that the proof of Proposition \ref{prop_fold} also works for bounded below, $S\ld$-free $K\ua_w$-modules if we replace
infinite sums by infinite products. However, this forces us to consider matrix factorizations with non-free entries
(these are called \textit{duplexes} in \cite{KR1}), and the author doesn't know how to think about them. 
\end{rem}

Now we go in the other direction:
\begin{prop}\label{prop_iota}
Let $S\ld$ be a regular local graded ring and $w\in S\ld$ be homogeneous of degree $d$. Then the assignment
\begin{align}\label{eq:iota}
\begin{tikzpicture}[description/.style={fill=white,inner sep=2pt}, baseline=(m-1-1.base)]
    \matrix (m) [matrix of math nodes, row sep=3em,
                 column sep=1em, text height=1.5ex, text depth=0.25ex,
                 inner sep=0pt, nodes={inner xsep=0.3333em, inner ysep=0.3333em}]
    {
       M^0\ld \pgfmatrixnextcell\pgfmatrixnextcell M^{-1}\ld \pgfmatrixnextcell\pgfmatrixnextcell M^0\ld
       \pgfmatrixnextcell\longmapsto \pgfmatrixnextcell ... \pgfmatrixnextcell 0 \pgfmatrixnextcell M^{-1}\ld
       \pgfmatrixnextcell\pgfmatrixnextcell M^0\ld \pgfmatrixnextcell 0 
       \pgfmatrixnextcell ...\\ 
    };
    \draw[->] ($(m-1-9.east) + (0,0.5mm)$) -- node[scale=0.75,above]{$g$} ($(m-1-11.west) + (0,0.5mm)$);
    \draw[->] ($(m-1-11.west) - (0,0.5mm)$) -- node[scale=0.75,below]{$f$} ($(m-1-9.east) - (0,0.5mm)$);    
    \draw[->] (m-1-1) -- node[above,scale=0.75]{$f$} (m-1-3);
    \draw[->] (m-1-3) -- node[above,scale=0.75]{$g$} (m-1-5);
    \draw[->] (m-1-7) -- (m-1-8);
    \draw[->] (m-1-8) -- (m-1-9);
    \draw[->] (m-1-11) -- (m-1-12);
    \draw[->] (m-1-12) -- (m-1-13);
\end{tikzpicture}\end{align}
(with $M^0\ld$ concentrated in cohomological degree $0$) induces a functor
$$\iota: \HMF^{\infty}(S\ld,w)\ \ \longrightarrow\ \ \D^{b,b}_{\fr}(K\ua_w)/\perf^\infty$$
which is right inverse to $\fold$.
\end{prop}
\begin{rem}
We do not claim here that $\iota$ is triangulated. However, we will see later in Theorem \ref{thm_bigtheorem} that
$\iota$ and $\fold$ are actually mutually inverse equivalences of categories. Since $\fold$ \textit{is} triangulated,
this gives the triangulated structure on $\iota$ for free.
\end{rem}

\begin{proof}[of Proposition \ref{prop_iota}]
Similar to the proof of Proposition \ref{prop_fold} we will proceed along the following steps:
\begin{enumerate}
\item\label{item:iota1} Extend \eqref{eq:iota} to a functor $\MF^\infty(S\ld,w)\to K\ua_w\mod^{b,b}_{\fr}$, i.e. say
  what happens to morphisms. 
\item\label{item:iota2} Check that homotopic morphisms of matrix factorizations give rise to equal morphisms in the
  stabilized derived category $\D^{b,b}_{\fr}(K\ua_w)/\perf^\infty$. This yields a functor $\iota:
  \HMF^\infty(S\ld,w)\to\D^{b,b}_{\fr}(K\ua_w)/\perf^\infty$.
\end{enumerate}
Once this is done, the claim follows because $\fold\circ\iota$ equals the identity on $\HMF^\infty(S\ld,w)$. 

\eqref{item:iota1}: Given a morphism 
\begin{equation*}\begin{tikzpicture}[description/.style={fill=white,inner sep=2pt}]
    \matrix (m) [matrix of math nodes, row sep=3em,
                 column sep=2.5em, text height=1.5ex, text depth=0.25ex,
                 inner sep=0pt, nodes={inner xsep=0.3333em, inner ysep=0.3333em}]
    {
       M\ld && M^0\ld & M^{-1}\ld & M^0\ld\\ N\ld && N^0\ld & N^{-1}\ld & N^0\ld\\
    };
    \draw[->] (m-1-1) -- node[scale=0.75,description]{$(\alpha,\beta)$} (m-2-1);

    \draw[->] (m-1-3) -- node[scale=0.75,above]{$f$}   (m-1-4);
    \draw[->] (m-2-3) -- node[scale=0.75,below]{$f\p$} (m-2-4);
    \draw[->] (m-1-4) -- node[scale=0.75,above]{$g$}   (m-1-5);
    \draw[->] (m-2-4) -- node[scale=0.75,below]{$g\p$} (m-2-5);
    
    \draw[->] (m-1-3) -- node[scale=0.75,description]{$\alpha$} (m-2-3);
    \draw[->] (m-1-4) -- node[scale=0.75,description]{$\beta$}  (m-2-4);
    \draw[->] (m-1-5) -- node[scale=0.75,description]{$\alpha$} (m-2-5);
\end{tikzpicture}\end{equation*}
of graded matrix factorizations, we define $\iota(\alpha,\beta)$ to be
\begin{equation*}\begin{tikzpicture}[description/.style={fill=white,inner sep=2pt}]
    \matrix (m) [matrix of math nodes, row sep=3em,
                 column sep=2.5em, text height=1.5ex, text depth=0.25ex,
                 inner sep=0pt, nodes={inner xsep=0.3333em, inner ysep=0.3333em}]
    {
       ... & 0 & M^{-1}\ld & M^0\ld & 0 & ...\\
       ... & 0 & N^{-1}\ld & N^0\ld & 0 & ...\\
    };
    \draw[->] ($(m-1-3.east) + (0,0.5mm)$) -- node[scale=0.75,above]{$g$} ($(m-1-4.west) + (0,0.5mm)$);
    \draw[->] ($(m-1-4.west) - (0,0.5mm)$) -- node[scale=0.75,below]{$f$} ($(m-1-3.east) - (0,0.5mm)$);
    \draw[->] ($(m-2-3.east) + (0,0.5mm)$) -- node[scale=0.75,above]{$g\p$} ($(m-2-4.west) + (0,0.5mm)$);
    \draw[->] ($(m-2-4.west) - (0,0.5mm)$) -- node[scale=0.75,below]{$f\p$} ($(m-2-3.east) - (0,0.5mm)$);

    \draw[->] (m-1-1) -- (m-1-2);
    \draw[->] (m-1-2) -- (m-1-3);
    \draw[->] (m-1-4) -- (m-1-5);
    \draw[->] (m-1-5) -- (m-1-6);

    \draw[->] (m-2-1) -- (m-2-2);
    \draw[->] (m-2-2) -- (m-2-3);
    \draw[->] (m-2-4) -- (m-2-5);
    \draw[->] (m-2-5) -- (m-2-6);

    \draw[->] (m-1-2) -- (m-2-2);
    \draw[->] (m-1-3) -- node[scale=0.75,description]{$\alpha$} (m-2-3);
    \draw[->] (m-1-4) -- node[scale=0.75,description]{$\beta$}  (m-2-4);
    \draw[->] (m-1-5) -- (m-2-5);
\end{tikzpicture}\end{equation*}
It is clear that this extends \eqref{eq:iota} to a functor $\MF^\infty(S\ld,w)\to K\ua_w\Mod^{b,b}_\fr$. 

\eqref{item:iota2}: Assume we have two morphisms $(\alpha,\beta)$ and $(\gamma,\delta)$ of graded matrix
factorizations, and suppose they are homotopic through a homotopy $D = (D^0,D^1)$: 
\begin{equation}\label{eq:iota2mfhtp}\begin{tikzpicture}[description/.style={fill=white,inner sep=2pt},
    baseline=(current bounding box.center)]
    \matrix (m) [matrix of math nodes, row sep=3em,
                 column sep=2.5em, text height=1.5ex, text depth=0.25ex,
                 inner sep=0pt, nodes={inner xsep=0.3333em, inner ysep=0.3333em}]
    {
       M^0\ld & M^{-1}\ld & M^0\ld\\ N^0\ld & N^{-1}\ld & N^0\ld\\
    };
    \draw[->] (m-1-1) -- node[scale=0.75,above]{$f$}   (m-1-2);
    \draw[->] (m-2-1) -- node[scale=0.75,below]{$f\p$} (m-2-2);
    \draw[->] (m-1-2) -- node[scale=0.75,above]{$g$}   (m-1-3);
    \draw[->] (m-2-2) -- node[scale=0.75,below]{$g\p$} (m-2-3);
 
    \draw[->] ($(m-1-1.south) + (-1mm,0)$) -- node[scale=0.75,left]{$\alpha$} ($(m-2-1.north) + (-1mm,0)$);
    \draw[->] ($(m-1-2.south) + (-1mm,0)$) -- node[scale=0.75,left]{$\beta$}  ($(m-2-2.north) + (-1mm,0)$);
    \draw[->] ($(m-1-3.south) + (-1mm,0)$) -- node[scale=0.75,left]{$\alpha$} ($(m-2-3.north) + (-1mm,0)$);

    \draw[->] ($(m-1-1.south) + (+1mm,0)$) -- node[scale=0.75,right]{$\gamma$} ($(m-2-1.north) + (+1mm,0)$);
    \draw[->] ($(m-1-2.south) + (+1mm,0)$) -- node[scale=0.75,right]{$\delta$} ($(m-2-2.north) + (+1mm,0)$);
    \draw[->] ($(m-1-3.south) + (+1mm,0)$) -- node[scale=0.75,right]{$\gamma$} ($(m-2-3.north) + (+1mm,0)$);

    \draw[dashed,->] (barycentric cs:m-1-2=0.75,m-2-1=0.25) -- node[scale=0.75,description]{$D^1$}
    (barycentric cs:m-1-2=0.2,m-2-1=0.8);
    \draw[dashed,->] (barycentric cs:m-1-3=0.75,m-2-2=0.25) -- node[scale=0.75,description]{$D^0$}
    (barycentric cs:m-1-3=0.2,m-2-2=0.8);
\end{tikzpicture}\end{equation}
We have to show that $\iota(\alpha,\beta)$ and $\iota(\gamma,\delta)$ are equal in
$\D^{b,b}_{\fr}(K\ua_w)/\perf^\infty$. As $$\D^{b,b}_{\fr}(K\ua_w)/\perf^\infty\ \stackrel{\cong}{\longrightarrow}\
\D^{-,b}_{\fr}(K\ua_w)/\perf^\infty$$ is an equivalence (Proposition \ref{prop_freetruncation}), it suffices to prove
that $\iota(\alpha,\beta)=\iota(\gamma,\delta)$ in $\D^{-,b}_{\fr}(K\ua_w)/\perf^\infty$. For this we will show that the
difference of the two upper horizontal maps in the $\D^{-,b}_{\fr}(K\ua_w)$-diagram 
\begin{equation*}\begin{tikzpicture}[description/.style={fill=white,inner sep=2pt}]
    \matrix (m) [matrix of math nodes, row sep=3em,
                 column sep=4em, text height=1.5ex, text depth=0.25ex,
                 inner sep=0pt, nodes={inner xsep=0.3333em, inner ysep=0.3333em}]
    {
      Q(K\ua_w,\iota(M\ld)) && Q(K\ua_w,\iota(N\ld)) \\
      \iota(M\ld) && \iota(N\ld)\\
    };

    \draw[->] ($(m-1-1.east) + (0,0.5mm)$) -- node[scale=0.75,above]{$Q(K\ua_w,\iota(\alpha,\beta))$} ($(m-1-3.west) +
    (0,0.5mm)$);     
    \draw[->] ($(m-1-1.east) - (0,0.5mm)$) -- node[scale=0.75,below]{$Q(K\ua_w,\iota(\gamma,\delta))$} ($(m-1-3.west) -
    (0,0.5mm)$);     

    \draw[->] (m-1-1) -- node[scale=0.75,description]{qis} (m-2-1);
    \draw[->] (m-1-3) -- node[scale=0.75,description]{qis} (m-2-3);

    \draw[->] ($(m-2-1.east) + (0,0.5mm)$) -- node[scale=0.75,above]{$\iota(\alpha,\beta)$} ($(m-2-3.west) +
    (0,0.5mm)$);     
    \draw[->] ($(m-2-1.east) - (0,0.5mm)$) -- node[scale=0.75,below]{$\iota(\gamma,\delta)$} ($(m-2-3.west) -
    (0,0.5mm)$);     
\end{tikzpicture}\end{equation*}
is homotopic, as a morphism of $K\ua_w$-modules, to a map factoring through a perfect $K\ua_w$-module. Since the vertical
maps are isomorphisms in $\D^{-,b}_{\fr}(K\ua_w)$, and since morphisms of $K\ua_w$-modules factoring through a perfect
module up to homotopy are zero in $\D^{-,b}_{\fr}(K\ua_w)/\perf^\infty$, this will prove that
$\iota(\alpha,\beta)=\iota(\gamma,\delta)$ in $\D^{-,b}(K\ua_w)/\perf^\infty$ as claimed.

The Bar resolution of $\iota(M\ld)$ is explicitly given as follows, see \ref{prop_barreskoszul}; to save space, we do
not keep track of internal gradings for the rest of the proof.
\begin{equation*}\begin{tikzpicture}[description/.style={fill=white,inner sep=2pt}]
    \matrix (m) [matrix of math nodes, row sep=3em,
                 column sep=4em, text height=1.5ex, text depth=0.25ex,
                 inner sep=0pt, nodes={inner xsep=0.3333em, inner ysep=0.3333em}]
    {
       ... & M^{-1}\ld\oplus M^0\ld & M^0\ld\oplus M^{-1}\ld & M^{-1}\ld\oplus M^0\ld & M^0\ld\\
    };
    \draw[->] ($(m-1-1.east) + (0,0.5mm)$) -- node[scale=0.5,above]{$\begin{pmatrix}f & w \\ -\id_{M^0\ld} &
        -g\end{pmatrix}$} ($(m-1-2.west) + (0,0.5mm)$);     
    \draw[->] ($(m-1-2.west) - (0,0.5mm)$) -- node[scale=0.5,below]{$\begin{pmatrix}0 & 0 \\ \id_{M^{-1}\ld} &
        0\end{pmatrix}$}($(m-1-1.east) - (0,0.5mm)$);  
    \draw[->] ($(m-1-2.east) + (0,0.5mm)$) --  node[scale=0.5,above]{$\begin{pmatrix}g & w \\ -\id_{M^{-1}\ld} &
        -f\end{pmatrix}$} ($(m-1-3.west) + (0,0.5mm)$);     
    \draw[->] ($(m-1-3.west) - (0,0.5mm)$) --  node[scale=0.5,below]{$\begin{pmatrix} 0 & 0 \\ \id_{M^0\ld} &
        0\end{pmatrix}$} ($(m-1-2.east) - (0,0.5mm)$);  
    \draw[->] ($(m-1-3.east) + (0,0.5mm)$) -- node[scale=0.5,above]{$\begin{pmatrix}f & w \\ -\id_{M^0\ld} &
        -g\end{pmatrix}$} ($(m-1-4.west) + (0,0.5mm)$);     
    \draw[->] ($(m-1-4.west) - (0,0.5mm)$) -- node[scale=0.5,below]{$\begin{pmatrix}0 & 0 \\ \id_{M^{-1}\ld} &
        0\end{pmatrix}$}($(m-1-3.east) - (0,0.5mm)$);   
    \draw[->] ($(m-1-4.east) + (0,0.5mm)$) -- node[scale=0.5,above]{$\begin{pmatrix}g & w\end{pmatrix}$}($(m-1-5.west)
    + (0,0.5mm)$);      
    \draw[->] ($(m-1-5.west) - (0,0.5mm)$) -- node[scale=0.5,below]{$\begin{pmatrix}0 \\ \id_{M^0\ld}\end{pmatrix}$}
    ($(m-1-4.east) - (0,0.5mm)$); 
\end{tikzpicture}\end{equation*}
The homotopy \eqref{eq:iota2mfhtp} of graded matrix factorizations yields the following homotopy of morphisms of
$K\ua_w$-modules: 
\begin{equation*}\begin{tikzpicture}[description/.style={fill=white,inner sep=2pt}]
    \matrix (m) [matrix of math nodes, row sep=6em,
                 column sep=2.3em, text height=1.5ex, text depth=0.25ex,
                 inner sep=0pt, nodes={inner xsep=0.3333em, inner ysep=0.3333em}]
    {
       ... & M^{-1}\ld\oplus M^0\ld && M^0\ld\oplus M^{-1}\ld && M^{-1}\ld\oplus M^0\ld && M^0\ld & 0 \\
       ... & M^{-1}\ld\oplus M^0\ld && M^0\ld\oplus M^{-1}\ld && M^{-1}\ld\oplus M^0\ld && M^0\ld & 0 \\
    };
    \draw[->] ($(m-1-1.east) + (0,0.5mm)$) -- ($(m-1-2.west) + (0,0.5mm)$);     
    \draw[->] ($(m-1-2.west) - (0,0.5mm)$) -- ($(m-1-1.east) - (0,0.5mm)$);  

    \draw[->] ($(m-1-8.east) + (0,0.5mm)$) -- ($(m-1-9.west) + (0,0.5mm)$);     
    \draw[->] ($(m-1-9.west) - (0,0.5mm)$) -- ($(m-1-8.east) - (0,0.5mm)$);  

    \draw[->] ($(m-1-2.east) + (0,0.5mm)$) --  node[scale=0.5,above]{$\begin{pmatrix}g & w \\ -\id_{M^{-1}\ld} &
        -f\end{pmatrix}$} ($(m-1-4.west) + (0,0.5mm)$);     
    \draw[->] ($(m-1-4.west) - (0,0.5mm)$) --  node[scale=0.5,below]{$\begin{pmatrix} 0 & 0 \\ \id_{M^0\ld} &
        0\end{pmatrix}$} ($(m-1-2.east) - (0,0.5mm)$);  
    \draw[->] ($(m-1-4.east) + (0,0.5mm)$) -- node[scale=0.5,above]{$\begin{pmatrix}f & w \\ -\id_{M^0\ld} &
        -g\end{pmatrix}$} ($(m-1-6.west) + (0,0.5mm)$);     
    \draw[->] ($(m-1-6.west) - (0,0.5mm)$) -- node[scale=0.5,below]{$\begin{pmatrix}0 & 0 \\ \id_{M^{-1}\ld} &
        0\end{pmatrix}$}($(m-1-4.east) - (0,0.5mm)$);   
    \draw[->] ($(m-1-6.east) + (0,0.5mm)$) -- node[scale=0.5,above]{$\begin{pmatrix}g & w\end{pmatrix}$}($(m-1-8.west)
    + (0,0.5mm)$);      
    \draw[->] ($(m-1-8.west) - (0,0.5mm)$) -- node[scale=0.5,below]{$\begin{pmatrix}0 \\ \id_{M^0\ld}\end{pmatrix}$}
    ($(m-1-6.east) - (0,0.5mm)$); 

    \draw[->] ($(m-2-1.east) + (0,0.6mm)$) -- ($(m-2-2.west) + (0,0.6mm)$);     
    \draw[->] ($(m-2-2.west) - (0,0.6mm)$) -- ($(m-2-1.east) - (0,0.6mm)$);  

    \draw[->] ($(m-2-8.east) + (0,0.6mm)$) -- ($(m-2-9.west) + (0,0.6mm)$);     
    \draw[->] ($(m-2-9.west) - (0,0.6mm)$) -- ($(m-2-8.east) - (0,0.6mm)$);  

    \draw[->] ($(m-2-2.east) + (0,0.6mm)$) --  node[scale=0.6,above]{$\begin{pmatrix}g & w \\ -\id_{M^{-1}\ld} &
        -f\end{pmatrix}$} ($(m-2-4.west) + (0,0.6mm)$);     
    \draw[->] ($(m-2-4.west) - (0,0.6mm)$) --  node[scale=0.6,below]{$\begin{pmatrix} 0 & 0 \\ \id_{M^0\ld} &
        0\end{pmatrix}$} ($(m-2-2.east) - (0,0.6mm)$);  
    \draw[->] ($(m-2-4.east) + (0,0.6mm)$) -- node[scale=0.6,above]{$\begin{pmatrix}f & w \\ -\id_{M^0\ld} &
        -g\end{pmatrix}$} ($(m-2-6.west) + (0,0.6mm)$);     
    \draw[->] ($(m-2-6.west) - (0,0.6mm)$) -- node[scale=0.6,below]{$\begin{pmatrix}0 & 0 \\ \id_{M^{-1}\ld} &
        0\end{pmatrix}$}($(m-2-4.east) - (0,0.6mm)$);   
    \draw[->] ($(m-2-6.east) + (0,0.6mm)$) -- node[scale=0.6,above]{$\begin{pmatrix}g & w\end{pmatrix}$}($(m-2-8.west)
    + (0,0.6mm)$);      
    \draw[->] ($(m-2-8.west) - (0,0.6mm)$) -- node[scale=0.6,below]{$\begin{pmatrix}0 \\ \id_{M^0\ld}\end{pmatrix}$}
    ($(m-2-6.east) - (0,0.6mm)$);

    \draw[->] (barycentric cs:m-1-4=0.75,m-2-2=0.25) -- node[scale=0.6,description]{$\begin{pmatrix} D^0 & 0 \\ 0 &
        D^1\end{pmatrix}$}(barycentric cs:m-1-4=0.25,m-2-2=0.75);
    \draw[->] (barycentric cs:m-1-6=0.75,m-2-4=0.25) -- node[scale=0.6,description]{$\begin{pmatrix} D^0 & 0 \\ 0 &
        D^1\end{pmatrix}$}(barycentric cs:m-1-6=0.25,m-2-4=0.75); 
    \draw[->] (barycentric cs:m-1-8=0.75,m-2-6=0.25) -- node[scale=0.6,description]{$\begin{pmatrix} D^0 \\
        0\end{pmatrix}$}(barycentric cs:m-1-8=0.25,m-2-6=0.75);  

    \draw[->] (m-1-2) -- node[scale=0.6,description]{$\begin{pmatrix}\beta-\delta & 0 \\ 0 &
        \alpha-\gamma\end{pmatrix}$} (m-2-2); 
    \draw[->] (m-1-4) -- node[scale=0.6,description]{$\begin{pmatrix}\alpha-\gamma & 0 \\ 0 &
        \beta-\delta\end{pmatrix}$} (m-2-4);
    \draw[->] (m-1-6) -- node[scale=0.6,description]{$\begin{pmatrix}\beta-\delta & 0 \\ -D^1 &
        \alpha-\gamma - D^1\circ f\end{pmatrix}$} (m-2-6); 
    \draw[->] (m-1-8) -- node[scale=0.6,description]{$\alpha-\beta-D^1\circ f$} (m-2-8);
\end{tikzpicture}\end{equation*}
Therefore, $Q(K\ua_w,\iota(\alpha,\beta))-Q(K\ua_w,\iota(\gamma,\delta))$ is homotopic to 
\begin{equation*}\begin{tikzpicture}[description/.style={fill=white,inner sep=2pt}]
    \matrix (m) [matrix of math nodes, row sep=6em,
                 column sep=2.2em, text height=1.5ex, text depth=0.25ex,
                 inner sep=0pt, nodes={inner xsep=0.3333em, inner ysep=0.3333em}]
    {
       ... & M^{-1}\ld\oplus M^0\ld && M^0\ld\oplus M^{-1}\ld && M^{-1}\ld\oplus M^0\ld && M^0\ld & 0 \\
       ... & M^{-1}\ld\oplus M^0\ld && M^0\ld\oplus M^{-1}\ld && M^{-1}\ld\oplus M^0\ld && M^0\ld & 0 \\
    }; 
    \draw[->] ($(m-1-1.east) + (0,0.5mm)$) -- ($(m-1-2.west) + (0,0.5mm)$);     
    \draw[->] ($(m-1-2.west) - (0,0.5mm)$) -- ($(m-1-1.east) - (0,0.5mm)$);  

    \draw[->] ($(m-1-8.east) + (0,0.5mm)$) -- ($(m-1-9.west) + (0,0.5mm)$);     
    \draw[->] ($(m-1-9.west) - (0,0.5mm)$) -- ($(m-1-8.east) - (0,0.5mm)$);  

    \draw[->] ($(m-1-2.east) + (0,0.5mm)$) --  node[scale=0.5,above]{$\begin{pmatrix}g & w \\ -\id_{M^{-1}\ld} &
        -f\end{pmatrix}$} ($(m-1-4.west) + (0,0.5mm)$);     
    \draw[->] ($(m-1-4.west) - (0,0.5mm)$) --  node[scale=0.5,below]{$\begin{pmatrix} 0 & 0 \\ \id_{M^0\ld} &
        0\end{pmatrix}$} ($(m-1-2.east) - (0,0.5mm)$);  
    \draw[->] ($(m-1-4.east) + (0,0.5mm)$) -- node[scale=0.5,above]{$\begin{pmatrix}f & w \\ -\id_{M^0\ld} &
        -g\end{pmatrix}$} ($(m-1-6.west) + (0,0.5mm)$);     
    \draw[->] ($(m-1-6.west) - (0,0.5mm)$) -- node[scale=0.5,below]{$\begin{pmatrix}0 & 0 \\ \id_{M^{-1}\ld} &
        0\end{pmatrix}$}($(m-1-4.east) - (0,0.5mm)$);   
    \draw[->] ($(m-1-6.east) + (0,0.5mm)$) -- node[scale=0.5,above]{$\begin{pmatrix}g & w\end{pmatrix}$}($(m-1-8.west)
    + (0,0.5mm)$);      
    \draw[->] ($(m-1-8.west) - (0,0.5mm)$) -- node[scale=0.5,below]{$\begin{pmatrix}0 \\ \id_{M^0\ld}\end{pmatrix}$}
    ($(m-1-6.east) - (0,0.5mm)$); 
 
    \draw[->] ($(m-2-1.east) + (0,0.6mm)$) -- ($(m-2-2.west) + (0,0.6mm)$);     
    \draw[->] ($(m-2-2.west) - (0,0.6mm)$) -- ($(m-2-1.east) - (0,0.6mm)$);  

    \draw[->] ($(m-2-8.east) + (0,0.6mm)$) -- ($(m-2-9.west) + (0,0.6mm)$);     
    \draw[->] ($(m-2-9.west) - (0,0.6mm)$) -- ($(m-2-8.east) - (0,0.6mm)$);  

    \draw[->] ($(m-2-2.east) + (0,0.6mm)$) --  node[scale=0.6,above]{$\begin{pmatrix}g & w \\ -\id_{M^{-1}\ld} &
        -f\end{pmatrix}$} ($(m-2-4.west) + (0,0.6mm)$);     
    \draw[->] ($(m-2-4.west) - (0,0.6mm)$) --  node[scale=0.6,below]{$\begin{pmatrix} 0 & 0 \\ \id_{M^0\ld} &
        0\end{pmatrix}$} ($(m-2-2.east) - (0,0.6mm)$);  
    \draw[->] ($(m-2-4.east) + (0,0.6mm)$) -- node[scale=0.6,above]{$\begin{pmatrix}f & w \\ -\id_{M^0\ld} &
        -g\end{pmatrix}$} ($(m-2-6.west) + (0,0.6mm)$);     
    \draw[->] ($(m-2-6.west) - (0,0.6mm)$) -- node[scale=0.6,below]{$\begin{pmatrix}0 & 0 \\ \id_{M^{-1}\ld} &
        0\end{pmatrix}$}($(m-2-4.east) - (0,0.6mm)$);   
    \draw[->] ($(m-2-6.east) + (0,0.6mm)$) -- node[scale=0.6,above]{$\begin{pmatrix}g & w\end{pmatrix}$}($(m-2-8.west)
    + (0,0.6mm)$);      
    \draw[->] ($(m-2-8.west) - (0,0.6mm)$) -- node[scale=0.6,below]{$\begin{pmatrix}0 \\ \id_{M^0\ld}\end{pmatrix}$}
    ($(m-2-6.east) - (0,0.6mm)$);

    \draw[->] (m-1-2) -- node[scale=0.6,description]{$0$} (m-2-2); 
    \draw[->] (m-1-4) -- node[scale=0.6,description]{$0$} (m-2-4);
    \draw[->] (m-1-6) -- node[scale=0.6,description]{$\begin{pmatrix}0 & 0 \\ D^1 &
        D^1\circ f\end{pmatrix}$} (m-2-6); 
    \draw[->] (m-1-8) -- node[scale=0.6,description]{$D^1\circ f$} (m-2-8);
\end{tikzpicture}\end{equation*}
which factors through the perfect $K\ua_w$-module
\begin{equation*}\begin{tikzpicture}[description/.style={fill=white,inner sep=2pt}]
    \matrix (m) [matrix of math nodes, row sep=3em,
                 column sep=2.1em, text height=1.5ex, text depth=0.25ex,
                 inner sep=0pt, nodes={inner xsep=0.3333em, inner ysep=0.3333em}]
    {
       ... & 0 & M^0\ld & M^0\ld & 0 &  ... & \cong &K\ua_w\otimes_{S\ld} M^0\ld.\\
    };
    \draw[->] ($(m-1-1.east) + (0,0.6mm)$) -- ($(m-1-2.west) + (0,0.6mm)$);     
    \draw[->] ($(m-1-2.west) - (0,0.6mm)$) -- ($(m-1-1.east) - (0,0.6mm)$); 

    \draw[->] ($(m-1-2.east) + (0,0.6mm)$) -- ($(m-1-3.west) + (0,0.6mm)$);     
    \draw[->] ($(m-1-3.west) - (0,0.6mm)$) -- ($(m-1-2.east) - (0,0.6mm)$); 

    \draw[->] ($(m-1-3.east) + (0,0.6mm)$) -- node[scale=0.75,above]{$w$} ($(m-1-4.west) + (0,0.6mm)$);     
    \draw[->] ($(m-1-4.west) - (0,0.6mm)$) -- node[scale=0.75,below]{$\id$}($(m-1-3.east) - (0,0.6mm)$); 

    \draw[->] ($(m-1-4.east) + (0,0.6mm)$) -- ($(m-1-5.west) + (0,0.6mm)$);     
    \draw[->] ($(m-1-5.west) - (0,0.6mm)$) -- ($(m-1-4.east) - (0,0.6mm)$); 

    \draw[->] ($(m-1-5.east) + (0,0.6mm)$) -- ($(m-1-6.west) + (0,0.6mm)$);     
    \draw[->] ($(m-1-6.west) - (0,0.6mm)$) -- ($(m-1-5.east) - (0,0.6mm)$); 
\end{tikzpicture}\end{equation*}
By the remarks above, this shows that in $\D^{b,b}(K\ua_w)/\perf^\infty$ we have 
$$Q(K\ua_w,\iota(\alpha,\beta))\ =\ Q(K\ua_w,\iota(\gamma,\delta))$$ and hence
$\iota(\alpha,\beta)=\iota(\gamma,\delta)$, as claimed.  
\end{proof}

\begin{theorem}\label{thm_bigtheorem}
There is a natural isomorphism $$\veps:\ \iota\circ\fold\ \stackrel{\cong}{\Longrightarrow}\
\id_{\D^{b,b}_{\fr}(K\ua_w)/\perf^\infty}$$ which together with the equality $\fold\circ\iota =
\id_{\HMF^\infty(S\ld,w)}$ forms an adjoint equivalence
\begin{equation*}\begin{tikzpicture}[baseline=(current bounding  box.center), description/.style={fill=white,inner sep=2pt}]
    \matrix (m) [matrix of math nodes, row sep=3em,
                 column sep=2.5em, text height=1.5ex, text depth=0.25ex]
    {
       \fold:\ \D^{b,b}_{\fr}(K\ua_w)/\perf^\infty &&&& \HMF^\infty(S\ld,w):\ \iota\\
    };
    \draw[->] ($(m-1-1.east) + (0, 0.7mm)$) -- node[above,scale=0.75]{$\cong$}($(m-1-5.west) + (0, 0.7mm)$); 
    \draw[->] ($(m-1-5.west) + (0,-0.5mm)$) -- ($(m-1-1.east) + (0,-0.5mm)$);
\end{tikzpicture}\end{equation*}
\end{theorem}
\begin{proof}
Let $F\ua\ld$ be a bounded, $S\ld$-free $K\ua_w$-module; without loss of generality we may assume
$F^k\ld = 0$ for $k>0$. Then $Q(K\ua_w,(\iota\circ\fold)(F\ua\ld))$ is explicitly given as
\begin{equation*}\begin{tikzpicture}[description/.style={fill=white,inner sep=2pt}]
    \matrix (m) [matrix of math nodes, row sep=3em,
                 column sep=4em, text height=1.5ex, text depth=0.25ex,
                 inner sep=0pt, nodes={inner xsep=0.3333em, inner ysep=0.3333em}]
    {
       ... & F\ueven\ld\langle -d\rangle \oplus F\uodd\ld\langle -d\rangle & F\uodd\ld\oplus F\ueven\ld\langle -d\rangle
       & F\ueven\ld\\ 
    };
    \draw[->] ($(m-1-1.east) + (0,0.5mm)$) -- node[scale=0.5,above]{$\begin{pmatrix}\d+s & w \\ -\id &
        -(\d+s)\end{pmatrix}$} ($(m-1-2.west) + (0,0.5mm)$);     
    \draw[->] ($(m-1-2.west) - (0,0.5mm)$) -- node[scale=0.5,below]{$\begin{pmatrix}0 & 0 \\ \id &
        0\end{pmatrix}$}($(m-1-1.east) - (0,0.5mm)$);  
    \draw[->] ($(m-1-2.east) + (0,0.5mm)$) --  node[scale=0.5,above]{$\begin{pmatrix}\d+s & w \\ -\id &
        -(\d+s)\end{pmatrix}$} ($(m-1-3.west) + (0,0.5mm)$);     
    \draw[->] ($(m-1-3.west) - (0,0.5mm)$) --  node[scale=0.5,below]{$\begin{pmatrix} 0 & 0 \\ \id &
        0\end{pmatrix}$} ($(m-1-2.east) - (0,0.5mm)$);  
    \draw[->] ($(m-1-3.east) + (0,0.5mm)$) -- node[scale=0.5,above]{$\begin{pmatrix}\d+s & w\end{pmatrix}$}($(m-1-4.west)
    + (0,0.5mm)$);      
    \draw[->] ($(m-1-4.west) - (0,0.5mm)$) -- node[scale=0.5,below]{$\begin{pmatrix}0 \\ \id\end{pmatrix}$}
    ($(m-1-3.east) - (0,0.5mm)$); 
\end{tikzpicture}\end{equation*}
where $F\ueven\ld := \bigoplus\limits_{n\geq 0} F^{-2n}\ld\langle nd\rangle$ and $F\uodd\ld :=
\bigoplus\limits_{n\geq 0} F^{-2n-1}\ld\langle nd\rangle$. We define $\veps_{F\ua\ld}$ as the roof
\begin{equation}\label{eq:epsilonroof}\begin{tikzpicture}[description/.style={fill=white,inner sep=2pt},
    baseline=(m-2-1.base)] 
    \matrix (m) [matrix of math nodes, row sep=3em,
                 column sep=4em, text height=1.5ex, text depth=0.25ex,
                 inner sep=0pt, nodes={inner xsep=0.3333em, inner ysep=0.3333em}]
    {
       ... & 0 & F\uodd\ld& F\ueven\ld\\
       ... & F\ueven\ld\langle -d\rangle \oplus F\uodd\ld\langle -d\rangle & F\uodd\ld\oplus F\ueven\ld\langle -d\rangle
       & F\ueven\ld\\ 
       ... & F^{-2}\ld & F^{-1}\ld & F^0\ld\\ 
    };
    \draw[->] (m-2-2) -- (m-1-2);
    \draw[->] (m-2-3) -- node[scale=0.75,description]{$\begin{pmatrix}\id & \d+s\end{pmatrix}$}(m-1-3);
    \draw[->] (m-2-4) -- (m-1-4);

   \draw[->] ($(m-1-1.east) + (0,0.5mm)$) -- ($(m-1-2.west) + (0,0.5mm)$); 
   \draw[->] ($(m-1-2.west) - (0,0.5mm)$) -- ($(m-1-1.east) - (0,0.5mm)$);  
   \draw[->] ($(m-1-2.east) + (0,0.5mm)$) -- ($(m-1-3.west) + (0,0.5mm)$); 
   \draw[->] ($(m-1-3.west) - (0,0.5mm)$) -- ($(m-1-2.east) - (0,0.5mm)$);  

    \draw[->] ($(m-1-3.east) + (0,0.5mm)$) -- node[scale=0.75,above]{$\d+s$} ($(m-1-4.west) + (0,0.5mm)$);     
    \draw[->] ($(m-1-4.west) - (0,0.5mm)$) -- node[scale=0.75,below]{$\d+s$}($(m-1-3.east) - (0,0.5mm)$);    

    \draw[->] ($(m-2-1.east) + (0,0.5mm)$) -- node[scale=0.6,above]{$\begin{pmatrix}\d+s & w \\ -\id &
        -(\d+s)\end{pmatrix}$} ($(m-2-2.west) + (0,0.5mm)$);     
    \draw[->] ($(m-2-2.west) - (0,0.5mm)$) -- node[scale=0.6,below]{$\begin{pmatrix}0 & 0 \\ \id &
        0\end{pmatrix}$}($(m-2-1.east) - (0,0.5mm)$);  
    \draw[->] ($(m-2-2.east) + (0,0.5mm)$) --  node[scale=0.6,above]{$\begin{pmatrix}\d+s & w \\ -\id &
        -(\d+s)\end{pmatrix}$} ($(m-2-3.west) + (0,0.5mm)$);     
    \draw[->] ($(m-2-3.west) - (0,0.5mm)$) --  node[scale=0.6,below]{$\begin{pmatrix} 0 & 0 \\ \id &
        0\end{pmatrix}$} ($(m-2-2.east) - (0,0.5mm)$);  
    \draw[->] ($(m-2-3.east) + (0,0.5mm)$) -- node[scale=0.6,above]{$\begin{pmatrix}\d+s & w\end{pmatrix}$}($(m-2-4.west)
    + (0,0.5mm)$);      
    \draw[->] ($(m-2-4.west) - (0,0.5mm)$) -- node[scale=0.6,below]{$\begin{pmatrix}0 \\ \id\end{pmatrix}$}
    ($(m-2-3.east) - (0,0.5mm)$); 

    \draw[->] ($(m-3-1.east) + (0,0.5mm)$) -- node[scale=0.75,above]{$\differential$} ($(m-3-2.west) + (0,0.5mm)$);     
    \draw[->] ($(m-3-2.west) - (0,0.5mm)$) -- node[scale=0.75,below]{$s$}  ($(m-3-1.east) - (0,0.5mm)$);  
    \draw[->] ($(m-3-2.east) + (0,0.5mm)$) -- node[scale=0.75,above]{$\differential$} ($(m-3-3.west) + (0,0.5mm)$);     
    \draw[->] ($(m-3-3.west) - (0,0.5mm)$) -- node[scale=0.75,below]{$s$}  ($(m-3-2.east) - (0,0.5mm)$);  
    \draw[->] ($(m-3-3.east) + (0,0.5mm)$) -- node[scale=0.75,above]{$\differential$} ($(m-3-4.west) + (0,0.5mm)$);      
    \draw[->] ($(m-3-4.west) - (0,0.5mm)$) -- node[scale=0.75,below]{$s$}  ($(m-3-3.east) - (0,0.5mm)$);

    \draw[->] (m-2-2) -- node[description,scale=0.75]{$\begin{pmatrix} \pr_{F^{-2}\ld} &
        s\circ\pr_{F^{-1}\ld}\end{pmatrix}$} (m-3-2); 
    \draw[->] (m-2-3) -- node[description,scale=0.75]{$\begin{pmatrix} \pr_{F^{-1}\ld} &
        s\circ\pr_{F^{0}\ld}\end{pmatrix}$}(m-3-3);
    \draw[->] (m-2-4) -- node[description,scale=0.75]{$\pr_{F^{0}\ld}$} (m-3-4);
\end{tikzpicture}\end{equation}
Observe that the internal grading shifts are such that this indeed preserves the grading. Clearly, this morphism is
natural in $F\ua\ld$. 

Note that, although the denominator $Q(K\ua_w,(\iota\circ\fold)(F\ua\ld))$ is not in $\D^{b,b}_{\fr}(K\ua_w)$,
$\veps_{F\ua\ld}$ \textit{is} a morphism in $\D^{b,b}_{\fr}(K\ua_w)$, because we defined the latter as a full
subcategory of $\D(K\ua_w)$. Nonetheless, we know from the equivalence
$\Ho^{b,b}_{\fr}(K\ua_w)/\Acyc^{b,b}_{\fr}(K\ua_w)\cong\D^{b,b}_{\fr}(K\ua_w)$ (Proposition  
\ref{prop_boundedasverdierquotient}) that $\veps_{F\ua\ld}$ can be represented by a roof having a denominator in
$\D^{b,b}_{\fr}(K\ua_w)$. More concretely, we can replace $Q(K\ua_w,(\iota\circ\fold)(F\ua\ld))$ by $\tau_{\geq 2n}
Q(K\ua_w,(\iota\circ\fold)(F\ua\ld))$ for $n\ll 0$ such that $F^k\ld=0$ for $k\leq 2n$; in degrees $2n$ to $2n+2$ this
truncation is explicitly given as follows:
\begin{equation*}\begin{tikzpicture}[description/.style={fill=white,inner sep=2pt}]
    \matrix (m) [matrix of math nodes, row sep=3em,
                 column sep=2.5em, text height=1.5ex, text depth=0.25ex,
                 inner sep=0pt, nodes={inner xsep=0.3333em, inner ysep=0.3333em}]
    {
       F\ueven\ld\langle nd\rangle & \left(F\uodd\ld\langle d\rangle\oplus F\ueven\ld\right)\langle nd\rangle &&
       \left(F\ueven\ld\oplus F\uodd\ld\right)\langle (n+1)d\rangle & ...\\ 
    };
    \draw[->] ($(m-1-1.east) + (0,0.5mm)$) -- node[scale=0.6,above]{$\begin{pmatrix} d+s \\ -\id\end{pmatrix}$}
    ($(m-1-2.west) + (0,0.5mm)$);      
    \draw[->] ($(m-1-2.west) - (0,0.5mm)$) -- node[scale=0.6,below]{$\begin{pmatrix}0 &
        \id\end{pmatrix}$}($(m-1-1.east) - (0,0.5mm)$);   
    \draw[->] ($(m-1-2.east) + (0,0.5mm)$) --  node[scale=0.6,above]{$\begin{pmatrix}\differential+s & w \\ -\id &
        -(\differential+s)\end{pmatrix}$} ($(m-1-4.west) + (0,0.5mm)$);     
    \draw[->] ($(m-1-4.west) - (0,0.5mm)$) --  node[scale=0.6,below]{$\begin{pmatrix} 0 & 0 \\ \id &
        0\end{pmatrix}$} ($(m-1-2.east) - (0,0.5mm)$); 

     \draw[->] ($(m-1-4.east) + (0,0.5mm)$) -- ($(m-1-5.west) + (0,0.5mm)$);     
     \draw[->] ($(m-1-5.west) - (0,0.5mm)$) -- ($(m-1-4.east) - (0,0.5mm)$);  
\end{tikzpicture}\end{equation*}
To see that $\veps$ as the counit and $\id=\fold\circ\iota$ as the unit form an adjunction
$\iota\dashv\fold$, we have to check the following:
\begin{enumerate}
\item\label{item:adj1} For each $M\ua\ld\in\HMF^\infty(S\ld,w)$, the map $$\iota(M\ua\ld) =
  \iota((\fold\circ\iota)(M\ua\ld)) = 
  (\iota\circ\fold)(\iota(M\ua\ld))\xrightarrow{\veps_{\iota(M\ua\ld)}}\iota(M\ua\ld)$$ is the identity. 
\item\label{item:adj2} For each $X\ua\ld\in K\ua_w\Mod_\fr$, the
  map $$\fold(F\ua\ld)\xrightarrow{\fold(\veps_{X\ua\ld})}\fold((\iota\circ\fold)(F\ua\ld)) =
  (\fold\circ\iota)(\fold(F\ua\ld)) = \fold(F\ua\ld)$$ is the identity. 
\end{enumerate}
Statement \eqref{item:adj1} holds because in the case where $F^k\ld = 0$ for $k \neq -1,0$, the numerator and
denominator in the roof \eqref{eq:epsilonroof} defining $\veps_{F\ua\ld}$ are equal. For \eqref{item:adj2}, we first
construct an explicit inverse for the homotopy equivalence $$\fold\big(\tau_{\geq 2n} 
  Q(K\ua_w,\iota(M\ua\ld))\ \longrightarrow\ \iota(M\ua\ld)\big)$$ for a matrix factorization $M\ua\ld =
  \left(M^0\ld\mor{f} M^{-1}\ld\mor{g} M^0\right)$ and $n\ll 0$:
\begin{equation}\label{eq:explicitinverse}\begin{tikzpicture}[description/.style={fill=white,inner
      sep=2pt},baseline=(current bounding box.center)]
    \matrix (m) [matrix of math nodes, row sep=5em,
                 column sep=4em, text height=1.5ex, text depth=0.25ex,
                 inner sep=0pt, nodes={inner xsep=0.3333em, inner ysep=0.3333em}]
    {
       ... & M^{-1}\ld\oplus M^0\ld & M^0\ld\oplus M^{-1}\ld & M^{-1}\ld\oplus M^0\ld & M^0\ld\\
        &  &  & M^{-1}\ld & M^0\ld\\
    };

    \draw[->] ($(m-2-5.north) + (+1mm,0)$) -- node[scale=0.6,description,pos=0.7]{$\begin{pmatrix}\id\\ 0\end{pmatrix}$}
    ($(m-2-5.north |- m-1-5.south) + (1mm,0)$); 
    \draw[->] ($(m-2-5.north) + (+0mm,0)$) -- ($(m-2-5.north) + (+0mm,8mm)$) -|
    node[scale=0.6,description,pos=0.75]{$\begin{pmatrix}\id\\ 0\end{pmatrix}$} (m-1-3.south); 
    \draw ($(m-2-5.north) + (-1mm,0)$) -- ($(m-2-5.north) + (-1mm,7mm)$) -- ($(m-2-5.north -| m-1-2) + (-2cm,7mm)$);
    \draw[dotted] ($(m-2-5.north -| m-1-2) + (-2cm,7mm)$) -- +(-1cm,0);

    \draw[draw=white, line width=2mm] ($(m-2-4.north) + (+1mm,0)$) -- ($(m-2-4.north |- m-1-4.south) + (1mm,0)$); 
    \draw[draw=white, line width=2mm] ($(m-2-4.north) + (+0mm,0)$) -- ($(m-2-4.north) + (+0mm,4mm)$) -| (m-1-2.south); 

    \draw[->] ($(m-2-4.north) + (+1mm,0)$) -- node[scale=0.6,description,pos=0.7]{$\begin{pmatrix}\id\\ 0\end{pmatrix}$}
    ($(m-2-4.north |- m-1-4.south) + (1mm,0)$); 
    \draw[->] ($(m-2-4.north) + (+0mm,0)$) -- ($(m-2-4.north) + (+0mm,4mm)$) -|
    node[scale=0.6,description,pos=0.82]{$\begin{pmatrix}\id\\ 0\end{pmatrix}$} (m-1-2.south); 
    \draw ($(m-2-4.north) + (-1mm,0)$) -- ($(m-2-4.north) + (-1mm,3mm)$) -- ($(m-2-4.north -| m-1-2) + (-2cm,3mm)$);
    \draw[dotted] ($(m-2-4.north -| m-1-2) + (-2cm,3mm)$) -- +(-1cm,0);

    \draw[->] ($(m-2-4.east) + (0,0.5mm)$) -- node[scale=0.6,above]{$g$} ($(m-2-5.west) + (0,0.5mm)$);     
    \draw[->] ($(m-2-5.west) - (0,0.5mm)$) -- node[scale=0.6,below]{$f$} ($(m-2-4.east) - (0,0.5mm)$); 
    
    \draw[->] ($(m-1-1.east) + (0,0.5mm)$) -- node[scale=0.6,above]{$\begin{pmatrix}f & w \\ -\id_{M^0\ld} &
        -g\end{pmatrix}$} ($(m-1-2.west) + (0,0.5mm)$);     
    \draw[->] ($(m-1-2.west) - (0,0.5mm)$) -- node[scale=0.6,below]{$\begin{pmatrix}0 & 0 \\ \id_{M^{-1}\ld} &
        0\end{pmatrix}$}($(m-1-1.east) - (0,0.5mm)$);  
    \draw[->] ($(m-1-2.east) + (0,0.5mm)$) --  node[scale=0.6,above]{$\begin{pmatrix}g & w \\ -\id_{M^{-1}\ld} &
        -f\end{pmatrix}$} ($(m-1-3.west) + (0,0.5mm)$);     
    \draw[->] ($(m-1-3.west) - (0,0.5mm)$) --  node[scale=0.6,below]{$\begin{pmatrix} 0 & 0 \\ \id_{M^0\ld} &
        0\end{pmatrix}$} ($(m-1-2.east) - (0,0.5mm)$);  
    \draw[->] ($(m-1-3.east) + (0,0.5mm)$) -- node[scale=0.6,above]{$\begin{pmatrix}f & w \\ -\id_{M^0\ld} &
        -g\end{pmatrix}$} ($(m-1-4.west) + (0,0.5mm)$);     
    \draw[->] ($(m-1-4.west) - (0,0.5mm)$) -- node[scale=0.6,below]{$\begin{pmatrix}0 & 0 \\ \id_{M^{-1}\ld} &
        0\end{pmatrix}$}($(m-1-3.east) - (0,0.5mm)$);   
    \draw[->] ($(m-1-4.east) + (0,0.5mm)$) -- node[scale=0.6,above]{$\begin{pmatrix}g & w\end{pmatrix}$}($(m-1-5.west)
    + (0,0.5mm)$);      
    \draw[->] ($(m-1-5.west) - (0,0.5mm)$) -- node[scale=0.6,below]{$\begin{pmatrix}0 \\ \id_{M^0\ld}\end{pmatrix}$}
    ($(m-1-4.east) - (0,0.5mm)$); 
\end{tikzpicture}\end{equation}       
The map from $M^0\ld$ into the truncated component $\tau_{\geq 2n} Q(K\ua_w,\iota(M\ua\ld))^{2n}\ld = M^0\ld$ is
the identity.

Now if $M\ld = \fold(F\ua\ld)$ as above, it is clear that the composition of \eqref{eq:explicitinverse} and the
numerator $\fold(Q(K\ua_w,(\iota\circ\fold)(F\ua\ld))\to F\ua\ld)$ from \eqref{eq:epsilonroof} is the identity on
$\fold(F\ua\ld)$, so \eqref{item:adj2} holds. This finishes the proof of the adjunction $\iota\dashv\fold$. 

It remains to show that $\veps_{F\ua\ld}:(\iota\circ\fold)(F\ua\ld)\to F\ua\ld$ is an isomorphism in
$\D^{b,b}_{\fr}(K\ua_w)$. For this, first define $$Q(K\ua_w,F\ua\ld)\ \supset\ U_{F\ua\ld}(n)\ :=\
\Span_{S\ld}\left\{a\otimes t^k\otimes m\ |\ 2k-|m|\leq n\right\}.$$ It is clear from the explicit description of the
  differential on $Q(K\ua_w,F\ua\ld)$ in Proposition \ref{prop_barreskoszul} that $\differential(U_{F\ua\ld}(n))\subset
  U_{F\ua\ld}(n)$, so $U_{F\ua\ld}(n)$ is a $K\ua_w$-submodule of $Q(K\ua_w,F\ua\ld)$. We put $D_{F\ua\ld}(n) :=
  Q(K\ua_w,F\ua\ld) / U_{F\ua\ld}(n-1)$ and denote $D_{F\ua\ld}(n\to m): D_{F\ua\ld}(n)\
  \longrightarrow  D_{F\ua\ld}(m)$ the projection map for $n\leq m$. We have $$D_{F\ua\ld}(0)\ =\ Q(K\ua_w,F\ua\ld)$$ and 
$$D_{F\ua\ld}(2n)\ \cong\ Q(K\ua_w,(\iota\circ\fold)(F\ua\ld))[2n]\langle -nd\rangle$$ for $n\gg 0$, so
$D_{F\ua\ld}(\ast)$ interpolates between the bar resolutions of $F\ua\ld$ and $(\iota\circ\fold)(F\ua\ld)$. We will now
prove that the quotient map $$D_{F\ua\ld}(0\to 2n):\ Q(K\ua_w,F\ua\ld)=D_{F\ua\ld}(0)\ \longrightarrow
D_{F\ua\ld}(2n) \cong Q(K\ua_w,(\iota\circ\fold)(F\ua\ld))[2n]\langle -nd\rangle$$ is an isomorphism in
$\D^{-,b}_{\fr}(K\ua_w)/\perf^\infty$ by showing that its cone is in $\perf^\infty$. By the octahedral axiom, it 
suffices to show that $\cone\left(D_{F\ua\ld}(n\to n+1)\right)\in\perf^\infty$ for each $n\in\ZZ$. For this, note that 
as $D_{F\ua\ld}(n\to n+1)$ is an epimorphism, we have an isomorphism in
$\D^{-,b}_{\fr}(K\ua_w)$
\begin{align*}
\cone\left(D_{F\ua\ld}(n\to n+1)\right) & \makebox[8mm][c]{$\cong$} \ker\left(D_{F\ua\ld}(n\to n+1)\right)\\
& \makebox[8mm][c]{$=$} K\ua_w\otimes_{S\ld} \Span_{S\ld}\left\{t^k\otimes m\ |\ 2k - |m| = n\right\},
\end{align*}
where the action of $K\ua_w$ and the differential on the right hand side is given by their respective action on the
first tensor factor $K\ua_w$. Thus, $\cone(D_{F\ua\ld}(n\to n+1))$ is of the form $K\ua_w\otimes_{S\ld} X\ld[k]$ for
some free $S\ld$-module $X\ld$ and some $k\in\ZZ$, and hence in $\perf^\infty$.

Next, note that applying the previous paragraph to $(\iota\circ\fold)(F\ua\ld)$ instead of $F\ua\ld$, we see that
$Q(K\ua_w,(\iota\circ\fold)(F\ua\ld))[2n]\langle -nd\rangle$ is also isomorphic to $D_{(\iota\circ\fold)(F\ua\ld)}(2n)$
for $n\ll 0$. Hence, we have the following diagram in $\D^{-,b}_{\fr}(K\ua_w)$, where the diagonal maps are have perfect
cones, i.e. are isomorphism in $\D^{-,b}_{\fr}(K\ua_w)/\perf^\infty$:
\begin{equation*}\begin{tikzpicture}[description/.style={fill=white,inner sep=2pt}]
    \matrix (m) [matrix of math nodes, row sep=3em,
                 column sep=2.5em, text height=1.5ex, text depth=0.25ex,
                 inner sep=0pt, nodes={inner xsep=0.3333em, inner ysep=0.3333em}]
    {
       Q(K\ua_w,(\iota\circ\fold)(F\ua\ld)) && Q(K\ua_w,F\ua\ld)\\ & Q(K\ua_w,(\iota\circ\fold)(F\ua\ld))[2n]\langle
       -nd\rangle \\
    };
    \draw[dashed,->] (m-1-1) -- node[above,scale=0.75]{$\psi$} (m-1-3);
    \draw[->] (m-1-1) -- ($(m-2-2.north) + (-2.2cm,1mm)$);
    \draw[->] (m-1-3) -- ($(m-2-2.north) + (+2.2cm,1mm)$);
\end{tikzpicture}\end{equation*}
We will know describe explicitly a map $\psi$ making the diagram commute. It is then automatically an isomorphism in
$\D^{-,b}_{\fr}(K\ua_w)/\perf^\infty$. Next, it will be clear  that the composition 
\begin{equation}\label{eq:compii}Q(K\ua_w,(\iota\circ\fold)(F\ua\ld))\
\stackrel{\psi}{\longrightarrow}\ Q(K\ua_w,F\ua\ld)\ \longrightarrow\ F\ua\ld\end{equation} is precisely the numerator 
in the roof \eqref{eq:epsilonroof} defining $\veps_{F\ua\ld}$, proving that $\veps_{F\ua\ld}$ is an isomorphism in
$\D^{-,b}_{\fr}(K\ua_w)/\perf^\infty$. 

The map $\psi$ is given as follows (negative powers of $t$ are to be interpreted as $0$):
\begin{equation*}\begin{tikzpicture}[description/.style={fill=white,inner sep=2pt}]
    \matrix (m) [matrix of math nodes, row sep=0.5em,
                 column sep=2.5em, text height=1.5ex, text depth=0.25ex,
                 inner sep=0pt, nodes={inner xsep=0.3333em, inner ysep=0.3333em}]
    {
       K\ua_w\otimes_{S\ld} S\ld[t]\otimes_{S\ld} \iota(\fold(F\ua\ld)) &&
       K\ua_w\otimes_{S\ld} S\ld[t]\otimes_{S\ld} F\ua\ld\\\\
       a\otimes t^k\otimes m^{-2l} && a\otimes t^{k-l}\otimes m^{-2l}\\
       a\otimes t^k\otimes m^{-2l-1} && a\otimes t^{k-l}\otimes m^{-2l-1}\\
    };
    \draw[->]  (m-1-1)  -- node[above,scale=0.75]{$\psi$} (m-1-3);
    \draw[|->] (m-4-1.east |- m-3-1.east) -- (m-4-3.west |- m-3-3.west);
    \draw[|->] (m-4-1) -- (m-4-3);
    \path (m-3-1) -- node[pos=0.25,rotate=90]{$\in$} (m-3-1 |- m-1-1);
    \path (m-3-3) -- node[pos=0.25,rotate=90]{$\in$} (m-3-3 |- m-1-3);
\end{tikzpicture}\end{equation*}
Here $m^{-2l}$ and $m^{-2l-1}$ denote elements in $F\ld^{-2l}$ and $F\ld^{-2l-1}$, respectively. Let us check carefully
that this makes sense, i.e. that both cohomological and internal degrees are preserved.
\begin{enumerate}
\item The cohomological degree of $a\otimes t^k\otimes m^{-2l}$ in 
$Q(K\ua_w,(\iota\circ\fold)(F\ua\ld))$ is $|a|-2k$, and the cohomological degree of $a\otimes t^{k-l}\otimes m^{-2l}$ in
$Q(K\ua_w,F\ua\ld)$ equals $|a| - 2(k-l) -2l = |a| - 2k$. Similarly, the cohomological degree of $a\otimes t^k\otimes
m^{-2l-1}$ in $Q(K\ua_w,(\iota\circ\fold)(F\ua\ld))$ is $|a| - 2k - 1$, and the cohomological degree of $a\otimes
t^{k-l}\otimes m^{-2l-1}$ in $Q(K\ua_w,F\ua\ld)$ is $|a| - 2(k-l) - 2l - 1 = |a| - 2k - 1$. 
\item Recalling that $F\ueven\ld = \bigoplus\limits_{n\in\ZZ} F^{-2n}\ld\langle nd\rangle$ and $F\uodd\ld =
\bigoplus\limits_{n\in\ZZ} F^{-2n-1}\ld\langle nd\rangle$, we see that the internal degree of $a\otimes t^k\otimes
m^{-2l}$ in $Q(K\ua_w,(\iota\circ\fold)(F\ua\ld))$ is $\deg(a) + kd - 2l - ld$, while the internal degree of
$a\otimes t^{k-l}\otimes m^{-2l}$ in $Q(K\ua_w,F\ua\ld)$ is $\deg(a) + (k-l)d - 2l$; similarly in the odd case.
\end{enumerate}
We leave it to the reader to check that $\psi$ respects the differential.

Finally, it is clear from the explicit description of $\psi$ that the composition \eqref{eq:compii} sends $a\otimes
t^k\otimes m^{-2l}$ to $a m^{-2l}$ if $k=l$ and to $0$ otherwise. Similarly, $a\otimes t^k\otimes m^{-2l-1}$ is sent to
$a m^{-2l-1}$ if $k=l$ and to $0$ otherwise. This shows that \eqref{eq:compii} equals the numerator in the roof
\eqref{eq:epsilonroof} defining $\veps_{F\ua\ld}$, finishing the proof of Theorem \ref{thm_bigtheorem}.
\end{proof}

\begin{rem}
The map $Q(K\ua_w,(\iota\circ\fold)(F\ua\ld))\to Q(K\ua_w,F\ua\ld)$ from the proof of Theorem \ref{thm_bigtheorem} is a
lift of the map we constructed in Remark \ref{rem_explicitcounit}.
\end{rem}

Theorem \ref{thm_bigtheorem} is also true in the finitely generated case, with the same proof.
\begin{definition}
We denote $\perf\subset\D^{b,b}_{\fr,\fg}(K\ua_w)$ the smallest thick triangulated subcategory of
$\D^{b,b}_{\fr,\fg}(K\ua_w)$ containing all finitely generated free $K\ua_w$-modules.
\end{definition}

\begin{theorem} Define $\fold$ and $\iota$ as in Proposition \ref{prop_fold} and \ref{prop_iota}. 
Then there is a natural isomorphism $$\veps:\ \iota\circ\fold\ \stackrel{\cong}{\Longrightarrow}\
\id_{\D^{b,b}_{\fr,\fg}(K\ua_w)/\perf}$$ which together with the equality $\fold\circ\iota =
\id_{\HMF(S\ld,w)}$ forms an adjoint equivalence
\begin{equation*}\begin{tikzpicture}[baseline=(current bounding  box.center), description/.style={fill=white,inner sep=2pt}]
    \matrix (m) [matrix of math nodes, row sep=3em,
                 column sep=2.5em, text height=1.5ex, text depth=0.25ex]
    {
       \fold:\ \D^{b,b}_{\fr,\fg}(K\ua_w)/\perf &&&& \HMF(S\ld,w):\ \iota\\
    };
    \draw[->] ($(m-1-1.east) + (0, 0.7mm)$) -- node[above,scale=0.75]{$\cong$}($(m-1-5.west) + (0, 0.7mm)$); 
    \draw[->] ($(m-1-5.west) + (0,-0.5mm)$) -- ($(m-1-1.east) + (0,-0.5mm)$);
\end{tikzpicture}\end{equation*}
\end{theorem}

Next we check that the equivalence $\fold: \D^{b,b}_{\fr,\fg}(K\ua_w)\cong\HMF(S\ld,w)$ indeed coincides with the composition
$$\D^{b,b}_{\fr,\fg}(K\ua_w)\cong\D^b_{\fg}(S\ld/(w))\cong\HMF(S\ld,w).$$

\begin{theorem}
Let $S\ld$ be a regular local graded ring and let $w\in\frm\setminus\{0\}$ be homogeneous of degree $d$. Then the
following diagram commutes up to natural isomorphism:
\begin{equation*}\begin{tikzpicture}[description/.style={fill=white,inner sep=2pt}]
    \matrix (m) [matrix of math nodes, row sep=3em,
                 column sep=2.5em, text height=1.5ex, text depth=0.25ex,
                 inner sep=0pt, nodes={inner xsep=0.3333em, inner ysep=0.3333em}]
    {
       \D^b_{\fg}(S\ld/(w))/\perf && \D^b_{\fg}(K\ua_w)/\perf \\ & \HMF(S\ld,w)\\
    };
    \draw[->] ($(m-1-3.west) + (0,1mm)$) -- node[above,scale=0.75]{$-\stackrel{\LL}{\otimes}_{K\ua_w} S\ld/(w)$}
    ($(m-1-1.east) + (0,1mm)$);
    \draw[->] ($(m-1-1.east) + (0,-0.5mm)$) -- node[below,scale=0.75]{$\VV$} ($(m-1-3.west) + (0,-0.5mm)$);
    \draw[->] ($(m-2-2.north) + (4mm,0)$) -- node[sloped,description,scale=0.75] {$\iota$} ($(m-1-3.south) + (-2mm,0)$);
    \draw[->] ($(m-1-3.south) + (4mm,0)$) -- node[sloped,description,scale=0.75]{$\fold$}($(m-2-2.north) + (1cm,0)$);
    \draw[->] (m-2-2) -- node[description,scale=0.75] {$\coker$}(m-1-1);
\end{tikzpicture}\end{equation*}
\end{theorem}
\begin{proof}
For a graded matrix factorization $M\ua\ld=\left(M^0\ld\mor{f} M^{-1}\ld\mor{g} M^0\ld\right)$ we have
\begin{align*}
\iota(M\ua\ld)\makebox[5mm][c]{$\stackrel[\mathclap{K\ua_w}]{\LL}{\otimes}$} S\ld/(w) & \makebox[8mm][c]{$\cong$}
Q(K\ua_w,\iota(M\ua\ld))\makebox[5mm][c]{$\stackrel[\mathclap{K\ua_w}]{}{\otimes}$} S\ld/(w)\\
& \makebox[8mm][c]{$\cong$} \left(...\mor{g} M^0\ld\langle -d\rangle\mor{f}
  M^{-1}\ld\mor{g} M^0\ld\to 0\to ...\right)\makebox[5mm][c]{$\stackrel[\mathclap{S\ld}]{}{\otimes}$} S\ld/(w)
\end{align*}
which is canonically isomorphic to $\coker(g)$ in $\D^b(S\ld/(w))$. This shows that
$$\left(-\stackrel{\LL}{\otimes}_{K\ua_w}S\ld/(w)\right)\circ\iota\ \ \cong\ \ \coker.$$
\end{proof}

We end this section with a funny description of the translation functor on $\D^{b,b}_{\fr}(K\ua_w)/\perf$ as swapping
the roles of $s$ and $\differential$ in $K\ua_w$:

\begin{cor}\label{cor_swapit}
For $F\ua\ld\in\D^{b,b}_{\fr}(K\ua_w)$  there is a canonical isomorphism in $\D^{b,b}_{\fr}(K\ua_w)/\perf$:
\begin{equation}\label{eq:funnyshift}\begin{tikzpicture}[description/.style={fill=white,inner sep=2pt},baseline=(m-1-1.base)]
    \matrix (m) [matrix of math nodes, row sep=3em,
                 column sep=1.3em, text height=1.5ex, text depth=0.25ex,
                 inner sep=0pt, nodes={inner xsep=0.3333em, inner ysep=0.3333em}]
    {
      ... & F^{-1}\ld & F^0\ld & F^1\ld & ... & \cong & ... & F^1\ld\langle -d\rangle & F^0\ld & F^{-1}\ld\langle d\rangle &
      ...\\
    };
    \node[above=4mm,scale=0.6] at (m-1-7) {$-2$};
    \node[above=4mm,scale=0.6] at (m-1-8) {$-1$};
    \node[above=4mm,scale=0.6] at (m-1-9) {$0$};
    \node[above=4mm,scale=0.6] at (m-1-10) {$1$};
    \node[above=4mm,scale=0.6] at (m-1-11) {$2$};

    \node[above=4mm,scale=0.6] at (m-1-1) {$-2$};
    \node[above=4mm,scale=0.6] at (m-1-2) {$-1$};
    \node[above=4mm,scale=0.6] at (m-1-3) {$0$};
    \node[above=4mm,scale=0.6] at (m-1-4) {$1$};
    \node[above=4mm,scale=0.6] at (m-1-5) {$2$};
    
    \draw[->] ($(m-1-1.east) + (0,0.5mm)$) -- node[above,scale=0.75]{$\differential$} ($(m-1-2.west) + (0,0.5mm)$);
    \draw[->] ($(m-1-2.west) + (0,-0.5mm)$) -- node[below,scale=0.75]{$s$} ($(m-1-1.east) + (0,-0.5mm)$);

    \draw[->] ($(m-1-2.east) + (0,0.5mm)$) -- node[above,scale=0.75]{$\differential$} ($(m-1-3.west) + (0,0.5mm)$);
    \draw[->] ($(m-1-3.west) + (0,-0.5mm)$) -- node[below,scale=0.75]{$s$} ($(m-1-2.east) + (0,-0.5mm)$);

    \draw[->] ($(m-1-3.east) + (0,0.5mm)$) -- node[above,scale=0.75]{$\differential$} ($(m-1-4.west) + (0,0.5mm)$);
    \draw[->] ($(m-1-4.west) + (0,-0.5mm)$) -- node[below,scale=0.75]{$s$} ($(m-1-3.east) + (0,-0.5mm)$);

    \draw[->] ($(m-1-4.east) + (0,0.5mm)$) -- node[above,scale=0.75]{$\differential$} ($(m-1-5.west) + (0,0.5mm)$);
    \draw[->] ($(m-1-5.west) + (0,-0.5mm)$) -- node[below,scale=0.75]{$s$} ($(m-1-4.east) + (0,-0.5mm)$);

    \draw[->] ($(m-1-7.east) + (0,0.5mm)$) -- node[above,scale=0.75]{$s$} ($(m-1-8.west) + (0,0.5mm)$);
    \draw[->] ($(m-1-8.west) + (0,-0.5mm)$) -- node[below,scale=0.75]{$\differential$} ($(m-1-7.east) + (0,-0.5mm)$);

    \draw[->] ($(m-1-8.east) + (0,0.5mm)$) -- node[above,scale=0.75]{$s$} ($(m-1-9.west) + (0,0.5mm)$);
    \draw[->] ($(m-1-9.west) + (0,-0.5mm)$) -- node[below,scale=0.75]{$\differential$} ($(m-1-8.east) + (0,-0.5mm)$);

    \draw[->] ($(m-1-9.east) + (0,0.5mm)$) -- node[above,scale=0.75]{$s$} ($(m-1-10.west) + (0,0.5mm)$);
    \draw[->] ($(m-1-10.west) + (0,-0.5mm)$) -- node[below,scale=0.75]{$\differential$} ($(m-1-9.east) + (0,-0.5mm)$);

    \draw[->] ($(m-1-10.east) + (0,0.5mm)$) -- node[above,scale=0.75]{$s$} ($(m-1-11.west) + (0,0.5mm)$);
    \draw[->] ($(m-1-11.west) + (0,-0.5mm)$) -- node[below,scale=0.75]{$\differential$} ($(m-1-10.east) + (0,-0.5mm)$);
\end{tikzpicture}\end{equation}
In particular, there is a natural isomorphism in $\D^{b,b}_{\fr}(K\ua_w)/\Perf$:
\begin{equation*}\begin{tikzpicture}[description/.style={fill=white,inner sep=2pt}]
    \matrix (m) [matrix of math nodes, row sep=3em,
                 column sep=1.75em, text height=1.5ex, text depth=0.25ex,
                 inner sep=0pt, nodes={inner xsep=0.3333em, inner ysep=0.3333em}]
    {
      F\ua\ld[1] & \cong & ... & F^2\ld\langle -d\rangle & F^1\ld & F^0\ld\langle d\rangle & F^{-1}\ld\langle 2d\rangle &
      ...\\
    };

    \node[above=4mm,scale=0.6] at (m-1-3) {$-2$};
    \node[above=4mm,scale=0.6] at (m-1-4) {$-1$};
    \node[above=4mm,scale=0.6] at (m-1-5) {$0$};
    \node[above=4mm,scale=0.6] at (m-1-6) {$1$};
    \node[above=4mm,scale=0.6] at (m-1-7) {$2$};
    \node[above=4mm,scale=0.6] at (m-1-8) {$3$};

    \draw[->] ($(m-1-3.east) + (0,0.5mm)$) -- node[above,scale=0.75]{$s$} ($(m-1-4.west) + (0,0.5mm)$);
    \draw[->] ($(m-1-4.west) + (0,-0.5mm)$) -- node[below,scale=0.75]{$\differential$} ($(m-1-3.east) + (0,-0.5mm)$);

    \draw[->] ($(m-1-4.east) + (0,0.5mm)$) -- node[above,scale=0.75]{$s$} ($(m-1-5.west) + (0,0.5mm)$);
    \draw[->] ($(m-1-5.west) + (0,-0.5mm)$) -- node[below,scale=0.75]{$\differential$} ($(m-1-4.east) + (0,-0.5mm)$);

    \draw[->] ($(m-1-5.east) + (0,0.5mm)$) -- node[above,scale=0.75]{$s$} ($(m-1-6.west) + (0,0.5mm)$);
    \draw[->] ($(m-1-6.west) + (0,-0.5mm)$) -- node[below,scale=0.75]{$\differential$} ($(m-1-5.east) + (0,-0.5mm)$);

    \draw[->] ($(m-1-6.east) + (0,0.5mm)$) -- node[above,scale=0.75]{$s$} ($(m-1-7.west) + (0,0.5mm)$);
    \draw[->] ($(m-1-7.west) + (0,-0.5mm)$) -- node[below,scale=0.75]{$\differential$} ($(m-1-6.east) + (0,-0.5mm)$);

    \draw[->] ($(m-1-7.east) + (0,0.5mm)$) -- node[above,scale=0.75]{$s$} ($(m-1-8.west) + (0,0.5mm)$);
    \draw[->] ($(m-1-8.west) + (0,-0.5mm)$) -- node[below,scale=0.75]{$\differential$} ($(m-1-7.east) + (0,-0.5mm)$);
\end{tikzpicture}\end{equation*}
Here the small numbers above an expression indicate its cohomological degree.
\end{cor}
\begin{proof}
This follows from Theorem \ref{thm_bigtheorem} together with the fact that the foldings of both sides in
\eqref{eq:funnyshift} are the same. 
\end{proof}

\subsection{Derived tensor products}\label{app_derivedtensorproducts}
\markright{\ref{app_derivedtensorproducts} Derived tensor products}

In this section we introduce and study the derived tensor product functor for modules over the Koszul-resolution
$K\ua_w$. We will see that this tensor product is compatible with the tensor product on matrix factorizations (Theorem
\ref{thm_compatible}), yielding a generalization of the statement about the compatibility of the stabilization functor
with tensor products of matrix factorizations from Section \ref{subsec_stabilization} (see Proposition
\ref{prop_tensorstabgeneral}).

\begin{definition}\label{def_derivedtensor}
The \textit{derived tensor product}
$$\D(K\ua_w)\times\D(K\ua_{w\p})\xrightarrow{\quad-\stackrel{\LL}{\otimes}_{S\ld}-\quad}\D(K\ua_w\otimes_{S\ld} 
K\ua_{w\p})$$ is defined as the composition
\begin{equation*}\begin{tikzpicture}[description/.style={fill=white,inner sep=2pt}]
    \matrix (m) [matrix of math nodes, row sep=3em,
                 column sep=2.5em, text height=1.5ex, text depth=0.25ex]
    {
       \D(K\ua_w)\times\D(K\ua_{w\p}) & \Ho(\Cof(K\ua_w))\times\Ho(\Cof(K\ua_{w\p}))\\
       \D(K\ua_w\otimes_{S\ld} K\ua_{w\p}) & \Ho(K\ua_w\otimes_{S\ld} K\ua_{w\p})\\
    };
    \draw[->] (m-1-1) -- node[scale=0.75,above]{$\cong$} (m-1-2);
    \draw[->] (m-1-2) -- node[scale=0.75,right]{$-\otimes_{S\ld}-$} (m-2-2);
    \draw[->] (m-2-2) -- (m-2-1);
    \draw[dashed, ->] (m-1-1) -- node[scale=0.75,left]{$-\stackrel{\LL}{\otimes}_{S\ld}-$}(m-2-1);
\end{tikzpicture}\end{equation*}
\end{definition}
\begin{rem}\label{rem_arbitraryresolution}
The derived tensor product depends on the choice of a quasi-inverse to the canonical
equivalence $$\Ho(\Cof(K\ua_{w^{(\prime)}}))\ \ \xrightarrow{\quad\cong\quad} \D(K\ua_{w^{(\prime)}}),$$ but any two
choices yield canonically isomorphic tensor products. In particular, by Fact \ref{fact_semifreecofibrant}, the derived
tensor product can be computed via semi-free resolutions. 
\end{rem}
Although the derived tensor product in Definition \ref{def_derivedtensor} is defined through cofibrant
resolutions, the derived tensor product of two bounded above $K\ua_{w^{(\prime)}}$-modules can actually be computed
through $S\ld$-free resolutions: 
\begin{fact}\label{fact_simpletensor}
For $M\ua\ld\in\D^{-,-}_{\fr}(K\ua_w)$, $N\ua\ld\in\D^{-,-}_{\fr}(K\ua_{w\p})$ there is a canonical isomorphism in
$\D(K\ua_w\otimes_{S\ld} K\ua_{w\p})$ 
$$M\ua\ld\stackrel{\LL}{\otimes}_{S\ld} N\ua\ld\quad\cong\quad M\ua\ld\otimes_{S\ld} N\ua\ld.$$
\end{fact}
\begin{proof}
If $p: QM\ua\ld\to M\ua\ld$ and $q: QN\ua\ld\to N\ua\ld$ denote bounded above semi-free replacements of $M\ua\ld$ and
$N\ua\ld$, respectively, both $p$ and $q$ are homotopy equivalences of complexes of $S\ld$-modules, since all complexes
involved are bounded above and $S\ld$-free. As the class of homotopy equivalences is stable under tensoring with
arbitrary complexes, we get canonical isomorphisms in $\D(K\ua_w\otimes_{S\ld} K\ua_{w\p})$
$$M\ua\ld\stackrel{\LL}{\otimes}_{S\ld} N\ua\ld\ \ \stackrel{\text{Def.}}{=}\ \ QM\ua\ld\otimes_{S\ld} QN\ua\ld\ \
\stackrel[\cong]{p\otimes\id}{\longrightarrow}\ \ M\ua\ld\otimes_{S\ld} QN\ua\ld\ \ \stackrel[\cong]{\id\otimes
  q}{\longrightarrow}\ \ M\ua\ld\otimes_{S\ld} N\ua\ld$$ as claimed.
\end{proof}

The derived tensor product on $\D(K\ua_{w^{(\prime)}})$ is also well behaved in the sense that is preserves complexes
with bounded cohomology; note that this is not true for the derived tensor product on $\D(S\ld/(w^{(\prime)}))$.
\begin{fact}\label{fact_preservesboundedness}
Let $M\ua\ld\in\D^b(K\ua_w)$ and $N\ua\ld\in\D^b(K\ua_{w\p})$. Then $M\ua\stackrel{\LL}{\otimes}_{S\ld}
N\ua\ld\in\D^b(K\ua_w\otimes_{S\ld} K\ua_{w\p})$. In other words, the dashed arrow in the following diagram exists:
\begin{equation*}\begin{tikzpicture}[description/.style={fill=white,inner sep=2pt}]
    \matrix (m) [matrix of math nodes, row sep=3em,
                 column sep=2.5em, text height=1.5ex, text depth=0.25ex]
    {
       \D^b(K\ua_w)\times\D^b(K\ua_{w\p}) && \D^b(K\ua_w\otimes_{S\ld} K\ua_{w\p}) \\
       \D(K\ua_w)\times\D(K\ua_{w\p}) && \D(K\ua_w\otimes_{S\ld} K\ua_{w\p}) \\
    };
    \draw[dashed,->] (m-1-1) -- node[scale=0.75,above]{$-\stackrel{\LL}{\otimes}_{S\ld}-$}(m-1-3);
    \draw[->] (m-2-1) -- node[scale=0.75,above]{$-\stackrel{\LL}{\otimes}_{S\ld}-$}(m-2-3);
    \draw[right hook->] (m-1-1) -- (m-2-1);
    \draw[right hook->] (m-1-3) -- (m-2-3); 
\end{tikzpicture}\end{equation*}
\end{fact}
\begin{proof}
By Fact \ref{fact_truncation} and Proposition \ref{prop_freetruncation} the inclusion $\D^{b,b}_{\fr}(K\ua_{w^{(\prime)}})\to
\D^{b}(K\ua_{w^{(\prime)}})$ is an equivalence. Hence we may without loss of generality assume that $M\ua\ld$ and
$N\ua\ld$ are bounded and $S\ld$-free, and in this case Fact \ref{fact_simpletensor} yields a canonical isomorphism
$M\ua\ld\stackrel{\LL}{\otimes}_{S\ld} N\ua\ld\cong M\ua\ld\otimes_{S\ld} 
N\ua\ld$. As $M\ua\ld\otimes_{S\ld} N\ua\ld$ is bounded, the claim follows.
\end{proof}

\begin{fact}\label{fact_canmorphkoszul}
For homogeneous $w, w\p\in S\ld$ of degree $d$, there is a canonical morphism of
dg-$S\ld$-algebras 
\begin{align*}
K\ua_{w+w\p}&\ \longrightarrow\ K\ua_w\otimes_{S\ld} K\ua_{w\p}\\
s&\ \longmapsto s\otimes 1 + 1\otimes s
\end{align*}
\end{fact}
\begin{proof}
The element $s\otimes 1 + 1\otimes s$ in $K\ua_w\otimes_{S\ld} K\ua_{w\p}$ satisfies
$\d(s\otimes 1 + 1\otimes s) = w + w\p$ and $(s\otimes 1+1\otimes s)^2 = 0$, and hence $$K\ua_{w+w\p}\ni s\ \longmapsto
s\otimes 1+1\otimes s\in K\ua_w\otimes_{S\ld} K\ua_{w\p}$$ indeed extends uniquely 
to a morphism of dg-$S\ld$-algebras $K\ua_{w+w\p}\to K\ua_w\otimes_{S\ld} K\ua_{w\p}$.
\end{proof}
Concatenating the derived tensor
product $$\D(K\ua_w)\times\D(K\ua_{w\p})\xrightarrow{-\stackrel{\LL}{\otimes}_{S\ld}-}\D(K\ua_w\otimes_{S\ld} 
K\ua_{w\p})$$ with 
the functor $\D(K\ua_w\otimes_{S\ld} K\ua_{w\p})\to\D(K\ua_{w+w\p})$ induced by the morphism $K\ua_{w+w\p}\to
K\ua_w\otimes_{S\ld} K\ua_{w\p}$ from Fact \ref{fact_canmorphkoszul} yields another derived tensor product functor
$$-\stackrel{\LL}{\otimes}_{S\ld}-:\ \ \D(K\ua_w)\times\D(K\ua_{w\p})\ \xrightarrow{-\stackrel{\LL}{\otimes}_{S\ld}-}\
\D(K\ua_w\otimes_{S\ld} K\ua_{w\p})\ \longrightarrow\ \D(K\ua_{w+w\p}).$$
We will now see that this derived tensor product is compatible with the tensor product functor
$$\HMF^\infty(S\ld,w)\times\HMF^\infty(S\ld,w\p)\ \xrightarrow{-\otimes_{S\ld}-}\HMF^\infty(S\ld,w+w\p)$$ with respect
to the canonical functor $$\D^b(K\ua_w)\ \cong\ \D^{b,b}_{\fr}(K\ua_w)\ 
\xrightarrow{\ \can\ }\ \D^{b,b}_{\fr}(K\ua_w)/\perf\ \stackrel{\fold}{\longrightarrow}\
\HMF^\infty(S\ld,w)$$ constructed in Section \ref{sec_bigtheorem} 
(see Theorem \ref{thm_bigtheorem}).

\begin{theorem}\label{thm_compatible}
Let $S\ld$ be a regular local graded ring and let $w,w\p\in S\ld$ be homogeneous of degree $d$. Then the diagram
\begin{equation*}\begin{tikzpicture}[description/.style={fill=white,inner sep=2pt}]
    \matrix (m) [matrix of math nodes, row sep=2.5em,
                 column sep=2.5em, text height=1.5ex, text depth=0.25ex,
                 inner sep=0pt, nodes={inner xsep=0.3333em, inner ysep=0.3333em}]
    {
       \D^b(K\ua_w)\times\D^b(K\ua_{w\p}) && \D^b(K\ua_{w+w\p})\\
       \HMF^\infty(S\ld,w)\times\HMF^\infty(S\ld,w\p) && \HMF^\infty(S\ld,w+w\p)\\
    };
    \draw[->] (m-1-1) -- node[above,scale=0.75]{$-\stackrel{\LL}{\otimes}_{S\ld}-$} (m-1-3);
    \draw[->] (m-2-1) -- node[above,scale=0.75]{$-\otimes_{S\ld}-$} (m-2-3);
    \draw[->] (m-1-1) -- (m-2-1);
    \draw[->] (m-1-3) -- (m-2-3);
\end{tikzpicture}\end{equation*}
commutes up to natural isomorphism.
\end{theorem}
\begin{proof}
By Fact \ref{fact_simpletensor} the upper square in the following diagram is commutative
\begin{equation*}\begin{tikzpicture}[description/.style={fill=white,inner sep=2pt}]
    \matrix (m) [matrix of math nodes, row sep=3em,
                 column sep=2.5em, text height=1.5ex, text depth=0.25ex,
                 inner sep=0pt, nodes={inner xsep=0.3333em, inner ysep=0.3333em}]
    {
       \D^b(K\ua_w)\times\D^b(K\ua_{w\p}) && \D^b(K\ua_{w+w\p})\\
       \D^{b,b}_{\fr}(K\ua_w)\times\D^{b,b}_{\fr}(K\ua_{w\p}) && \D^{b,b}_{\fr}(K\ua_{w+w\p})\\
       \HMF^\infty(S\ld,w)\times\HMF^\infty(S\ld,w\p) && \HMF^\infty(S\ld,w+w\p)\\
    };
    \draw[->] (m-1-1) -- node[above,scale=0.75]{$-\stackrel{\LL}{\otimes}_{S\ld} -$} (m-1-3);
    \draw[->] (m-2-1) -- node[above,scale=0.75]{$-\otimes_{S\ld} -$} (m-2-3);
    \draw[->] (m-3-1) -- node[above,scale=0.75]{$-\otimes_{S\ld} -$} (m-3-3);

    \draw[->] (m-2-1) -- node[left,scale=0.75]{$\can$} node[right,scale=0.75]{$\cong$} (m-1-1);
    \draw[->] (m-2-3) -- node[right,scale=0.75]{$\can$} node[left,scale=0.75]{$\cong$} (m-1-3);

    \draw[->] (m-2-1) -- node[description,scale=0.75]{$\fold\times\fold$} (m-3-1);
    \draw[->] (m-2-3) -- node[description,scale=0.75]{$\fold$} (m-3-3);
\end{tikzpicture}\end{equation*}
Hence, to prove the theorem it suffices to prove that the lower square is commutative. This is the
same calculation as in Proposition \ref{prop_tensorstabgeneral}: If
$M\ua\ld\in\D^{b,b}_{\fr}(K\ua_w)$ and $N\ua\ld\in\D^{b,b}_{\fr}(K\ua_{w\p})$, we have
\begin{align*}
\makebox[8mm][c]{}& \fold(M\ua\ld)\stackrel[S\ld]{}{\otimes} \fold(N\ua\ld)\\
 \makebox[8mm][c]{$=$} &\left(M\ueven\ld\stackrel[S\ld]{}{\otimes}
  N\uodd\ld\oplus M\uodd\ld\stackrel[S\ld]{}{\otimes} N\ueven\ld\ \leftrightarrows\ M\ueven\ld\stackrel[S\ld]{}{\otimes}
  N\ueven\ld\oplus M\uodd\ld\stackrel[S\ld]{}{\otimes} N\uodd\ld\langle d\rangle\right)\\
 \makebox[8mm][c]{$=$} & \left(\left(M\ua\ld\otimes_{S\ld}
  N\ua\ld\right)\uodd\ \leftrightarrows\ \left(M\ua\ld\otimes_{S\ld} N\ua\ld\right)\ueven\right)\\ 
 \makebox[8mm][c]{$=$} & \fold\left(M\ua\ld\otimes_{S\ld} N\ua\ld\right)
\end{align*}
as claimed (we omit the details about differentials and internal grading).
\end{proof}

\begin{theorem}\label{thm_naturalmorphism}
Let $S\ld$ be a regular local graded ring, and let $w, w\p\in S\ld$ be homogeneous the same degree. Then, if $M\ld\in 
S\ld/(w)\Mod$ and $N\ld\in S\ld/(w\p)\Mod$, there is a canonical morphism in $\HMF^\infty(S\ld,w+w\p)$
\begin{align}
\label{eq:natmorph}M\ld\stab{w}\otimes_{S\ld} N\ld\stab{w\p}\ \longrightarrow\ \left(M\ld\otimes_{S\ld}
  N\ld\right)\stab{w+w\p}\end{align}
which is an isomorphism if $\tor^k_S(M\ld,N\ld)\ld = 0$ for all $k>0$. 
\end{theorem}
\begin{proof}
We have to compare the images of $(M\ld,N\ld)$ under the composed functors from the upper left to the lower right corner
in the following diagram:
\begin{equation*}\begin{tikzpicture}[description/.style={fill=white,inner sep=2pt}]
    \matrix (m) [matrix of math nodes, row sep=2.5em,
                 column sep=2.5em, text height=1.5ex, text depth=0.25ex,
                 inner sep=0pt, nodes={inner xsep=0.3333em, inner ysep=0.3333em}]
    {
       S\ld/(w)\Mod\times S\ld/(w\p)\Mod && S\ld/(w+w\p)\Mod\\
       \D^b(K\ua_w)\times\D^b(K\ua_{w\p}) && \D^b(K\ua_{w+w\p})\\
       \HMF^\infty(S\ld,w)\times\HMF^\infty(S\ld,w\p) && \HMF^\infty(S\ld,w+w\p)\\
    };
    \draw[->] (m-2-1) -- node[above,scale=0.75]{$-\stackrel{\LL}{\otimes}_{S\ld}-$} (m-2-3);
    \draw[->] (m-3-1) -- node[above,scale=0.75]{$-\otimes_{S\ld}-$} (m-3-3);
    \draw[->] (m-2-1) -- (m-3-1);
    \draw[->] (m-2-3) -- (m-3-3);
    \draw[->] (m-1-1) -- node[above,scale=0.75]{$-\otimes_{S\ld}-$} (m-1-3);
    \draw[->] (m-1-1) -- node[description,scale=0.75]{$\can$} (m-2-1);
    \draw[->] (m-1-3) -- node[description,scale=0.75]{$\can$} (m-2-3);
    \node[rotate=30] at (barycentric cs:m-1-3=0.5,m-2-1=0.5) {$\Longrightarrow$};
    \node at (barycentric cs:m-2-3=0.5,m-3-1=0.5) {$\circlearrowright$};
\end{tikzpicture}\end{equation*}
The lower square commutes by Theorem \ref{thm_compatible}, and the upper square admits a natural transformation as
indicated, given by the canonical morphism $M\ld\stackrel{\LL}{\otimes}_{S\ld} N\ld\to M\ld\otimes_{S\ld} N\ld$. This
morphism is an isomorphism if and only if $\tor^k_S(M\ld,N\ld)\ld=0$ for all $k>0$, and the claim follows.
\end{proof}

\begin{rem}
It is not clear to the author what the morphism \eqref{eq:natmorph} looks like explicitly, because for bounded
$S\ld$-free resolutions $Q M\ld\to M\ld$ and $QN\ld\to N\ld$ with square-zero nullhomotopy for the multiplication by $w$
and $w\p$, respectively, the morphism
$$QM\ld\otimes_{S\ld} QN\ld\ \cong\ M\ld\stackrel{\LL}{\otimes}_{S\ld} N\ld\ \longrightarrow\
M\ld\otimes_{S\ld} N\ld\ \cong\ Q(M\ld\otimes_{S\ld} N\ld)$$ in $\D(K\ua_{w+w\p})$ need not be representable by a
morphism of $K\ua_{w+w\p}$-modules $Q(M\ld)\otimes_{S\ld} Q(N\ld)\to Q(M\ld\otimes_{S\ld} N\ld)$, but only by a roof
between these modules. This is why we didn't succeed in constructing it directly in Section \ref{subsec:compatibility}. 
\end{rem}

\subsection{Duality for modules over the Koszul resolution}\label{sec_koszulduality}
\markright{\ref{sec_koszulduality} Duality for modules over the Koszul resolution}

We know from Theorem \ref{thm_bigtheorem} that $\fold: \D^{b,b}_{\fr,\fg}(K\ua_w)/\Perf\cong\HMF(S\ld,w)$ and from
Definition \ref{def_wdual} and Fact \ref{fact_wdualcompatible} that $\HMF(S\ld,w)$ admits a duality $(-)\uc$ compatible
with the usual duality $\hom_{S\ld/(w)}(-,S\ld/(w))\ld$ on $\uMCM(S\ld,w)$. It is therefore natural to ask what the duality
on $\D^{b,b}_{\fr,\fg}(K\ua_w) / \Perf$ obtained from pulling back $(-)\uc$ along $\fold$ looks like, and whether it
admits a lifting to a duality on $\D^{b,b}_{\fg}(K\ua_w)$. In this section we will see that such a lifting exists and is
given by component-wise dualizing a dg-module over $K\ua_w$ (for the precise definition, see Definition
\ref{def_koszuldual}). This coincides with the duality established by Frankild and J\o rgensen in
\cite{frankildjorgensen}; they defined the notion of a Gorenstein dg-algebra in terms of the existence of a
duality and established the Gorensteinness of Koszul algebras over Gorenstein local rings (for arbitrary sequences in
the maximal ideal).

\begin{definition}\label{def_koszuldual}
The \textit{dual} of a $K\ua_w$-module $M\ua\ld$, denoted $D(M)\ua\ld$, is the $K\ua_w$-module defined
by $$D(M)^n\ld := \hom_{S}(M^{-(n+1)}\ld,S\ld)\ld\langle -d\rangle.$$ Its differential is given by $\d(f) :=
(-1)^{n+1} f\circ\differential^{-(n+2)}_{M\ua\ld}$ for $f\in D(M)^n\ld$, and the action of $s$ is given by $s.f := (-1)^n f\circ
s$. This gives a contravariant endofunctor on the category of $K\ua_w$-modules.
\end{definition}

Definition \ref{def_koszuldual} coincides with the duality induced by $\hom_{K\ua_w}\left(-,K\ua_w\right)\ua\ld$:

\begin{fact}\label{fact_twodualities}
Let $M\ua\ld$ be a $K\ua_w$-module. Then there is a natural isomorphism of $K\ua_w$-modules
$$D(M)\ua\ld\ \ \cong\ \ \hom_{K\ua_w}(M\ua\ld,K\ua_w)\ua\ld.$$
\end{fact}
\begin{proof}
An element of $\hom_{K\ua_w}(M\ua\ld,K\ua_w)^n_k$ is given by a diagram
\begin{equation*}\begin{tikzpicture}[description/.style={fill=white,inner sep=2pt}]
    \matrix (m) [matrix of math nodes, row sep=3em,
                 column sep=2.5em, text height=1.5ex, text depth=0.25ex,
                 inner sep=0pt, nodes={inner xsep=0.3333em, inner ysep=0.3333em}]
    {
       ... & M^{-n-2}\ld & M^{-n-1}\ld & M^{-n}\ld & M^{-n+1}\ld & ...\\
       ... & 0 & S\ld\langle -d\rangle & S\ld & 0 & ... \\
    };
    \draw[->] (m-1-2) -- node[above,scale=0.75]{$s$} (m-1-1);
    \draw[->] (m-1-3) -- node[above,scale=0.75]{$s$} (m-1-2);
    \draw[->] (m-1-4) -- node[above,scale=0.75]{$s$} (m-1-3);
    \draw[->] (m-1-5) -- node[above,scale=0.75]{$s$} (m-1-4);
    \draw[->] (m-1-6) -- node[above,scale=0.75]{$s$} (m-1-5);

    \draw[->] (m-1-2) -- (m-2-2);
    \draw[->] (m-1-3) -- node[description,scale=0.75]{$\alpha$}(m-2-3);
    \draw[->] (m-1-4) -- node[description,scale=0.75]{$\beta$}(m-2-4);
    \draw[->] (m-1-5) -- (m-2-5);

    \draw[->] (m-2-2) -- (m-2-1);
    \draw[->] (m-2-3) -- (m-2-2);
    \draw[->] (m-2-4) -- node[above,scale=0.75]{$\id$} (m-2-3);
    \draw[->] (m-2-5) -- (m-2-4);
    \draw[->] (m-2-6) -- (m-2-5);    
\end{tikzpicture}\end{equation*}
where each vertical map raises the internal degree by $k$ and where each square commutes up to the sign
$(-1)^{n}$. This forces $\beta = (-1)^n \alpha\circ s$, while $\alpha$ can be chosen freely in
$\hom_{S}(M^{-n-1}\ld,S\ld)_k$. Hence we have a canonical isomorphism $$\hom_{K\ua_w}(M\ua\ld,K\ua_w)^n_k\ \ \cong\ \
D(M)^n_k,$$ and it is easily checked that this isomorphism is compatible with the differential and the action of
$K\ua_w$ on both sides (see  Definition \ref{def_homomorphismcomplex}).
\end{proof}

\begin{fact}\label{fact_dualityderived}
The duality functor $$\hom_{K\ua_w}(-,K\ua_w)\ua\ld = D: K\ua_w\Mod\to K\ua_w\Mod$$ takes quasi-isomorphisms between  
bounded above, $S\ld$-free $K\ua_w$-modules to quasi-isomorphisms. Therefore, its derived
functor $$\bfR\hom_{K\ua_w}(-,K\ua_w)\ua\ld: \D(K\ua_w)\longrightarrow \D(K\ua_w)$$ may be computed
naively on the subcategory $\D^{-,-}_{\fr,\fg}(K\ua_w)$, and the diagram
\begin{equation*}\begin{tikzpicture}[description/.style={fill=white,inner sep=2pt}]
    \matrix (m) [matrix of math nodes, row sep=3em,
                 column sep=7em, text height=1.5ex, text depth=0.25ex,
                 inner sep=0pt, nodes={inner xsep=0.3333em, inner ysep=0.3333em}]
    {
       \D^b_\fg(K\ua_w) & \D^b_\fg(K\ua_w)\\
       \D^{b,b}_{\fr,\fg}(K\ua_w) & \D^{b,b}_{\fr,\fg}(K\ua_w)\\
    };
    \draw[->] (m-1-1) -- node[above,scale=0.75]{$\bfR\hom_{K\ua_w}(-,K\ua_w)\ua\ld$} (m-1-2);
    \draw[->] (m-2-1) -- node[below,scale=0.75]{$D$} (m-2-2);
    \draw[->] (m-2-1) -- node[left,scale=0.75]{$\incl$} node[right,scale=0.75]{$\cong$} (m-1-1);
    \draw[->] (m-2-2) -- node[right,scale=0.75]{$\incl$} node[left,scale=0.75]{$\cong$} (m-1-2);
\end{tikzpicture}\end{equation*}
is well-defined and commutative up to natural isomorphism.
\end{fact}
\begin{proof}
It is clear that $D\circ [1]\cong [-1]\circ D$ and that $D(\cone(\alpha))\cong\cone(D(\alpha))$ for a morphism $\alpha$ of
$K\ua_w$-modules. Thus, to prove the first statement we only have to check that the dual of an acyclic, bounded above and
$S\ld$-free is acyclic. However, such a module is contractible as a complex of $S\ld$-modules, and so is its
component-wise $S\ld$-dual. This proves the first statement, and the second statement is an immediate
consequence.
\end{proof}

\begin{prop}\label{prop_compkoszulhmf}
The following diagram commutes up to natural isomorphism:
\begin{equation*}\begin{tikzpicture}[description/.style={fill=white,inner sep=2pt}]
    \matrix (m) [matrix of math nodes, row sep=3em,
                 column sep=2.5em, text height=1.5ex, text depth=0.25ex,
                 inner sep=0pt, nodes={inner xsep=0.3333em, inner ysep=0.3333em}]
    {
       \D^{b,b}_{\fr,\fg}(K\ua_w) && \D^{b,b}_{\fr,\fr}(K\ua_w)\\
       \HMF(S\ld,w) && \HMF(S\ld,w)\\
    };
    \draw[->] (m-1-1) -- node[scale=0.75,above]{$D$} (m-1-3);
    \draw[->] (m-2-1) -- node[scale=0.75,below]{$(-1)\uc$} (m-2-3);

    \draw[->] (m-1-1) -- node[left,scale=0.75]{$\fold$} (m-2-1);
    \draw[->] (m-1-3) -- node[right,scale=0.75]{$\fold$} (m-2-3);
\end{tikzpicture}\end{equation*}
\end{prop}
\begin{proof}
For $F\ua\ld\in\D^{b,b}_{\fr,\fg}(K\ua_w)$ we have
\begin{align*}
\left(\fold(F\ua\ld)\right)\uc & = \left(\bigoplus\limits_{n\in\ZZ} F^{2n}\ld\langle
  -nd\rangle\longrightarrow\bigoplus\limits_{n\in\ZZ} F^{2n-1}\ld\langle -nd\rangle\longrightarrow
  \bigoplus\limits_{n\in\ZZ} F^{2n}\ld\langle -nd\rangle\right)\uc\\
& = \left(\bigoplus\limits_{n\in\ZZ} \left(F^{2n-1}\ld\right)\us\langle
  (n-1)d\rangle\longrightarrow\bigoplus\limits_{n\in\ZZ} \left(F^{2n}\ld\right)\us\langle
  (n-1)d\rangle\longrightarrow\bigoplus\limits_{k\in\ZZ} \left(F^{2n-1}\ld\right)\us\langle (n-1)d\rangle\right)\\
& = \left(\bigoplus\limits_{n\in\ZZ} D(F\ua\ld)^{2n}\ld\langle -nd\rangle \longrightarrow \bigoplus\limits_{n\in\ZZ}
  D(F\ua\ld)^{2n-1}\ld\langle -nd\rangle \longrightarrow \bigoplus\limits_{n\in\ZZ} D(F\ua\ld)^{2n}\ld\langle
  -nd\rangle\right)\\
& = \fold\left(D(F\ua\ld)\right).
\end{align*}
\end{proof}

\begin{prop}\label{prop_compderivkoszul} 
For $w\neq 0$, the following diagram commutes up to natural isomorphism:
\begin{equation*}\begin{tikzpicture}[description/.style={fill=white,inner sep=2pt}]
    \matrix (m) [matrix of math nodes, row sep=3em,
                 column sep=5em, text height=1.5ex, text depth=0.25ex,
                 inner sep=0pt, nodes={inner xsep=0.3333em, inner ysep=0.3333em}]
    {
       \D^b_\fg(S\ld/(w)) && \D^b_\fg(S\ld/(w))\\
       \D^{b,b}_{\fr,\fg}(K\ua_w) && \D^{b,b}_{\fr,\fr}(K\ua_w)\\
    };
    \draw[->] (m-2-1) -- node[scale=0.75,below]{$D$} (m-2-3);
    \draw[->] (m-1-1) -- node[scale=0.75,above]{$\bfR \hom_{S\ld/(w)}(-,S\ld/(w))\ua\ld$} (m-1-3);

    \draw[->] (m-2-1) -- node[right,scale=0.75]{$\cong$} node[left,scale=0.75]  {$-\stackrel{\LL}{\otimes}_{K\ua_w}
      S\ld/(w)$} (m-1-1); 
    \draw[->] (m-2-3) -- node[left,scale=0.75]{$\cong$} node[right,scale=0.75] {$-\stackrel{\LL}{\otimes}_{K\ua_w}
      S\ld/(w)$} (m-1-3); 
\end{tikzpicture}\end{equation*}
\end{prop}

\begin{proof}
By Fact \ref{fact_dualityderived} it suffices to show that the diagram
\begin{equation*}\begin{tikzpicture}[description/.style={fill=white,inner sep=2pt}]
    \matrix (m) [matrix of math nodes, row sep=3em,
                 column sep=5em, text height=1.5ex, text depth=0.25ex,
                 inner sep=0pt, nodes={inner xsep=0.3333em, inner ysep=0.3333em}]
    {
       \D^b_\fg(S\ld/(w)) && \D^b_\fg(S\ld/(w))\\
       \D^b_{\fg}(K\ua_w) && \D^b_{\fg}(K\ua_w)\\
    };
    \draw[->] (m-2-1) -- node[scale=0.75,below]{$\bfR\hom_{K\ua_w}(-,K\ua_w)\ua\ld$} (m-2-3);
    \draw[->] (m-1-1) -- node[scale=0.75,above]{$\bfR \hom_{S\ld/(w)}(-,S\ld/(w))\ua\ld$} (m-1-3);

    \draw[->] (m-2-1) -- node[right,scale=0.75]{$\cong$} node[left,scale=0.75]  {$-\stackrel{\LL}{\otimes}_{K\ua_w}
      S\ld/(w)$} (m-1-1); 
    \draw[->] (m-2-3) -- node[left,scale=0.75]{$\cong$} node[right,scale=0.75] {$-\stackrel{\LL}{\otimes}_{K\ua_w}
      S\ld/(w)$} (m-1-3); 
\end{tikzpicture}\end{equation*}
commutes up to natural isomorphism. To prove this, we note that for $M\ua\ld\in\D(K\ua_w)$ we have a natural isomorphism
in $\D(S\ld/(w))$
\begin{align}\label{eq:duality1}
\bfR\hom_{S\ld/(w)}\left(M\ua\ld\stackrel{\LL}{\otimes}_{K\ua_w} S\ld/(w),S\ld/(w)\right)\ua\ld & \cong
\bfR\hom_{K\ua_w}\left(M\ua\ld,\ _{K\ua_w} S\ld/(w)\right)\ua\ld.
\end{align}
Further, we have a natural isomorphism in $\D(K\ua_w)$: 
\begin{align}\label{eq:duality2}
\bfR\hom_{K\ua_w}\left(M\ua\ld,_{K\ua_w} S\ld/(w)\right)\ua\ld\ \cong\ \bfR\hom_{K\ua_w}\left(M\ua\ld,
  K\ua_w\right)\ua\ld\end{align}
Putting \eqref{eq:duality1} and \eqref{eq:duality2} together, we see that the composition
$$\D^b_{\fg}(K\ua_w)\xrightarrow{-\stackrel{\LL}{\otimes}_{K\ua_w}S\ld/(w)}
\D^b_{\fg}(S\ld/(w))\xrightarrow{\bfR\hom_{S\ld/(w)}(-,S\ld/(w))\ua\ld}
\D^b_\fg(S\ld/(w))\longrightarrow\D^b_{\fg}(K\ua_w)$$ is naturally isomorphic to
$$\D^b_{\fg}(K\ua_w)\xrightarrow{\ \bfR\hom_{K\ua_w}(-,K\ua_w)\ua\ld}\D^b_\fg(K\ua_w)$$
so the claim follows.
\end{proof}

Together, Propositions \ref{prop_compkoszulhmf} and \ref{prop_compderivkoszul} yield the following compatibility of
stabilization and duality:

\begin{theorem}\label{thm_compduality}
For $w\neq 0$ the following diagram commutes up to natural isomorphism:
\begin{equation*}\begin{tikzpicture}[description/.style={fill=white,inner sep=2pt}]
    \matrix (m) [matrix of math nodes, row sep=3em,
                 column sep=5em, text height=1.5ex, text depth=0.25ex,
                 inner sep=0pt, nodes={inner xsep=0.3333em, inner ysep=0.3333em}]
    {
       \D^b_\fg(S\ld/(w)) && \D^b_\fg(S\ld/(w))\\
       \HMF(S\ld,w) && \HMF(S\ld,w)\\
    };
    \draw[->] (m-1-1) -- node[above,scale=0.75]{$\bfR\hom_{S\ld/(w)}(-,S\ld/(w))\ua\ld$} (m-1-3);
    \draw[->] (m-2-1) -- node[below,scale=0.75]{$(-1)\uc$} (m-2-3);

    \draw[->] (m-1-1) -- node[left,scale=0.75]{stab} (m-2-1);
    \draw[->] (m-1-3) -- node[right,scale=0.75]{stab} (m-2-3);
\end{tikzpicture}\end{equation*}
\end{theorem}

\begin{rem} Theorem \ref{thm_compduality} generalizes Proposition \ref{prop_wdual} for if $M\ld$ is Cohen-Macaulay 
  module over $S\ld$ with defect $n := \dim(S\ld) - \depth(M\ld)$, it is also Cohen-Macaulay over $S\ld/(w)$ with defect
  $n-1$, and hence we have $$\bfR\hom_{S\ld/(w)}(M\ld,S\ld/(w))\ua\ld\
  \cong\ \ext_{S\ld/(w)}^{n-1}(M\ld,S\ld/(w))\ld [-n+1]\ \cong\ \ext_{S\ld}^n(M\ld,S\ld)\ld[-n+1]$$ as claimed (for the
  last isomorphism, see \cite[Lemma 3.1.16]{BrunsHerzog})). Thus, the Cohen-Macaulay modules over $S\ld$ enter because
  they are precisely the modules $M\ld$ for which $\bfR\hom_{S\ld/(w)}(M\ld,S\ld/(w))\ua\ld$ is concentrated in a single
  degree. 
\end{rem}

\end{appendix}

\newpage

%%%%%%%%%%%%%%% REFERENZEN %%%%%%%%%%%%%%%%%%

\bibliographystyle{amsalphaeprint}
\addcontentsline{toc}{section}{References}
\bibliography{referenzen}

%%%%%%%%%%%%%%%%%%%%%%%%%%%%%%%%%%%%%%%%%%%%%

\end{document}